\documentclass[a4paper]{amsart}

\usepackage{enumerate, amsmath, amsfonts, amssymb, amsthm, wasysym, graphics, graphicx, xcolor, url, hyperref, hypcap, a4wide, stmaryrd, shuffle, xargs, multicol, overpic, pdflscape, multirow, hvfloat, minibox, accents}
\usepackage{hyperxmp}
\hypersetup{pdfproducer={Vincent Pilaud}}
\hypersetup{colorlinks=true, citecolor=darkblue, linkcolor=darkblue, urlcolor=darkblue}
\usepackage[all]{xy}
\usepackage[bottom]{footmisc}
\usepackage{tikz}\usetikzlibrary{trees,snakes,shapes,arrows,matrix,calc}
\graphicspath{{figures/}}
\makeatletter
\def\input@path{{figures/}}
\makeatother

\title{Cambrian Hopf Algebras}

\thanks{VP was partially supported by MICINN grant MTM2011-22792, and ANR grant EGOS~(12\,JS02\,002\,01)}

\author{Gr\'egory Chatel}
\author{Vincent Pilaud}

\address[GC]{LIGM, Univ.\,Paris-Est Marne-la-Vall\'ee}
\email{gregory.chatel@univ-paris-est.fr}

\address[VP]{CNRS \& LIX, \'Ecole Polytechnique, Palaiseau}
\email{vincent.pilaud@lix.polytechnique.fr}
\urladdr{http://www.lix.polytechnique.fr/~pilaud/}

\newtheorem{theorem}{Theorem}
\newtheorem{corollary}[theorem]{Corollary}
\newtheorem{proposition}[theorem]{Proposition}
\newtheorem{lemma}[theorem]{Lemma}
\newtheorem{definition}[theorem]{Definition}

\theoremstyle{definition}
\newtheorem{example}[theorem]{Example}
\newtheorem{remark}[theorem]{Remark}

\newcommand{\R}{\mathbb{R}} 
\newcommand{\N}{\mathbb{N}} 
\newcommand{\Z}{\mathbb{Z}} 
\newcommand{\C}{\mathbb{C}} 
\newcommand{\HH}{\mathbb{H}} 
\newcommand{\fS}{\mathfrak{S}} 
\newcommand{\fP}{\mathfrak{P}} 
\newcommand{\fX}{\mathfrak{X}} 
\renewcommand{\b}[1]{\mathbf{#1}} 

\newcommand{\set}[2]{\left\{ #1 \;\middle|\; #2 \right\}} 
\newcommand{\bigset}[2]{\big\{ #1 \;|\; #2 \big\}} 
\newcommand{\ssm}{\smallsetminus} 
\newcommand{\symdif}{\,\triangle\,} 
\newcommand{\eqdef}{\mbox{\,\raisebox{0.2ex}{\scriptsize\ensuremath{\mathrm:}}\ensuremath{=}\,}} 
\newcommand{\polar}{^\diamond} 
\renewcommand{\implies}{\Rightarrow} 
\newcommand{\sep}{|} 
\newcommand{\bsep}{\,|\,} 

\newcommand{\FQSym}{\mathsf{FQSym}} 
\newcommand{\Bax}{^{\upharpoonleft\!\!\downharpoonright}} 
\newcommand{\PBT}{\mathsf{PBT}} 
\newcommand{\Rec}{\mathsf{Rec}} 
\newcommand{\CambTrees}{\mathrm{Camb}} 
\newcommand{\SchrCambTrees}{\mathrm{SchrCamb}} 
\newcommand{\Camb}{\mathsf{Camb}} 
\newcommand{\BaxCamb}{\mathsf{BaxCamb}} 
\newcommand{\SchrCamb}{\mathsf{SchrCamb}} 
\newcommand{\OrdPart}{\mathsf{OrdPart}} 
\newcommand{\WQSym}{\mathsf{WQSym}} 
\newcommand{\NCQSym}{\mathsf{NCQSym}} 
\newcommand{\Tril}{\mathsf{Tril}} 

\newcommand{\product}{\cdot} 
\newcommand{\coproduct}{\triangle} 
\newcommand{\shiftedShuffle}{\,\bar\shuffle\,} 
\newcommand{\convolution}{\star} 
\newcommand{\mirror}[1]{\stackrel{\leftarrow}{#1}} 

\newcommand{\F}{\mathbb{F}} 
\newcommand{\G}{\mathbb{G}} 
\newcommand{\PCamb}{\mathbb{P}} 
\newcommand{\QCamb}{\mathbb{Q}} 
\newcommand{\HCamb}{\mathbb{H}} 
\newcommand{\ECamb}{\mathbb{E}} 
\newcommand{\PBax}{\mathbb{P}} 
\newcommand{\QBax}{\mathbb{Q}} 
\newcommand{\XRec}{\mathbb{X}} 
\newcommand{\indecomposables}{\mathrm{Ind}} 

\newcommand{\CambCorresp}{\Theta} 
\newcommand{\BaxCorresp}{\CambCorresp\Bax} 
\newcommand{\SchrCambCorresp}{\Theta^\star} 
\newcommand{\surjectionPermAsso}{\PSymbol} 
\newcommand{\surjectionSchrPermAsso}{\PSymbol^\star} 
\newcommand{\surjectionPermAssoBax}{\surjectionPermAsso\Bax} 
\newcommand{\surjectionAssoPara}{\mathbf{can}} 
\newcommand{\surjectionSchrAssoPara}{\mathbf{can}^\star} 
\newcommand{\surjectionPermPara}{\mathbf{rec}} 
\newcommand{\surjectionSchrPermPara}{\mathbf{rec}^\star} 
\newcommand{\PSymbol}{\mathbf{P}} 
\newcommand{\QSymbol}{\mathbf{Q}} 

\newcommandx{\graphG}[1][1=G]{\mathrm{#1}} 
\newcommandx{\tree}[1][1=T]{\mathrm{#1}} 
\newcommandx{\tuple}[1][1=T]{\mathcal{#1}} 
\newcommandx{\poset}{{\tree[P]}} 
\newcommand{\ground}{\mathrm{V}} 
\newcommand{\signature}{\varepsilon} 
\newcommand{\signatures}{\mathcal{E}} 
\newcommand{\psignature}{\signature_p} 
\newcommand{\vsignature}{\signature_v} 
\newcommand{\update}{\delta} 
\newcommand{\extension}{\vartriangleleft} 
\newcommand{\generatingTree}{\mathcal{T}} 
\newcommand{\level}{m} 
\newcommand{\BC}{B} 
\newcommand{\BCMat}{\mathbf{B}} 
\newcommand{\switch}{\mathsf{switch}} 
\newcommand{\up}[1]{\overline{#1}}
\newcommand{\upr}[1]{{\red \overline{#1}}}
\newcommand{\upb}[1]{{\darkblue \overline{#1}}}
\newcommand{\upw}[1]{\underaccent{\phantom{.}}{\up{#1}}}
\newcommand{\down}[1]{\underline{#1}}
\newcommand{\downr}[1]{{\red \underline{#1}}}
\newcommand{\downb}[1]{{\darkblue \underline{#1}}}
\newcommand{\downw}[1]{\accentset{\phantom{.}}{\down{#1}}}
\newcommand{\simtilde}[1]{\accentset{\vspace{-.01cm}\sim}{#1}}
\newcommand{\uptilde}[1]{\accentset{\vspace{-.03cm}\sim}{#1}}
\newcommand{\downtilde}[1]{\underaccent{\,\sim}{#1}}
\newcommand{\linearExtensions}{\mathcal{L}} 
\newcommand{\minimalLinearExtension}{\mu} 
\newcommand{\maximalLinearExtension}{\omega} 
\newcommand{\increasingTree}{\mathrm{IT}} 
\newcommand{\decreasingTree}{\mathrm{DT}} 
\newcommand{\Tex}{\includegraphics{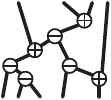}} 
\newcommand{\TexTwin}{\includegraphics{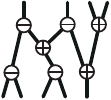}} 
\newcommand{\TexPair}{\includegraphics{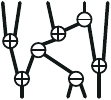}} 
\newcommand{\TexSchroder}{\includegraphics{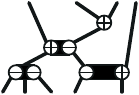}} 
\newcommand{\redMinus}{\raisebox{-.035cm}{\red$\boxminus$}} 
\newcommand{\redPlus}{\raisebox{-.035cm}{\red$\boxplus$}} 
\newcommand{\blueMinus}{{\darkblue\ominus}} 
\newcommand{\bluePlus}{{\darkblue\oplus}} 
\newcommand{\blueCirc}{{\text{\darkblue\tiny$\bigcirc$}}} 
\newcommand{\redSquare}{{\text{\raisebox{-.05cm}{\red$\square$}}}} 
\newcommand{\edgecut}[2]{\left( #1 \;\middle\|\; #2 \right)} 
\newcommand{\switchSign}{\chi} 
\newcommand{\unionOp}[2]{{#1 \!\! \upharpoonleft\!\!\downharpoonright \! #2}} 
\newcommand{\graphTwin}{\unionOp{\tree_\circ}{\tree_\bullet}}
\newcommand{\dash}{\text{-}}

\newcommandx{\Asso}[1][1=\signature]{\mathsf{Asso}(#1)} 
\newcommandx{\Perm}[1][1=n]{\mathsf{Perm}(#1)} 
\newcommandx{\Para}[1][1=n]{\mathsf{Para}(#1)} 
\newcommandx{\polygon}[1][1=\signature]{\mathrm{P}^{#1}} 
\newcommand{\Cone}{\mathrm{C}} 
\newcommandx{\BaxAsso}[1][1=\signature]{\mathsf{BaxAsso}(#1)} 

\DeclareMathOperator{\conv}{conv} 
\DeclareMathOperator{\cone}{cone} 
\DeclareMathOperator{\coinv}{coinv} 

\newcommand{\fref}[1]{Figure~\ref{#1}} 
\newcommand{\ie}{\textit{i.e.}~} 
\newcommand{\eg}{\textit{e.g.}~} 
\newcommand{\versus}{\textit{vs.}~} 
\newcommand{\perse}{\textit{per se}} 
\definecolor{darkblue}{rgb}{0,0,0.7} 
\definecolor{green}{RGB}{57,181,74} 
\newcommand{\darkblue}{\color{darkblue}} 
\newcommand{\green}{\color{green}} 
\newcommand{\red}{\color{red}} 
\newcommand{\defn}[1]{\emph{\darkblue #1}} 
\usepackage{todonotes}

\newcommand{\para}[1]{\medskip\noindent\framebox{\textsc{#1}}} 
\DeclareRobustCommand{\verylonghookrightarrow}{\lhook\joinrel\relbar\joinrel\relbar\joinrel\relbar\joinrel\relbar\joinrel\relbar\joinrel\relbar\joinrel\relbar\joinrel\relbar\joinrel\relbar\joinrel\relbar\joinrel\rightarrow}
\DeclareRobustCommand{\verylongtwoheadrightarrow}{\relbar\joinrel\relbar\joinrel\relbar\joinrel\relbar\joinrel\relbar\joinrel\relbar\joinrel\relbar\joinrel\relbar\joinrel\relbar\joinrel\relbar\joinrel\twoheadrightarrow}
\DeclareRobustCommand{\verylongleftrightarrow}{\leftarrow\joinrel\relbar\joinrel\relbar\joinrel\relbar\joinrel\relbar\joinrel\relbar\joinrel\rightarrow}
\DeclareRobustCommand{\verylongrightarrow}{\joinrel\relbar\joinrel\relbar\joinrel\relbar\joinrel\relbar\joinrel\relbar\joinrel\rightarrow}
\newcommand{\contractionPoset}{\subseteq} 
\newcommand{\SchrCambPoset}{<} 

\usetikzlibrary{positioning,arrows,matrix,chains,shadows,shapes}
\tikzstyle{Point} = [fill, radius=0.08]
\newcommand{\bannedGap}{}
\newcommand{\freeGap}{\textcolor{blue}{{\scalebox{.5}{$\bullet$}}}}

\makeatletter
\def\part{\@startsection{part}{1}%
\z@{.7\linespacing\@plus\linespacing}{.5\linespacing}%
{\LARGE\sffamily\centering}}
\@addtoreset{section}{part}
\makeatother

\makeatletter
\def\l@section{\@tocline{1}{3pt}{0pc}{}{}}
\makeatother
\let\oldtocpart=\tocpart
\renewcommand{\tocpart}[2]{\hspace{0em}\bf\large\oldtocpart{#1}{#2}}
\let\oldtocsection=\tocsection
\renewcommand{\tocsection}[2]{\hspace{0em}\bf\oldtocsection{#1}{#2}}

\begin{document}

\begin{abstract}
Cambrian trees are oriented and labeled trees which fulfill local conditions around each node generalizing the conditions for classical binary search trees. Based on the bijective correspondence between signed permutations and leveled Cambrian trees, we define the Cambrian Hopf algebra generalizing J.-L.~Loday and M.~Ronco's algebra on binary trees. We describe combinatorially the products and coproducts of both the Cambrian algebra and its dual in terms of operations on Cambrian trees. We also define multiplicative bases of the Cambrian algebra and study structural and combinatorial properties of their indecomposable elements. Finally, we extend to the Cambrian setting different algebras connected to binary trees, in particular S.~Law and N.~Reading's Baxter Hopf algebra on quadrangulations and S.~Giraudo's equivalent Hopf algebra on twin binary trees, and F.~Chapoton's Hopf algebra on all faces of the associahedron.
\end{abstract}

\maketitle

\vspace*{-.2cm}
\tableofcontents


\section*{Introduction}

The background of this paper is the fascinating interplay between the combinatorial, geometric and algebraic structures of permutations, binary trees and binary sequences (see Table~\ref{tab:structures}):
\begin{enumerate}[$\star$]
\item \textbf{Combinatorially}, the descent map from permutations to binary sequences factors via binary trees through the BST insertion and the canopy map. These maps define lattice homomorphisms from the weak order via the Tamari lattice to the boolean~lattice.
\item \textbf{Geometrically}, the permutahedron is contained in Loday's associahedron~\cite{Loday} which is in turn contained in the parallelepiped generated by the simple roots. These polytopes are just obtained by deleting inequalities from the facet description of the permutahedron. See~\fref{fig:permutahedronAssociahedronCube}.
\item \textbf{Algebraically}, these maps translate to Hopf algebra inclusions from M.~Malvenuto and C.~Reute\-nauer's algebra on permutations~\cite{MalvenutoReutenauer} via J.-L.~Loday and M.~Ronco's algebra on binary trees~\cite{LodayRonco} to L.~Solomon's descent algebra~\cite{Solomon}.
\end{enumerate}

\bigskip

\begin{table}[h]
	\begin{tabular}{c@{\quad}||@{\quad}c@{\quad}|@{\quad}c@{\quad}|@{\quad}c}
		Combinatorics & Permutations & Binary trees & Binary sequences \\[.2cm]
		\multirow{2}{*}{Geometry} & Permutahedron & Loday's & Parallelepiped \\
		& $\conv(\fS_n)$ & associahedron~\cite{Loday} & gen. by $e_{i+1}-e_i$\\[.2cm]
		\multirow{2}{*}{Algebra} & Malvenuto-Reutenauer & Loday-Ronco & Solomon \\
		& Hopf algebra~\cite{MalvenutoReutenauer} & Hopf algebra~\cite{LodayRonco} & descent algebra~\cite{Solomon}
	\end{tabular}
	\vspace{-.1cm}
	\caption{Related combinatorial, geometric and algebraic structures.}
	\vspace{-.3cm}
	\label{tab:structures}
\end{table}

\begin{figure}[h]
  \centerline{\includegraphics[width=\textwidth]{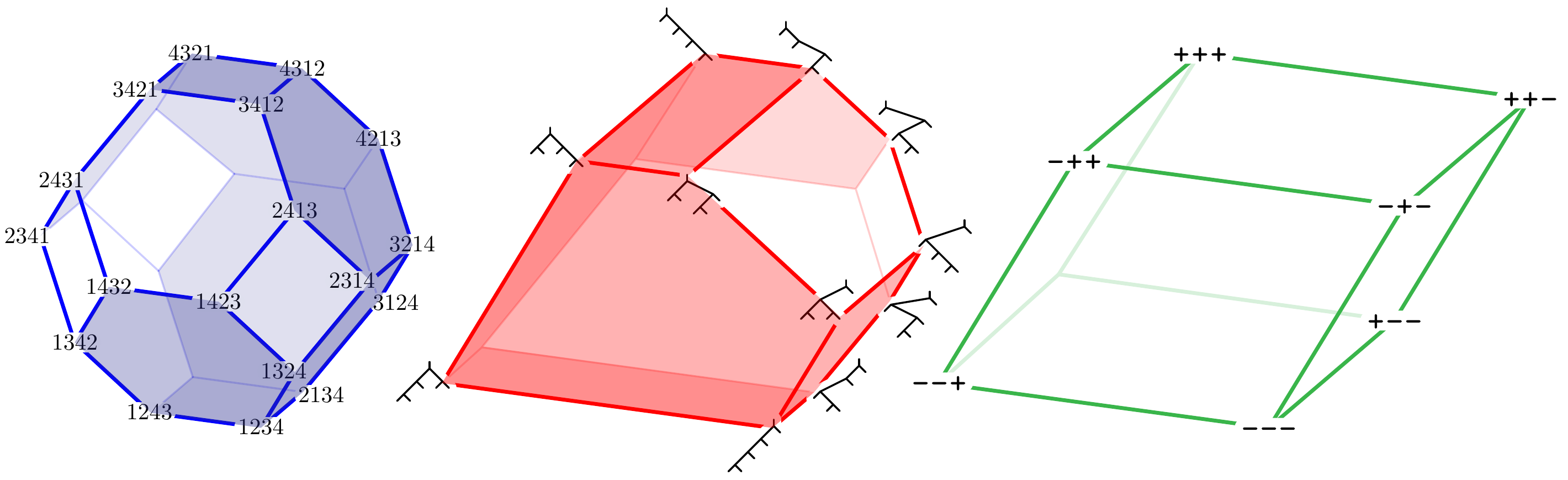}}
  \caption{The $3$-dimensional permutahedron (blue, left), Loday's associahedron (red, middle), and parallelepiped (green, right). Shaded facets are preserved to get the next polytope.}
  \label{fig:permutahedronAssociahedronCube}
\end{figure}

These structures and their connections have been partially extended in several directions in particular to the Cambrian lattices of N.~Reading~\cite{Reading-CambrianLattices, ReadingSpeyer} and their polytopal realizations by C.~Hohlweg, C.~Lange, and H.~Thomas~\cite{HohlwegLange, HohlwegLangeThomas}, to the graph associahedra of M.~Carr and S.~Devadoss~\cite{CarrDevadoss, Devadoss}, the nested complexes and their realizations as generalized associahedra by A.~Postnikov~\cite{Postnikov} (see also~\cite{PostnikovReinerWilliams, FeichtnerSturmfels, Zelevinsky}), or to the $m$-Tamari lattices of F.~Bergeron and L.-F.~Pr\'eville-Ratelle~\cite{BergeronPrevilleRatelle} (see also~\cite{BousquetMelouFusyPrevilleRatelle, BousquetMelouChapuyPrevilleRatelle}) and the Hopf algebras on these $m$-structures recently constructed by J.-C.~Novelli and J.-Y.~Thibon~\cite{NovelliThibon, Novelli}.

This paper explores combinatorial and algebraic aspects of Hopf algebras related to the type~$A$ Cambrian lattices. N.~Reading provides in~\cite{Reading-CambrianLattices} a procedure to map a signed permutation of~$\fS_n$ into a triangulation of a certain convex $(n+3)$-gon. The dual trees of these triangulations naturally extend rooted binary trees and were introduced and studied as ``spines''~\cite{LangePilaud} or ``mixed cobinary trees''~\cite{IgusaOstroff}. We prefer here the term ``Cambrian trees'' in reference to N.~Reading's work. The map~$\kappa$ from signed permutations to Cambrian trees is known to encode combinatorial and geometric properties of the Cambrian structures: the Cambrian lattice is the quotient of the weak order under the fibers of~$\kappa$, each maximal cone of the Cambrian fan is the incidence cone of a Cambrian tree~$\tree$ and is refined by the braid cones of the permutations in the fiber~$\kappa^{-1}(\tree)$, etc.

In this paper, we use this map~$\kappa$ for algebraic purposes. In the first part, we introduce the Cambrian Hopf algebra~$\Camb$ as a subalgebra of the Hopf algebra~$\FQSym_\pm$ on signed permutations, and the dual Cambrian algebra~$\Camb^*$ as a quotient algebra of the dual Hopf algebra~$\FQSym_\pm^*$. Their bases are indexed by all Cambrian trees. Our approach extends that of F.~Hivert, \mbox{J.-C.~Novelli} and \mbox{J.-Y.~Thibon}~\cite{HivertNovelliThibon-algebraBinarySearchTrees} to construct J.-L.~Loday and M.~Ronco's Hopf algebra on binary trees~\cite{LodayRonco} as a subalgebra of C.~Malvenuto and C.~Reutenauer's Hopf algebra on permutations~\cite{MalvenutoReutenauer}. We also use this map~$\kappa$ to describe both the product and coproduct in the algebras~$\Camb$ and~$\Camb^*$ in terms of simple combinatorial operations on Cambrian trees. From the combinatorial description of the product in~$\Camb$, we derive multiplicative bases of the Cambrian algebra~$\Camb$ and study the structural and enumerative properties of their indecomposable elements.

In the second part of this paper, we study Baxter-Cambrian structures, extending in the Cambrian setting the constructions of S.~Law and N.~Reading on quadrangulations~\cite{LawReading} and that of S.~Giraudo on twin binary trees~\cite{Giraudo}. We define Baxter-Cambrian lattices as quotients of the weak order under the intersections of two opposite Cambrian congruences. Their elements can be labeled by pairs of twin Cambrian trees, \ie Cambrian trees with opposite signatures whose union forms an acyclic graph. We study in detail the number of such pairs of Cambrian trees for arbitrary signatures. Following~\cite{LawReading}, we also observe that the Minkowski sums of opposite associahedra of C.~Hohlweg and C.~Lange~\cite{HohlwegLange} provide polytopal realizations of the Baxter-Cambrian lattices. Finally, we introduce the Baxter-Cambrian Hopf algebra~$\BaxCamb$ as a subalgebra of the Hopf algebra~$\FQSym_\pm$ on signed permutations and its dual~$\BaxCamb^*$ as a quotient algebra of the dual Hopf algebra~$\FQSym_\pm^*$. Their bases are indexed by pairs of twin Cambrian trees, and it is also possible to describe both the product and coproduct in the algebras~$\BaxCamb$ and~$\BaxCamb^*$ in terms of simple combinatorial operations on Cambrian trees. We also extend our results to arbitrary tuples of Cambrian trees, resulting to the Cambrian tuple algebra.

\enlargethispage{.6cm}
The last part of the paper is devoted to Schr\"oder-Cambrian structures. We consider Schr\"oder-Cambrian trees which correspond to all faces of all C.~Hohlweg and C.~Lange's associahedra~\cite{HohlwegLange}. We define the Schr\"oder-Cambrian lattice as a quotient of the weak order on ordered partitions defined in~\cite{KrobLatapyNovelliPhanSchwer}, thus extending N.~Reading's type~$A$ Cambrian lattices~\cite{Reading-CambrianLattices} to all faces of the associahedron. Finally, we consider the Schr\"oder-Cambrian Hopf algebra~$\SchrCamb$, generalizing the algebra defined by F.~Chapoton in~\cite{Chapoton}. 


\part{The Cambrian Hopf Algebra}
\label{part:CambrianAlgebra}


\section{Cambrian trees}
\label{sec:CambrianTrees}

In this section, we recall the definition and properties of ``Cambrian trees'', generalizing standard binary search trees. They were introduced independently by K.~Igusa and J.~Ostroff in~\cite{IgusaOstroff} as ``mixed cobinary trees'' in the context of cluster algebras and quiver representation theory and by C.~Lange and V.~Pilaud in~\cite{LangePilaud} as ``spines'' (\ie oriented and labeled dual trees) of triangulations of polygons to revisit the multiple realizations of the associahedron of C.~Hohlweg and C.~Lange~\cite{HohlwegLange}. Here, we use the term ``Cambrian trees'' to underline their connection with the type~$A$ Cambrian lattices of N.~Reading~\cite{Reading-CambrianLattices}. Although motivating and underlying this paper, these interpretations are not needed for the combinatorial and algebraic constructions presented here, and we only refer to them when they help to get geometric intuition on our~statements.


\subsection{Cambrian trees and increasing trees}

Consider a directed tree~$\tree$ on a vertex set~$\ground$ and a vertex~$v \in \ground$. We call \defn{children} (resp.~\defn{parents}) of~$v$ the sources of the incoming arcs (resp.~the targets of the outgoing arcs) at~$v$ and \defn{descendants} (resp.~\defn{ancestor}) \defn{subtrees} of~$v$ the subtrees attached to them. The main characters of our paper are the following trees, which generalize standard binary search trees. Our definition is adapted from~\cite{IgusaOstroff, LangePilaud}.

\begin{definition}
A \defn{Cambrian tree} is a directed tree~$\tree$ with vertex set~$\ground$ endowed with a bijective vertex labeling $p : \ground \to [n]$ such that for each vertex~$v \in \ground$,
\begin{enumerate}[(i)]
\item $v$ has either one parent and two children (its descendant subtrees are called \defn{left and right~subtrees}) or one child and two parents (its ancestor subtrees are called \defn{left and right subtrees});
\item all labels are smaller (resp.~larger) than~$p(v)$ in the left (resp.~right) subtree~of~$v$.
\end{enumerate}
The \defn{signature} of~$\tree$ is the $n$-tuple~$\signature(\tree) \in \pm^n$ defined by~$\signature(\tree)_{p(v)} = -$ if~$v$ has two children and~$\signature(\tree)_{p(v)} = +$ if~$v$ has two parents.
Denote by~$\CambTrees(\signature)$ the set of Cambrian trees with signature~$\signature$, by~$\CambTrees(n) = \bigsqcup_{\signature \in \pm^n} \CambTrees(\signature)$ the set of all Cambrian trees on~$n$ vertices, and by~$\CambTrees \eqdef \bigsqcup_{n \in \N} \CambTrees(n)$ the set of all Cambrian trees.
\end{definition}

\begin{definition}
An \defn{increasing tree} is a directed tree~$\tree$ with vertex set~$\ground$ endowed with a bijective vertex labeling~$q : \ground \to [n]$ such that~$v \to w$ in~$\tree$ implies~$q(v) < q(w)$.
\end{definition}

\begin{definition}
A \defn{leveled Cambrian tree} is a directed tree~$\tree$ with vertex set~$\ground$ endowed with two bijective vertex labelings~${p,q : V \to [n]}$ which respectively define a Cambrian and an increasing tree.
\end{definition}

\begin{figure}[b]
  \centerline{\includegraphics{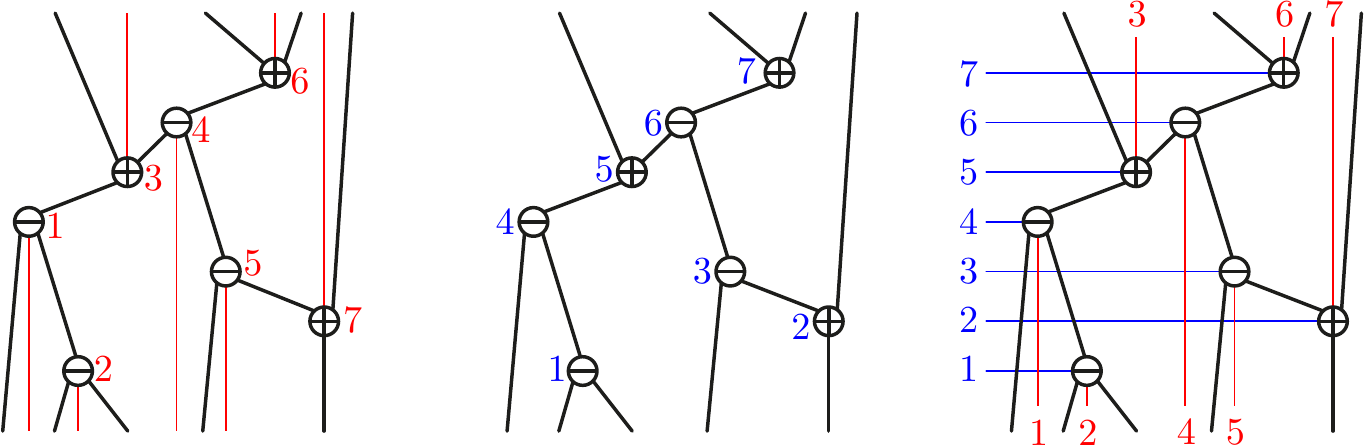}}
  \caption{A Cambrian tree (left), an increasing tree (middle), and a leveled Cambrian tree (right).}
  \label{fig:leveledCambrianTree}
\end{figure}

\enlargethispage{-.6cm}
In other words, a leveled Cambrian tree is a Cambrian tree endowed with a linear extension of its transitive closure.
\fref{fig:leveledCambrianTree} provides examples of a Cambrian tree (left), an increasing tree (middle), and a leveled Cambrian tree (right). All edges are oriented bottom-up. Throughout the paper, we represent leveled Cambrian trees on an $(n \times n)$-grid as follows (see \fref{fig:leveledCambrianTree}):
\begin{enumerate}[(i)]
\item each vertex~$v$ appears at position $(p(v), q(v))$;
\item negative vertices (with one parent and two children) are represented by~$\ominus$, while positive vertices (with one child and two parents) are represented by~$\oplus$;
\item we sometimes draw a vertical red wall below the negative vertices and above the positive vertices to mark the separation between the left and right subtrees of each vertex.
\end{enumerate}

\begin{remark}[Spines of triangulations]
\label{rem:triangulation}
Cambrian trees can be seen as spines (\ie oriented and labeled dual trees) of triangulations of labeled polygons. Namely, consider an $(n+2)$-gon~$\polygon$ with vertices labeled by~$0, \dots, n+1$ from left to right, and where vertex~$i$ is located above the diagonal~$[0,n+1]$ if~$\signature_i = +$ and below it if~$\signature_i = -$. We associate to a triangulation~$\sigma$ of~$\polygon$ its dual tree, with a node labeled by~$j$ for each triangle~$ijk$ of~$\sigma$ where~$i<j<k$, and an edge between any two adjacent triangles oriented from the triangle below to the triangle above their common diagonal. See \fref{fig:triangulation} and refer to~\cite{LangePilaud} for details. Throughout the paper, we denote by~$\tree^*$ the triangulation of~$\polygon$ dual to the $\signature$-Cambrian tree~$\tree$, and we use this interpretation to provide the reader with some geometric intuition of definitions and results of this paper.

\begin{figure}[h]
  \centerline{\includegraphics{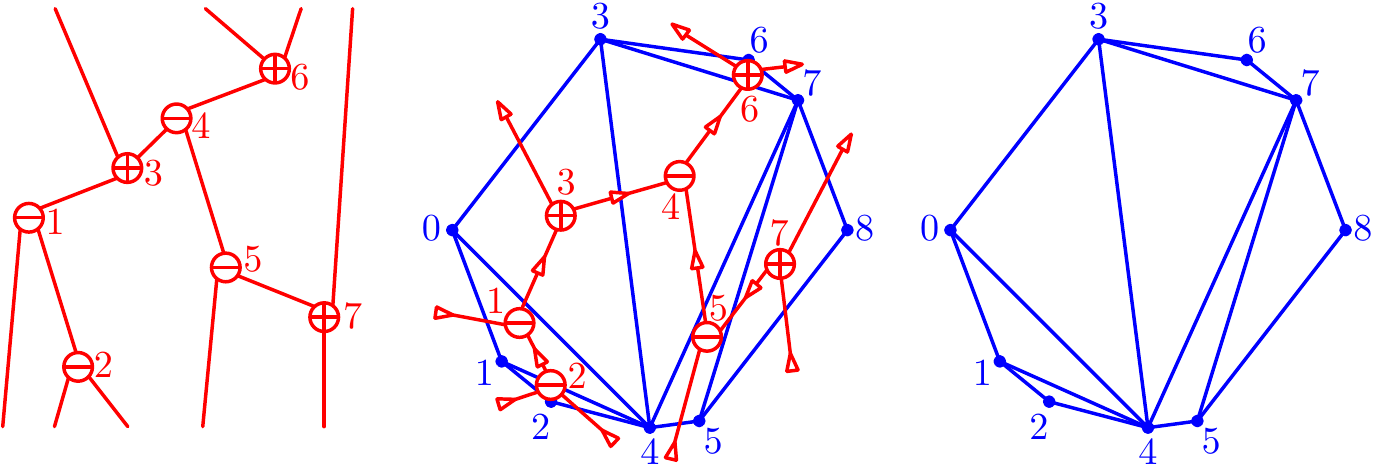}}
  \caption{Cambrian trees (left) and triangulations (right) are dual to each other (middle).}
  \label{fig:triangulation}
\end{figure}
\end{remark}

\begin{proposition}[\cite{LangePilaud, IgusaOstroff}]
For any signature~$\signature \in \pm^n$, the number of $\signature$-Cambrian trees is the Catalan number~$C_n = \frac{1}{n+1}\binom{2n}{n}$. Therefore,~$|\CambTrees(n)| = 2^n C_n$. See~\href{https://oeis.org/A151374}{\cite[A151374]{OEIS}}.
\end{proposition}

There are several ways to prove this statement (to our knowledge, the last two are original): 
\begin{enumerate}[(i)]
\item From the description of~\cite{LangePilaud} given in the previous remark, the number of $\signature$-Cambrian trees is the number of triangulations of a convex $(n+2)$-gon, counted by the Catalan number.
\item There are natural bijections between $\signature$-Cambrian trees and binary trees. One simple way is to reorient all edges of a Cambrian tree towards an arbitrary leaf to get a binary tree, but the inverse map is more difficult to explain, see~\cite{IgusaOstroff}.
\item Cambrian trees are in bijection with certain pattern avoiding signed permutations, see Section~\ref{subsec:CambrianClasses}. In Proposition~\ref{prop:GeneratingTree}, we show that the shape of the generating tree for these permutations is independent of~$\signature$.
\item In Lemma~\ref{lem:switchSign}, we give an explicit bijection between $\signature$- and $\signature'$-Cambrian trees, where~$\signature$ and~$\signature'$ only differ by swapping two consecutive signs or switching the sign of~$1$ (or that of~$n$).
\end{enumerate}


\subsection{Cambrian correspondence}
\label{subsec:CambrianCorrespondence}

\enlargethispage{-.4cm}
We represent graphically a permutation~$\tau \in \fS_n$ by the $(n \times n)$-table, with rows labeled by positions from bottom to top and columns labeled by values from left to right, and with a dot in row~$i$ and column~$\tau(i)$ for all~$i \in [n]$. (This unusual choice of orientation is necessary to fit later with the existing constructions of~\cite{LodayRonco, HivertNovelliThibon-algebraBinarySearchTrees}.)

A \defn{signed permutation} is a permutation table where each dot receives a~$+$ or~$-$ sign, see the top left corner of \fref{fig:insertionAlgorithm}. We could equivalently think of a permutation where the positions or the values receive a sign, but it will be useful later to switch the signature from positions to values. The \defn{p-signature} (resp.~\defn{v-signature}) of a signed permutation~$\tau$ is the sequence~$\psignature(\tau)$ (resp.~$\vsignature(\tau)$) of signs of~$\tau$ ordered by positions from bottom to top (resp.~by values from left to right). For a signature~$\signature \in \pm^n$, we denote by~$\fS_\signature$ (resp.~by~$\fS^\signature$) the set of signed permutations~$\tau$ with p-signature~$\psignature(\tau) = \signature$ (resp.~with v-signature~$\vsignature(\tau) = \signature$). Finally, we denote by
\[
\fS_\pm \eqdef \bigsqcup_{\substack{n \in \N \\ \signature \in \pm^n}} \fS_\signature = \bigsqcup_{\substack{n \in \N \\ \signature \in \pm^n}} \fS^\signature
\]
the set of all signed permutations.

In concrete examples, we underline negative positions/values while we overline positive positions/values: for example, we write~$\down{2}\up{7}\down{51}\up{3}\down{4}\up{6}$ for the signed permutation represented on the top left corner of \fref{fig:insertionAlgorithm}, where~${\tau = [2,7,5,1,3,4,6]}$, ${\psignature = {-}{+}{-}{-}{+}{-}{+}}$ and~${\vsignature = {-}{-}{+}{-}{-}{+}{+}}$.

Following~\cite{LangePilaud}, we now present an algorithm to construct a leveled $\signature$-Cambrian tree~$\CambCorresp(\tau)$ from a signed permutation~$\tau \in \fS^\signature$. \fref{fig:insertionAlgorithm} illustrates this algorithm on the permutation~$\down{2}\up{7}\down{51}\up{3}\down{4}\up{6}$. As a preprocessing, we represent the table of~$\tau$ (with signed dots in positions~$(\tau(i),i)$ for~$i \in [n]$) and draw a vertical wall below the negative vertices and above the positive vertices. We then sweep the table from bottom to top (thus reading the permutation~$\tau$ from left to right) as follows. The procedure starts with an incoming strand in between any two consecutive negative values. A negative dot~$\ominus$ connects the two strands immediately to its left and immediately to its right to form a unique outgoing strand. A positive dot~$\oplus$ separates the only visible strand (not hidden by a wall) into two outgoing strands. The procedure finishes with an outgoing strand in between any two consecutive positive values. See \fref{fig:insertionAlgorithm}.

\begin{figure}
  \centerline{\includegraphics[width=\textwidth]{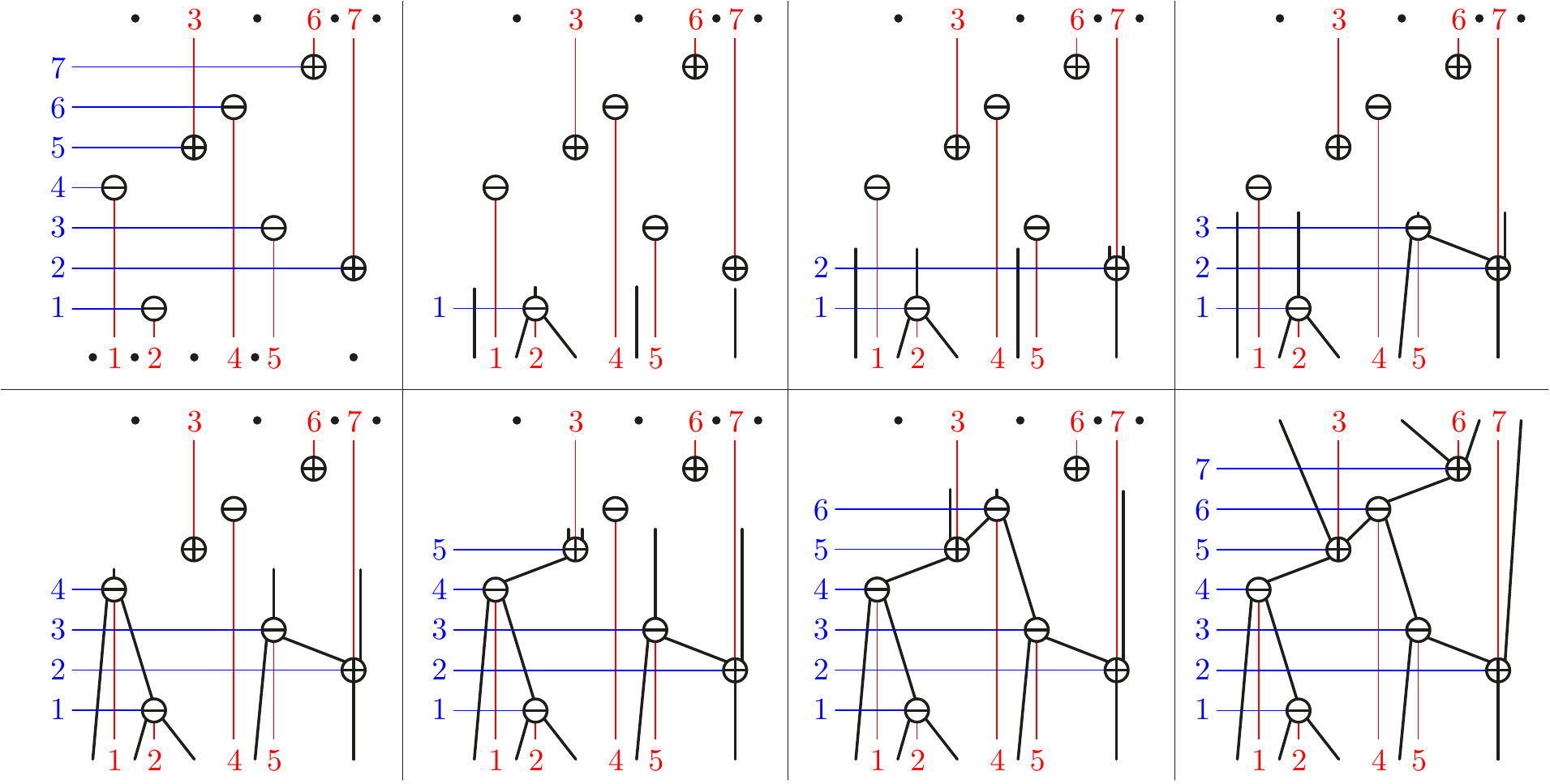}}
  \caption{The insertion algorithm on the signed permutation~$\down{2}\up{7}\down{51}\up{3}\down{4}\up{6}$.}
  \label{fig:insertionAlgorithm}
\end{figure}

\begin{proposition}[\cite{LangePilaud}]
The map~$\CambCorresp$ is a bijection from signed permutations to leveled Cambrian trees.
\end{proposition}

\begin{remark}[Cambrian correspondence]
The \defn{Robinson-Schensted correspondence} is a bijection between permutations and pairs of standard Young tableaux of the same shape. Schensted's algorithm~\cite{Schensted} gives an efficient algorithmic way to create the pair of tableaux~$(\PSymbol(\tau), \QSymbol(\tau))$ corresponding to a given permutation~$\tau$ by successive insertions: the first tableau~$\PSymbol(\tau)$ (\defn{insertion tableau}) remembers the inserted elements of~$\tau$ while the second tableau~$\QSymbol(\tau)$ (\defn{recording tableau}) remembers the order in which the elements have been inserted. F.~Hivert, J.-C.~Novelli and \mbox{J.-Y.~Thibon} defined in~\cite{HivertNovelliThibon-algebraBinarySearchTrees} a similar correspondence, called \defn{sylvester correspondence}, between permutations and pairs of labeled trees of the same shape. In the sylvester correspondence, the first tree (insertion tree) is a standard binary search tree and the second tree (recording tree) is an increasing binary tree. The \defn{Cambrian correspondence} can as well be thought of as a correspondence between signed permutations and pairs of trees of the same shape, where the first tree (insertion tree) is Cambrian and the second tree (recording tree) is increasing. This analogy motivates the following definition.
\end{remark}

\begin{definition}
Given a signed permutation~$\tau \in \fS^\signature$, its \defn{$\PSymbol$-symbol} is the insertion Cambrian tree~$\surjectionPermAsso(\tau)$ defined by~$\CambCorresp(\tau)$ and its \defn{$\QSymbol$-symbol} is the recording increasing tree~$\QSymbol(\tau)$ defined by~$\CambCorresp(\tau)$.
\end{definition}

The following characterization of the fibers of~$\surjectionPermAsso$ is immediate from the description of the algorithm. We denote by~$\linearExtensions(\graphG)$ the set of linear extensions of a directed graph~$\graphG$.

\begin{proposition}
The signed permutations~$\tau \in \fS^\signature$ such that~$\surjectionPermAsso(\tau) = \tree$ are precisely the linear extensions of (the transitive closure of)~$\tree$. \end{proposition}

\begin{example}
When~$\signature = (+)^n$, the procedure constructs a binary search tree~$\surjectionPermAsso(\tau)$ pointing up by successive insertions from the left. Equivalently, $\surjectionPermAsso(\tau)$ can be constructed as the increasing tree of~$\tau^{-1}$. Here, the \defn{increasing tree}~$\increasingTree(\pi)$ of a permutation~$\pi = \pi' 1 \pi''$ is defined inductively by grafting the increasing tree~$\increasingTree(\pi')$ on the left and the increasing tree~$\increasingTree(\pi'')$ on the right of the bottom root labeled by~$1$. When~$\signature = (-)^n$, this procedure constructs bottom-up a binary search tree~$\surjectionPermAsso(\tau)$ pointing down. This tree would be obtained by successive binary search tree insertions from the right. Equivalently, $\surjectionPermAsso(\tau)$ can be constructed as the decreasing tree of~$\tau^{-1}$. Here, the \defn{decreasing tree}~$\decreasingTree(\pi)$ of a permutation~$\pi = \pi' n \pi''$ is defined inductively by grafting the decreasing tree~$\decreasingTree(\pi')$ on the left and the decreasing tree~$\decreasingTree(\pi'')$ on the right of the top root labeled by~$n$. These observations are illustrated on~\fref{fig:constantSigns}.

\begin{figure}[h]
  \centerline{\includegraphics{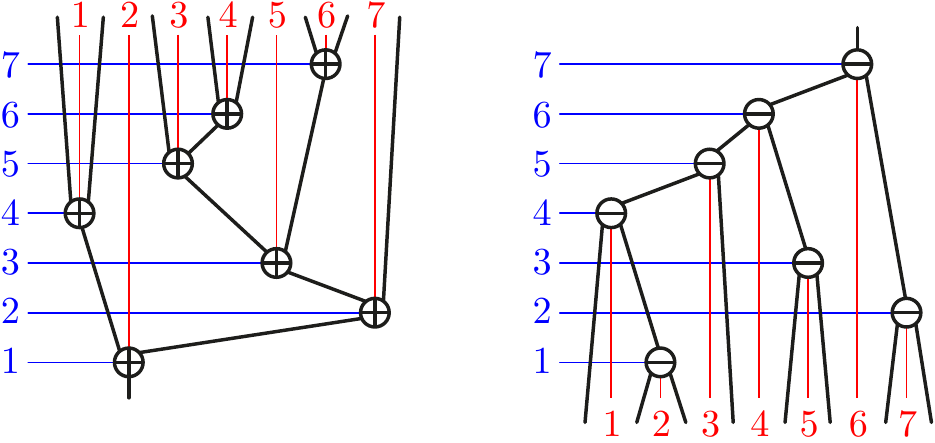}}
  \caption{The insertion procedure produces binary search trees when the signature is constant positive (left) or constant negative (right).}
  \label{fig:constantSigns}
\end{figure}
\end{example}

\begin{remark}[Cambrian correspondence on triangulations]
\label{rem:CambrianCorrespondenceTriangulations}
N.~Reading~\cite{Reading-CambrianLattices} first described the map~$\surjectionPermAsso$ on the triangulations of the polygon~$\polygon$ (remember Remark~\ref{rem:triangulation}). Namely, the triangulation~$\surjectionPermAsso(\tau)^*$ is the union of the paths~$\pi_0, \dots, \pi_n$ where~$\pi_i$ is the path between vertices~$0$ and~$n+1$ of~$\polygon$ passing through the vertices in the symmetric difference~$\signature^{-1}(-) \symdif \tau([i])$.
\end{remark}


\subsection{Cambrian congruence}

Following the definition of the sylvester congruence in~\cite{HivertNovelliThibon-algebraBinarySearchTrees}, we now characterize by a congruence relation the signed permutations~$\tau \in \fS^\signature$ which have the same $\PSymbol$-symbol~$\surjectionPermAsso(\tau)$. This Cambrian congruence goes back to the original definition of N.~Reading~\cite{Reading-CambrianLattices}.

\begin{definition}[\cite{Reading-CambrianLattices}]
\label{def:CambrianCongruence}
For a signature~$\signature \in \pm^n$, the \defn{$\signature$-Cambrian congruence} is the equivalence relation on~$\fS^\signature$ defined as the transitive closure of the rewriting rules
\begin{gather*}
UacVbW \equiv_\signature UcaVbW \quad\text{if } a < b < c \text{ and } \signature_b = -, \\
UbVacW \equiv_\signature UbVcaW \quad\text{if } a < b < c \text{ and } \signature_b = +,
\end{gather*}
where~$a,b,c$ are elements of~$[n]$ while~$U,V,W$ are words on~$[n]$. The \defn{Cambrian congruence} is the equivalence relation on all signed permutations~$\fS_\pm$ obtained as the union of all $\signature$-Cambrian congruences:
\[
\equiv \; \eqdef \bigsqcup_{\substack{n \in \N \\ \signature \in \pm^n}} \!\! \equiv_\signature.
\]
\end{definition}

\begin{proposition}
\label{prop:CambrianClass}
Two signed permutations~$\tau, \tau' \in \fS^\signature$ are $\signature$-Cambrian congruent if and only if they have the same $\PSymbol$-symbol:
\[
\tau \equiv_\signature \tau' \iff \surjectionPermAsso(\tau) = \surjectionPermAsso(\tau').
\]
\end{proposition}

\begin{proof}
It boils down to observe that two consecutive vertices~$a,c$ in a linear extension~$\tau$ of a $\signature$-Cambrian tree~$\tree$ can be switched while preserving a linear extension~$\tau'$ of~$\tree$ precisely when they belong to distinct subtrees of a vertex~$b$ of~$\tree$. It follows that the vertices~$a,c$ lie on either sides of~$b$ so that we have~$a < b < c$. If~$\signature_b = -$, then~$a,c$ appear before~$b$ and~$\tau = UacVbW$ can be switched to~$\tau' = UcaVbW$, while if~$\signature_b = +$, then~$a,c$ appear after~$b$ and~$\tau = UbVacW$ can be switched to~$\tau' = UbVcaW$.
\end{proof}


\subsection{Cambrian classes and generating trees}
\label{subsec:CambrianClasses}

We now focus on the equivalence classes of the Cambrian congruence. Remember that the \defn{(right) weak order} on~$\fS^\signature$ is defined as the inclusion order of coinversions, where a \defn{coinversion} of~$\tau \in \fS^\signature$ is a pair of values~$i < j$ such that~${\tau^{-1}(i) > \tau^{-1}(j)}$ (no matter the signs on~$\tau$). In this paper, we always work with the right weak order, that we simply call weak order for brevity. The following statement is due to N.~Reading~\cite{Reading-CambrianLattices}.

\begin{proposition}[\cite{Reading-CambrianLattices}]
All $\signature$-Cambrian classes are intervals of the weak order on~$\fS^\signature$.
\end{proposition}

Therefore, the $\signature$-Cambrian trees are in bijection with the weak order maximal permutations of~$\signature$-Cambrian classes. Using Definition~\ref{def:CambrianCongruence} and Proposition~\ref{prop:CambrianClass}, one can prove that these permutations are precisely the permutations in~$\fS^\signature$ that avoid the signed patterns~$b \dash ac$ with~$\signature_b = +$ and~$ac \dash b$ with~$\signature_b = -$ (for brevity, we write~$\up{b} \dash ac$ and~$ac \dash \down{b}$). It enables us to construct a generating tree~$\generatingTree_{\signature}$ for these permutations. This tree has~$n$ levels, and the nodes at level~$\level$ are labeled by the permutations of~$[\level]$ whose values are signed by the restriction of~$\signature$ to~$[\level]$ and avoiding the two patterns~$\up{b} \dash ac$ and~$ac \dash \down{b}$. The parent of a permutation in~$\generatingTree_\signature$ is obtained by deleting its maximal value. See ~\fref{fig:GeneratingTree} for examples of such trees. The following statement provides another proof that the number of~$\signature$-Cambrian trees on~$n$ nodes is always the Catalan number~$C_n = \frac{1}{n+1}\binom{2n}{n}$, as well as an explicit bijection between~$\signature$- and $\signature'$-Cambrian trees for distinct signatures~$\signature, \signature' \in \pm^n$.

\begin{figure}[h]
  \centerline{\begin{tikzpicture}
  \node(T1) at (0,0) {
    \begin{tikzpicture}[xscale=1.2, yscale=1.3]
      \node(P2341) at (0,0) {$\freeGap\up{2}\bannedGap\down{3}\bannedGap\down{4}\freeGap\down{1}\bannedGap$};
      \node(P4231) at (1,0) {$\freeGap\down{4}\freeGap\up{2}\bannedGap\down{3}\freeGap\down{1}\bannedGap$};
      \node(P3241) at (2,0) {$\freeGap\down{3}\bannedGap\up{2}\bannedGap\down{4}\freeGap\down{1}\bannedGap$};
      \node(P3421) at (3,0) {$\freeGap\down{3}\bannedGap\down{4}\freeGap\up{2}\freeGap\down{1}\bannedGap$};
      \node(P4321) at (4,0) {$\freeGap\down{4}\freeGap\down{3}\freeGap\up{2}\freeGap\down{1}\bannedGap$};
      \node(P1234) at (5,0) {$\freeGap\down{1}\bannedGap\up{2}\bannedGap\down{3}\bannedGap\down{4}\freeGap$};
      \node(P4123) at (6,0) {$\freeGap\down{4}\freeGap\down{1}\bannedGap\up{2}\bannedGap\down{3}\freeGap$};
      \node(P1324) at (7,0) {$\freeGap\down{1}\bannedGap\down{3}\bannedGap\up{2}\bannedGap\down{4}\freeGap$};
      \node(P1342) at (8,0) {$\freeGap\down{1}\bannedGap\down{3}\bannedGap\down{4}\freeGap\up{2}\freeGap$};
      \node(P4132) at (9,0) {$\freeGap\down{4}\freeGap\down{1}\bannedGap\down{3}\freeGap\up{2}\freeGap$};
      \node(P3124) at (10,0) {$\freeGap\down{3}\bannedGap\down{1}\bannedGap\up{2}\bannedGap\down{4}\freeGap$};
      \node(P3142) at (11,0) {$\freeGap\down{3}\bannedGap\down{1}\bannedGap\down{4}\freeGap\up{2}\freeGap$};
      \node(P3412) at (12,0) {$\freeGap\down{3}\bannedGap\down{4}\freeGap\down{1}\freeGap\up{2}\freeGap$};
      \node(P4312) at (13,0) {$\freeGap\down{4}\freeGap\down{3}\freeGap\down{1}\freeGap\up{2}\freeGap$};
      \node(P231) at (0.5,1) {$\freeGap\up{2}\bannedGap\down{3}\freeGap\down{1}\bannedGap$};
      \node(P321) at (3,1) {$\freeGap\down{3}\freeGap\up{2}\freeGap\down{1}\bannedGap$};
      \node(P123) at (5.5,1) {$\freeGap\down{1}\bannedGap\up{2}\bannedGap\down{3}\freeGap$};
      \node(P132) at (8,1) {$\freeGap\down{1}\bannedGap\down{3}\freeGap\up{2}\freeGap$};
      \node(P312) at (11.5,1) {$\freeGap\down{3}\freeGap\down{1}\freeGap\up{2}\freeGap$};
      \node(P21) at (1.75,2) {$\freeGap\up{2}\freeGap\down{1}\bannedGap$};
      \node(P12) at (8,2) {$\freeGap\down{1}\freeGap\up{2}\freeGap$};
      \node(P1) at (4.875,3) {$\freeGap\down{1}\freeGap$};

      \draw (P21) -- (P1);
      \draw (P12) -- (P1);

      \draw (P123) -- (P12);
      \draw (P132) -- (P12);
      \draw (P312) -- (P12);

      \draw (P231) -- (P21);
      \draw (P321) -- (P21);

      \draw (P3412) -- (P312);
      \draw (P3124) -- (P312);
      \draw (P4312) -- (P312);
      \draw (P3142) -- (P312);

      \draw (P1342) -- (P132);
      \draw (P4132) -- (P132);
      \draw (P1324) -- (P132);

      \draw (P1234) -- (P123);
      \draw (P4123) -- (P123);

      \draw (P3241) -- (P321);
      \draw (P3421) -- (P321);
      \draw (P4321) -- (P321);

      \draw (P2341) -- (P231);
      \draw (P4231) -- (P231);
    \end{tikzpicture}
  };

  \node(T2) at (0, -5) {
    \begin{tikzpicture}[xscale=1.2, yscale=1.3]
      \node(P4123) at (0,0) {$\freeGap\up{1}\bannedGap\down{2}\bannedGap\down{3}\bannedGap\down{4}\freeGap$};
      \node(P1234) at (1,0) {$\freeGap\down{4}\freeGap\up{1}\bannedGap\down{2}\bannedGap\down{3}\freeGap$};
      \node(P4312) at (2,0) {$\freeGap\down{3}\bannedGap\up{1}\bannedGap\down{2}\bannedGap\down{4}\freeGap$};
      \node(P3412) at (3,0) {$\freeGap\down{3}\bannedGap\down{4}\freeGap\up{1}\bannedGap\down{2}\freeGap$};
      \node(P3124) at (4,0) {$\freeGap\down{4}\freeGap\down{3}\freeGap\up{1}\bannedGap\down{2}\freeGap$};
      \node(P4213) at (5,0) {$\freeGap\down{2}\bannedGap\up{1}\bannedGap\down{3}\bannedGap\down{4}\freeGap$};
      \node(P2134) at (6,0) {$\freeGap\down{4}\freeGap\down{2}\bannedGap\up{1}\bannedGap\down{3}\freeGap$};
      \node(P4231) at (7,0) {$\freeGap\down{2}\bannedGap\down{3}\bannedGap\up{1}\bannedGap\down{4}\freeGap$};
      \node(P2341) at (8,0) {$\freeGap\down{2}\bannedGap\down{3}\bannedGap\down{4}\freeGap\up{1}\freeGap$};
      \node(P2314) at (9,0) {$\freeGap\down{4}\freeGap\down{2}\bannedGap\down{3}\freeGap\up{1}\freeGap$};
      \node(P4321) at (10,0) {$\freeGap\down{3}\bannedGap\down{2}\bannedGap\up{1}\bannedGap\down{4}\freeGap$};
      \node(P3421) at (11,0) {$\freeGap\down{3}\bannedGap\down{2}\bannedGap\down{4}\freeGap\up{1}\freeGap$};
      \node(P3241) at (12,0) {$\freeGap\down{3}\bannedGap\down{4}\freeGap\down{2}\freeGap\up{1}\freeGap$};
      \node(P3214) at (13,0) {$\freeGap\down{4}\freeGap\down{3}\freeGap\down{2}\freeGap\up{1}\freeGap$};
      \node(P123) at (0.5,1) {$\freeGap\up{1}\bannedGap\down{2}\bannedGap\down{3}\freeGap$};
      \node(P312) at (3,1) {$\freeGap\down{3}\freeGap\up{1}\bannedGap\down{2}\freeGap$};
      \node(P213) at (5.5,1) {$\freeGap\down{2}\bannedGap\up{1}\bannedGap\down{3}\freeGap$};
      \node(P231) at (8,1) {$\freeGap\down{2}\bannedGap\down{3}\freeGap\up{1}\freeGap$};
      \node(P321) at (11.5,1) {$\freeGap\down{3}\freeGap\down{2}\freeGap\up{1}\freeGap$};
      \node(P12) at (1.75,2) {$\freeGap\up{1}\bannedGap\down{2}\freeGap$};
      \node(P21) at (8,2) {$\freeGap\down{2}\freeGap\up{1}\freeGap$};
      \node(P1) at (4.875,3) {$\freeGap\up{1}\freeGap$};

      \draw (P4321) -- (P321);
      \draw (P3421) -- (P321);
      \draw (P3241) -- (P321);
      \draw (P3214) -- (P321);

      \draw (P4231) -- (P231);
      \draw (P2341) -- (P231);
      \draw (P2314) -- (P231);

      \draw (P4213) -- (P213);
      \draw (P2134) -- (P213);

      \draw (P4312) -- (P312);
      \draw (P3412) -- (P312);
      \draw (P3124) -- (P312);

      \draw (P4123) -- (P123);
      \draw (P1234) -- (P123);

      \draw (P213) -- (P21);
      \draw (P231) -- (P21);
      \draw (P321) -- (P21);

      \draw (P123) -- (P12);
      \draw (P312) -- (P12);

      \draw (P12) -- (P1);
      \draw (P21) -- (P1);
    \end{tikzpicture}
  };
\end{tikzpicture}}
  \caption{The generating trees~$\generatingTree_\signature$ for the signatures~$\signature = {-}{+}{-}{-}$ (top) and ${\signature = {+}{-}{-}{-}}$ (bottom). Free gaps are marked with a blue dot.}
  \label{fig:GeneratingTree}
\end{figure}
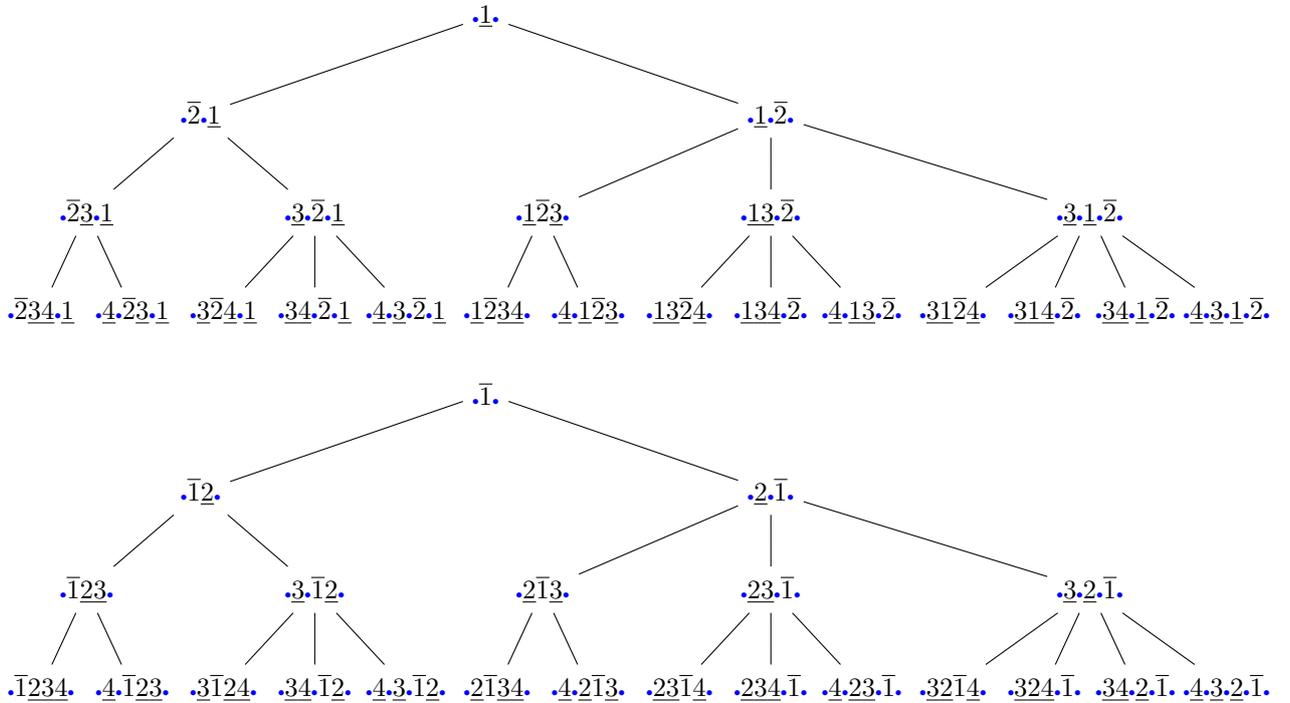

\begin{proposition}
\label{prop:GeneratingTree}
For any signatures~$\signature, \signature' \in \pm^n$, the generating trees~$\generatingTree_\signature$ and~$\generatingTree_{\signature'}$ are isomorphic.
\end{proposition}

For the proof, we consider the possible positions of~$\level+1$ in the children of a permutation~$\tau$ at level~$\level$ in~$\generatingTree_\signature$. Index by~$\{0, \dots, \level\}$ from left to right the gaps before the first letter, between two consecutive letters, and after the last letter of~$\tau$. We call \defn{free gaps} the gaps in~$\{0, \dots, \level\}$ where placing~$\level+1$ does not create a pattern~$ac \dash \down{b}$ or~$\up{b} \dash ac$. They are marked with a blue point in \fref{fig:GeneratingTree}.

\begin{lemma}
\label{lem:GeneratingTree}
A permutation with~$k$ free gaps has $k$ children in~$\generatingTree_\signature$, whose numbers of free gaps range from~$2$ to~$k+1$.
\end{lemma}

\begin{proof}
Let~$\tau$ be a permutation at level~$\level$ in~$\generatingTree_\signature$ with $k$ free gaps. Let~$\sigma$ be the child of~$\tau$ in~$\generatingTree_\signature$ obtained by inserting~$\level+1$ at a free gap~$j \in \{0, \dots, \level\}$. If~$\signature_{\level + 1}$ is negative (resp.~positive), then the free gaps of~$\sigma$ are $0$, $j+1$ and the free gaps of~$\tau$ after~$j$ (resp.~before~$j+1$). The result follows.
\end{proof}

\begin{proof}[Proof of Proposition~\ref{prop:GeneratingTree}]
Order the children of a permutation of~$\generatingTree_\signature$ from left to right by increasing number of free gaps as in \fref{fig:GeneratingTree}. Lemma~\ref{lem:GeneratingTree} shows that the shape of the resulting tree is independent of~$\signature$. It ensures that the trees~$\generatingTree_\signature$ and~$\generatingTree_{\signature'}$ are isomorphic and provides an explicit bijection between the~$\signature$-Cambrian trees and~$\signature'$-Cambrian trees.
\end{proof}


\subsection{Rotations and Cambrian lattices}

We now present rotations in Cambrian trees, a local operation which transforms a $\signature$-Cambrian tree into another $\signature$-Cambrian tree where a single oriented cut differs (see Proposition~\ref{prop:rotation}).

\begin{definition}
Let~$i \to j$ be an edge in a Cambrian tree~$\tree$, with~$i < j$. Let~$L$ denote the left subtree of~$i$ and~$B$ denote the remaining incoming subtree of~$i$, and similarly, let~$R$ denote the right subtree of~$j$ and~$A$ denote the remaining outgoing subtree of~$j$. Let~$\tree'$ be the oriented tree obtained from~$\tree$ just reversing the orientation of~$i \to j$ and attaching the subtrees~$L$ and~$A$ to~$i$ and the subtrees~$B$ and~$R$ to~$j$. The transformation from~$\tree$ to~$\tree'$ is called \defn{rotation} of the edge~$i \to j$. See \fref{fig:rotation}.

\begin{figure}[h]
  \centerline{\begin{minipage}{2cm} rotation \\ of $i \to j$ \end{minipage} \hspace{-1cm} $\begin{array}{c} \tree \\ \!\rotatebox{-90}{$\verylongrightarrow\;\;\;$} \\ \,\tree' \end{array} \quad \vcenter{\hbox{\includegraphics{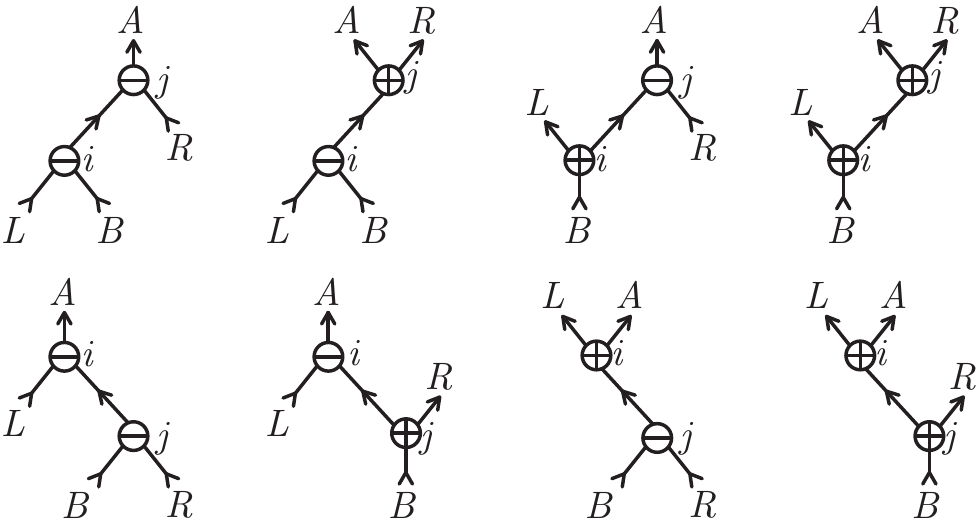}}}$}
  \caption{Rotations in Cambrian trees: the tree~$\tree$ (top) is transformed into the tree~$\tree'$ (bottom) by rotation of the edge~$i \to j$. The four cases correspond to the possible signs of~$i$~and~$j$.}
  \label{fig:rotation}
\end{figure}
\end{definition}

The following proposition states that rotations are compatible with Cambrian trees and their edge cuts. An \defn{edge cut} in a Cambrian tree~$\tree$ is the ordered partition~$\edgecut{X}{Y}$ of the vertices of~$\tree$ into the set~$X$ of vertices in the source set and the set~$Y$ of vertices in the target set of an oriented edge of~$\tree$.

\begin{proposition}[\cite{LangePilaud}]
\label{prop:rotation}
The result~$\tree'$ of the rotation of an edge~$i \to j$ in a $\signature$-Cambrian tree~$\tree$ is a $\signature$-Cambrian tree. Moreover, $\tree'$ is the unique $\signature$-Cambrian tree with the same edge cuts as~$\tree$, except the cut defined by the edge~$i \to j$.
\end{proposition}

\begin{remark}[Rotations and flips]
Rotating an edge~$e$ in a $\signature$-Cambrian tree~$\tree$ corresponds to flipping the dual diagonal~$e^*$ of the dual triangulation~$\tree^*$ of the polygon~$\polygon$. See~\cite[Lemma~13]{LangePilaud}.
\end{remark}

Define the \defn{increasing rotation graph} on~$\CambTrees(\signature)$ to be the graph whose vertices are the $\signature$-Cambrian trees and whose arcs are increasing rotations~$\tree \to \tree'$, \ie where the edge~$i \to j$ in~$\tree$ is reversed to the edge~$i \leftarrow j$ in~$\tree'$ for~$i < j$. See \fref{fig:CambrianLattices} for an illustration. The following statement, adapted from N.~Reading's work~\cite{Reading-CambrianLattices}, asserts that this graph is acyclic, that its transitive closure defines a lattice, and that this lattice is closely related to the weak order. See \fref{fig:lattices}.

\hvFloat[floatPos=p, capWidth=h, capPos=r, capAngle=90, objectAngle=90, capVPos=c, objectPos=c]{figure}
{\includegraphics[width=1.65\textwidth]{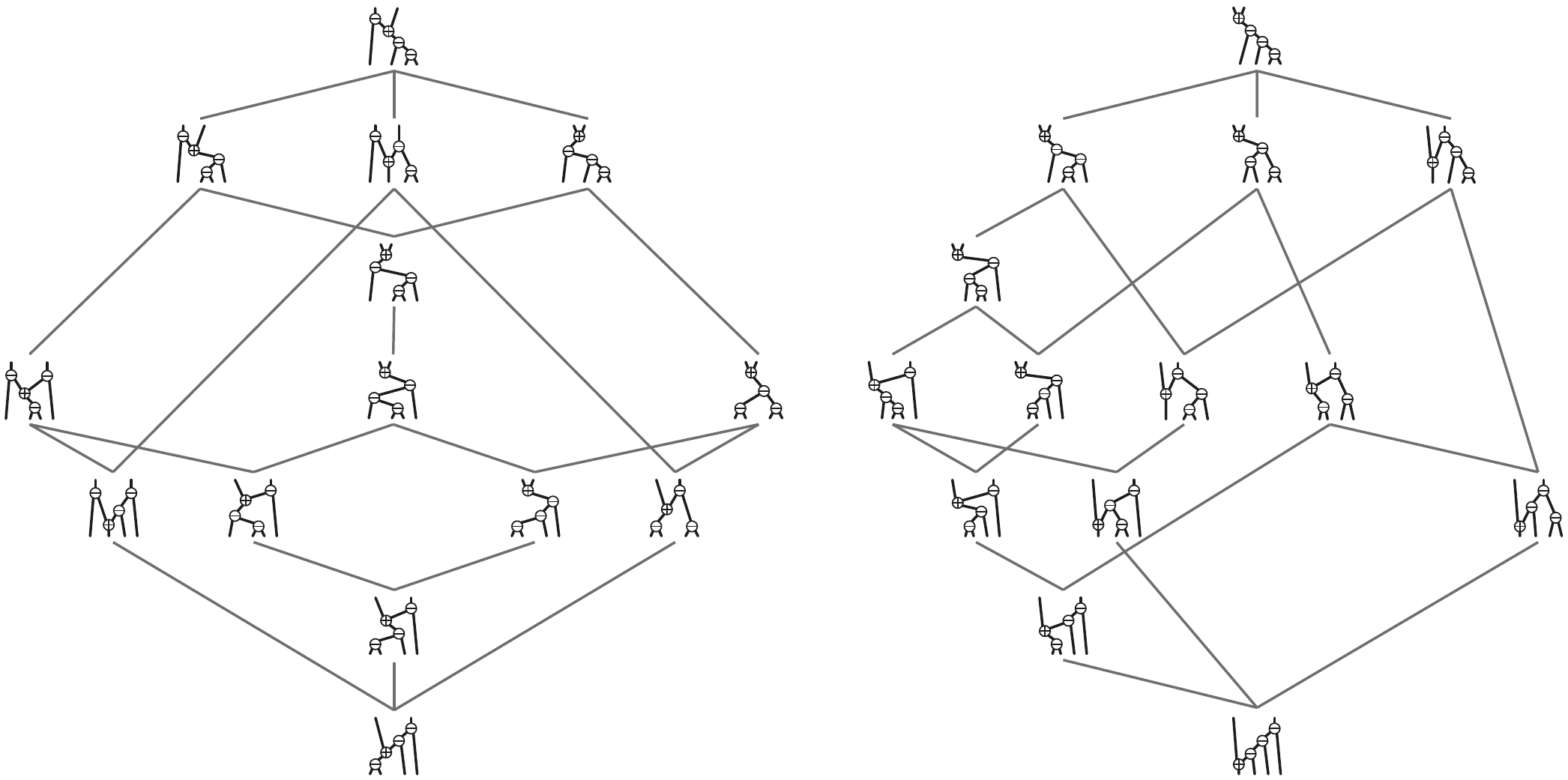}}
{The $\signature$-Cambrian lattices on $\signature$-Cambrian trees, for the signatures $\signature = {-}{+}{-}{-}$ (left) and ${\signature = {+}{-}{-}{-}}$~(right).}
{fig:CambrianLattices}

\begin{proposition}[\cite{Reading-CambrianLattices}]
The transitive closure of the increasing rotation graph on~$\CambTrees(\signature)$ is a lattice, called \defn{$\signature$-Cambrian lattice}. The map~$\surjectionPermAsso : \fS^\signature \to \CambTrees(\signature)$ defines a lattice homomorphism from the weak order on~$\fS^\signature$ to the $\signature$-Cambrian lattice on~$\CambTrees(\signature)$.
\end{proposition}

Note that the minimal (resp.~maximal) $\signature$-Cambrian tree is an oriented path from~$1$ to~$n$ (resp.~from~$n$ to~$1$) with an additional incoming leaf at each negative vertex and an additional outgoing leaf at each positive vertex. See \fref{fig:CambrianLattices}.

\begin{example}
When~$\signature = (-)^n$, the Cambrian lattice is the classical Tamari lattice~\cite{TamariFestschrift}. It can be defined equivalently by left-to-right rotations in planar binary trees, by slope increasing flips in triangulations of~$\polygon[(-)^n]$, or as the quotient of the weak order by the sylvester congruence.
\end{example}


\subsection{Canopy}
\label{subsec:canopy}

The canopy of a binary tree was already used by J.-L.~Loday in~\cite{LodayRonco, Loday} but the name was coined by X.~Viennot~\cite{Viennot}. It was then extended to Cambrian trees (or spines) in~\cite{LangePilaud} to define a surjection from the associahedron~$\Asso$ to the parallelepiped~$\Para$ generated by the simple roots. The main observation is that the vertices~$i$ and~$i+1$ are always comparable in a Cambrian tree (otherwise, they would be in distinct subtrees of a vertex~$j$ which should then lie in between~$i$ and~$i+1$).

\begin{definition}
The \defn{canopy} of a Cambrian tree~$\tree$ is the sequence~$\surjectionAssoPara(\tree) \in \pm^{n-1}$ defined by ${\surjectionAssoPara(\tree)_i = -}$ if $i$ is above~$i+1$ in~$\tree$ and~${\surjectionAssoPara(\tree)_i = +}$ if $i$ is below~$i+1$ in~$\tree$.
\end{definition}

For example, the canopy of the Cambrian tree of \fref{fig:leveledCambrianTree}\,(left) is~${-}{+}{+}{-}{+}{-}$. The canopy of~$\tree$ behaves nicely with the linear extensions of~$\tree$ and with the Cambrian lattice. To state this, we define for a permutation~$\tau \in \fS^\signature$ the sequence~$\surjectionPermPara(\tau) \in \pm^{n-1}$, where~$\surjectionPermPara(\tree)_i = -$ if~$\tau^{-1}(i) > \tau^{-1}(i+1)$ and~$\surjectionPermPara(\tree)_i = +$ otherwise. In other words, $\surjectionPermPara(\tau)$ records the \defn{recoils} of the permutation~$\tau$, \ie the \defn{descents} of the inverse permutation of~$\tau$.

\begin{proposition}
\label{prop:commutativeDiagram}
The maps~$\surjectionPermAsso, \surjectionAssoPara$, and~$\surjectionPermPara$ define the following commutative diagram of lattice homomorphisms:
\[
\begin{tikzpicture}
  \matrix (m) [matrix of math nodes,row sep=1.5em,column sep=5em,minimum width=2em]
  {
     \fS^\signature  	&					& \pm^{n-1}	\\
						& \CambTrees(\signature) &			\\
  };
  \path[->>]
    (m-1-1) edge node [above] {$\surjectionPermPara$} (m-1-3)
                 edge node [below] {$\surjectionPermAsso$} (m-2-2.west)
    (m-2-2.east) edge node [below] {$\quad\surjectionAssoPara$} (m-1-3);
\end{tikzpicture}
\]
\end{proposition}

The fibers of these maps on the weak orders of~$\fS_\signature$ for $\signature = {-}{+}{-}{-}$ and $\signature = {+}{-}{-}{-}$ are represented in \fref{fig:lattices}.

\begin{figure}[h]
  \centerline{\includegraphics[width=\textwidth]{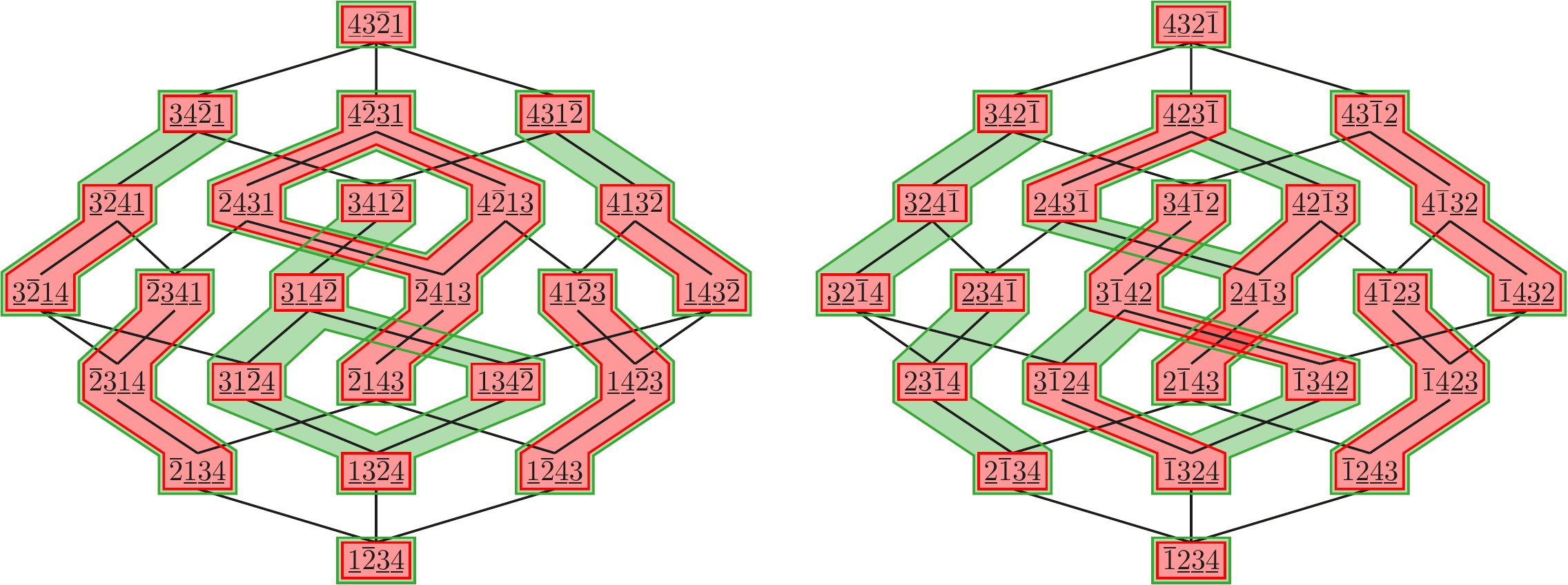}}
  \caption{The fibers of the maps~$\surjectionPermAsso$ (red) and~$\surjectionPermPara$ (green) on the weak orders of~$\fS_\signature$ for $\signature = {-}{+}{-}{-}$ (left) and $\signature = {+}{-}{-}{-}$ (right).}
  \label{fig:lattices}
\end{figure}


\subsection{Geometric realizations}
\label{subsec:geometricRealizations}

We close this section with geometric interpretations of the Cambrian trees, Cambrian classes, Cambrian correspondence, and Cambrian lattices. We denote by~$e_1, \dots, e_n$ the canonical basis of~$\R^n$ and by~$\HH$ the hyperplane of~$\R^n$ orthogonal to~$\sum e_i$. Define the \defn{incidence cone}~$\Cone(\tree)$ and the \defn{braid cone}~$\Cone\polar(\tree)$ of a directed tree~$\tree$ as
\[
\Cone(\tree) \eqdef \cone\set{e_i-e_j}{\text{for all } i \to j \text{ in } \tree}
\quad\text{and}\quad
\Cone\polar(\tree) \eqdef \set{\b{x} \in \HH}{x_i \le x_j \text{ for all } i \to j \text{ in } \tree}.
\]
These two cones lie in the space~$\HH$ and are polar to each other. For a permutation~${\tau \in \fS_n}$, we denote by~$\Cone(\tau)$ and~$\Cone\polar(\tau)$ the incidence and braid cone of the chain~$\tau(1) \to \dots \to \tau(n)$. Finally, for a sign vector~$\chi \in \pm^{n-1}$, we denote by~$\Cone(\tau)$ and~$\Cone\polar(\tau)$ the incidence and braid cone of the oriented path~$1 - \dots - n$, where~$i \to i+1$ if~$\chi_i = +$ and $i \leftarrow i+1$ if~$\chi_i = -$.

\vspace{1.5cm}
These cones (together with all their faces) form complete simplicial fans in~$\HH$:
\begin{enumerate}[(i)]
\item the cones~$\Cone\polar(\tau)$, for all permutations~$\tau \in \fS_n$, form the \defn{braid fan}, which is the normal fan of the \defn{permutahedron}~$\Perm \eqdef \conv\bigset{\sum_{i \in [n]} \tau(i) e_i}{\tau \in \fS_n}$;
\item the cones~$\Cone\polar(\tree)$, for all $\signature$-Cambrian trees~$\tree$, form the \defn{$\signature$-Cambrian fan}, which is the normal fan of the \defn{$\signature$-associahedron}~$\Asso$ of C.~Hohlweg and C.~Lange~\cite{HohlwegLange} (see also~\cite{LangePilaud});
\item the cones~$\Cone\polar(\chi)$, for all sign vectors~$\chi \in \pm^{n-1}$, form the \defn{boolean fan}, which is the normal fan of the parallelepiped~$\Para \eqdef \bigset{\b{x} \in \HH}{i(2n+1-i) \le 2 \sum_{j \le i} x_j \le i(i+1) \text{ for all } i \in [n]}$.
\end{enumerate}
In fact, $\Asso$ is obtained by deleting certain inequalities in the facet description of~$\Perm$, and similarly, $\Para$ is obtained by deleting facets of~$\Asso$. In particular, we have the geometric inclusions~$\Perm \subset \Asso \subset \Para$. See \fref{fig:permutahedraAssociahedraCubes} for $3$-dimensional examples.

\begin{figure}[h]
  \centerline{
  	\begin{overpic}[width=\textwidth]{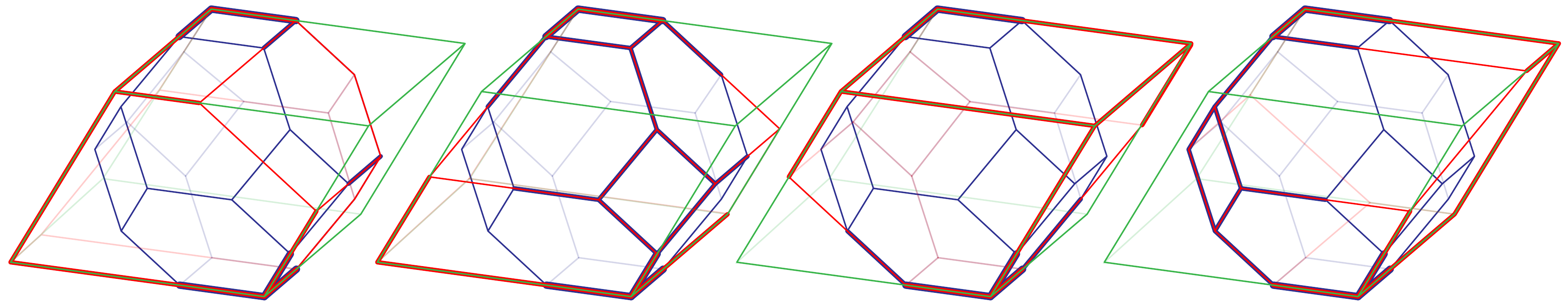}
			\put(12,-2){${-}{-}{-}{-}$}
			\put(36,-2){${-}{+}{-}{-}$}
			\put(59,-2){${-}{-}{+}{-}$}
			\put(82,-2){${-}{+}{+}{-}$}
	\end{overpic}
  }
  \vspace{.2cm}
  \caption{The polytope inclusion~$\Perm[4] \subset \Asso \subset \Para[4]$ for different signatures~$\signature \in \pm^4$. The permutahedron~$\Perm[4]$ is represented in red, the associahedron~$\Asso$ in blue, and the parallelepiped~$\Para[4]$ in green.}
  \label{fig:permutahedraAssociahedraCubes}
\end{figure}

\enlargethispage{-.4cm}
The incidence and braid cones also characterize the maps~$\surjectionPermAsso$, $\surjectionAssoPara$, and~$\surjectionPermPara$ as follows
\begin{gather*}
\tree = \surjectionPermAsso(\tau) \iff \Cone(\tree) \subseteq \Cone(\tau) \iff \Cone\polar(\tree) \supseteq \Cone\polar(\tau), \\
\chi = \surjectionAssoPara(\tree) \iff \Cone(\chi) \subseteq \Cone(\tree) \iff \Cone\polar(\chi) \supseteq \Cone\polar(\tree), \\
\chi = \surjectionPermPara(\tau) \iff \Cone(\chi) \subseteq \Cone(\tau) \iff \Cone\polar(\chi) \supseteq \Cone\polar(\tau).
\end{gather*}
In particular, Cambrian classes are formed by all permutations whose braid cone belong to the same Cambrian cone. Finally, the $1$-skeleta of the permutahedron~$\Perm$, associahedron~$\Asso$ and parallelepiped~$\Para$, oriented in the direction~$(n, \dots, 1) - (1, \dots, n) = \sum_{i \in [n]} (n+1-2i) \, e_i$ are the Hasse diagrams of the weak order, the Cambrian lattice and the boolean lattice respectively. These geometric properties originally motivated the definition of Cambrian trees in~\cite{LangePilaud}.


\section{Cambrian Hopf Algebra}
\label{sec:CambrianAlgebra}

In this section, we introduce the Cambrian Hopf algebra~$\Camb$ as a subalgebra of the Hopf algebra~$\FQSym_\pm$ on signed permutations, and the dual Cambrian algebra~$\Camb^*$ as a quotient algebra of the dual Hopf algebra~$\FQSym_\pm^*$. We describe both the product and coproduct in these algebras in terms of combinatorial operations on Cambrian trees. These results extend the approach of F.~Hivert, J.-C.~Novelli and~J.-Y.~Thibon~\cite{HivertNovelliThibon-algebraBinarySearchTrees} to construct the algebra of \mbox{J.-L.~Loday} and M.~Ronco on binary trees~\cite{LodayRonco} as a subalgebra of the algebra of C.~Malvenuto and C.~Reutenauer on permutations~\cite{MalvenutoReutenauer}.

We immediately mention that a different generalization was studied by N.~Reading in~\cite{Reading-HopfAlgebras}. His idea was to construct a subalgebra of C.~Malvenuto and C.~Reutenauer's algebra~$\FQSym$ using equivalent classes of a congruence relation defined as the union~$\bigcup_{n \in \N} \equiv_{\signature_n}$ of $\signature_n$-Cambrian relation for one fixed signature~$\signature_n \in \pm^n$ for each~$n \in \N$. In order to obtain a valid Hopf algebra, the choice of~$(\signature_n)_{n \in \N}$ has to satisfy certain compatibility relations: N.~Reading characterizes the ``translational'' (resp.~``insertional'') families~$\equiv_n$ of lattice congruences on~$\fS_n$ for which the sums over the elements of the congruence classes of~$(\equiv_n)_{n \in \N}$ form the basis of a subalgebra (resp.~subcoalgebra) of~$\FQSym$. These conditions make the choice of~$(\signature_n)_{n \in \N}$ rather constrained. In contrast, by constructing a subalgebra of~$\FQSym_\pm$ rather than~$\FQSym$, we consider simultaneously all Cambrian relations for all signatures. In particular, our Cambrian algebra contains all Hopf algebras of~\cite{Reading-HopfAlgebras} as subalgebras.


\subsection{Signed shuffle and convolution products}
\label{subsec:products}

\enlargethispage{-.2cm}
For~$n,n' \in \N$, let
\[
\fS^{(n,n')} \eqdef \set{\tau \in \fS_{n+n'}}{\tau(1) < \dots < \tau(n) \text{ and } \tau(n+1) < \dots < \tau(n+n')}
\]
denote the set of permutations of~$\fS_{n+n'}$ with at most one descent, at position~$n$. 
The \defn{shifted concatenation}~$\tau\bar\tau'$, the \defn{shifted shuffle product}~$\tau \shiftedShuffle \tau'$, and the \defn{convolution product}~$\tau \convolution \tau'$ of two (unsigned) permutations~$\tau \in \fS_n$ and~$\tau' \in \fS_{n'}$ are classically defined by
\begin{gather*}
\tau\bar\tau' \eqdef [\tau(1), \dots, \tau(n), \tau'(1) + n, \dots, \tau'(n') + n] \in \fS_{n+n'}, \\
\tau \shiftedShuffle \tau' \eqdef \bigset{(\tau\bar\tau') \circ \pi^{-1}}{\pi \in \fS^{(n,n')}} 
\qquad\text{and}\qquad
\tau \convolution \tau' \eqdef \bigset{\pi \circ (\tau\bar\tau')}{\pi \in \fS^{(n,n')}}.
\end{gather*}
For example,
\begin{align*}
{\red 12} \shiftedShuffle {\darkblue 231} & = \{ {\red 12}{\darkblue 453}, {\red 1}{\darkblue 4}{\red 2}{\darkblue 53}, {\red 1}{\darkblue 45}{\red 2}{\darkblue 3}, {\red 1}{\darkblue 453}{\red 2}, {\darkblue 4}{\red 12}{\darkblue 53}, {\darkblue 4}{\red 1}{\darkblue 5}{\red 2}{\darkblue 3}, {\darkblue 4}{\red 1}{\darkblue 53}{\red 2}, {\darkblue 45}{\red 12}{\darkblue 3}, {\darkblue 45}{\red 1}{\darkblue 3}{\red 2}, {\darkblue 453}{\red 12} \}, \\
{\red 12} \convolution {\darkblue 231} & = \{ {\red 12}{\darkblue 453}, {\red 13}{\darkblue 452}, {\red 14}{\darkblue 352}, {\red 15}{\darkblue 342}, {\red 23}{\darkblue 451}, {\red 24}{\darkblue 351}, {\red 25}{\darkblue 341}, {\red 34}{\darkblue 251}, {\red 35}{\darkblue 241}, {\red 45}{\darkblue 231} \}.
\end{align*}
These operations can be visualized graphically on the tables of the permutations~$\tau, \tau'$. Remember that the table of~$\tau$ contains a dot at coordinates~$(\tau(i),i)$ for each~$i \in [n]$. The table of the shifted concatenation~$\tau\bar\tau'$ contains the table of~$\tau$ as the bottom left block and the table of~$\tau'$ as the top right block. The tables in the shifted shuffle product~$\tau \shiftedShuffle \tau'$ (resp.~in the convolution product~$\tau \convolution \tau'$) are then obtained by shuffling the rows (resp.~columns) of the table of~$\tau\bar\tau'$. In particular, we obtain the table of~$\tau$ if we erase all dots in the~$n'$ rightmost columns (resp. topmost rows) of a table in the shifted shuffle product~$\tau \shiftedShuffle \tau'$ (resp.~in the convolution product~$\tau \convolution \tau'$). See~\fref{fig:shuffleConvolution}.

\begin{figure}[b]
  \centerline{\includegraphics{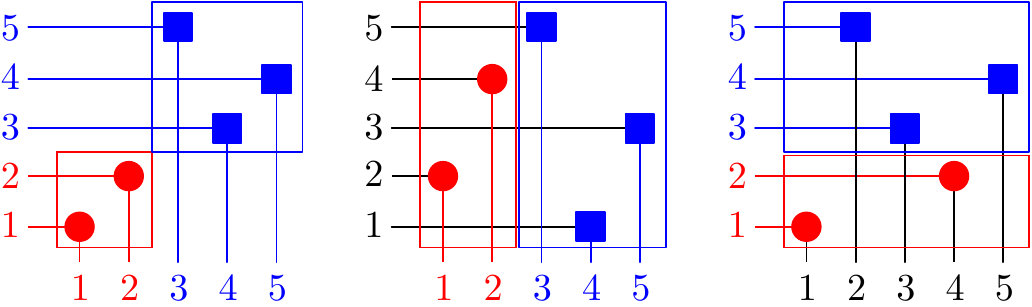}}
  \caption{The table of the shifted concatenation~$\tau\bar\tau'$ (left) has two blocks containing the tables of the permutations~$\tau = 12$ and~$\tau' = 231$. Elements of the shifted shuffle product~$\tau \shiftedShuffle \tau'$ (middle) and of the convolution product~$\tau \convolution \tau'$~(right) are obtained by shuffling respectively the rows and columns of the table of~$\tau\bar\tau'$.}
  \label{fig:shuffleConvolution}
\end{figure}

These definitions extend to signed permutations. The \defn{signed shifted shuffle product}~$\tau \shiftedShuffle \tau'$ is defined as the shifted product of the permutations where signs travel with their values, while the \defn{signed convolution product}~$\tau \convolution \tau'$ is defined as the convolution product of the permutations where signs stay at their positions. For example,
\begin{align*}
\upr{1}\downr{2} \shiftedShuffle \downb{23}\upb{1} & = \{ \upr{1}\downr{2}\downb{45}\upb{3}, \upr{1}\downb{4}\downr{2}\downb{5}\upb{3}, \upr{1}\downb{45}\downr{2}\upb{3}, \upr{1}\downb{45}\upb{3}\downr{2}, \downb{4}\upr{1}\downr{2}\downb{5}\upb{3}, \downb{4}\upr{1}\downb{5}\downr{2}\upb{3}, \downb{4}\upr{1}\downb{5}\upb{3}\downr{2}, \downb{45}\upr{1}\downr{2}\upb{3}, \downb{45}\upr{1}\upb{3}\downr{2}, \downb{45}\upb{3}\upr{1}\downr{2} \}, \\
\upr{1}\downr{2} \convolution \downb{23}\upb{1} & = \{ \upr{1}\downr{2}\downb{45}\upb{3}, \upr{1}\downr{3}\downb{45}\upb{2}, \upr{1}\downr{4}\downb{35}\upb{2}, \upr{1}\downr{5}\downb{34}\upb{2}, \upr{2}\downr{3}\downb{45}\upb{1}, \upr{2}\downr{4}\downb{35}\upb{1}, \upr{2}\downr{5}\downb{34}\upb{1}, \upr{3}\downr{4}\downb{25}\upb{1}, \upr{3}\downr{5}\downb{24}\upb{1}, \upr{4}\downr{5}\downb{23}\upb{1} \}.
\end{align*}
Note that the shifted shuffle is compatible with signed values, while the convolution is compatible with signed positions in the sense that
\[
\fS^\signature \shiftedShuffle \fS^{\signature'} = \fS^{\signature\signature'}
\qquad\text{while}\qquad
\fS_\signature \convolution \fS_{\signature'} = \fS_{\signature\signature'}.
\]
In any case, both~$\shiftedShuffle$ and~$\convolution$ are compatible with the distribution of positive and negative signs,~\ie
\[
|\tau \shiftedShuffle \tau'|_+ = |\tau|_+ + |\tau'|_+ = |\tau \convolution \tau'|_+
\qquad\text{and}\qquad
|\tau \shiftedShuffle \tau'|_- = |\tau|_- + |\tau'|_- = |\tau \convolution \tau'|_-.
\]


\subsection{Subalgebra of $\FQSym_\pm$}
\label{subsec:subalgebra}

We denote by~$\FQSym_\pm$ the Hopf algebra with basis~$(\F_\tau)_{\tau \in \fS_\pm}$ and whose product and coproduct are defined by
\[
\F_\tau \product \F_{\tau'} = \sum_{\sigma \in \tau \shiftedShuffle \tau'} \F_\sigma
\qquad\text{and}\qquad
\coproduct \F_\sigma = \sum_{\sigma \in \tau \convolution \tau'} \F_\tau \otimes \F_{\tau'}.
\]
This Hopf algebra is bigraded by the size and the number of positive signs of the signed permutations.
It naturally extends to signed permutations the Hopf algebra~$\FQSym$ on permutations defined by C.~Malvenuto and C.~Reute\-nauer~\cite{MalvenutoReutenauer}.

We denote by~$\Camb$ the vector subspace of~$\FQSym_\pm$ generated by the elements
\[
\PCamb_{\tree} \eqdef \sum_{\substack{\tau \in \fS_\pm \\ \surjectionPermAsso(\tau) = \tree}} \F_\tau = \sum_{\tau \in \linearExtensions(\tree)} \F_\tau,
\]
for all Cambrian trees~$\tree$.
For example, for the Cambrian tree of \fref{fig:leveledCambrianTree}\,(left), we have
\[
\PCamb_{\!\!\Tex} = \begin{array}[t]{c}
\phantom{+}\F_{\down{21}\up{37}\down{54}\up{6}} + \F_{\down{21}\up{73}\down{54}\up{6}} + \F_{\down{21}\up{7}\down{5}\up{3}\down{4}\up{6}} + \F_{\down{2}\up{7}\down{1}\up{3}\down{54}\up{6}} + \F_{\down{2}\up{7}\down{15}\up{3}\down{4}\up{6}} \\
+ \; \F_{\down{2}\up{7}\down{51}\up{3}\down{4}\up{6}} + \F_{\up{7}\down{21}\up{3}\down{54}\up{6}} + \F_{\up{7}\down{215}\up{3}\down{4}\up{6}} + \F_{\up{7}\down{251}\up{3}\down{4}\up{6}} + \F_{\up{7}\down{521}\up{3}\down{4}\up{6}}.
\end{array}
\]

\begin{theorem}
\label{thm:cambSubalgebra}
$\Camb$ is a Hopf subalgebra of~$\FQSym_\pm$.
\end{theorem}

\begin{proof}
We first prove that $\Camb$ is a subalgebra of~$\FQSym_\pm$. To do this, we just need to show that the Cambrian congruence is compatible with the shuffle product, \ie that the product of two Cambrian classes can be decomposed into a sum of Cambrian classes. Consider two signatures~${\signature \in \pm^n}$ and~$\signature' \in \pm^{n'}$, two Cambrian trees~$\tree \in \CambTrees(\signature)$ and~$\tree' \in \CambTrees(\signature')$, and two congruent permutations~$\sigma \equiv_{\signature\signature'} \tilde\sigma \in \fS^{\signature\signature'}$. We want to show that~$\F_\sigma$ appears in the product~$\PCamb_{\tree} \product \PCamb_{\tree'}$ if and only if~$\F_{\tilde\sigma}$ does. We can assume that~$\sigma = UacVbW$ and~$\tilde\sigma = UcaVbW$ for some letters~$a < b < c$ and words~$U,V,W$ with~$(\signature\signature')_b = -$. Suppose moreover that~$\F_\sigma$ appears in the product~$\PCamb_{\tree} \product \PCamb_{\tree'}$, and let~$\tau \in \linearExtensions(\tree)$ and~$\tau' \in \linearExtensions(\tree')$ such that~$\sigma \in \tau \shiftedShuffle \tau'$. We distinguish three cases:
\begin{enumerate}[(i)]
\item If~$a \le n$ and~$n < c$, then~$\tilde\sigma$ also belongs~$\tau \shiftedShuffle \tau'$, and thus~$\F_{\tilde\sigma}$ appears in the product~$\PCamb_{\tree} \product \PCamb_{\tree'}$.
\item If~$a < b < c \le n$, then~$\tau = \hat U ac \hat V b \hat W$ is $\signature$-congruent to~$\tilde\tau = \hat U ca \hat V b \hat W$, and thus~$\tilde\tau \in \linearExtensions(\tree)$. Since~$\tilde\sigma \in \tilde\tau \shiftedShuffle \tau'$, we obtain that~$\F_{\tilde\sigma}$ appears in the product~$\PCamb_{\tree} \product \PCamb_{\tree'}$.
\item If~$n < a < b < c$, the argument is similar, exchanging~$ac$ to~$ca$ in~$\tau'$.
\end{enumerate}
The proof for the other rewriting rule of Definition~\ref{def:CambrianCongruence} is symmetric, and the general case for~${\sigma \equiv_{\signature\signature'} \tilde\sigma}$ follows by transitivity.

We now prove that~$\Camb$ is a subcoalgebra of~$\FQSym^{\pm}$. We just need to show that the Cambrian congruence is compatible with the deconcatenation coproduct, \ie that the coproduct of a Cambrian class is a sum of tensor products of Cambrian classes. Consider a Cambrian tree~$\tree \in \CambTrees(\eta)$, and Cambrian congruent permutations~$\tau \equiv_\signature \tilde\tau \in \fS^\signature$ and~$\tau' \equiv_{\signature'} \tilde\tau' \in \fS^{\signature'}$. We want to show that~$\F_\tau \otimes \F_{\tau'}$ appears in the coproduct~$\coproduct(\PCamb_{\tree})$ if and only if~$\F_{\tilde\tau} \otimes \F_{\tilde\tau'}$ does. We can assume that~$\tau = UacVbW$ and~$\tilde\tau = UcaVbW$ for some letters~$a < b < c$ and words~$U,V,W$ with~$\signature_b = -$, while~$\tau' = \tilde\tau'$. Suppose moreover that~$\F_\tau \otimes \F_{\tau'}$ appears in the coproduct~$\coproduct(\PCamb_{\tree})$, \ie that there exists~$\sigma \in (\tau \convolution \tau') \cap \linearExtensions(\tree)$. Since~$\sigma \in \tau \convolution \tau'$, it can be written as~$\sigma = \hat U \hat a \hat c \hat V \hat b \hat W \hat \tau'$ for some letters~$\hat a < \hat b < \hat c$ and words~$\hat U,\hat V,\hat W,\hat \tau'$ with~$\eta_{\hat b} = -$. Therefore~$\tilde\sigma = \hat U \hat c \hat a \hat V \hat b \hat W \hat \tau'$ is $\eta$-congruent to~$\sigma$ and in the convolution product~$\tilde\tau \convolution \tilde\tau'$. It follows that~$\F_{\tilde\tau} \otimes \F_{\tilde\tau'}$ also appears in the coproduct~$\coproduct(\PCamb_{\tree})$. The proofs for the other rewriting rule on~$\tau$, as well as for both rewriting rules on~$\tau'$, are symmetric, and the general case for~$\tau \equiv_\signature \tilde\tau$ and~$\tau' \equiv_{\signature'} \tilde\tau'$ follows by transitivity.
\end{proof}

Another proof of this statement would be to show that the Cambrian congruence yields a $\varphi$-good mono\"id~\cite{Priez}. In the remaining of this section, we provide direct descriptions of the product and coproduct of $\PCamb$-basis elements of~$\Camb$ in terms of combinatorial operations on Cambrian trees.

\para{Product}
The product in the Cambrian algebra can be described in terms of intervals in Cambrian lattices.
Given two Cambrian trees~$\tree, \tree'$, we denote by~$\raisebox{-6pt}{$\tree$}\nearrow \raisebox{4pt}{$\bar \tree'$}$ the tree obtained by grafting the rightmost outgoing leaf of~$\tree$ on the leftmost incoming leaf of~$\tree$ and shifting all labels of~$\tree'$. Note that the resulting tree is~$\signature\signature'$-Cambrian, where~$\signature\signature'$ is the concatenation of the signatures~$\signature = \signature(\tree)$ and~$\signature' = \signature(\tree')$. We define similarly~$\raisebox{4pt}{$\tree$} \nwarrow \raisebox{-6pt}{$\bar \tree'$}$. Examples are given in \fref{fig:exampleProduct}.

\begin{figure}[h]
  \centerline{\includegraphics{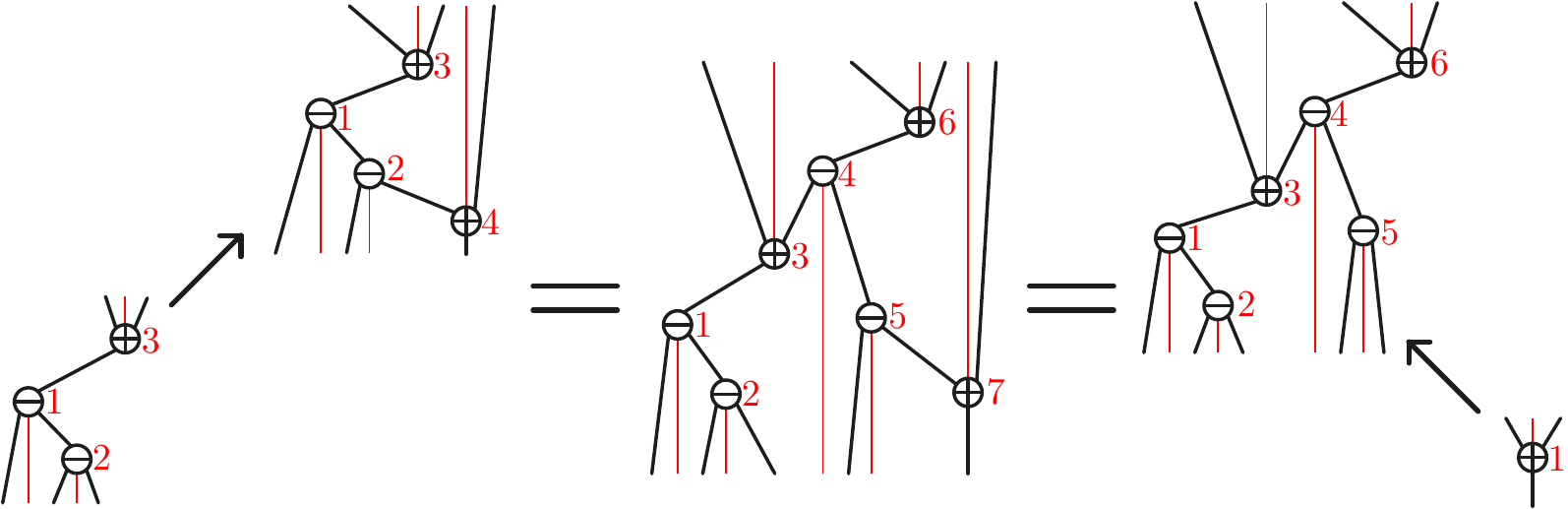}}
  \vspace{-.2cm}
  \caption{Grafting Cambrian trees.}
  \label{fig:exampleProduct}
\end{figure}

\begin{proposition}
\label{prop:product}
For any Cambrian trees~$\tree, \tree'$, the product~$\PCamb_{\tree} \product \PCamb_{\tree'}$ is given by
\[
\PCamb_{\tree} \product \PCamb_{\tree'}  = \sum_{\tree[S]} \PCamb_{\tree[S]},
\]
where~$\tree[S]$ runs over the interval between~$\raisebox{-6pt}{$\tree$}\nearrow \raisebox{4pt}{$\bar \tree'$}$ and~$\raisebox{4pt}{$\tree$} \nwarrow \raisebox{-6pt}{$\bar \tree'$}$ in the $\signature(\tree)\signature(\tree')$-Cambrian lattice.
\end{proposition}

\begin{proof}
For any Cambrian tree~$\tree$, the linear extensions~$\linearExtensions(\tree)$ form an interval of the weak order~\cite{Reading-CambrianLattices}. Moreover, the shuffle product of two intervals of the weak order is an interval of the weak order. Therefore, the product~$\PCamb_{\tree} \product \PCamb_{\tree'}$ is a sum of~$\PCamb_{\tree[S]}$ where~$\tree[S]$ runs over an interval of the Cambrian lattice. It remains to characterize the minimal and maximal elements of this interval.

Let~$\minimalLinearExtension_{\tree}$ and~$\maximalLinearExtension_{\tree}$ denote respectively the smallest and the greatest linear extension of~$\tree$ in weak order. The product~$\PCamb_{\tree} \product \PCamb_{\tree'}$ is the sum of~$\PCamb_{\tree[S]}$ over the interval
\[
[\minimalLinearExtension_{\tree}, \maximalLinearExtension_{\tree}] \shiftedShuffle [\minimalLinearExtension_{\tree'}, \omega_{\tree'}] = [\minimalLinearExtension_{\tree} \bar \minimalLinearExtension_{\tree'}, \bar \maximalLinearExtension_{\tree'} \maximalLinearExtension_{\tree}],
\]
where~$\bar~$ denotes as usual the shifting operator on permutations. The result thus follows from the fact that
\[
\surjectionPermAsso(\minimalLinearExtension_{\tree} \bar \mu_{\tree'}) = \raisebox{-6pt}{$\tree$}\nearrow \raisebox{4pt}{$\bar \tree'$}
\qquad\text{and}\qquad
\surjectionPermAsso(\bar \maximalLinearExtension_{\tree'} \maximalLinearExtension_{\tree}) = \raisebox{4pt}{$\tree$} \nwarrow \raisebox{-6pt}{$\bar \tree'$}.
\qedhere
\]
\end{proof}

For example, we can compute the product
\[
\begin{array}{@{}c@{${} = {}$}c@{+}c@{+}c@{}@{}c}
\PCamb_{\!\!\includegraphics{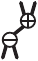}} \product \PCamb_{\includegraphics{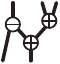}} &
 \multicolumn{3}{l}{\F_{\down{1}\up{2}} \product \big( \F_{\up{2}\down{1}\up{3}} + \F_{\up{2}\up{3}\down{1}} \big)}
 \\[-.4cm]
& \begin{pmatrix} \quad \F_{\down{1}\up{2}\up{4}\down{3}\up{5}} + \F_{\down{1}\up{2}\up{4}\up{5}\down{3}} + \F_{\down{1}\up{4}\up{2}\down{3}\up{5}} \\ + \; \F_{\down{1}\up{4}\up{2}\up{5}\down{3}} + \F_{\down{1}\up{4}\up{5}\up{2}\down{3}} + \F_{\up{4}\down{1}\up{2}\down{3}\up{5}} \\ + \; \F_{\up{4}\down{1}\up{2}\up{5}\down{3}} + \F_{\up{4}\down{1}\up{5}\up{2}\down{3}} + \F_{\up{4}\up{5}\down{1}\up{2}\down{3}} \end{pmatrix}
& \begin{pmatrix} \quad \F_{\down{1}\up{4}\down{3}\up{2}\up{5}} + \F_{\down{1}\up{4}\down{3}\up{5}\up{2}} \\ + \; \F_{\down{1}\up{4}\up{5}\down{3}\up{2}}  + \F_{\up{4}\down{1}\down{3}\up{2}\up{5}} \\ + \; \F_{\up{4}\down{1}\down{3}\up{5}\up{2}} + \F_{\up{4}\down{1}\up{5}\down{3}\up{2}} \\ + \; \F_{\up{4}\up{5}\down{1}\down{3}\up{2}} \end{pmatrix}
& \begin{pmatrix} \quad \F_{\up{4}\down{3}\down{1}\up{2}\up{5}} + \F_{\up{4}\down{3}\down{1}\up{5}\up{2}} \\ + \; \F_{\up{4}\down{3}\up{5}\down{1}\up{2}} + \F_{\up{4}\up{5}\down{3}\down{1}\up{2}} \end{pmatrix}
\\[.8cm]
& \PCamb_{\!\!\includegraphics{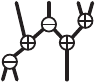}} & \PCamb_{\!\!\includegraphics{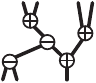}} & \PCamb_{\!\!\includegraphics{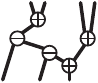}} & .
\end{array}
\]
The first equality is obtained by computing the linear extensions of the two factors, the second by computing the shuffle product and grouping terms according to their $\PSymbol$-symbol, displayed in the last line. Proposition~\ref{prop:product} enables us to shortcut the computation by avoiding to resort to the~$\F$-basis.

\para{Coproduct}
The coproduct in the Cambrian algebra can also be described in combinatorial terms. Define a \defn{cut} of a Cambrian tree~$\tree[S]$ to be a set~$\gamma$ of edges such that any geodesic vertical path in~$\tree[S]$ from a down leaf to an up leaf contains precisely one edge of~$\gamma$. Such a cut separates the tree~$\tree$ into two forests, one above~$\gamma$ and one below~$\gamma$, denoted~$A(\tree[S], \gamma)$ and~$B(\tree[S],\gamma)$, respectively. An example is given in \fref{fig:exampleCoproduct}.

\begin{figure}[h]
  \centerline{\includegraphics{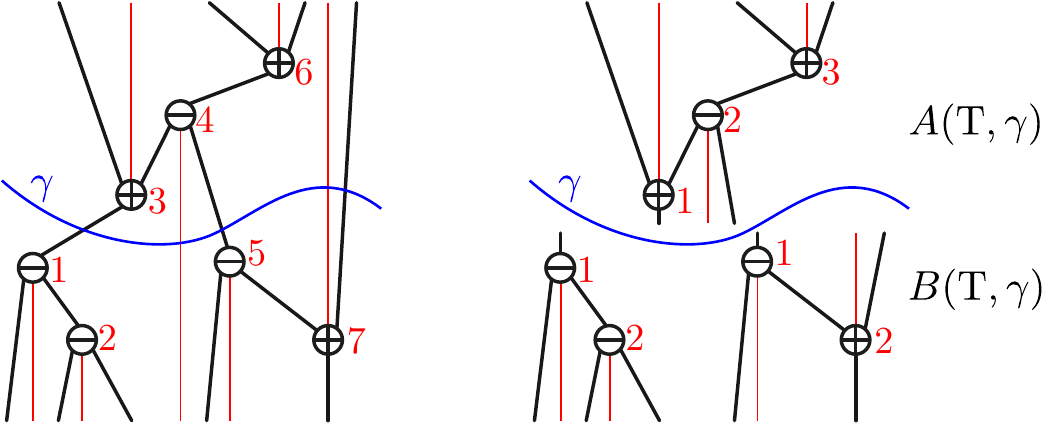}}
  \caption{A cut~$\gamma$ of a Cambrian tree~$\tree$ defines two forests~$A(\tree, \gamma)$ and~$B(\tree,\gamma)$.}
  \label{fig:exampleCoproduct}
\end{figure}

\begin{proposition}
\label{prop:coproduct}
For any Cambrian tree~$\tree[S]$, the coproduct~$\coproduct \PCamb_{\tree[S]}$ is given by
\[
\coproduct \PCamb_{\tree[S]} = \sum_{\gamma} \bigg( \prod_{\tree \in B(\tree[S],\gamma)} \PCamb_{\tree} \bigg) \otimes \bigg( \prod_{\tree' \in A(\tree[S], \gamma)} \PCamb_{\tree'} \bigg),
\]
where~$\gamma$ runs over all cuts of~$\tree[S]$.
\end{proposition}

\begin{proof}
Let~$\sigma$ be a linear extension of~$\tree[S]$ and~$\tau, \tau' \in \fS_\pm$ such that~$\sigma \in \tau \convolution \tau'$. As discussed in Section~\ref{subsec:products}, the tables of~$\tau$ and~$\tau'$ respectively appear in the bottom and top rows of the table of~$\sigma$. We can therefore associate a cut of~$\tree[S]$ to each element which appears in the coproduct~$\coproduct \PCamb_{\tree[S]}$.

Reciprocally, given a cut~$\gamma$ of~$\tree[S]$, we are interested in the linear extensions of~$\tree[S]$ where all indices below~$\gamma$ appear before all indices above~$\gamma$. These linear extensions are precisely the permutations formed by a linear extension of~$B(\tree, \gamma)$ followed by a linear extension of~$A(\tree, \gamma)$. But the linear extensions of a forest are obtained by shuffling the linear extensions of its connected components. The result immediately follows since the product~$\PCamb_{\tree} \product \PCamb_{\tree'}$ precisely involves the shuffle of the linear extensions of~$\tree$ with the linear extensions of~$\tree'$.
\end{proof}

For example, we can compute the coproduct
\[
\begin{array}{@{}c@{${} = {}$}c@{$\,+\,$}c@{$\,+\,$}c@{$\,+\,$}c@{$\,+\,$}c@{$\,+\,$}c@{}}
\coproduct \PCamb_{\includegraphics{exmProductB}} & \multicolumn{6}{l}{\coproduct \big( \F_{\up{2}\down{1}\up{3}} + \F_{\up{2}\up{3}\down{1}} \big)}
\\
& 1 \otimes \big( \F_{\up{2}\down{1}\up{3}} + \F_{\up{2}\up{3}\down{1}} \big)
& \F_{\up{1}} \otimes \F_{\down{1}\up{2}}
& \F_{\up{1}} \otimes \F_{\up{2}\down{1}}
& \F_{\up{2}\down{1}} \otimes \F_{\up{1}}
& \F_{\up{1}\up{2}} \otimes \F_{\down{1}}
& \big( \F_{\up{2}\down{1}\up{3}} + \F_{\up{2}\up{3}\down{1}} \big) \otimes 1
\\[.2cm]
& 1 \otimes \PCamb_{\!\includegraphics{exmProductB}}
& \PCamb_{\includegraphics{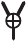}} \otimes \PCamb_{\!\!\includegraphics{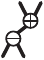}}
& \PCamb_{\includegraphics{exmTreeY}} \otimes \PCamb_{\!\includegraphics{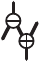}}
& \PCamb_{\!\includegraphics{exmTreeYAg}} \otimes \PCamb_{\includegraphics{exmTreeY}}
& \PCamb_{\includegraphics{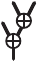}} \otimes \PCamb_{\!\includegraphics{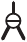}}
& \PCamb_{\!\includegraphics{exmProductB}} \otimes 1
\\
& 1 \otimes \PCamb_{\!\includegraphics{exmProductB}}
& \multicolumn{2}{c@{$\,+\,$}}{\PCamb_{\includegraphics{exmTreeY}} \otimes \big( \PCamb_{\!\includegraphics{exmTreeA}} \product \PCamb_{\includegraphics{exmTreeY}} \big)}
& \PCamb_{\!\includegraphics{exmTreeYAg}} \otimes \PCamb_{\includegraphics{exmTreeY}}
& \PCamb_{\includegraphics{exmTreeYYd}} \otimes \PCamb_{\!\includegraphics{exmTreeA}}
& \PCamb_{\!\includegraphics{exmProductB}} \otimes 1.
\end{array}
\]
\enlargethispage{-.2cm}
Proposition~\ref{prop:coproduct} enables us to shortcut the computation by avoiding to resort to the $\F$-basis. We compute directly the last line, which corresponds to the five possible cuts of the Cambrian tree~\raisebox{-.3cm}{\includegraphics{exmProductB}}.

\para{Matriochka algebras}
To conclude, we connect the Cambrian algebra to the recoils algebra~$\Rec$, defined as the Hopf subalgebra of~$\FQSym_\pm$ generated by the elements
\[
\XRec_\chi \eqdef \sum_{\substack{\tau \in \fS_\pm \\ \surjectionPermPara(\tau) = \chi}} \F_\tau
\]
for all sign vectors~$\chi \in \pm^{n-1}$. The commutative diagram of Proposition~\ref{prop:commutativeDiagram} ensures~that
\[
\XRec_\chi = \sum_{\substack{\tree \in \CambTrees \\ \surjectionAssoPara(\tree) = \chi}} \PCamb_{\tree},
\]
and thus that~$\Rec$ is a subalgebra of~$\Camb$. In other words, the Cambrian algebra is sandwiched between the signed permutation algebra and the recoils algebra~$\Rec \subset \Camb \subset \FQSym_\pm$. This property has to be compared with the polytope inclusions discussed in Section~\ref{subsec:geometricRealizations}.


\subsection{Quotient algebra of~$\FQSym_\pm^*$}
\label{subsec:quotientAlgebra}

We switch to the dual Hopf algebra~$\FQSym_\pm^*$ with basis $(\G_\tau)_{\tau \in \fS_\pm}$ and whose product and coproduct are defined by
\[
\G_\tau \product \G_{\tau'} = \sum_{\sigma \in \tau \convolution \tau'} \G_\sigma
\qquad\text{and}\qquad
\coproduct \G_\sigma = \sum_{\sigma \in \tau \shiftedShuffle \tau'} \G_\tau \otimes \G_{\tau'}.
\]
The following statement is automatic from Theorem~\ref{thm:cambSubalgebra}.

\begin{theorem}
The graded dual~$\Camb^*$ of the Cambrian algebra is isomorphic to the image of~$\FQSym_\pm^*$ under the canonical projection
\[
\pi : \C\langle A \rangle \longrightarrow \C\langle A \rangle / \equiv,
\]
where~$\equiv$ denotes the Cambrian congruence. The dual basis~$\QCamb_{\tree}$ of~$\PCamb_{\tree}$ is expressed as~$\QCamb_{\tree} = \pi(\G_\tau)$, where~$\tau$ is any linear extension of~$\tree$.
\end{theorem}

Similarly as in the previous section, we can describe combinatorially the product and coproduct of $\QCamb$-basis elements of~$\Camb^*$ in terms of operations on Cambrian trees.

\para{Product}
Call \defn{gaps} the $n+1$ positions between two consecutive integers of~$[n]$, including the position before~$1$ and the position after~$n$. A gap~$\gamma$ defines a \defn{geodesic vertical path}~$\lambda(\tree,\gamma)$ in a Cambrian tree~$\tree$ from the bottom leaf which lies in the same interval of consecutive negative labels as~$\gamma$ to the top leaf which lies in the same interval of consecutive positive labels as~$\gamma$. See \fref{fig:exampleCoproductDual}. A multiset~$\Gamma$ of gaps therefore defines a \defn{lamination}~$\lambda(\tree,\Gamma)$ of~$\tree$, \ie a multiset of pairwise non-crossing geodesic vertical paths in~$\tree$ from down leaves to up leaves. When cut along the paths of a lamination, the Cambrian tree~$\tree$ splits into a forest.

Consider two Cambrian trees~$\tree$ and~$\tree'$ on~$[n]$ and~$[n']$ respectively. For any shuffle~$s$ of their signatures~$\signature$ and~$\signature'$, consider the multiset~$\Gamma$ of gaps of~$[n]$ given by the positions of the negative signs of~$\signature'$ in~$s$ and the multiset~$\Gamma'$ of gaps of~$[n']$ given by the positions of the positive signs of~$\signature$ in~$s$. We denote by~$\tree \,{}_s\!\backslash \tree'$ the Cambrian tree obtained by connecting the up leaves of the forest defined by the lamination~$\lambda(\tree,\Gamma)$ to the down leaves of the forest defined by the lamination~$\lambda(\tree',\Gamma')$.

\begin{example}
Consider the Cambrian trees~$\tree^\blueCirc$ and~$\tree^\redSquare$ of \fref{fig:exampleProductDual}. To distinguish signs in~$\tree^\blueCirc$ and~$\tree^\redSquare$, we circle the signs in~$\signature(\tree^\blueCirc) = \blueMinus\blueMinus\bluePlus$ and square the signs in~$\signature(\tree^\redSquare) = \redMinus\redMinus\redPlus\redMinus$. Consider now an arbitrary shuffle~$s = \redMinus\blueMinus\blueMinus\redMinus\redPlus\bluePlus\redMinus$ of these two signatures. The resulting laminations of~$\tree^\blueCirc$ and~$\tree^\redSquare$, as well as the Cambrian tree~$\tree^\blueCirc {}_s\!\backslash \tree^\redSquare$ are represented in \fref{fig:exampleProductDual}.

\begin{figure}[h]
  \centerline{\includegraphics{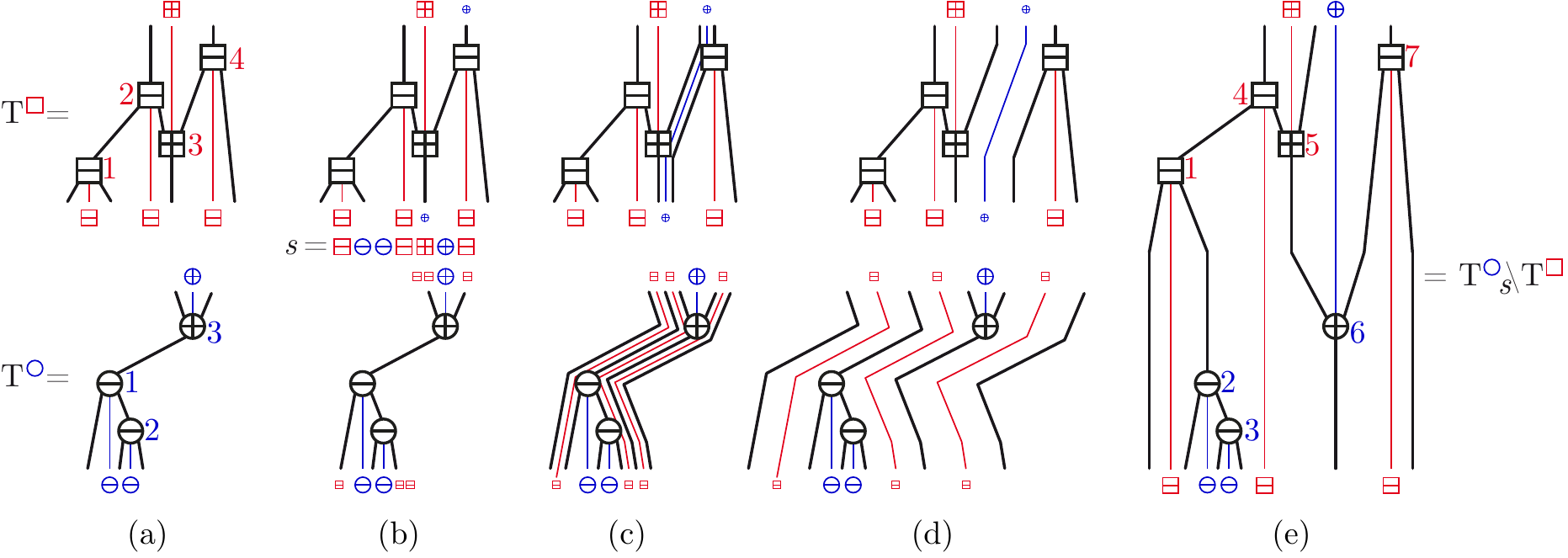}}
  \caption[Combinatorial interpretation of the product in the~$\QCamb$-basis of~$\Camb^*$.]{(a) The two Cambrian trees~$\tree^\blueCirc$ and~$\tree^\redSquare$. (b) Given the shuffle $s = \redMinus\blueMinus\blueMinus\redMinus\redPlus\bluePlus\redMinus$, the positions of the~$\redMinus$ are reported in~$\tree^\blueCirc$ and the positions of the~$\bluePlus$ are reported in~$\tree^\redSquare$. (c) The corresponding laminations. (d) The trees are split according to the laminations. (e) The resulting Cambrian tree~$\tree^\blueCirc {}_s\!\backslash \tree^\redSquare$.}
  \label{fig:exampleProductDual}
\end{figure}
\end{example}

\begin{proposition}
\label{prop:productDual}
For any Cambrian trees~$\tree, \tree'$, the product~$\QCamb_{\tree} \product \QCamb_{\tree'}$ is given by
\[
\QCamb_{\tree} \product \QCamb_{\tree'} = \sum_s \QCamb_{\tree \,{}_s\!\backslash \tree'},
\]
where~$s$ runs over all shuffles of the signatures of~$\tree$ and~$\tree'$.
\end{proposition}

\begin{proof}
Let~$\tau$ and~$\tau'$ be linear extensions of~$\tree$ and~$\tree'$ respectively, let~$\sigma \in \tau \convolution \tau'$ and let~$\tree[S] = \surjectionPermAsso(\sigma)$. As discussed in Section~\ref{subsec:products}, the convolution~$\tau \convolution \tau'$ shuffles the columns of the tables of~$\tau$ and~$\tau'$ while preserving the order of their rows. According to the description of the insertion algorithm~$\CambCorresp$, the tree~$\tree[S]$ thus consists in~$\tree$ below and~$\tree'$ above, except that the vertical walls falling from the negative nodes of~$\tree'$ split~$\tree$ and similarly the vertical walls rising from the positive nodes of~$\tree$ split~$\tree'$. This corresponds to the description of~$\tree \,{}_s\!\backslash \tree'$, where~$s$ is the shuffle of the signatures of~$\tree$ and~$\tree'$ given by~$\sigma$.
\end{proof}

For example, we can compute the product
\[
\hspace*{-1cm}\begin{array}{@{}c@{${} = {}$}c@{$\,+\,$}c@{$\,+\,$}c@{$\,+\,$}c@{$\,+\,$}c@{$\,+\,$}c@{$\,+\,$}c@{$\,+\,$}c@{$\,+\,$}c@{$\,+\,$}c@{\;}c@{}}
\QCamb_{\includegraphics{exmProductA}} \product \QCamb_{\includegraphics{exmProductB}} &
 \multicolumn{10}{l}{\G_{\down{1}\up{2}} \product \G_{\up{2}\down{1}\up{3}}}
 \\[-.2cm]
& \G_{\down{1}\up{2}\up{4}\down{3}\up{5}} & \G_{\down{1}\up{3}\up{4}\down{2}\up{5}} & \G_{\down{1}\up{4}\up{3}\down{2}\up{5}} & \G_{\down{1}\up{5}\up{3}\down{2}\up{4}} & \G_{\down{2}\up{3}\up{4}\down{1}\up{5}} & \G_{\down{2}\up{4}\up{3}\down{1}\up{5}} & \G_{\down{2}\up{5}\up{3}\down{1}\up{4}} & \G_{\down{3}\up{4}\up{2}\down{1}\up{5}} & \G_{\down{3}\up{5}\up{2}\down{1}\up{4}} & \G_{\down{4}\up{5}\up{2}\down{1}\up{3}}
\\[.2cm]
& \QCamb_{\!\includegraphics{exmProduct1}} & \QCamb_{\!\includegraphics{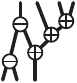}} & \QCamb_{\!\includegraphics{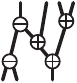}} & \QCamb_{\!\includegraphics{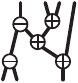}} & \QCamb_{\!\includegraphics{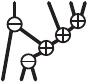}} & \QCamb_{\!\includegraphics{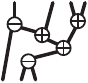}} & \QCamb_{\!\includegraphics{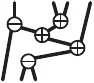}} & \QCamb_{\!\includegraphics{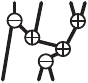}} & \QCamb_{\!\includegraphics{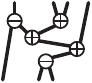}} & \QCamb_{\!\includegraphics{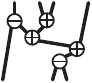}} & .
\end{array}
\]
Note that the~$10$ resulting Cambrian trees correspond to the~$10$ possible shuffles of~${-}{+}$ and~${-}{+}{+}$.

\para{Coproduct}
To describe the coproduct of~$\QCamb$-basis elements of~$\Camb^*$, we also use gaps and vertical paths in Cambrian trees. Namely, for a gap~$\gamma$, we denote by~$L(\tree[S],\gamma)$ and~$R(\tree[S],\gamma)$ the left and right Cambrian subtrees of~$\tree[S]$ when split along the path~$\lambda(\tree[S], \gamma)$. An example is given in \fref{fig:exampleCoproductDual}.

\begin{figure}[b]
  \centerline{\includegraphics{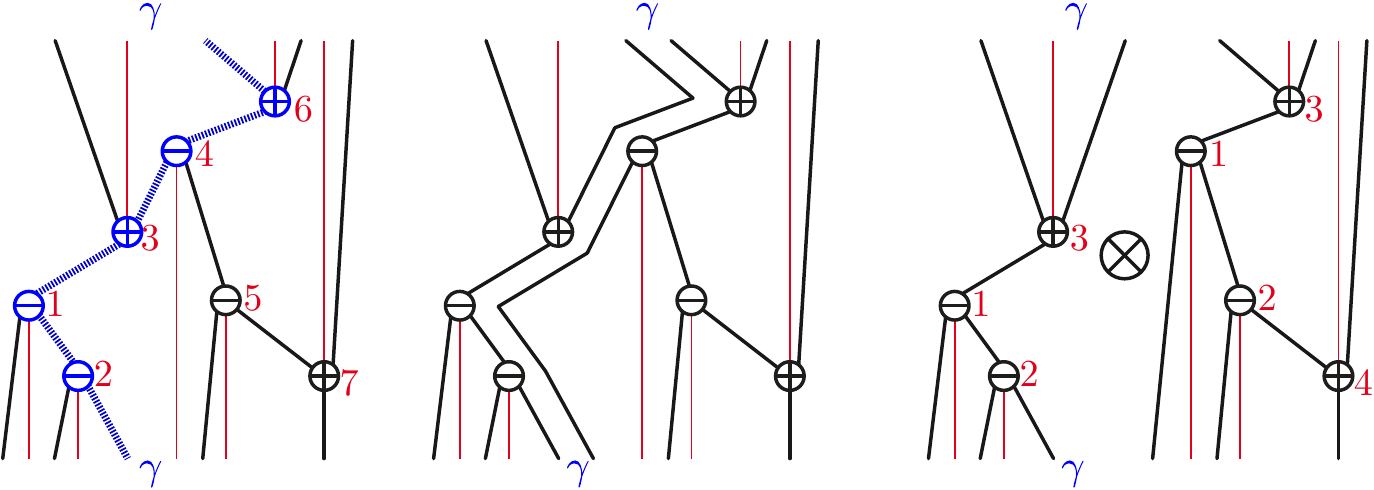}}
  \caption{A gap~$\gamma$ between~$3$ and~$4$ (left) defines a vertical cut (middle) which splits the Cambrian tree (right).}
  \label{fig:exampleCoproductDual}
\end{figure}

\begin{proposition}
\label{prop:coproductDual}
For any Cambrian tree~$\tree[S]$, the coproduct~$\coproduct\QCamb_{\tree[S]}$ is given by
\[
\coproduct\QCamb_{\tree[S]} = \sum_{\gamma} \QCamb_{L(\tree[S],\gamma)} \otimes \QCamb_{R(\tree[S],\gamma)},
\]
where~$\gamma$ runs over all gaps between vertices of~$\tree[S]$.
\end{proposition}

\begin{proof}
Let~$\sigma$ be a linear extension of~$\tree[S]$ and~$\tau, \tau' \in \fS_\pm$ such that~$\sigma \in \tau \shiftedShuffle \tau'$. As discussed in Section~\ref{subsec:products}, $\tau$ and~$\tau'$ respectively appear on the left and right columns of~$\sigma$. Let~$\gamma$ denote the vertical gap separating~$\tau$ from~$\tau'$. Applying the insertion algorithm to~$\tau$ and~$\tau'$ separately then yields the trees~$L(\tree[S],\gamma)$ and~$R(\tree[S],\gamma)$. The description follows.
\end{proof}

For example, we can compute the coproduct
\[
\hspace*{-1cm}\begin{array}{@{}c@{${} = {}$}c@{${} + {}$}c@{${} + {}$}c@{${} + {}$}c@{}}
\coproduct \QCamb_{\includegraphics{exmProductB}} &
 \multicolumn{4}{l}{\coproduct \G_{\up{2}\down{1}\up{3}}}
 \\[-.2cm]
& 1 \otimes \G_{\up{2}\down{1}\up{3}}
& \G_{\down{1}} \otimes \G_{\up{1}\up{2}}
& \G_{\up{2}\down{1}} \otimes \G_{\up{1}}
& \G_{\up{2}\down{1}\up{3}} \otimes 1
\\[.2cm]
& 1 \otimes \QCamb_{\includegraphics{exmProductB}}
& \QCamb_{\includegraphics{exmTreeA}} \otimes \QCamb_{\includegraphics{exmTreeYYd}}
& \QCamb_{\includegraphics{exmTreeYAg}} \otimes \QCamb_{\includegraphics{exmTreeY}}
& \QCamb_{\includegraphics{exmProductB}} \otimes 1.
\end{array}
\]
Note that the last line can indeed be directly computed using the paths defined by the four possible gaps of the Cambrian tree~\raisebox{-.3cm}{\includegraphics{exmProductB}}.


\subsection{Duality}

As proven in~\cite{HivertNovelliThibon-algebraBinarySearchTrees}, the duality~$\tau \mapsto \tau^{-1}$ between the Hopf algebras~$\FQSym$ and~$\FQSym^*$ induces a duality between the Hopf algebras~$\PBT$ and~$\PBT^*$. That is to say that the composition~$\Psi$ of the applications
\[
\begin{array}{ccccccc}
\PBT & \verylonghookrightarrow & \FQSym & \verylongleftrightarrow & \FQSym^* & \verylongtwoheadrightarrow & \PBT^* \\
& \PCamb_{\tree} \mapsto \!\!\sum\limits_{\tau \in \linearExtensions(\tree)} \F_\tau & & \tau \mapsto \tau^{-1} & & \G_\tau \mapsto \QCamb_{\surjectionPermAsso(\tau)}
\end{array}
\]
is an isomorphism between~$\PBT$ and~$\PBT^*$. This property is no longer true for the Cambrian algebra~$\Camb$ and its dual~$\Camb^*$. Namely, the composition~$\Psi$ of the applications
\[
\begin{array}{ccccccc}
\Camb & \verylonghookrightarrow & \FQSym_\pm & \verylongleftrightarrow & \FQSym_\pm^* & \verylongtwoheadrightarrow & \Camb^* \\
& \PCamb_{\tree} \mapsto \!\!\sum\limits_{\tau \in \linearExtensions(\tree)} \F_\tau & & \F_\tau \mapsto \G_\tau^{-1} & & \G_\tau \mapsto \QCamb_{\surjectionPermAsso(\tau)}
\end{array}
\]
is not an isomorphism. It is indeed not injective as
\[
\Psi \big( \PCamb_{\raisebox{-.25cm}{\includegraphics{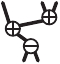}}} \big) = \QCamb_{\raisebox{-.25cm}{\includegraphics{exmProductB}}} = \Psi \big( \PCamb_{\raisebox{-.25cm}{\includegraphics{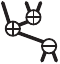}}} \big).
\]
Indeed, their images along the three maps are given by
\[
\begin{array}{ccccccc@{\,}l}
\PCamb_{\raisebox{-.25cm}{\includegraphics{ctrexmDualityA}}} & \longmapsto & \F_{\down{2}\up{13}} & \longmapsto & \G_{\up{2}\down{1}\up{3}} & \longmapsto & \QCamb_{\raisebox{-.25cm}{\includegraphics{exmProductB}}} & , \text{ and}\\
\PCamb_{\raisebox{-.25cm}{\includegraphics{ctrexmDualityB}}} & \longmapsto & \F_{\down{3}\up{12}} & \longmapsto & \G_{\up{23}\down{1}} & \longmapsto & \QCamb_{\raisebox{-.25cm}{\includegraphics{exmProductB}}} & .
\end{array}
\]


\section{Multiplicative bases}
\label{sec:multiplicativeBases}

In this section, we define multiplicative bases of~$\Camb$ and study the indecomposable elements of~$\Camb$ for these bases. We prove in Sections~\ref{subsec:structuralProp} and~\ref{subsec:enumerativeProp} both structural and enumerative properties of the set of indecomposable elements.


\subsection{Multiplicative bases and indecomposable elements}

For a Cambrian tree~$\tree$, we define
\[
\ECamb^{\tree} \eqdef \sum_{\tree \le \tree'} \PCamb_{\tree'}
\qquad\text{and}\qquad
\HCamb^{\tree} \eqdef \sum_{\tree' \le \tree} \PCamb_{\tree'}.
\]
To describe the product of two elements of the~$\ECamb$- or $\HCamb$-basis, remember that the Cambrian trees
\[
\raisebox{-6pt}{$\tree$}\nearrow \raisebox{4pt}{$\bar \tree'$}
\qquad\text{and}\qquad
\raisebox{4pt}{$\tree$} \nwarrow \raisebox{-6pt}{$\bar \tree'$}
\]
are defined to be the trees obtained by shifting all labels of~$\tree'$ and grafting for the first one the rightmost outgoing leaf of~$\tree$ on the leftmost incoming leaf of~$\tree'$, and for the second one the rightmost incoming leaf of~$\tree$ on the leftmost outgoing leaf of~$\tree'$. Examples are given in \fref{fig:exampleProduct}.

\begin{proposition}
\label{prop:productMultiplicativeBases}
$(\ECamb^{\tree})_{\tree \in \CambTrees}$ and~$(\HCamb^{\tree})_{\tree \in \CambTrees}$ are multiplicative bases of~$\Camb$:
\[
\ECamb^{\tree} \product \ECamb^{\tree'} = \ECamb^{\raisebox{-5pt}{\scriptsize$\tree$}\nearrow \raisebox{4pt}{\scriptsize$\bar \tree'$}}
\qquad\text{and}\qquad
\HCamb^{\tree} \product \HCamb^{\tree'} = \HCamb^{\raisebox{4pt}{\scriptsize$\tree$}\nwarrow \raisebox{-5pt}{\scriptsize$\bar \tree'$}}.
\]
\end{proposition}

\begin{proof}
Let~$\maximalLinearExtension_{\tree}$ denote the maximal linear extension of~$\tree$ in weak order. Since~$\bigset{\linearExtensions(\tilde\tree)}{\tilde\tree \le \tree}$ partitions the weak order interval~$[12 \cdots n, \maximalLinearExtension_{\tree}]$, we have
\[
\HCamb^{\tree} = \sum_{\tilde\tree \le \tree} \PCamb_{\tilde\tree} = \sum_{\tilde\tree \le \tree} \sum_{\tau \in \linearExtensions(\tilde\tree)} \F_\tau = \sum_{\tau \le \maximalLinearExtension_{\tree}} \F_\tau.
\]
Since the shuffle product of two intervals of the weak order is an interval of the weak order, the product~$\HCamb^{\tree} \product \HCamb^{\tree'}$ is the sum of~$\F_\tau$ over the interval
\[
[12 \cdots n, \maximalLinearExtension_{\tree}] \shiftedShuffle [12 \cdots n', \maximalLinearExtension_{\tree'}] = [12 \cdots (n+n'), \bar \maximalLinearExtension_{\tree'}\maximalLinearExtension_{\tree}].
\]
The result thus follows from the fact that
\[
\surjectionPermAsso(\bar\maximalLinearExtension_{\tree'}\maximalLinearExtension_{\tree}) = \raisebox{4pt}{$\tree$} \nwarrow \raisebox{-6pt}{$\bar \tree'$}.
\]
The proof is symmetric for~$\ECamb^{\tree}$, replacing lower interval and~$[12 \cdots n, \maximalLinearExtension_{\tree}]$ by the upper interval~$[\minimalLinearExtension_{\tree}, n \cdots 21]$.
\end{proof}

As the multiplicative bases~$(\ECamb^{\tree})_{\tree \in \CambTrees}$ and~$(\HCamb^{\tree})_{\tree \in \CambTrees}$ have symmetric properties, we focus our analysis on the $\ECamb$-basis. The reader is invited to translate the results below to the $\HCamb$-basis. We consider multiplicative decomposability. Remember that an \defn{edge cut} in a Cambrian tree~$\tree[S]$ is the ordered partition~$\edgecut{X}{Y}$ of the vertices of~$\tree[S]$ into the set~$X$ of vertices in the source set and the set~$Y$ of vertices in the target set of an oriented edge~$e$ of~$\tree[S]$.

\begin{proposition}
The following properties are equivalent for a Cambrian tree~$\tree[S]$:
\begin{enumerate}[(i)]
\item $\ECamb^{\tree[S]}$ can be decomposed into a product~$\ECamb^{\tree[S]} = \ECamb^{\tree} \product \ECamb^{\tree'}$ for non-empty Cambrian trees~$\tree, \tree'$; \label{enum:decomposable}
\item $\edgecut{[k]}{[n] \ssm [k]}$ is an edge cut of~$\tree[S]$ for some~$k \in [n]$; \label{enum:cut}
\item at least one linear extension~$\tau$ of~$\tree[S]$ is decomposable, \ie $\tau([k]) = [k]$ for some~$k \in [n]$. \label{enum:perm}
\end{enumerate}
The tree~$\tree[S]$ is then called \defn{$\ECamb$-decomposable} and the edge cut~$\edgecut{[k]}{[n] \ssm [k]}$ is called \defn{splitting}.
\end{proposition}

\begin{proof}
The equivalence~\eqref{enum:decomposable} $\iff$ \eqref{enum:cut} is an immediate consequence of the description of the product~$\ECamb^{\tree} \product \ECamb^{\tree'}$ in Proposition~\ref{prop:productMultiplicativeBases}. The implication \eqref{enum:cut} $\implies$ \eqref{enum:perm} follows from the fact that for any cut~$\edgecut{X}{Y}$ of a directed acyclic graph~$G$, there exists a linear extension of~$G$ which starts with~$X$ and finishes with~$Y$. Reciprocally, if~$\tau$ is a decomposable linear extension of~$\tree[S]$, then the insertion algorithm creates two blocks and necessarily relates the bottom-left block to the top-right block by a splitting edge.
\end{proof}

For example, \fref{fig:exampleProduct} shows that~$\surjectionPermAsso(\down{2}\up{7}\down{51}\up{3}\down{4}\up{6})$ is both $\ECamb$- and $\HCamb$-decomposable. In the remaining of this section, we study structural and enumerative properties of $\ECamb$-indecompo\-sable elements of~$\CambTrees$. We denote by~$\indecomposables_\signature$ the set of $\ECamb$-indecomposable elements of~$\CambTrees(\signature)$.

\begin{example}
\label{exm:rightTilting}
For~$\signature = (-)^n$, the $\ECamb$-indecomposable $\signature$-Cambrian trees are \defn{right-tilting} binary trees, \ie binary trees whose root has no left child. Similarly, for~$\signature = (+)^n$, the $\ECamb$-indecomposable $\signature$-Cambrian trees are \defn{left-tilting} binary trees oriented upwards. See \fref{fig:minIndecomposable} for illustrations.
\end{example}


\subsection{Structural properties}
\label{subsec:structuralProp}

The objective of this section is to prove the following property of the $\ECamb$-indecomposable elements of~$\CambTrees(\signature)$.

\begin{proposition}
\label{prop:upperIdeal}
For any signature~$\signature \in \pm^n$, the set~$\indecomposables_\signature$ of~$\ECamb$-indecomposable $\signature$-Cambrian trees forms a principal upper ideal of the $\signature$-Cambrian lattice.
\end{proposition}

To prove this statement, we need the following result.

\begin{lemma}
\label{lem:rotationIndecomposable}
Let~$\tree$ be a $\signature$-Cambrian tree, let~$i \to j$ be an edge of~$\tree$ with~$i < j$, and let~$\tree'$ be the $\signature$-Cambrian tree obtained by rotating~$i \to j$ in~$\tree$. Then
\begin{enumerate}[(i)]
\item if~$\tree$ is $\ECamb$-indecomposable, then so is~$\tree'$;
\item if~$\tree$ is $\ECamb$-decomposable while~$\tree'$ is not, then~$\signature_i = +$ or~$i = 1$, and~$\signature_j = -$ or~$j = n$.
\end{enumerate}
\end{lemma}

\begin{proof}
As observed in Proposition~\ref{prop:rotation}, the Cambrian trees~$\tree$ and~$\tree'$ have the same edge cuts, except the cut defined by edge~$i \to j$. Using notations of \fref{fig:rotation}, the edge cut~$C \eqdef \edgecut{i \cup L \cup B}{j \cup R \cup A}$ of~$\tree$ is replaced by the edge cut~$C' \eqdef \edgecut{j \cup R \cup B}{i \cup L \cup A}$ of~$\tree'$. Since $i < j$, the edge cut~$C'$ cannot be splitting. Therefore, $\tree'$ is always $\ECamb$-indecomposable when~$\tree$ is $\ECamb$-indecomposable.

Assume conversely that~$\tree$ is $\ECamb$-decomposable while $\tree'$ is not. This implies that~$C$ is splitting while~$C'$ is not. Since~$C$ is splitting we have~$i \cup L \cup B < j \cup R \cup A$ (where we write~$X < Y$ if~$x < y$ for all~$x \in X$ and~$y \in Y$). If~$\signature_i = -$, then~$L < i < B$, and thus~$L < \{i,j\} \cup R \cup A \cup B$. If moreover~$1 < i$, then~$1 < \{i,j\} \cup R \cup A \cup B$ and thus~$1 \in L \ne \varnothing$. This would imply that the cut of~$\tree'$ defined by the edge~$L \to i$ would be splitting. Contradiction. We prove similarly that~$\signature_j = -$ or~$j = n$.
\end{proof}

\begin{proof}[Proof or Proposition~\ref{prop:upperIdeal}]
We already know from Lemma~\ref{lem:rotationIndecomposable}\,(i) that~$\indecomposables_\signature$ is an upper set of the $\signature$-Cambrian lattice. To see that this upper set is a principal upper ideal, we characterize the unique $\ECamb$-indecomposable $\signature$-Cambrian tree~$\tree_\bullet$ whose decreasing rotations all create a splitting edge cut. We proceed in three steps.

\para{Claim~A} All negative vertices~$i > 1$ of~$\tree_\bullet$ have no right child, while all positive vertices~$j < n$ of~$\tree_\bullet$ have no left child. \\[.1cm]
\textit{Proof.} Assume by means of contradiction that a negative vertex~$i > 1$ has a right child~$j$. Let~$\tree$ be the Cambrian tree obtained by rotation of the edge~$i \leftarrow j$ in~$\tree_\bullet$. Since this rotation is decreasing (because~$i < j$), $\tree$ is decomposable while~$\tree_\bullet$ is not. This contradicts Lemma~\ref{lem:rotationIndecomposable}\,(ii).

\medskip
Claim~A ensures that the Cambrian tree~$\tree_\bullet$ is a path with additional leaves incoming at negative vertices and outgoing at positive vertices. Therefore, $\tree_\bullet$ admits a unique linear extension~$\tau_\bullet$. The next two claims determine~$\tau_\bullet$ and thus~$\tree_\bullet = \surjectionPermAsso(\tau_\bullet)$.

As vertex~$1$ has no left child and vertex~$n$ has no right child, we consider that~$1$ behaves as a positive vertex and~$n$ behaves as a negative vertex. We thus define~${N \eqdef \{n_1 < \dots < n_{N-1} < n_N = n\}}$ and~$P \eqdef \{1 = p_1 < p_2 < \dots < p_P\}$, where~$n_1 < \dots < n_{N-1}$ are the negative vertices and~$p_2 < \dots < p_P$ are the positive vertices among~$\{2, \dots, n-1\}$.

\para{Claim~B} The sets~$N$ and~$P$ both appear in increasing order in~$\tau_\bullet$. \\[.1cm]
\textit{Proof.} If~~$i$ appears in~$\tau_\bullet$ before~$j \in N$, then~$i$ lies in the left child of~$j$ (since~$j$ has no right child), so that~$i < j$. In particular, $N$ is sorted in~$\tau_\bullet$. The proof is symmetric for positive vertices.

\para{Claim~C} In~$\tau_\bullet$, vertex~$p_k$ appears immediately after the first vertex in~$N$ larger than~$p_{k+1}$. \\[.1cm]
\textit{Proof.} Let~$n_\ell$ denote the first vertex in~$N$ larger than~$p_{k+1}$. If~$p_k$ appears before~$n_\ell$ in~$\tau_\bullet$, then~$\tau_\bullet$ is a decomposable permutation (since~$\tau([p_{k+1}-1]) = [p_{k+1}-1]$). If~$p_k$ appears after~$n_{\ell+1}$ in~$\tau_\bullet$, then the Cambrian tree obtained by rotation of the incoming edge at~$p_k$ in~$\tree_\bullet$ remains indecomposable. Therefore, $p_k$ appears precisely in between~$n_\ell$ and~$n_{\ell+1}$.
\end{proof}

\enlargethispage{.1cm}
For example, \fref{fig:minIndecomposable} illustrates the generator of the $\ECamb$-indecomposable $\signature$-Cambrian trees for~$\signature = {-}{-}{+}{-}{-}{+}{+}$, $\signature = (-)^7$, and~$\signature = (+)^7$. The last two are right- and left-tilting trees respectively.

\begin{figure}[h]
  \centerline{\includegraphics{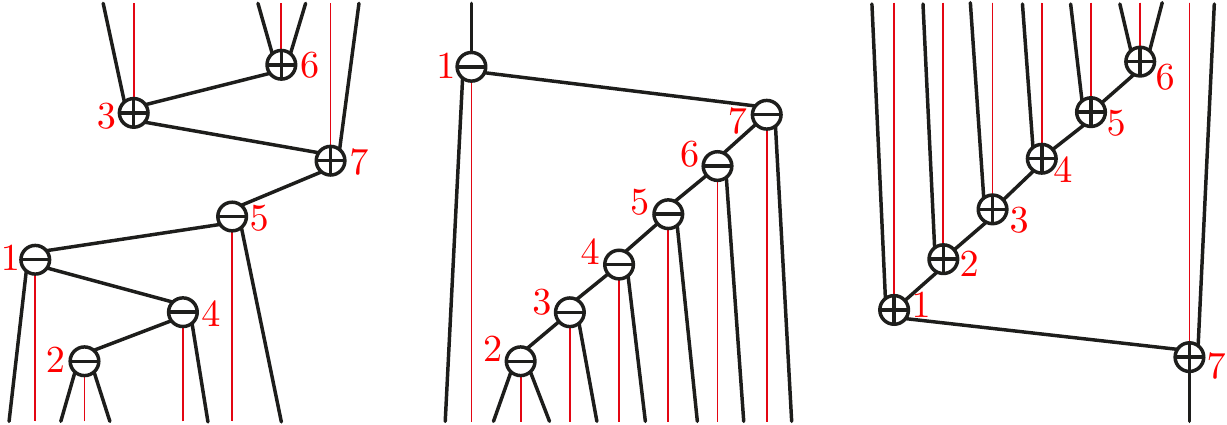}}
  \caption{The generators of the principal upper ideals of $\ECamb$-indecomposable $\signature$-Cambrian trees for~$\signature = {-}{-}{+}{-}{-}{+}{+}$ (left), $\signature = (-)^7$ (middle), $\signature = (+)^7$~(right).}
  \label{fig:minIndecomposable}
\end{figure}


\subsection{Enumerative properties}
\label{subsec:enumerativeProp}

We now consider enumerative properties of $\ECamb$-indecomposable elements. We want to show that the number of $\ECamb$-indecomposable $\signature$-Cambrian trees is independent of the signature~$\signature$.

\begin{proposition}
\label{prop:numberIndecomposables}
For any signature~$\signature \in \pm^n$, there are~$C_{n-1}$ $\ECamb$-indecomposable $\signature$-Cambrian trees. Therefore, there are~$2^n C_{n-1}$ $\ECamb$-indecomposable Cambrian trees on~$n$ vertices.
\end{proposition}

This result is immediate for the signature~$\signature = (-)^n$ as $\ECamb$-indecomposable elements are right-tilting binary trees (see Example~\ref{exm:rightTilting}), which are clearly counted by the Catalan number~$C_{n-1}$. To show Proposition~\ref{prop:numberIndecomposables}, we study the behavior of Cambrian trees and their decompositions under local transformations on signatures of~$[n]$. We believe that these transformations are interesting \perse. For example, they provide an alternative proof that there are $C_n$ $\signature$-Cambrian trees for any signature~$\signature \in \pm^n$.

Let~$\switchSign_0 : \pm^n \to \pm^n$ and~$\switchSign_n : \pm^n \to \pm^n$ denote the transformations which switch the signs of~$1$ and~$n$, respectively. Denote by~$\Psi_0(\tree)$ and~$\Psi_n(\tree)$ the trees obtained from a Cambrian tree~$\tree$ by changing the direction of the leftmost and rightmost leaf of~$\tree$ respectively. For~$i \in [n-1]$, let~$\switchSign_i : \pm^n \to \pm^n$ denote the transformation which switches the signs at positions~$i$ and~$i+1$. The transformation~$\signature \to \switchSign_i(\signature)$ is only relevant when~$\signature_i \ne \signature_{i+1}$. In this situation, we denote by~$\Psi_i(\tree)$ the tree obtained from a $\signature$-Cambrian tree~$\tree$ by
\begin{itemize}
\item reversing the arc from the positive to the negative vertex of~$\{i,i+1\}$ if it exists,
\item exchanging the labels of~$i$ and~$i+1$ otherwise.
\end{itemize}
This transformation is illustrated on \fref{fig:switchSign} when~$\signature_i = +$ and~$\signature_{i+1} = -$.

\begin{figure}[h]
  \centerline{\includegraphics{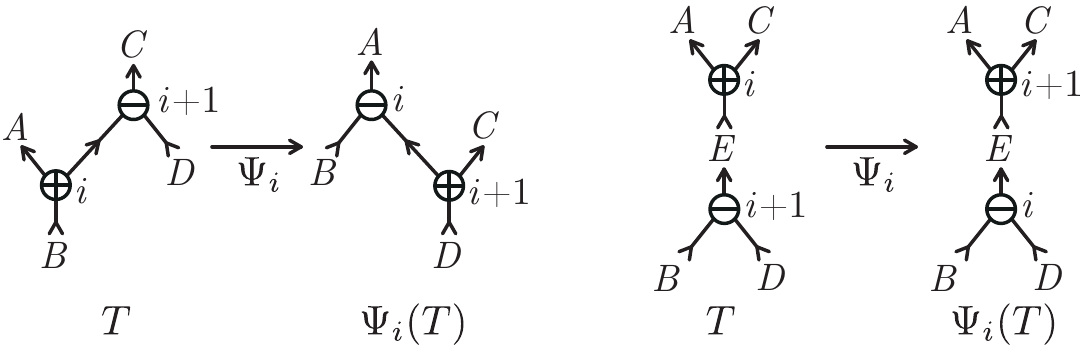}}
  \caption{The transformation~$\Psi_i$ when~$\signature_i = +$ and~$\signature_{i+1} = -$. The tree~$\Psi_i(\tree)$ is obtained by reversing the arc from ~$i$ to~$i+1$ if it exists (left), and just exchanging the labels of~$i$ and~$i+1$ otherwise (right).}
  \label{fig:switchSign}
\end{figure}

To show that~$\Psi_i$ transforms $\signature$-Cambrian trees to $\switchSign_i(\signature)$-Cambrian trees and preserves the number of $\ECamb$-indecomposable elements, we need the following lemma. Note that this lemma also explains why \fref{fig:switchSign} covers all possibilities when~$\signature_i = +$ and~$\signature_{i+1} = -$.

\begin{lemma}
If~$\signature_i = +$ and~$\signature_{i+1} = -$, then the following assertions are equivalent for a $\signature$-Cambrian tree~$\tree$:
\begin{enumerate}[(i)]
\item $\edgecut{[i]}{[n] \ssm [i]}$ is an edge cut of~$\tree$; \label{enum:cut2}
\item $i$ is smaller than~$i+1$ in~$\tree$; \label{enum:smaller}
\item $i$ is in the left subtree of~$i+1$ and~$i+1$ is in the right subtree of~$i$; \label{enum:subtrees}
\item $i$ is the left child of~$i+1$ and~$i+1$ is the right parent of~$i$. \label{enum:edge}
\end{enumerate}
A similar statement holds in the case when~$\signature_i = -$ and~$\signature_{i+1} = +$.
\end{lemma}

\begin{proof}
Since~$i$ and~$i+1$ are comparable in~$\tree$ (see Section~\ref{subsec:canopy}), the fact that~$\edgecut{[i]}{[n] \ssm [i]}$ is an edge cut of~$\tree$ implies that~$i$ is smaller than~$i+1$ in~$\tree$. This shows that~\eqref{enum:cut2} $\implies$ \eqref{enum:smaller}.

If~$i$ is smaller than~$i+1$ in~$\tree$, then~$i$ is in a subtree of~$i+1$, and thus in the left one, and similarly, $i+1$ is in the right subtree of~$i$. This shows that~\eqref{enum:smaller} $\implies$ \eqref{enum:subtrees}.

Assume now that~$i$ is in the left subtree of~$i+1$ and~$i+1$ is in the right subtree of~$i$, and consider the path from~$i$ to~$i+1$ in~$\tree$. Since it lies in the right subtree of~$i$ and in the left subtree of~$i+1$, any label along this path should be greater than~$i$ and smaller than~$i+1$. This path is thus a single arc. This shows that~\eqref{enum:subtrees} $\implies$ \eqref{enum:edge}.

Finally, assume that $i$ is the left child of~$i+1$ and~$i+1$ is the right parent of~$i$ in~$\tree$. Then the cut corresponding to the arc~$e$ of~$\tree$ from~$i$ to~$i+1$ is~$\edgecut{[i]}{[n] \ssm [i]}$. Indeed, all elements in the source of~$e$ are in the left subtree and thus smaller than~$i+1$, while all elements in the target of~$e$ are in the right subtree and thus greater than~$i$. This shows that~\mbox{\eqref{enum:edge} $\implies$ \eqref{enum:cut2}}.
\end{proof}

\begin{lemma}
\label{lem:switchSign}
For~$0 \le i \le n$, the map~$\Psi_i$ defines a bijection from~$\signature$-Cambrian trees to~$\switchSign_i(\signature)$-Cambrian trees and preserves the number of $\ECamb$-indecomposable elements.
\end{lemma}

\begin{proof}
The result is immediate for~$i = 0$ and~$i = n$. Assume thus that~$i \in [n-1]$ and that~$\signature_i = +$ while~$\signature_{i+1} = -$. We first prove that~$\Psi_i$ sends~$\signature$-Cambrian trees to~$\switchSign_i(\signature)$-Cambrian trees. It clearly transforms trees to trees. To see that~$\Psi_i(\tree)$ is $\switchSign_i(\signature)$-Cambrian, we distinguish two cases:
\begin{itemize}
\item \fref{fig:switchSign}\,(left) illustrates the case when $\tree$ has an arc in from~$i$ to~$i+1$. All labels in~$B$ are smaller than~$i$ since they are distinct from~$i$ and in the left subtree of~$i+1$, and all labels in the right subtree of~$i$ in~$\Psi_i(\tree)$ are greater than~$i$ since they were in the right subtree of~$i$ in~$\tree$. Therefore, the labels around vertex~$i$ of~$\Psi_i(\tree)$ respect the Cambrian rules. We argue similarly around~$i+1$. All other vertices have the same signs and subtrees.
\item \fref{fig:switchSign}\,(right) illustrates the case when $\tree$ has no arc in from~$i$ to~$i+1$. All labels in~$B$ (resp.~$D$) are smaller (resp.~greater) than~$i$ since they are distinct from~$i$ and in the left (resp.~ right) subtree of~$i+1$, so the labels around vertex~$i$ of~$\Psi_i(\tree)$ respect the Cambrian rules. We argue similarly around~$i+1$. All other vertices have the same signs and subtrees.
\end{itemize}
Alternatively, it is also easy to see~$\Psi_i$ transforms~$\signature$-Cambrian trees to~$\switchSign_i(\signature)$-Cambrian trees using the interpretation of Cambrian trees as dual trees of triangulations (see Remark~\ref{rem:triangulation}).

Although~$\Psi_i$ does not preserve $\ECamb$-indecomposable elements, we now check that~$\Psi_i$ preserves the number of $\ECamb$-indecomposable elements. Write~$\signature = \underline{\signature}\overline{\signature}$ with~$\underline{\signature} : [i] \to \{\pm\}$ and~$\overline{\signature} : [n] \ssm [i] \to \{\pm\}$, and let~$\underline{I} = |\indecomposables_{\underline{\signature}}|$ and~$\overline{I} = |\indecomposables_{\overline{\signature}}|$. We claim that
\begin{itemize}
\item the map~$\Psi_i$ transforms precisely $\underline{I} \cdot \overline{I}$\, $\ECamb$-decomposable $\signature$-Cambrian trees to $\ECamb$-indecomposable $\switchSign_i(\signature)$-Cambrian trees. Indeed, $\tree$ is $\ECamb$-decomposable while $\Psi_i(\tree)$ is $\ECamb$-indecomposable if and only of~$\tree$ has an arc from~$i$ to~$i+1$ whose source and target subtrees are $\ECamb$-indecomposable $\underline{\signature}$- and $\overline{\signature}$-Cambrian trees, respectively.
\item the map~$\Psi_i$ transforms precisely $\underline{I} \cdot \overline{I}$\, $\ECamb$-indecomposable $\signature$-Cambrian trees to $\ECamb$-decomposable $\switchSign_i(\signature)$-Cambrian trees. Indeed, assume that~$\tree$ is $\ECamb$-indecomposable while $\Psi_i(\tree)$ is $\ECamb$-decomposable.  We claim that~$\edgecut{[i]}{[n] \ssm [i]}$ is the only splitting edge cut of~$\Psi_i(\tree)$. Indeed, for~$j \ne i$, both $i$ and~$i+1$ belong either to~$[j]$ or to~$[n] \ssm [j]$, and~$\edgecut{[j]}{[n] \ssm [j]}$ is an edge cut of~$\Psi_i(\tree)$ if and only if it is an edge cut of~$\tree$. Moreover, the $\underline{\signature}$- and~$\overline{\signature}$-Cambrian trees~$\underline{\tree[S]}$ and~$\overline{\tree[S]}$ induced by~$\Psi_i(\tree)$ on~$[i]$ and~$[n] \ssm [i]$ are both $\ECamb$-indecomposable. Otherwise, a splitting edge cut~$\edgecut{[j]}{[i] \ssm [j]}$ of~$\underline{\tree[S]}$ would define a splitting edge cut~$\edgecut{[j]}{[n] \ssm [j]}$ of~$\Psi_i(\tree)$. Conversely, if~$\underline{\tree[S]}$ and~$\overline{\tree[S]}$ are both $\ECamb$-indecomposable, then so is~$\tree$.
\end{itemize}
We conclude that~$\Psi_i$ globally preserves the number of $\ECamb$-indecomposable Cambrian trees.
\end{proof}

\begin{proof}[Proof of Proposition~\ref{prop:numberIndecomposables}]
Starting from the fully negative signature~$(-)^n$, we can reach any signature~$\signature$ by the transformations~$\switchSign_0, \dots, \switchSign_{n-1}$: we can make positive signs appear on vertex~$1$ (using the map~$\switchSign_0$) and make these positive signs travel towards their final position in~$\signature$ (using the maps~$\switchSign_i$). More precisely, if~$p_1 < \dots < p_P$ denote the positions of the positive signs of~$\signature$, then~$\signature = \big( \prod_{j \in [P]} \switchSign_{p_j} \circ \switchSign_{p_{j-1}} \circ \dots \circ \switchSign_{p_1} \circ \switchSign_0 \big) \big( (-)^n \big)$. The result thus follows from Lemma~\ref{lem:switchSign}.
\end{proof}

\begin{proposition}
The Cambrian algebra~$\Camb$ is free.
\end{proposition}

\begin{proof}
As the generating function~$B(u)$ of the Catalan numbers satisfies the functional equation~$B(u) = 1 + uB(u)^2$, we obtain by substitution~$u = 2t$ that
\[
\frac{1}{1 - \sum_{n \ge 1} 2^n C_{n-1} t^n} = \sum_{n \ge 0} 2^n C_n t^n.
\]
The result immediately follows from Proposition~\ref{prop:numberIndecomposables}.
\end{proof}


\part{The Baxter-Cambrian Hopf Algebra}
\label{part:BaxterCambrian}


\section{Twin Cambrian trees}
\label{sec:BaxterTrees}

We now consider twin Cambrian trees and the resulting Baxter-Cambrian algebra. It provides a straightforward generalization to the Cambrian setting of the work of S.~Law and N.~Reading on quadrangulations~\cite{LawReading} and S.~Giraudo on twin binary trees~\cite{Giraudo}. The bases of these algebras are counted by the Baxter numbers. In Section~\ref{subsec:BaxterCambrianNumbers} we provide references for the various Baxter families and their bijective correspondences, and we discuss the Cambrian counterpart of these numbers. Definitions and combinatorial properties of twin Cambrian trees are given in this section, while the algebraic aspects are treated in the next section.


\subsection{Twin Cambrian trees}
\label{subsec:twinCambrianTrees}

This section deals with the following pairs of Cambrian trees.

\begin{definition}
\label{def:twinCambrianTrees}
Two $\signature$-Cambrian trees~$\tree_\circ, \tree_\bullet$ are \defn{twin} if the union~$\graphTwin$ of~$\tree_\circ$ with the reverse of~$\tree_\bullet$ (reversing the orientations of all edges) is acyclic.
\end{definition}

\begin{definition}
Let~$\tree_{\circ}, \tree_{\bullet}$ be two leveled~$\signature$-Cambrian trees with labelings~$p_{\circ}, q_{\circ}$ and $p_{\bullet}, q_{\bullet}$ respectively. We say that they are \defn{twin} if $q_{\circ}(p_{\circ}^{-1}(i)) = n - q_{\bullet}(p_{\bullet}^{-1}(i))$ for all~$i \in [n]$. In other words, when labeled as Cambrian trees, the bottom-up order of the vertices of~$\tree_{\circ}$ and $\tree_{\bullet}$ are opposite.
\end{definition}

Examples of twin Cambrian trees and twin leveled Cambrian trees are represented in \fref{fig:twinCambrianTrees}. Note that twin leveled Cambrian trees are twin Cambrian trees~$\tree_\circ, \tree_\bullet$ endowed with a linear extension of the transitive closure of~$\graphTwin$.

\begin{figure}[h]
  \centerline{\includegraphics{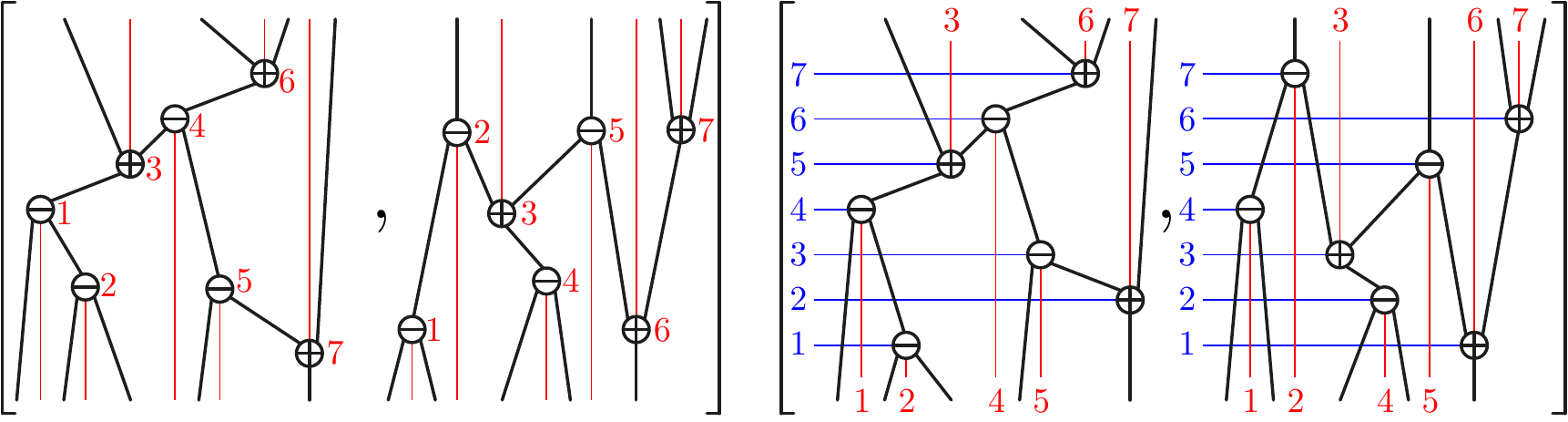}}
  \caption{A pair of twin Cambrian trees (left), and a pair of twin leveled Cambrian trees (right).}
  \label{fig:twinCambrianTrees}
\end{figure}

If~$\tree_\circ, \tree_\bullet$ are two $\signature$-Cambrian trees, they necessarily have opposite canopy (see Section~\ref{subsec:canopy}), meaning that $\surjectionAssoPara(\tree_\circ)_i = -\surjectionAssoPara(\tree_\bullet)_i$ for all~$i \in [n-1]$. The reciprocal statement for the constant signature~$(-)^n$ is proved by S.~Giraudo in~\cite{Giraudo}.

\begin{proposition}[\cite{Giraudo}]
\label{prop:twinBinaryTrees}
Two binary trees are twin if and only if they have opposite canopy.
\end{proposition}

We conjecture that this statement holds for general signatures. Consider two $\signature$-Cambrian trees~$\tree_\circ, \tree_\bullet$ with opposite canopies. It is easy to show that~$\graphTwin$ cannot have trivial cycles, meaning that~$\tree_\circ$ and~$\tree_\bullet$ cannot both have a path from~$i$ to~$j$ for~$i \ne j$. To prove that~$\graphTwin$ has no cycles at all, a good method is to design an algorithm to extract a linear extension of~$\graphTwin$. This approach was used in~\cite{Giraudo} for the signature~$(-)^n$. In this situation, it is clear that the root of~$\tree_\bullet$ is minimal in~$\tree_\circ$ (by the canopy assumption), and we therefore pick it as the first value of a linear extension of~$\graphTwin$. The remaining of the linear extension is constructed inductively. In the general situation, it turns out that not all maximums in~$\tree_\bullet$ are minimums in~$\tree_\circ$ (and reciprocally). It is thus not clear how to choose the first value of a linear extension of~$\graphTwin$.

\begin{remark}[Reversing~$\tree_\bullet$]
\label{rem:reversing}
\enlargethispage{.2cm}
It is sometimes useful to reverse the second tree~$\tree_\bullet$ in a pair~$[\tree_\circ,\tree_\bullet]$ of twin Cambrian trees. The resulting Cambrian trees have opposite signature and their union is acyclic. In this section, we have chosen the orientation of Definition~\ref{def:twinCambrianTrees} to fit with the notations and results in~\cite{Giraudo}. We will have to switch to the opposite convention in Section~\ref{sec:tuples} when we will extend our results on twin Cambrian trees to arbitrary tuples of Cambrian trees.
\end{remark}


\subsection{Baxter-Cambrian correspondence}

We obtain the \defn{Baxter-Cambrian correspondence} between permutations of~$\fS^\signature$ and pairs of twin leveled $\signature$-Cambrian trees by inserting with the map~$\CambCorresp$ from Section~\ref{subsec:CambrianCorrespondence} a permutation~$\tau = \tau_1 \cdots \tau_n \in \fS^\signature$ and its \defn{mirror}~${\mirror{\tau} = \tau_n \cdots \tau_1 \in \fS^\signature}$.

\begin{proposition}
The map~$\BaxCorresp$ defined by~$\BaxCorresp(\tau) = \big[ \CambCorresp(\tau), \CambCorresp(\mirror{\tau}) \big]$ is a bijection from signed permutations to pairs of twin leveled Cambrian trees.
\end{proposition}

\begin{proof}
If~$p,q : V \to [n]$ denote the Cambrian and increasing labelings of the Cambrian tree~$\CambCorresp(\tau)$, then~$\tau = q \circ p$. This yields that the leveled $\signature$-Cambrian trees~$\CambCorresp(\tau)$ and~$\CambCorresp(\mirror{\tau})$ are twin and the map~$\BaxCorresp$ is bijective. 
\end{proof}

As for Cambrian trees, we focus on the $\PSymbol\Bax$-symbol of this correspondence.

\begin{proposition}
The map~$\surjectionPermAssoBax$ defined by~$\surjectionPermAssoBax(\tau) = \big[ \surjectionPermAsso(\tau), \surjectionPermAsso(\mirror{\tau}) \big]$ is a surjection from signed permutations to pairs of twin Cambrian trees.
\end{proposition}

\begin{proof}
The fiber~$(\surjectionPermAssoBax)^{-1}([\tree_\circ, \tree_\bullet])$ of a pair of twin $\signature$-Cambrian trees~$\tree_\circ, \tree_\bullet$ is the set~$\linearExtensions(\graphTwin)$ of linear extensions of the graph~$\graphTwin$. This set is non-empty since~$\graphTwin$ is acyclic by definition of twin Cambrian trees.
\end{proof}


\subsection{Baxter-Cambrian congruence}
\label{subsec:BaxterCongruence}

We now characterize by a congruence relation the signed permutations~$\tau \in \fS^{\signature}$ which have the same $\PSymbol\Bax$-symbol.

\begin{definition}
For a signature~${\signature \in \pm^n}$, the \defn{$\signature$-Baxter-Cambrian congruence} is the equivalence relation on $\fS^{\signature}$ defined as the transitive closure of the rewriting rules
\begin{gather*}
UbVadWcX \equiv\Bax_\signature UbVdaWcX \quad\text{if } a < b, c < d \text{ and } \signature_b = \signature_c, \\
UbVcWadX \equiv\Bax_\signature UbVcWdaX \quad\text{if } a < b, c < d \text{ and } \signature_b \not= \signature_c, \\
UadVbWcX \equiv\Bax_\signature UdaVbWcX \quad\text{if } a < b, c < d \text{ and } \signature_b \not= \signature_c,
\end{gather*}
where~$a, b, c, d$ are elements of~$[n]$ while~$U, V, W, X$ are words on~$[n]$. The \defn{Baxter-Cambrian congruence} is the equivalence relation on all signed permutations~$\fS_{\pm}$ obtained as the union of all~$\signature$-Baxter-Cambrian congruences:
\[
\equiv\Bax \; \eqdef \bigsqcup_{\substack{n \in \N \\ \signature \in \pm^n}} \!\! \equiv\Bax_\signature.
\]
\end{definition}

\begin{proposition}
Two signed permutations~$\tau, \tau' \in \fS^{\signature}$ are~$\signature$-Baxter-Cambrian congruent if and only if they have the same $\PSymbol\Bax$-symbol:
\[
\tau \equiv\Bax_\signature \tau' \iff \surjectionPermAssoBax(\tau) = \surjectionPermAssoBax(\tau').
\]
\end{proposition}

\begin{proof}
The proof of this proposition consists essentially in seeing that~$\surjectionPermAssoBax(\tau) = \surjectionPermAssoBax(\tau')$ if and only if~$\tau \equiv \tau'$ and~$\mirror{\tau} \, \equiv \, \mirror{\tau}\!\!'$ (by definition of $\surjectionPermAssoBax$). The definition of the~$\signature$-Baxter-Cambrian equivalence $\equiv\Bax_{\signature}$ is exactly the translation of this observation in terms of rewriting rules.
\end{proof}

\begin{proposition}
\label{prop:intersectionCambrianClasses}
The~$\signature$-Baxter-Cambrian class indexed by a pair~$[\tree_\circ, \tree_\bullet]$ of twin $\signature$-Cambrian trees is the intersection of the $\signature$-Cambrian class indexed by~$\tree_\circ$ with the $(-\signature)$-Cambrian class indexed by the reverse of~$\tree_\bullet$.
\end{proposition}

\begin{proof}
The~$\signature$-Baxter-Cambrian class indexed by~$[\tree_\circ, \tree_\bullet]$ is the set of linear extensions of~$\graphTwin$, \ie of permutations which are both linear extensions of~$\tree_\circ$ and linear extensions of the reverse of~$\tree_\bullet$. The former form the $\signature$-Cambrian class indexed by~$\tree_\circ$ while the latter form the $(-\signature)$-Cambrian class indexed by the reverse of~$\tree_\bullet$. This is illustrated in \fref{fig:twoOppositeCambrians}.
\begin{figure}[h]
  \centerline{\includegraphics[width=\textwidth]{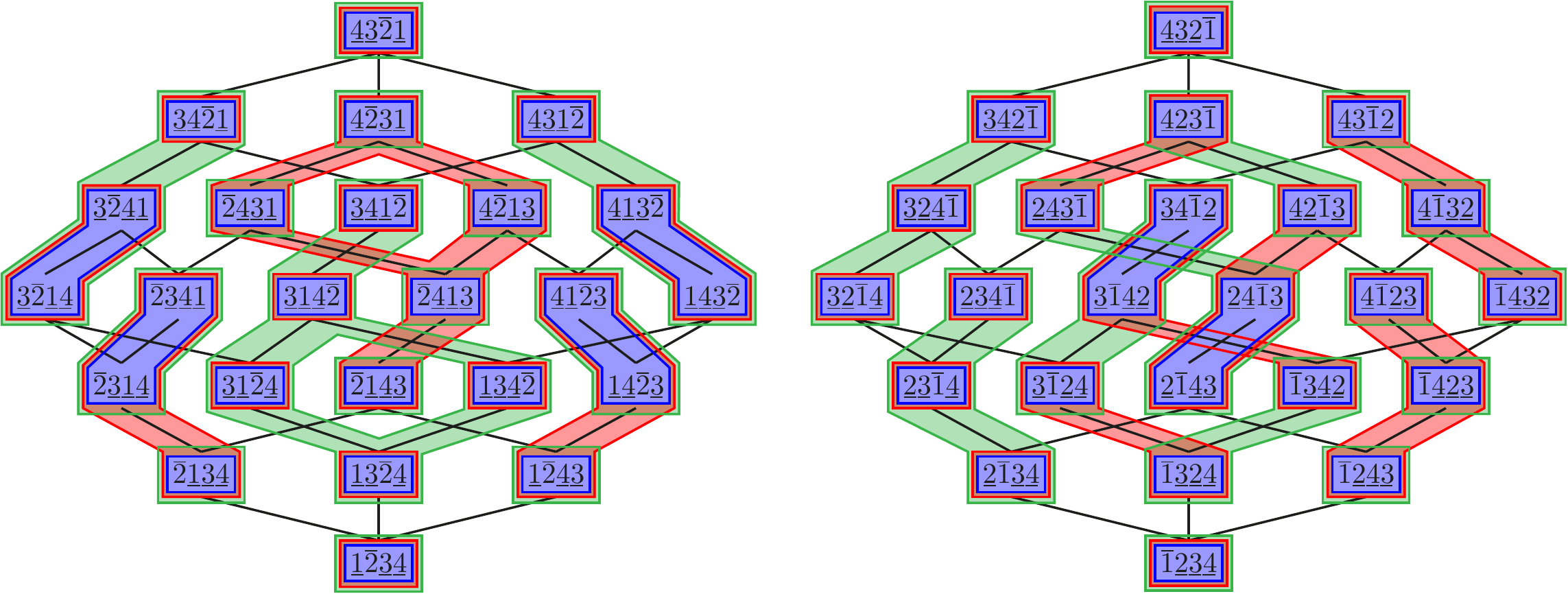}}
  \caption{The Baxter-Cambrian classes of~$\equiv_\signature\Bax$ (blue) are the intersections of the Cambrian classes of~$\equiv_\signature$ (red) and~$\equiv_{-\signature}$ (green). Illustrated for the signatures $\signature = {-}{+}{-}{-}$ (left) and ${\signature = {+}{-}{-}{-}}$~(right).}
  \label{fig:twoOppositeCambrians}
\end{figure}
\end{proof}


\subsection{Rotations and Baxter-Cambrian lattices}

We now present the rotation operation on pairs of twin $\signature$-Cambrian trees.

\begin{definition}
\label{def:rotationBaxter}
Let~$[\tree_\circ, \tree_\bullet]$ be a pair of $\signature$-Cambrian trees and~$i \to j$ be an edge of~$\graphTwin$. We say that the edge~$i \to j$ is \defn{rotatable} if
\begin{itemize}
\item either~$i \to j$ is an edge in~$\tree_\circ$ and~$j \to i$ is an edge in~$\tree_\bullet$,
\item or~$i \to j$ is an edge in~$\tree_\circ$ while~$i$ and~$j$ are incomparable in~$\tree_\bullet$,
\item or $i$ and~$j$ are incomparable in~$\tree_\circ$ while $j \to i$ is an edge in~$\tree_\bullet$.
\end{itemize}
If~$i \to j$ is rotatable in~$[\tree_\circ, \tree_\bullet]$, its \defn{rotation} transforms~$[\tree_\circ, \tree_\bullet]$ to the pair of trees~$[\tree'_\circ, \tree'_\bullet]$,~where
\begin{itemize}
\item $\tree'_\circ$ is obtained by rotation of~$i \to j$ in~$\tree_\circ$ if possible and~$\tree'_\circ = \tree_\circ$ otherwise, and
\item $\tree'_\bullet$ is obtained by rotation of~$j \to i$ in~$\tree_\bullet$ if possible and~$\tree'_\bullet = \tree_\bullet$ otherwise.
\end{itemize}
\end{definition}

\begin{proposition}
Rotating a rotatable edge~$i \to j$ in a pair~$[\tree_\circ, \tree_\bullet]$ of twin $\signature$-Cambrian trees yields a pair~$[\tree'_\circ, \tree'_\bullet]$ of twin $\signature$-Cambrian trees.
\end{proposition}

\begin{proof}
By Proposition~\ref{prop:rotation}, the trees~$\tree_\circ, \tree_\bullet$ are $\signature$-Cambrian trees. To see that they are twins, observe that switching~$i$ and~$j$ in a linear extension of~$\graphTwin$ yields a linear extension of~$\unionOp{\tree'_\circ}{\tree'_\bullet}$.
\end{proof}

\begin{remark}[Number of rotatable edges]
Note that a pair~$[\tree_\circ, \tree_\bullet]$ of~$\signature$-Cambrian trees has always at least~$n-1$ rotatable edges. This will be immediate from the considerations of Section~\ref{subsec:geometricRealizationsBaxter}.
\end{remark}

Consider the \defn{increasing rotation graph} whose vertices are pairs of twin $\signature$-Cambrian trees and whose arcs are increasing rotations~$[\tree_\circ, \tree_\bullet] \to [\tree'_\circ, \tree'_\bullet]$, \ie for which~$i < j$ in Definition~\ref{def:rotationBaxter}. This graph is illustrated on \fref{fig:BaxterCambrianLattices} for the signatures $\signature = {-}{+}{-}{-}$ and ${\signature = {+}{-}{-}{-}}$.

\hvFloat[floatPos=p, capWidth=h, capPos=r, capAngle=90, objectAngle=90, capVPos=c, objectPos=c]{figure}
{\includegraphics[width=1.65\textwidth]{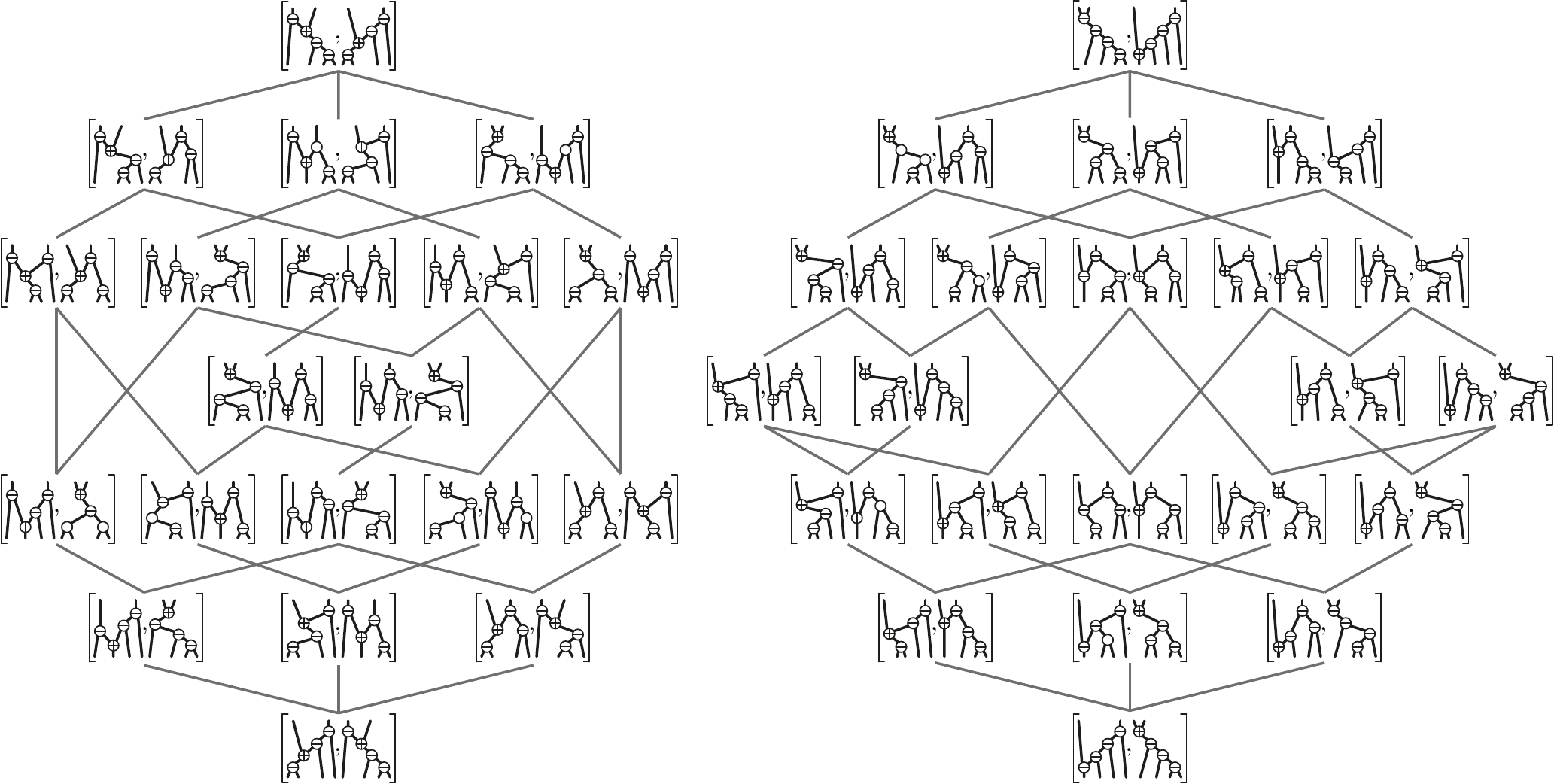}}
{The $\signature$-Baxter-Cambrian lattices on pairs of twin $\signature$-Cambrian trees, for the signatures $\signature = {-}{+}{-}{-}$ (left) and ${\signature = {+}{-}{-}{-}}$~(right).}
{fig:BaxterCambrianLattices}

\begin{proposition}
For any cover relation~$\tau < \tau'$ in the weak order on~$\fS^\signature$, either~$\surjectionPermAssoBax(\tau) = \surjectionPermAssoBax(\tau')$ or~$\surjectionPermAssoBax(\tau) \to \surjectionPermAssoBax(\tau')$ in the increasing rotation graph.
\end{proposition}

\begin{proof}
Let~$i,j \in [n]$ be such that~$\tau'$ is obtained from~$\tau$ by switching two consecutive values~$ij$ to~$ji$. If~$i$ and~$j$ are incomparable in~$\surjectionPermAsso(\tau)$, then~$\surjectionPermAsso(\tau) = \surjectionPermAsso(\tau')$. Otherwise, there is an edge~$i \to j$ in~$\surjectionPermAsso(\tau)$, and~$\surjectionPermAsso(\tau')$ is obtained by rotating~$i \to j$ in~$\surjectionPermAsso(\tau)$. The same discussion is valid for the trees~$\surjectionPermAsso(\mirror{\tau})$ and~$\surjectionPermAsso(\mirror{\tau}\!\!')$ and edge~$j \to i$. The result immediately follows.
\end{proof}

It follows that the increasing rotation graph on pairs of twin $\signature$-Cambrian trees is acyclic and we call \defn{$\signature$-Baxter-Cambrian poset} its transitive closure. In other words, the previous statement says that the map~$\surjectionPermAssoBax$ defines a poset homomorphism from the weak order on~$\fS^\signature$ to the $\signature$-Baxter-Cambrian poset. The following statement extends the results of N.~Reading~\cite{Reading-CambrianLattices} on Cambrian lattices and S.~Law and N.~Reading~\cite{LawReading} on the lattice of diagonal rectangulations.

\begin{proposition}
The~$\signature$-Baxter-Cambrian poset is a lattice quotient of the weak order on~$\fS^\signature$.
\end{proposition}

\begin{proof}
By Proposition~\ref{prop:intersectionCambrianClasses}, the $\signature$-Baxter-Cambrian congruence is the intersection of two Cambrian congruences. The statement follows since the Cambrian congruences are lattice congruences of the weak order~\cite{Reading-CambrianLattices} and an intersection of lattice congruences is a lattice congruence.
\end{proof}

\begin{remark}[Cambrian \versus Baxter-Cambrian lattices]
\enlargethispage{.1cm}
Using the definition of~$\BaxCorresp$, we also notice that the~$\signature$-Cambrian classes are unions of~$\signature$-Baxter-Cambrian classes, therefore the Cambrian lattice is a lattice quotient of the Baxter-Cambrian lattice. \fref{fig:latticesBis} illustrates the Baxter-Cambrian, Cambrian, and boolean congruence classes on the weak orders of~$\fS_\signature$ for the signatures $\signature = {-}{+}{-}{-}$ and ${\signature = {+}{-}{-}{-}}$.

\begin{figure}[h]
  \centerline{\includegraphics[width=\textwidth]{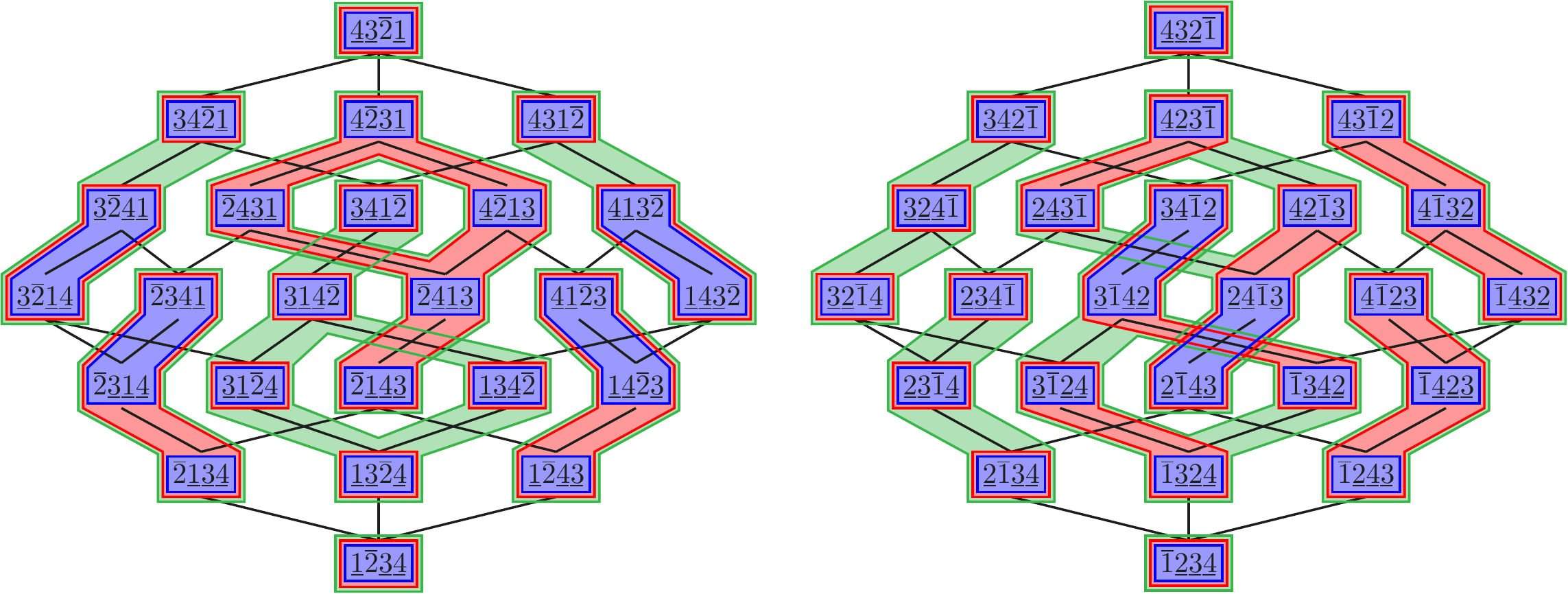}}
  \caption{Baxter-Cambrian (blue), Cambrian (red), and boolean (green) congruence classes on the weak orders of~$\fS_\signature$ for the signatures $\signature = {-}{+}{-}{-}$ (left) and ${\signature = {+}{-}{-}{-}}$~(right). The number of Baxter-Cambrian classes is not constant.}
  \label{fig:latticesBis}
\end{figure}
\end{remark}

\begin{remark}[Extremal elements and pattern avoidance]
\label{rem:patternsBaxterCambrian}
Since the Baxter-Cambrian classes are generated by rewriting rules, we immediately obtain that the minimal elements of the Baxter-Cambrian classes are precisely the signed permutations avoiding the patterns:
\[
\down{b} \dash da \dash \down{c}, \; \up{b} \dash da \dash \up{c}, \; \down{c}  \dash da \dash \down{b}, \; \up{c} \dash da \dash \up{b}, \; \down{b} \dash \up{c} \dash da, \; \up{b} \dash \down{c} \dash da, \; \down{c} \dash \up{b} \dash da, \; \up{c} \dash \down{b} \dash da, \; da \dash \down{b} \dash \up{c}, \; da \dash \up{b} \dash \down{c}, \; da \dash \down{c} \dash \up{b}, \; da \dash \up{c} \dash \down{b}.
\]
Similarly, the maximal elements of the Baxter-Cambrian classes are precisely the signed permutations avoiding the patterns:
\begin{equation}
\label{eqn:patternsBaxterCambrian}\tag{$\star$}
\down{b} \dash ad \dash \down{c}, \; \up{b} \dash ad \dash \up{c}, \; \down{c} \dash ad \dash \down{b}, \; \up{c} \dash ad \dash \up{b}, \; \down{b} \dash \up{c} \dash ad, \; \up{b} \dash \down{c} \dash ad, \; \down{c} \dash \up{b} \dash ad, \; \up{c} \dash \down{b} \dash ad, \; ad \dash \down{b} \dash \up{c}, \; ad \dash \up{b} \dash \down{c}, \; ad \dash \down{c} \dash \up{b}, \; ad \dash \up{c} \dash \down{b}.
\end{equation}
\end{remark}


\subsection{Baxter-Cambrian numbers}
\label{subsec:BaxterCambrianNumbers}

In contrast to the number of $\signature$-Cambrian trees, the number of pairs of twin $\signature$-Cambrian trees does depend on the signature~$\signature$. For example, there are~$22$ pairs of twin $({-}{-}{-}{-})$-Cambrian trees and only~$20$ pairs of twin $({-}{+}{-}{-})$-Cambrian trees. See Figures~\ref{fig:BaxterCambrianLattices}, \ref{fig:latticesBis} and~\ref{fig:GeneratingTreeBaxterCambrian}.

\enlargethispage{-.2cm}
For a signature~$\signature$, we define the \defn{$\signature$-Baxter-Cambrian number}~$\BC_\signature$ to be the number of pairs of twin $\signature$-Cambrian trees. We immediately observe that~$\BC_\signature$ is preserved when we change the first and last sign of~$\signature$, inverse simultaneously all signs of~$\signature$, or reverse the signature~$\signature$:
\[
\BC_\signature = \BC_{\switchSign_0(\signature)} = \BC_{\switchSign_n(\signature)} = \BC_{-\signature} = \BC_{\mirror{\signature}},
\]
where~$\switchSign_0$ and~$\switchSign_n$ change the first and last sign,~$(-\signature)_i = -\signature_i$ and $(\mirror{\signature})_i = \signature_{n+1-i}$. Table~\ref{table:numberBaxterCambrianEachSignature} shows the $\signature$-Baxter-Cambrian number~$\BC_\signature$ for all small signatures~$\signature$ up to these transformations. Table~\ref{table:numbersBaxterCambrian} records all possible $\signature$-Baxter-Cambrian numbers~$\BC_\signature$ for signatures~$\signature$ of sizes~$n \le 10$.

\begin{table}[h]
	\centerline{
	\begin{tabular}{c|l@{\qquad}l@{\qquad}l}
		$n = 4$ & $\BC_{{+}{+}{+}{+}} = 22$ & $\BC_{{+}{+}{-}{+}} = 20$ & \\[.1cm]
		\hline \\[-.3cm]
		$n = 5$ & $\BC_{{+}{+}{+}{+}{+}} = 92$ & $\BC_{{+}{+}{+}{-}{+}} = 78$ & $\BC_{{+}{+}{-}{+}{+}} = 70$ \\[.1cm]
		\hline \\[-.3cm]
		$n = 6$ & $\BC_{{+}{+}{+}{+}{+}{+}} = 422$ & $\BC_{{+}{+}{+}{+}{-}{+}} = 342$ & $\BC_{{+}{+}{+}{-}{-}{+}} = 316$ \\
				& $\BC_{{+}{+}{-}{-}{+}{+}} = 284$  & $\BC_{{+}{+}{+}{-}{+}{+}} = 282$ & $\BC_{{+}{+}{-}{+}{-}{+}} = 252$ \\[.1cm]
		\hline \\[-.3cm]
		$n = 7$ & $\BC_{{+}{+}{+}{+}{+}{+}{+}} = 2074$ & $\BC_{{+}{+}{+}{+}{+}{-}{+}} = 1628$ & $\BC_{{+}{+}{+}{+}{-}{-}{+}} = 1428$ \\
				& $\BC_{{+}{+}{-}{-}{-}{+}{+}} = 1298$ & $\BC_{{+}{+}{+}{+}{-}{+}{+}} = 1270$ & $\BC_{{+}{+}{+}{-}{-}{+}{+}} = 1172$ \\
				& $\BC_{{+}{+}{+}{-}{+}{+}{+}} = 1162$ & $\BC_{{+}{+}{-}{+}{+}{-}{+}} = 1044$ & $\BC_{{+}{+}{+}{-}{+}{-}{+}} = 1036$ \\
				& 											 & $\BC_{{+}{+}{-}{+}{-}{+}{+}} = 924$  & 										   \\
	\end{tabular}
	}
	\caption{The number~$\BC_\signature$ of twin $\signature$-Cambrian trees for all small signatures~$\signature$ (up to first/last sign change, simultaneous inversion of all signs, and reverse).}
	\label{table:numberBaxterCambrianEachSignature}
\end{table}

\begin{table}[h]
	\centerline{
	\begin{tabular}{c|l}
		$n = 4$ & $22$ {\scriptsize $(1)$}, $20$ {\scriptsize $(1)$} \\[.1cm]
		\hline \\[-.3cm]
		$n = 5$ & $92$ {\scriptsize $(1)$}, $78$ {\scriptsize $(2)$}, $70$ {\scriptsize $(1)$} \\[.1cm]
		\hline \\[-.3cm]
		$n = 6$ & $422$ {\scriptsize $(1)$}, $342$ {\scriptsize $(2)$}, $316$ {\scriptsize $(1)$}, $284$ {\scriptsize $(1)$}, $282$ {\scriptsize $(2)$}, $252$ {\scriptsize $(1)$} \\[.1cm]
		\hline \\[-.3cm]
		$n = 7$ & $2074$ {\scriptsize $(1)$}, $1628$ {\scriptsize $(2)$}, $1428$ {\scriptsize $(2)$}, $1298$ {\scriptsize $(1)$}, $1270$ {\scriptsize $(2)$}, $1172$ {\scriptsize $(2)$}, $1162$ {\scriptsize $(1)$}, $1044$ {\scriptsize $(2)$}, $1036$ {\scriptsize $(2)$}, $924$ {\scriptsize $(1)$} \\[.1cm]
		\hline \\[-.3cm]
		$n = 8$ & $10754$ {\scriptsize $(1)$}, $8244$ {\scriptsize $(2)$}, $6966$ {\scriptsize $(2)$}, $6612$ {\scriptsize $(1)$}, $6388$ {\scriptsize $(1)$}, $6182$ {\scriptsize $(2)$}, $5498$ {\scriptsize $(2)$}, $5380$ {\scriptsize $(2)$}, $5334$ {\scriptsize $(2)$}, $4902$ {\scriptsize $(1)$}, \\
				& $4884$ {\scriptsize $(2)$}, $4748$ {\scriptsize $(2)$}, $4392$ {\scriptsize $(1)$}, $4362$ {\scriptsize $(2)$}, $4356$ {\scriptsize $(2)$}, $4324$ {\scriptsize $(1)$}, $3882$ {\scriptsize $(1)$}, $3880$ {\scriptsize $(2)$}, $3852$ {\scriptsize $(2)$}, $3432$ {\scriptsize $(1)$}\\[.1cm]
		\hline \\[-.3cm]
		$n = 9$ & $58202$ {\scriptsize $(1)$}, $43812$ {\scriptsize $(2)$}, $35998$ {\scriptsize $(2)$}, $33240$ {\scriptsize $(1)$}, $32908$ {\scriptsize $(2)$}, $31902$ {\scriptsize $(2)$}, $27660$ {\scriptsize $(2)$}, $26602$ {\scriptsize $(2)$}, $26392$ {\scriptsize $(2)$}, \\
				& $25768$ {\scriptsize $(2)$}, $24888$ {\scriptsize $(1)$}, $24528$ {\scriptsize $(2)$}, $23530$ {\scriptsize $(1)$}, $23466$ {\scriptsize $(2)$}, $22768$ {\scriptsize $(2)$}, $20888$ {\scriptsize $(2)$}, $20886$ {\scriptsize $(2)$}, $20718$ {\scriptsize $(2)$}, \\
				& $20244$ {\scriptsize $(2)$}, $20218$ {\scriptsize $(2)$}, $20082$ {\scriptsize $(2)$}, $18544$ {\scriptsize $(1)$}, $18518$ {\scriptsize $(2)$}, $18430$ {\scriptsize $(2)$}, $18376$ {\scriptsize $(2)$}, $17874$ {\scriptsize $(2)$}, $16470$ {\scriptsize $(2)$}, \\
				& $16454$ {\scriptsize $(1)$}, $16358$ {\scriptsize $(2)$}, $16344$ {\scriptsize $(2)$}, $16342$ {\scriptsize $(2)$}, $16234$ {\scriptsize $(1)$}, $14550$ {\scriptsize $(4)$}, $14454$ {\scriptsize $(2)$}, $12870$ {\scriptsize $(1)$} \\[.1cm]
		\hline \\[-.3cm]
		$n = 10$ & $326240$ {\scriptsize $(1)$}, $242058$ {\scriptsize $(2)$}, $194608$ {\scriptsize $(2)$}, $180678$ {\scriptsize $(1)$}, $172950$ {\scriptsize $(2)$}, $172304$ {\scriptsize $(2)$}, $166568$ {\scriptsize $(1)$}, $146622$ {\scriptsize $(2)$}, \\ 
				& $139100$ {\scriptsize $(2)$}, $138130$ {\scriptsize $(2)$}, $131994$ {\scriptsize $(2)$}, $129870$ {\scriptsize $(2)$}, $129600$ {\scriptsize $(2)$}, $124896$ {\scriptsize $(2)$}, $122716$ {\scriptsize $(2)$}, $120800$ {\scriptsize $(1)$}, \\
				& $113754$ {\scriptsize $(2)$}, $111274$ {\scriptsize $(2)$}, $107072$ {\scriptsize $(2)$}, $106854$ {\scriptsize $(1)$}, $106382$ {\scriptsize $(2)$}, $105606$ {\scriptsize $(2)$}, $101084$ {\scriptsize $(3)$}, $101028$ {\scriptsize $(2)$}, \\
				& $100426$ {\scriptsize $(2)$}, $98730$ {\scriptsize $(2)$}, $97524$ {\scriptsize $(2)$}, $94908$ {\scriptsize $(1)$}, $94372$ {\scriptsize $(1)$}, $93854$ {\scriptsize $(2)$}, $89952$ {\scriptsize $(2)$}, $89324$ {\scriptsize $(2)$}, $89276$ {\scriptsize $(2)$}, \\
				& $88966$ {\scriptsize $(2)$}, $86638$ {\scriptsize $(2)$}, $86034$ {\scriptsize $(2)$}, $86026$ {\scriptsize $(2)$}, $79826$ {\scriptsize $(2)$}, $79384$ {\scriptsize $(2)$}, $79226$ {\scriptsize $(2)$}, $79076$ {\scriptsize $(2)$}, $79018$ {\scriptsize $(2)$}, \\
				& $78580$ {\scriptsize $(1)$}, $78528$ {\scriptsize $(2)$}, $76542$ {\scriptsize $(2)$}, $76526$ {\scriptsize $(2)$}, $76484$ {\scriptsize $(2)$}, $76072$ {\scriptsize $(2)$}, $70450$ {\scriptsize $(2)$}, $70316$ {\scriptsize $(1)$}, $69866$ {\scriptsize $(4)$}, \\
				& $69838$ {\scriptsize $(2)$}, $69810$ {\scriptsize $(2)$}, $69400$ {\scriptsize $(2)$}, $69314$ {\scriptsize $(1)$}, $67694$ {\scriptsize $(2)$}, $62124$ {\scriptsize $(3)$}, $62120$ {\scriptsize $(1)$}, $62096$ {\scriptsize $(2)$}, $61766$ {\scriptsize $(2)$}, \\
				& $61746$ {\scriptsize $(2)$}, $61706$ {\scriptsize $(2)$}, $61682$ {\scriptsize $(2)$}, $61376$ {\scriptsize $(1)$}, $54956$ {\scriptsize $(2)$}, $54920$ {\scriptsize $(2)$}, $54892$ {\scriptsize $(1)$}, $54626$ {\scriptsize $(2)$}, $48620$ {\scriptsize $(1)$}
	\end{tabular}
	}
	\caption{All possible $\signature$-Baxter-Cambrian numbers~$\BC_\signature$ for signatures~$\signature$ of sizes~$n \le 10$. Numbers in parenthesis indicate the multiplicity of each Baxter number: for example, the second line indicates that there are~$8$ (resp.~$16$, resp.~$8$) signatures~$\signature$ of~$\pm^5$ such that~$\BC_\signature = 92$ (resp.~$78$, resp.~$70$).}
	\label{table:numbersBaxterCambrian}
\end{table}

In the following statements, we provide an inductive formula to compute all $\signature$-Baxter-Cambrian numbers, using a two-parameters refinement. The proof is based on ideas similar to Proposition~\ref{prop:GeneratingTree}. The pairs of twin $\signature$-Cambrian trees are in bijection with the weak order maximal permutations of $\signature$-Baxter-Cambrian classes. These permutations are precisely the permutations avoiding the patterns~\eqref{eqn:patternsBaxterCambrian} in Remark~\ref{rem:patternsBaxterCambrian}. We consider the  generating tree~$\generatingTree_\signature\Bax$ for these permutations. This tree has~$n$ levels, and the nodes at level~$\level$ are labeled by permutations of~$[\level]$ whose values are signed by the restriction of~$\signature$ to~$[\level]$ and avoiding the patterns~\eqref{eqn:patternsBaxterCambrian}. The parent of a permutation in~$\generatingTree_\signature\Bax$ is obtained by deleting its maximal value. See \fref{fig:GeneratingTreeBaxterCambrian}. 

As in the proof of Proposition~\ref{prop:GeneratingTree}, we consider the possible positions of~$\level + 1$ in the children of a permutation~$\tau$ at level~$\level$ in this generating tree~$\generatingTree_\signature\Bax$. Index by~$\{0, \dots, \level\}$ from left to right the gaps before the first letter, between two consecutive letters, and after the last letter of~$\tau$. \defn{Free gaps} are those where placing~$\level + 1$ does not create a pattern of~\eqref{eqn:patternsBaxterCambrian}. Free gaps are marked with a blue dot in \fref{fig:GeneratingTreeBaxterCambrian}. It is important to observe that gap~$0$ as well as the gaps immediately after~$\level-1$ and~$\level$ are always free, no matter~$\tau$ or the signature~$\signature$.

Define the \defn{free-gap-type} of~$\tau$ to be the pair~$(\ell,r)$ where~$\ell$ (resp.~$r$) denote the number of free gaps on the left (resp.~right) of~$\level$ in~$\tau$. For a signature~$\signature$, let~$\BC_\signature(\ell,r)$ denote the number of free-gap-type~$(\ell,r)$ weak order maximal permutations of $\signature$-Baxter-Cambrian classes. These refined Baxter-Cambrian numbers enables us to write inductive equations.

\begin{proposition}
\label{prop:inductionBaxterCambrian}
Consider two signatures~$\signature \in \pm^n$ and~$\signature' \in \pm^{n-1}$, where~$\signature'$ is obtained by deleting the last sign of~$\signature$. Then
\[
\BC_\signature(\ell,r) = 
\begin{cases}
\sum\limits_{\ell' \ge \ell} \BC_{\signature'}(\ell',r-1) \; + \; \sum\limits_{r' \ge r} \BC_{\signature'}(\ell-1,r')
& \text{if } \signature_{n-1} = \signature_{n}, \vspace*{.2cm}\quad(=) \\
\delta_{\ell = 1} \cdot \delta_{r \ge 2} \cdot\!\!\! \sum\limits_{\substack{\ell' \ge r-1 \\ r' \ge 1}} \BC_{\signature'}(\ell',r') \; + \; \delta_{\ell \ge 2} \cdot \delta_{r = 1} \cdot\!\!\! \sum\limits_{\substack{\ell' \ge 1 \\ r' \ge \ell-1}} \BC_{\signature'}(\ell',r') \;
& \text{if } \signature_{n-1} \ne \signature_{n}, \quad(\ne)
\end{cases}
\]
where~$\delta$ denote the Kronecker~$\delta$ (defined by~$\delta_X = 1$ if~$X$ is satisfied and~$0$ otherwise).
\end{proposition}

\begin{proof}
\hvFloat[floatPos=p, capWidth=h, capPos=r, capAngle=90, objectAngle=90, capVPos=c, objectPos=c]{figure}
{\scalebox{.9}{
\begin{tikzpicture}
  \node(T2) at (0,0) {
    \begin{tikzpicture}[xscale=1.3, yscale=1.5]
	  \node(P4321)  at (0,0) {$\freeGap\down{4}\freeGap\down{3}\freeGap\up{2}\bannedGap\down{1}\bannedGap$};
	  \node(P3421)  at (1,0) {$\freeGap\down{3}\freeGap\down{4}\freeGap\up{2}\bannedGap\down{1}\bannedGap$};
	  \node(P3241)  at (2,0) {$\freeGap\down{3}\freeGap\up{2}\freeGap\down{4}\freeGap\down{1}\bannedGap$};
	  \node(P4231)  at (3,0) {$\freeGap\down{4}\freeGap\up{2}\freeGap\down{3}\freeGap\down{1}\bannedGap$};
	  \node(P2431)  at (4,0) {$\freeGap\up{2}\bannedGap\down{4}\freeGap\down{3}\freeGap\down{1}\bannedGap$};
	  \node(P2341)  at (5,0) {$\freeGap\up{2}\bannedGap\down{3}\freeGap\down{4}\freeGap\down{1}\bannedGap$};
	  \node(P4213)  at (6,0) {$\freeGap\down{4}\freeGap\up{2}\freeGap\down{1}\freeGap\down{3}\freeGap$};
	  \node(P2413)  at (7,0) {$\freeGap\up{2}\bannedGap\down{4}\freeGap\down{1}\freeGap\down{3}\freeGap$};
	  \node(P2143)  at (8,0) {$\freeGap\up{2}\bannedGap\down{1}\bannedGap\down{4}\freeGap\down{3}\freeGap$};
	  \node(P2134)  at (9,0) {$\freeGap\up{2}\bannedGap\down{1}\bannedGap\down{3}\freeGap\down{4}\freeGap$};
	  \node(P4312)  at (10,0) {$\freeGap\down{4}\freeGap\down{3}\freeGap\down{1}\bannedGap\up{2}\bannedGap$};
	  \node(P3412)  at (11,0) {$\freeGap\down{3}\freeGap\down{4}\freeGap\down{1}\bannedGap\up{2}\bannedGap$};
	  \node(P3142)  at (12,0) {$\freeGap\down{3}\freeGap\down{1}\freeGap\down{4}\freeGap\up{2}\bannedGap$};
	  \node(P3124)  at (13,0) {$\freeGap\down{3}\freeGap\down{1}\freeGap\up{2}\freeGap\down{4}\freeGap$};
	  \node(P4132)  at (14,0) {$\freeGap\down{4}\freeGap\down{1}\bannedGap\down{3}\freeGap\up{2}\bannedGap$};
	  \node(P1342)  at (15,0) {$\freeGap\down{1}\bannedGap\down{3}\freeGap\down{4}\freeGap\up{2}\bannedGap$};
	  \node(P1324)  at (16,0) {$\freeGap\down{1}\bannedGap\down{3}\freeGap\up{2}\freeGap\down{4}\freeGap$};
	  \node(P4123)  at (17,0) {$\freeGap\down{4}\freeGap\down{1}\bannedGap\up{2}\freeGap\down{3}\freeGap$};
	  \node(P1243)  at (18,0) {$\freeGap\down{1}\bannedGap\up{2}\bannedGap\down{4}\freeGap\down{3}\freeGap$};
	  \node(P1234)  at (19,0) {$\freeGap\down{1}\bannedGap\up{2}\bannedGap\down{3}\freeGap\down{4}\freeGap$};
	  \node(P321)   at (1,1) {$\freeGap\down{3}\freeGap\up{2}\freeGap\down{1}\bannedGap$};
	  \node(P231)   at (4,1) {$\freeGap\up{2}\freeGap\down{3}\freeGap\down{1}\bannedGap$};
	  \node(P213)   at (7.5,1) {$\freeGap\up{2}\freeGap\down{1}\freeGap\down{3}\freeGap$};
	  \node(P312)   at (11.5,1) {$\freeGap\down{3}\freeGap\down{1}\freeGap\up{2}\freeGap$};
	  \node(P132)   at (15,1) {$\freeGap\down{1}\bannedGap\down{3}\freeGap\up{2}\freeGap$};
	  \node(P123)   at (18,1) {$\freeGap\down{1}\bannedGap\up{2}\freeGap\down{3}\freeGap$};
	  \node(P21)    at (4,2) {$\freeGap\up{2}\freeGap\down{1}\freeGap$};
	  \node(P12)    at (15,2) {$\freeGap\down{1}\freeGap\up{2}\freeGap$};
	  \node(P1) at (9.5,3) {$\freeGap\down{1}\freeGap$};
	
	  \draw (P21) -- (P1);
	  \draw (P12) -- (P1);
	
	  \draw (P321) -- (P21);
	  \draw (P231) -- (P21);
	  \draw (P213) -- (P21);
	
	  \draw (P312) -- (P12);
	  \draw (P132) -- (P12);
	  \draw (P123) -- (P12);
	
	  \draw (P4321) -- (P321);
	  \draw (P3421) -- (P321);
	  \draw (P3241) -- (P321);
	
	  \draw (P4231) -- (P231);
	  \draw (P2431) -- (P231);
	  \draw (P2341) -- (P231);
	
	  \draw (P4213) -- (P213);
	  \draw (P2413) -- (P213);
	  \draw (P2143) -- (P213);
	  \draw (P2134) -- (P213);
	
	  \draw (P4312) -- (P312);
	  \draw (P3412) -- (P312);
	  \draw (P3142) -- (P312);
	  \draw (P3124) -- (P312);
	
	  \draw (P4132) -- (P132);
	  \draw (P1342) -- (P132);
	  \draw (P1324) -- (P132);
	
	  \draw (P4123) -- (P123);
	  \draw (P1243) -- (P123);
	  \draw (P1234) -- (P123);
	\end{tikzpicture}
  };

  \node(T2) at (0.3, -7) {
    \begin{tikzpicture}[xscale=1.3, yscale=1.5]
	  \node(P4321)  at (0,0) {$\freeGap\down{4}\freeGap\down{3}\freeGap\down{2}\freeGap\up{1}\freeGap$};
	  \node(P3421)  at (1,0) {$\freeGap\down{3}\freeGap\down{4}\freeGap\down{2}\freeGap\up{1}\freeGap$};
	  \node(P3241)  at (2,0) {$\freeGap\down{3}\freeGap\down{2}\bannedGap\down{4}\freeGap\up{1}\freeGap$};
	  \node(P3214)  at (3,0) {$\freeGap\down{3}\freeGap\down{2}\bannedGap\up{1}\bannedGap\down{4}\freeGap$};
	  \node(P4231)  at (4,0) {$\freeGap\down{4}\freeGap\down{2}\bannedGap\down{3}\freeGap\up{1}\freeGap$};
	  \node(P2431)  at (5,0) {$\freeGap\down{2}\freeGap\down{4}\freeGap\down{3}\freeGap\up{1}\freeGap$};
	  \node(P2341)  at (6,0) {$\freeGap\down{2}\freeGap\down{3}\freeGap\down{4}\freeGap\up{1}\freeGap$};
	  \node(P2314)  at (7,0) {$\freeGap\down{2}\freeGap\down{3}\freeGap\up{1}\bannedGap\down{4}\freeGap$};
	  \node(P4213)  at (8,0) {$\freeGap\down{4}\freeGap\down{2}\bannedGap\up{1}\bannedGap\down{3}\freeGap$};
	  \node(P2413)  at (9,0) {$\freeGap\down{2}\freeGap\down{4}\freeGap\up{1}\bannedGap\down{3}\freeGap$};
	  \node(P2134)  at (10,0) {$\freeGap\down{2}\freeGap\up{1}\bannedGap\down{3}\freeGap\down{4}\freeGap$};
	  \node(P4312)  at (11,0) {$\freeGap\down{4}\freeGap\down{3}\freeGap\up{1}\bannedGap\down{2}\freeGap$};
	  \node(P3412)  at (12,0) {$\freeGap\down{3}\freeGap\down{4}\freeGap\up{1}\bannedGap\down{2}\freeGap$};
	  \node(P3124)  at (13,0) {$\freeGap\down{3}\freeGap\up{1}\bannedGap\down{2}\bannedGap\down{4}\freeGap$};
	  \node(P4132)  at (14,0) {$\freeGap\down{4}\freeGap\up{1}\bannedGap\down{3}\freeGap\down{2}\freeGap$};
	  \node(P1432)  at (15,0) {$\freeGap\up{1}\freeGap\down{4}\freeGap\down{3}\freeGap\down{2}\freeGap$};
	  \node(P1342)  at (16,0) {$\freeGap\up{1}\freeGap\down{3}\freeGap\down{4}\freeGap\down{2}\freeGap$};
	  \node(P1324)  at (17,0) {$\freeGap\up{1}\freeGap\down{3}\freeGap\down{2}\bannedGap\down{4}\freeGap$};
	  \node(P4123)  at (18,0) {$\freeGap\down{4}\freeGap\up{1}\bannedGap\down{2}\bannedGap\down{3}\freeGap$};
	  \node(P1423)  at (19,0) {$\freeGap\up{1}\freeGap\down{4}\freeGap\down{2}\bannedGap\down{3}\freeGap$};
	  \node(P1243)  at (20,0) {$\freeGap\up{1}\freeGap\down{2}\freeGap\down{4}\freeGap\down{3}\freeGap$};
	  \node(P1234)  at (21,0) {$\freeGap\up{1}\freeGap\down{2}\freeGap\down{3}\freeGap\down{4}\freeGap$};
	  \node(P321)   at (1.5,1) {$\freeGap\down{3}\freeGap\down{2}\freeGap\up{1}\freeGap$};
	  \node(P231)   at (5.5,1) {$\freeGap\down{2}\freeGap\down{3}\freeGap\up{1}\freeGap$};
	  \node(P213)   at (9,1) {$\freeGap\down{2}\freeGap\up{1}\bannedGap\down{3}\freeGap$};
	  \node(P312)   at (12,1) {$\freeGap\down{3}\freeGap\up{1}\bannedGap\down{2}\freeGap$};
	  \node(P132)   at (15.5,1) {$\freeGap\up{1}\freeGap\down{3}\freeGap\down{2}\freeGap$};
	  \node(P123)   at (19.5,1) {$\freeGap\up{1}\freeGap\down{2}\freeGap\down{3}\freeGap$};
	  \node(P21)    at (5.5,2) {$\freeGap\down{2}\freeGap\up{1}\freeGap$};
	  \node(P12)    at (15.5,2) {$\freeGap\up{1}\freeGap\down{2}\freeGap$};
	  \node(P1) at (10.5,3) {$\freeGap\up{1}\freeGap$};
	
	  \draw (P21) -- (P1);
	  \draw (P12) -- (P1);
	
	  \draw (P321) -- (P21);
	  \draw (P231) -- (P21);
	  \draw (P213) -- (P21);
	
	  \draw (P312) -- (P12);
	  \draw (P132) -- (P12);
	  \draw (P123) -- (P12);
	
	  \draw (P4321) -- (P321);
	  \draw (P3421) -- (P321);
	  \draw (P3241) -- (P321);
	  \draw (P3214) -- (P321);
	
	  \draw (P4231) -- (P231);
	  \draw (P2431) -- (P231);
	  \draw (P2341) -- (P231);
	  \draw (P2314) -- (P231);
	
	  \draw (P4213) -- (P213);
	  \draw (P2413) -- (P213);
	  \draw (P2134) -- (P213);
	
	  \draw (P4312) -- (P312);
	  \draw (P3412) -- (P312);
	  \draw (P3124) -- (P312);
	
	  \draw (P4132) -- (P132);
	  \draw (P1432) -- (P132);
	  \draw (P1342) -- (P132);
	  \draw (P1324) -- (P132);
	
	  \draw (P4123) -- (P123);
	  \draw (P1423) -- (P123);
	  \draw (P1243) -- (P123);
	  \draw (P1234) -- (P123);
	\end{tikzpicture}
  };

  \node(T3) at (0.3, -10) {
	};
\end{tikzpicture}
}}
{The generating trees~$\generatingTree_\signature\Bax$ for the signatures~$\signature = {-}{+}{-}{-}$ (top) and~$\signature = {+}{-}{-}{-}$ (bottom). Free gaps are marked with a blue dot.}
{fig:GeneratingTreeBaxterCambrian}
Assume first that~$\signature_{n-1} = \signature_n$. Consider two permutations~$\tau$ and~$\tau'$ at level~$n$ and~$n-1$ in~$\generatingTree_\signature\Bax$ such that~$\tau'$ is obtained by deleting~$n$ in~$\tau$. Denote by~$\alpha$ and~$\beta$ the gaps immediately after~$n-1$ and~$n$ in~$\tau$, by~$\alpha'$ the gap immediately after~$n-1$ in~$\tau'$, and by~$\beta'$ the gap in~$\tau'$ where we insert~$n$ to get~$\tau$. Then, besides gaps~$0$, $\alpha$ and~$\beta$, the free gaps of~$\tau$ are precisely the free gaps of~$\tau'$ not located between gaps~$\alpha'$ and~$\beta'$. Indeed,
\begin{itemize}
\item inserting~$d \eqdef n+1$ just after a value~$a$ located between~$b \eqdef n-1$ and~$c \eqdef n$ in~$\tau$ would create a pattern~$b \dash ad \dash c$ or~$c \dash ad \dash b$ with~$\signature_b = \signature_c$;
\item conversely, consider a gap~$\gamma$ of~$\tau$ not located between~$\alpha$ and~$\beta$. If inserting~$n+1$ at~$\gamma$ in~$\tau$ creates a forbidden pattern of~\eqref{eqn:patternsBaxterCambrian} with~$c = n$, then inserting~$n$ at~$\gamma$ in~$\tau'$ would also create the same forbidden pattern of~\eqref{eqn:patternsBaxterCambrian} with~$c = n-1$. Therefore, all free gaps not located between gaps~$\alpha'$ and~$\beta'$ remain free.
\end{itemize}
Let~$(\ell,r)$ denote the free-gap-type of~$\tau$ and~$(\ell',r')$ denote the free-gap-type of~$\tau'$. We obtain that
\begin{itemize}
\item $\ell' \ge \ell$ and~$r' = r-1$ if~$n$ is inserted on the left of~$n-1$;
\item $\ell' = \ell-1$ and~$r' \ge r$ if~$n$ is inserted on the right of~$n-1$.
\end{itemize}
The formula follows immediately when~$\signature_{n-1} = \signature_n$.

Assume now that~$\signature_{n-1} = -\signature_n$, and keep the same notations as before. Using similar arguments, we observe that besides gaps~$0$, $\alpha$ and~$\beta$, the free gaps of~$\tau$ are precisely the free gaps of~$\tau'$ located between gaps~$\alpha'$ and~$\beta'$. Therefore, we obtain that
\begin{itemize}
\item $\ell = 1$, $r \ge 2$, and~$\ell' \ge r-1$ if~$n$ is inserted on the left of~$n-1$;
\item $\ell \ge 2$, $r = 1$, and~$r' \ge \ell-1$ if~$n$ is inserted on the right of~$n-1$.
\end{itemize}
The formula follows for~$\signature_{n-1} = -\signature_n$.
\end{proof}

\enlargethispage{.3cm}
Before applying these formulas to obtain bounds on~$\BC_\signature$ for arbitrary signatures~$\signature$, let us consider two special signatures: the constant and the alternating signature.

\para{Alternating signature}
Since it is the easiest, we start with the \defn{alternating signature}~$({+}{-})^{\frac{n}{2}}$ (where we define~$({+}{-})^{\frac{n}{2}}$ to be~$({+}{-})^{m}{+}$ when~$n = 2m+1$ is odd).

\begin{proposition}
\label{prop:BaxterAlternating}
The Baxter-Cambrian numbers for alternating signatures are central binomial coefficients (see \href{https://oeis.org/A000984}{\cite[A000984]{OEIS}}):
\[
\BC_{({+}{-})^{\frac{n}{2}}} = \binom{2n-2}{n-1}.
\]
\end{proposition}

\begin{proof}
We prove by induction on~$n$ that the refined Baxter-Cambrian numbers are
\[
\BC_{({+}{-})^{\frac{n}{2}}}(\ell,r) = \delta_{\ell = 1} \cdot \delta_{r \ge 2} \cdot \binom{2n-2-r}{n-r} + \delta_{\ell \ge 2} \cdot \delta_{r = 1} \cdot \binom{2n-2-\ell}{n-\ell}.
\]
This is true for~$n = 2$ since~$\BC_{{+}{-}}(1,2) = 1$ (counting the permutation~$21$) and~$\BC_{{+}{-}}(2,1) = 1$ (counting the permutation~$12$). Assume now that it is true for some~$n \in \N$. Then Equation~$(\ne)$ of Proposition~\ref{prop:inductionBaxterCambrian} shows that
\begin{align*}
\BC_{({+}{-})^{\frac{n+1}{2}}}(\ell,r) 
& = \delta_{\ell = 1} \cdot \delta_{r \ge 2} \cdot\!\!\! \sum_{\ell' \ge r-1} \binom{2n-2-\ell'}{n-\ell'} + \delta_{\ell \ge 2} \cdot \delta_{r = 1} \cdot\!\!\!  \sum_{r' \ge \ell-1} \binom{2n-2-r'}{n-r'} \\
& = \delta_{\ell = 1} \cdot \delta_{r \ge 2} \cdot \binom{2n - r}{n + 1 - r} + \delta_{\ell \ge 2} \cdot \delta_{r = 1} \cdot \binom{2n - \ell}{n + 1 - \ell},
\end{align*}
since a sum of binomial coefficients along a diagonal~$\sum_{i = 0}^p \binom{q+i}{i}$ simplifies to the binomial coefficient~$\binom{q+p+1}{p}$ by multiple applications of Pascal's rule. Finally, we conclude observing that
\[
\BC_{({+}{-})^{\frac{n}{2}}} = \sum_{\ell,r \in [n]} \BC_{({+}{-})^{\frac{n}{2}}}(\ell,r) = 2 \sum_{u \ge 2} \binom{2n-2-u}{n-u} = 2 \binom{2n-3}{n-2} = \binom{2n-2}{n-1}.
\]
Remark~\ref{rem:BaxterGeneratingFunctions} provides an alternative analytic proof for this result.
\end{proof}

\begin{remark}[Properties of the generating tree~$\generatingTree_{({+}{-})^{\frac{n}{2}}}\Bax$] Observe that:
\begin{enumerate}[(i)]
\item A permutation at level~$\level$ with $k$ free gaps has~$k$ children, whose numbers of free gaps are~$3, 3, 4, 5, \dots, k+1$ respectively (compare to Lemma~\ref{lem:GeneratingTree}). This can already be observed on the generating tree~$\generatingTree_{{+}{-}{+}{-}}\Bax$ of \fref{fig:GeneratingTreeBaxterCambrian}.
\item For a permutation~$\tau$ at level~$\level$ with $k$ free gaps, there are precisely~$\binom{k + 2p -2}{p}$ permutations at level~$\level + p$ that have~$\tau$ as a subword, for any~$p \in \N$.
\item The number of permutations at level~$\level$ with~$k$ free gaps is~$2 \binom{2\level - 1 - k}{\level + 1 - k}$. Counting permutations at level~$\level$ and~$\level+1$ according to their number of free gaps gives
\[
\binom{2\level-2}{\level-1} = \sum_{k \ge 3} 2 \binom{2\level - 1 - k}{\level + 1 - k}
\qquad\text{and}\qquad
\binom{2\level}{\level} = \sum_{k \ge 3} 2k \binom{2\level - 1 - k}{\level + 1 - k}.
\]
\item Slight perturbations of the alternating signature yields interesting signatures for which we can give closed formulas for the Baxter-Cambrian number. For example, consider the signature~${+}{+}({+}{-})^{\frac{n}{2}-1}$ obtained from the alternating one by switching the second sign. Its Baxter-Cambrian number is given by a sum of three almost-central binomial coefficients:
\[
\BC_{{+}{+}({+}{-})^{\frac{n}{2}-1}} = 2\binom{2n - 6}{n - 4} + \binom{2n - 2}{n - 1}.
\]
\end{enumerate}
\end{remark}

\para{Constant signature}
We now consider the \defn{constant signature}~$(+)^n$. The number~$\BC_{({+})^n}$ is the classical \defn{Baxter number} (see \href{https://oeis.org/A001181}{\cite[A001181]{OEIS}}) defined by
\[
\BC_{(+)^n} = B_n = \binom{n+1}{1}^{-1} \binom{n+1}{2}^{-1} \sum_{k=1}^n \binom{n+1}{k-1} \binom{n+1}{k} \binom{n+1}{k+1}.
\]
These numbers have been extensively studied, see in particular~\cite{ChungGrahamHoggattKleiman, Mallows, DulucqGuibert1, DulucqGuibert2, YaoChenChengGraham, FelsnerFusyNoyOrden, BonichonBousquetMelouFusy, LawReading, Giraudo}. The Baxter number~$B_n$ counts several families:
\begin{itemize}
\item Baxter permutations of~$[n]$, \ie permutations avoiding the patterns~$b \dash da \dash c$ and~$c \dash ad \dash b$,
\item weak order maximal (resp.~minimal) permutations of Baxter congruence classes on~$\fS_n$, \ie permutations avoiding the patterns~$b \dash ad \dash c$ and~$c \dash ad \dash b$ (resp.~$b \dash da \dash c$ and~$c \dash da \dash b$),
\item pairs of twin binary trees on~$n$ nodes,
\item diagonal rectangulations of an~$n \times n$ grid,
\item plane bipolar orientations with~$n$ edges,
\item non-crossing triples of path with~$k-1$ north steps and~$n-k$ east steps, for all~$k \in [n]$,
\item etc.
\end{itemize}
Bijections between all these \defn{Baxter families} are discussed in~\cite{DulucqGuibert1, DulucqGuibert2, FelsnerFusyNoyOrden, BonichonBousquetMelouFusy}.

\begin{remark}[Two proofs of the summation formula]
\label{rem:BaxterGeneratingFunctions}
There are essentially two ways to obtain the above summation formula for Baxter numbers: it was first proved analytically in~\cite{ChungGrahamHoggattKleiman}, and then bijectively in~\cite{Viennot-Baxter, DulucqGuibert2, FelsnerFusyNoyOrden}. Let us shortly comment on these two techniques and discuss the limits of their extension to the Baxter-Cambrian setting.
\begin{enumerate}[(i)]
\item The \defn{bijective proofs} in~\cite{DulucqGuibert2, FelsnerFusyNoyOrden} transform pairs of binary trees to triples of non-crossing paths, and then use the Gessel-Viennot determinant lemma~\cite{GesselViennot} to get the summation formula. The middle path of these triples is given by the canopy of the twin binary trees, while the other two paths are given by the structure of the trees. We are not yet able to adapt this technique to provide summation formulas for all Baxter-Cambrian numbers.
\item The \defn{analytic proof} in~\cite{ChungGrahamHoggattKleiman} is based on Equation~$(=)$ of Proposition~\ref{prop:inductionBaxterCambrian} and can be partially adapted to arbitrary signatures as follows. Define the \defn{extension} of a signature~${\signature \in \pm^n}$ by a signature~$\update \in \pm^m$ to be the signature~$\signature \extension \update \in \pm^{n+m}$ such that~$(\signature \extension \update)_i = \signature_i$ for~$i \in [n]$ and~$(\signature \extension \update)_{n+j} = \update_j \cdot (\signature \extension \update)_{n+j-1}$ for~$j \in m$. For example, ${+}{+}{-} \extension {+}{-}{-}{+} = {+}{+}{-}{-}{+}{-}{-}$. Then for any~$\signature \in \pm^n$ and~$\update \in \pm^m$, we have
\[
\BC_{\signature\extension\update} = \sum_{\ell, r \ge 1} X_\update(\ell, r) \, \BC_\signature(\ell,r),
\]
where the coefficients~$X_\update(\ell,r)$ are obtained inductively from the formulas of Proposition~\ref{prop:inductionBaxterCambrian}. Namely, for any~$\ell, r \ge 1$, we have~${X_\varnothing(\ell,r) = 1}$ and
\begin{align*}
X_{(+\update)}(\ell,r) & = \sum\limits_{1 \le \ell' \le \ell} X_{\update}(\ell',r+1) \; + \; \sum\limits_{1 \le r' \le r} X_{\update}(\ell+1,r'), \\
X_{(-\update)}(\ell,r) & = \sum\limits_{2 \le \ell' \le r+1} X_{\update}(\ell',1) \; + \; \sum\limits_{2 \le r' \le \ell+1} X_{\update}(1,r').
\end{align*}
These equations translate on the generating function~$\fX_\update(u,v) \eqdef \sum_{\ell, r \ge 1} X_\update(\ell,r) u^{\ell-1} v^{r-1}$ to the formulas~$\fX_\varnothing(u,v) = \frac{1}{(1-u)(1-v)}$ and
\begin{align*}
\fX_{(+\update)}(u,v) & = \frac{\fX_\update(u,v) - \fX_\update(u,0)}{(1-u) v} + \frac{\fX_\update(u,v)-\fX_\update(0,v)}{u (1-v)}, \\
\fX_{(-\update)}(u,v) & = \frac{\fX_\update(v,0) - \fX_\update(0,0)}{(1-u) (1-v) v} + \frac{\fX_\update(0,u) - \fX_\update(0,0)}{u (1-u) (1-v)}.
\end{align*}
Note that the~$u/v$-symmetry of~$\fX_\update(u,v)$ is reflected in a symmetry on these inductive equations. We can thus write this generating function~$\fX_\update(u,v)$ as
\[
\fX_\update(u,v) = \sum_{\substack{i,j \ge 0 \\ k \in [|\update|+1]}} Y_\update^{i, j, k} \frac{(-u)^i \, (-v)^j}{(1-u)^{|\update|+2-k} (1-v)^k},
\]
where the non-vanishing coefficients~$Y_\update^{i,j,k}$ are computed inductively by~$Y_\varnothing^{0,0,1} = 1$~and
\begin{align*}
Y_{(+\update)}^{i,j,k} & = \binom{k}{j+1} Y_\update^{i,0,k} - Y_\update^{i,j+1,k} + \binom{|\update|+3-k}{i+1} Y_\update^{0,j,k-1} - Y_\update^{i+1,j,k-1}, \\
Y_{(-\update)}^{i,j,k} & = \binom{k-1}{j} \! \bigg[ \! \binom{|\update|+2-k}{i+1} Y_\update^{0,0,k} - Y_\update^{i+1,0,k} \bigg] + \binom{|\update|+2-k}{i} \! \bigg[ \! \binom{k-1}{j+1} Y_\update^{0,0,k-1} - Y_\update^{0,j+1,k-1} \bigg].
\end{align*}
We used that~$Y_\update^{i,j,k} = Y_\update^{j,i,|\update|+2-k}$ to simplify the second equation. Note that this decomposition of~$\fX_\update$ is not unique and the inductive equations on~$Y_\update^{i,j,k}$ follow from a particular choice of such a decomposition.

At that stage, F.~Chung, R.~Graham, V.~Hoggatt, and M.~Kleiman~\cite{ChungGrahamHoggattKleiman}, guess and check that the first equation is always satisfied by
\[
Y_{(+)^{n-1}}^{i,j,k} = \frac{\binom{n+1}{k} \binom{n+1}{k+i+1} \binom{n+1}{k-j-1} \big[ \! \binom{k+i-2}{i} \binom{n+j-k-1}{j} - \binom{k+i-2}{i-1} \binom{n+j-k-1}{j-1} \! \big]}{\binom{n+1}{1} \binom{n+1}{2}} 
\]
from which they derive immediately that
\begin{align*}
\BC_{(+)^n} & = \BC_{+\extension(+)^{n-1}} = \sum_{\ell, r \ge 1} X_{(+)^{n-1}}(\ell, r) \, \BC_+(\ell,r) = X_{(+)^{n-1}}(1,1) = \fX_{(+)^{n-1}}(0,0) \\
& = \sum_{k \in [n]} Y_{(+)^{n-1}}^{0,0,k} = \binom{n+1}{1}^{-1} \binom{n+1}{2}^{-1} \sum_{k=1}^n \binom{n+1}{k-1} \binom{n+1}{k} \binom{n+1}{k+1}.
\end{align*}

Unfortunately, we have not been able to guess a closed formula for the coefficients~$Y_{(-)^n}^{i,j,k}$. Note that it would be sufficient to understand the coefficients~$Y_{(-)^n}^{i,0,k}$ for which we observed empirically that
\[
Y_{(-)^n}^{0,0,k} = C_n, \quad Y_{(-)^n}^{i,0,0} = Y_{(-)^n}^{i,0,1} = \binom{2n}{n-1-i}\binom{n-1+i}{i}\bigg/n \quad\text{and}\quad Y_{(-)^n}^{i,0,n+1-i} = \sum_{p = i}^{n-1} C_{n-1-p}C_p.
\]
See~\href{https://oeis.org/A000108}{\cite[A000108]{OEIS}}, \href{https://oeis.org/A234950}{\cite[A234950]{OEIS}}, and \href{https://oeis.org/A028364}{\cite[A028364]{OEIS}}.
This would provide an alternative proof of Proposition~\ref{prop:BaxterAlternating} as we would obtain that
\[
\BC_{({+}{-})^{\frac{n}{2}}} = \BC_{+\extension(-)^{n-1}} = \fX_{(-)^{n-1}}(0,0) = \sum_{k \in [n]} Y_{(-)^{n-1}}^{0,0,k} = n C_{n-1} = \binom{2n-2}{n-1}.
\]
However, even if we were not able to work out the coefficients~$Y_{(-)^n}^{i,0,k}$, we still obtain another proof Proposition~\ref{prop:BaxterAlternating} by checking directly on the inductive equations on~$\fX_\update(u,v)$ that
\[
\fX_{(-)^n}(u,v) = \sum_{k \in [n]} \binom{2n-1-k}{n-1} \bigg[ \frac{1}{(1-u) (1-v)^{k+1}} + \frac{1}{(1-u)^{k+1} (1-v)} \bigg],
\]
from which we obtain
\begin{align*}
\BC_{({+}{-})^{\frac{n}{2}}} & = \BC_{+\extension(-)^{n-1}} = \fX_{(-)^{n-1}}(0,0) = \sum_{k \in [n-1]} 2\binom{2n-3-k}{n-2} \\
& = 2 \sum_{k = 2}^{n} \binom{2n-2-k}{n-2} = 2\binom{2n-2}{n-1} - 2\binom{2n-3}{n-2} = \binom{2n-2}{n-1}.
\end{align*}
For the prior to last equality, choose $n-1$ positions among~$2n-2$ and group according to the first position~$k$.
\end{enumerate}
\end{remark}

\para{Arbitrary signatures}
We now come back to an arbitrary signature~$\signature$. We were not able to derive summation formulas for arbitrary signatures using the techniques presented in Remark~\ref{rem:BaxterGeneratingFunctions} above. However, we use here the inductive formulas of Proposition~\ref{prop:inductionBaxterCambrian} to bound the Baxter-Cambrian number~$\BC_\signature$ for an arbitrary signature~$\signature$.

For this, we consider the matrix~$\BCMat_\signature \eqdef \big( \BC_\signature(\ell, r) \big)_{\ell, r \in [n]}$. The inductive formulas of Proposition~\ref{prop:inductionBaxterCambrian} provide an efficient inductive algorithm to compute this matrix~$\BCMat_\signature$ and thus the~$\signature$-Baxter-Cambrian number~${\BC_\signature = \sum_{\ell, r \in [n]} \BC_\signature(\ell,r)}$. Namely, if~$\signature$ is obtained by adding a sign at the end of~$\signature'$, then each entry of~$\BCMat_\signature$ is the sum of entries of~$\BCMat_{\signature'}$ in a region depending on whether~$\signature_n = \signature_{n-1}$. These regions are sketched in \fref{fig:rulesMatrixComputation} and examples of such computations appear in \fref{fig:examplesComputationsMatrices}.

\begin{figure}[h]
  \medskip
  \centerline{
  	\begin{overpic}[scale=1]{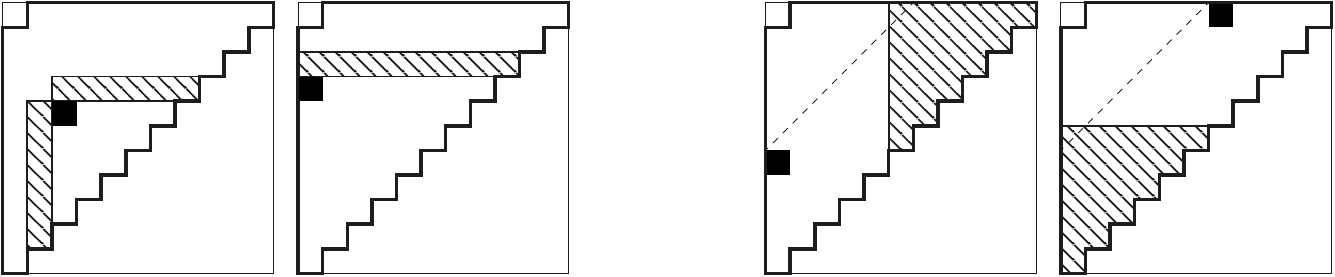}
		\put(15,-2.5){$\signature_n = \signature_{n-1}$}
		\put(72,-2.5){$\signature_n = -\signature_{n-1}$}
	\end{overpic}
  }
  \medskip
  \caption{Inductive computation of~$\BCMat_\signature$: the black entry of~$\BCMat_\signature$ is the sum of the entries of~$\BCMat_{\signature'}$ over the shaded region. Entries outside the upper triangular region always vanish. When~$\signature_n = -\signature_{n-1}$, the only non-vanishing entries of~$\BCMat_\signature$ are in the first row or in the first column.}
  \label{fig:rulesMatrixComputation}
\end{figure}

\begin{figure}[h]
  \centerline{\includegraphics{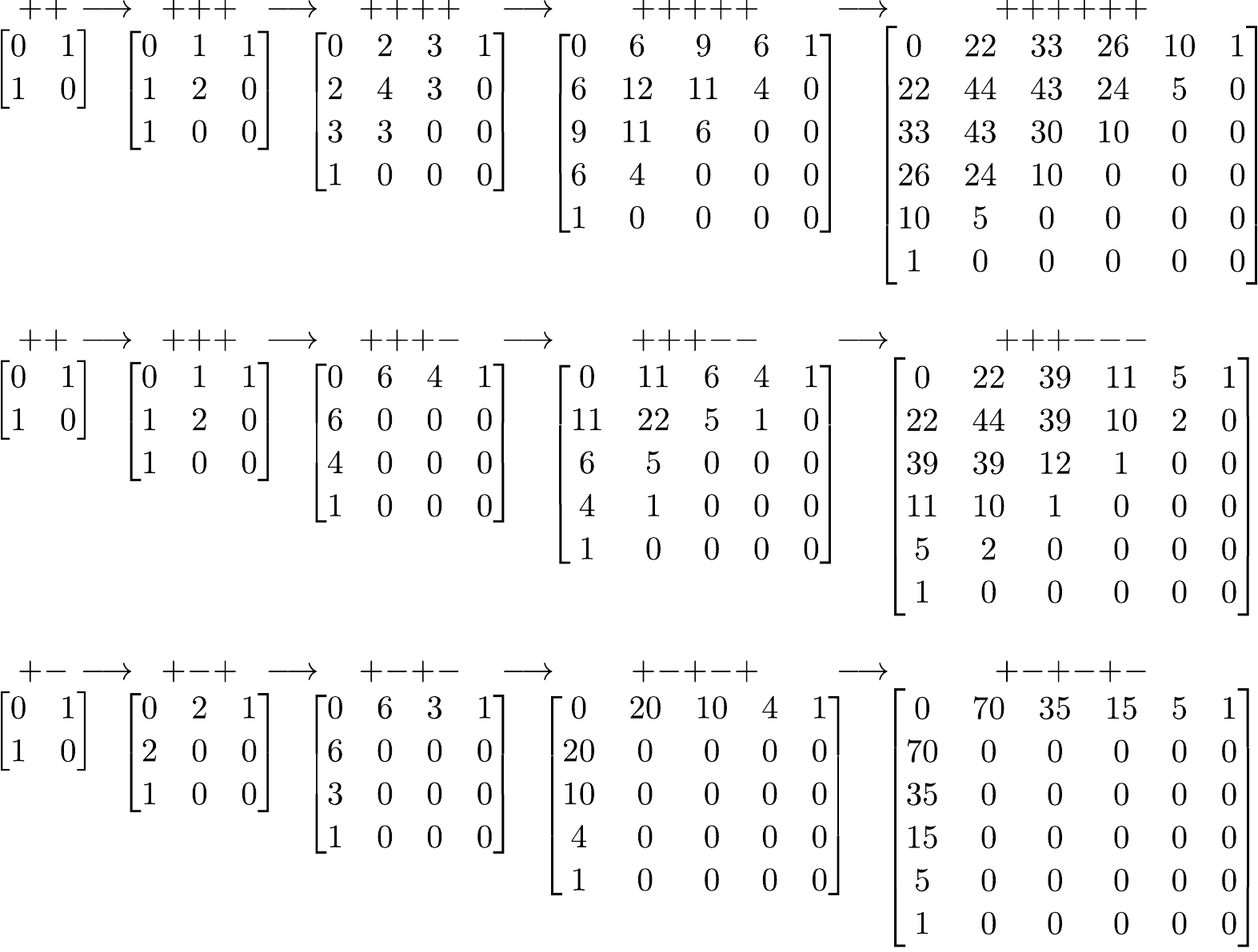}}
  \vspace{-.3cm}
  \caption{Inductive computation of~$\BCMat_\signature$, for~$\signature = ({+})^6$, $({+})^3({-})^3$ and~$({+}{-})^3$.}
  \label{fig:examplesComputationsMatrices}
\end{figure}

We observe that the transformations of \fref{fig:rulesMatrixComputation} are symmetric with respect to the diagonal of the matrix. Since~$\BCMat_{\signature_1\signature_2} = \begin{bmatrix} 0 & 1 \\ 1 & 0 \end{bmatrix}$ is symmetric, and~$\BCMat_\signature$ is obtained from~$\BCMat_{\signature_1\signature_2}$ by successive applications of these symmetric transformations, we obtain that~$\BCMat_\signature$ is always symmetric. Although this fact may seem natural to the reader, it is not at all immediate as there is an asymmetry on the three forced free gaps: for example gap~$0$ is always free.

\newcommand{\SE}{^\textsc{se}}
For a matrix~$M \eqdef (m_{i,j})$, we consider the matrix~$M\SE \eqdef \big( m\SE_{i,j} \big)$ where
\[
m\SE_{i,j} \eqdef \sum_{p \ge i, \; q \ge j} m_{p,q}
\]
is the sum of all entries located south-east of~$(i,j)$ (in matrix notation). Observe that~$(\BCMat_\signature)\SE_{1,1}$ is the sum of all entries of~$\BCMat_\signature$, and thus equals the $\signature$-Baxter-Cambrian number~$\BC_\signature$. Using \fref{fig:rulesMatrixComputation}, we obtain a similar rule to compute the entries of~$\BC_\signature\SE$ as sums of entries of~$\BC_{\signature'}\SE$ when~$\signature$ is obtained by adding a sign at the end of~$\signature'$. This rule is presented in \fref{fig:rulesSESMatrixComputation}.
 
\begin{figure}[h]
  \medskip
  \centerline{
  	\begin{overpic}[scale=1]{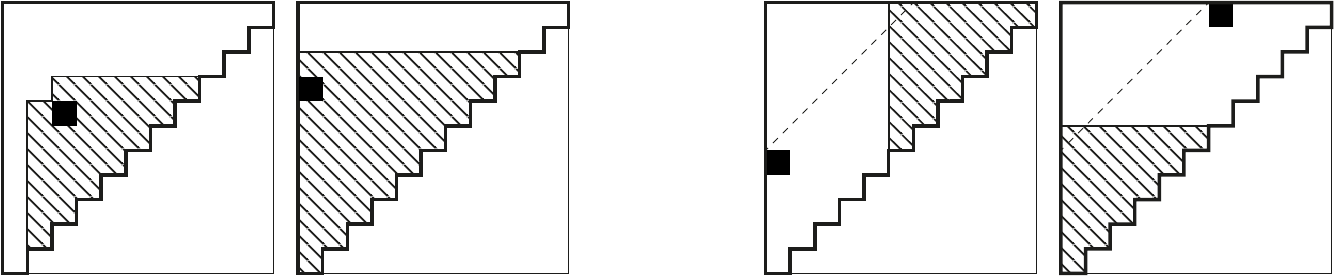}
		\put(15,-2.5){$\signature_n = \signature_{n-1}$}
		\put(72,-2.5){$\signature_n = -\signature_{n-1}$}
	\end{overpic}
  }
  \medskip
  \caption{Inductive computation of~$\BCMat_\signature\SE$: the black entry of~$\BCMat_\signature\SE$ is the sum of the entries of~$\BCMat_{\signature'}\SE$ over the shaded region. Entries outside the triangular shape always vanish. When~$\signature_n = -\signature_{n-1}$, the only non-vanishing entries of~$\BCMat_\signature\SE$ are in the first row or in the first column.}
  \label{fig:rulesSESMatrixComputation}
\end{figure}

This matrix interpretation of the formulas of Proposition~\ref{prop:inductionBaxterCambrian} provides us with tools to bound the Baxter-Cambrian numbers. For a signature~$\signature$, we denote by~$\switch(\signature)$ the set of gaps where~$\signature$ switches sign.

\begin{proposition}
\label{prop:BCSE}
For any two signatures~$\signature, \tilde\signature \in  \pm^n$, if~$\switch(\signature) \subset \switch(\tilde\signature)$ then~$\BC_\signature > \BC_{\tilde\signature}$.
\end{proposition}

\begin{proof}
For two matrices~$M \eqdef (m_{i,j})$ and~$\tilde M \eqdef (\tilde m_{i,j})$, we write~$M \succcurlyeq \tilde M$ when~$m_{i,j} \ge \tilde m_{i,j}$ for all indices~$i,j$ (entrywise comparison), and we write~$M \succ \tilde M$ when~$M \succcurlyeq \tilde M$ and~$M \ne \tilde M$. Consider four signatures~$\signature, \tilde\signature \in \pm^n$ and~$\signature', \tilde\signature' \in \pm^{n-1}$ such that~$\signature'$ (resp.~$\tilde\signature'$) is obtained by deleting the last sign of~$\signature$ (resp.~$\tilde\signature$). From \fref{fig:rulesSESMatrixComputation}, and using the fact that~$\BCMat_\signature$ is symmetric, we obtain that:
\begin{itemize}
\item if~$\signature_n = \signature_{n-1}$ while~$\tilde\signature_n = -\tilde\signature_{n-1}$, then~$\BCMat_{\signature'}\SE \succcurlyeq \BCMat_{\tilde\signature'}\SE$ implies $\BCMat_{\signature}\SE \succ \BCMat_{\tilde\signature}\SE$.
\item if either both~$\signature_n = \signature_{n-1}$ and~$\tilde\signature_n = \tilde\signature_{n-1}$, or both~$\signature_n = -\signature_{n-1}$ and~$\tilde\signature_n = -\tilde\signature_{n-1}$, then~$\BCMat_{\signature'}\SE \succ \BCMat_{\tilde\signature'}\SE$ implies $\BCMat_{\signature}\SE \succ \BCMat_{\tilde\signature}\SE$.
\end{itemize}
By repeated applications of these observations, we therefore obtain that~$\switch(\signature) \subset \switch(\tilde\signature)$ implies~$\BCMat_{\signature}\SE \succ \BCMat_{\tilde\signature}\SE$, and thus~$\BC_\signature > \BC_{\tilde\signature}$.
\end{proof}

\begin{corollary}
\label{coro:boundsBaxterCambrianNumbers}
Among all signatures of~$\pm^n$, the constant signature maximizes the Baxter-Cambrian number, while the alternating signature minimizes it: for all~$\signature \in \pm^n$,
\[
\binom{2n-2}{n-1} = \BC_{({+}{-})^{\frac{n}{2}}} \le \BC_\signature \le \BC_{(+)^n} = \binom{n+1}{1}^{-1} \binom{n+1}{2}^{-1} \sum_{k=1}^n \binom{n+1}{k-1} \binom{n+1}{k} \binom{n+1}{k-1}.
\]
\end{corollary}

\begin{remark}
The proof of Proposition~\ref{prop:BCSE} may seem unnecessarily intricate. Observe however that the situation is rather subtle:
\begin{itemize}
\item If~$\switch(\signature) \not\subseteq \switch(\tilde\signature)$, we may have~$\BC_\signature < \BC_{\tilde\signature}$ even if~$|\switch(\signature)| <  |\switch(\tilde\signature)|$. The smallest example is given by~$\BC_{{+}{+}{+}{-}{+}{+}{-}{-}{-}} = 18376 < 18544 = \BC_{{+}{+}{-}{+}{+}{+}{-}{+}{+}}$.
\item We may have~$\BCMat_\signature\SE \succcurlyeq \BCMat_{\tilde\signature}\SE$ but~$\BCMat_\signature \not\succcurlyeq \BCMat_{\tilde\signature}$. See the third column of \fref{fig:examplesComputationsMatrices}.
\end{itemize}
\end{remark}


\subsection{Geometric realizations}
\label{subsec:geometricRealizationsBaxter}

Using similar tools as in Section~\ref{subsec:geometricRealizations} and following~\cite{LawReading}, we present geometric realizations for pairs of twin Cambrian trees, for the Baxter-Cambrian lattice, and for the Baxter-Cambrian $\PSymbol\Bax$-symbol. For a partial order~$\prec$ on~$[n]$, we still define its \defn{incidence cone}~$\Cone(\prec)$ and its \defn{braid cone}~$\Cone\polar(\prec)$ as
\[
\Cone(\prec) \eqdef \cone\set{e_i-e_j}{\text{for all } i \prec j}
\quad\text{and}\quad
\Cone\polar(\prec) \eqdef \set{\b{x} \in \HH}{x_i \le x_j \text{ for all } i \prec j}.
\]
The cones~$\Cone(\graphTwin)$ for all pairs~$[\tree_\circ, \tree_\bullet]$ of twin $\signature$-Cambrian trees form (together with all their faces) a complete polyhedral fan that we call the \defn{$\signature$-Baxter-Cambrian fan}. It is the common refinement of the $\signature$- and~$(-\signature)$-Cambrian fans. It is therefore the normal fan of the Minkowski sum of the associahedra~$\Asso$ and~$\Asso[-\signature]$. We call this polytope \defn{Baxter-Cambrian associahedron} and denote it by~$\BaxAsso$. Note that~$\BaxAsso$ is clearly centrally symmetric (since~$\Asso = -\Asso[-\signature]$) but not necessarily simple. Examples are illustrated on \fref{fig:MinkowskiSums}. The graph of~$\BaxAsso$, oriented in the direction ${(n, \dots, 1)-(1, \dots, n) = \sum_{i \in [n]} (n+1-2i) \, e_i}$, is the Hasse diagram of the $\signature$-Baxter-Cambrian lattice. Finally, the Baxter-Cambrian $\PSymbol\Bax$-symbol can be read geometrically as
\[
[\tree_\circ, \tree_\bullet] = \surjectionPermAssoBax(\tau) \iff \Cone(\graphTwin) \subseteq \Cone(\tau) \iff \Cone\polar(\graphTwin) \supseteq \Cone\polar(\tau).
\]

\begin{figure}[h]
  \vspace*{.5cm}
  \begin{overpic}[scale=.35]{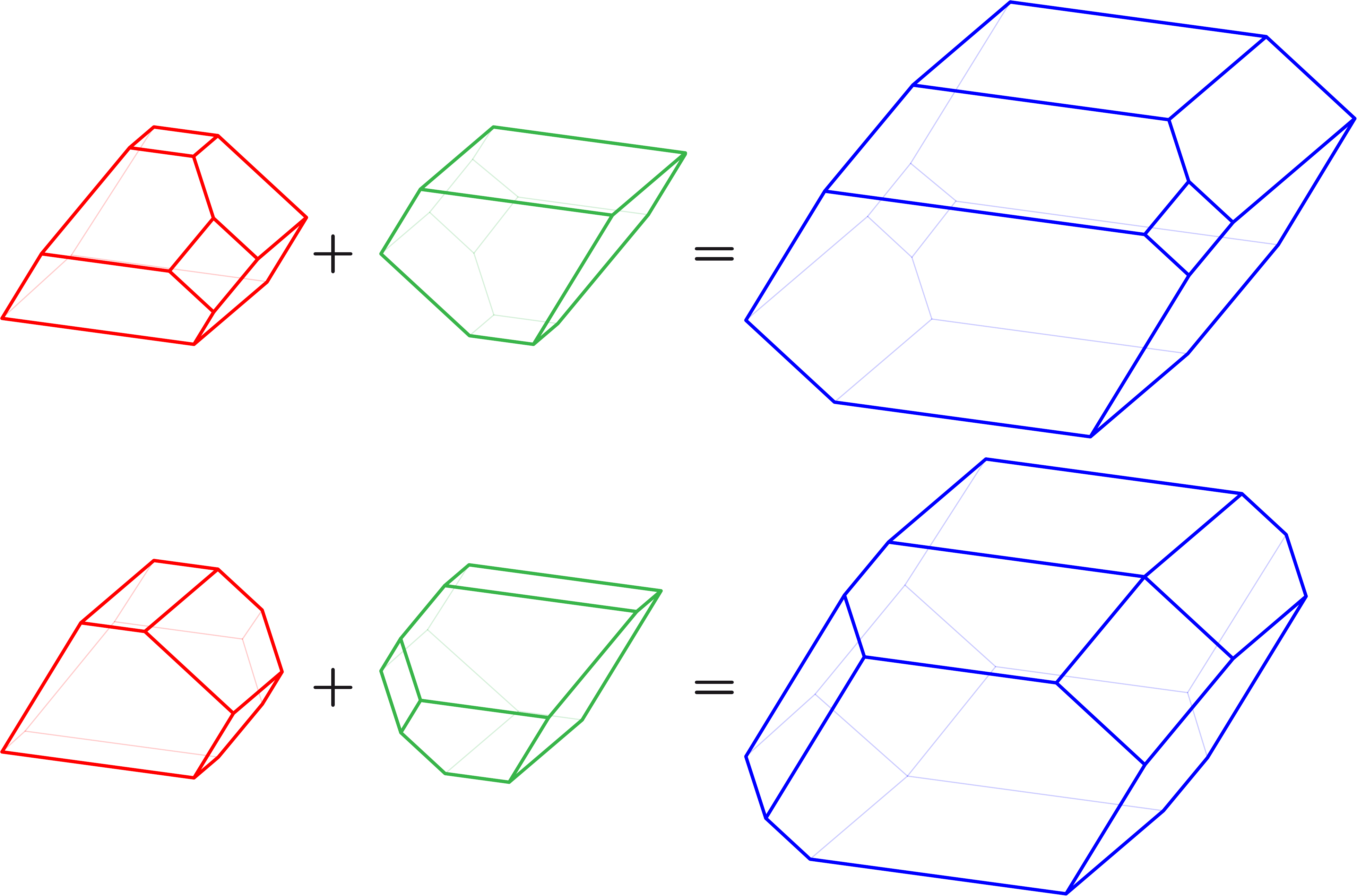}
	\put( 6,59){$\red\Asso[{-}{+}{-}{-}]$}
	\put(32,59){$\green\Asso[{+}{-}{+}{+}]$}
	\put( 6,5){$\red\Asso[{+}{-}{-}{-}]$}
	\put(32,5){$\green\Asso[{-}{+}{+}{+}]$}
	\put(62,68){$\darkblue\Asso[{-}{+}{-}{-}] + \Asso[{+}{-}{+}{+}]$}
	\put(62,-3){$\darkblue\Asso[{+}{-}{-}{-}] + \Asso[{-}{+}{+}{+}]$}
  \end{overpic}
  \vspace{.1cm}
  \caption{The Minkowski sum (blue, right) of the associahedra~$\Asso$ (red, left) and $\Asso[-\signature]$ (green, middle) gives a realization of the $\signature$-Baxter-Cambrian lattice. Illustrated with the signatures $\signature = {-}{+}{-}{-}$ (top) and ${\signature = {+}{-}{-}{-}}$~(bottom) whose $\signature$-Baxter-Cambrian lattice are represented in Figures~\ref{fig:twoOppositeCambrians}, \ref{fig:BaxterCambrianLattices}, and~\ref{fig:latticesBis}.}
  \label{fig:MinkowskiSums}
\end{figure}


\section{Baxter-Cambrian Hopf Algebra}
\label{sec:BaxterCambrianAlgebra}

In this section, we define the Baxter-Cambrian Hopf algebra~$\BaxCamb$, extending simultaneously the Cambrian Hopf algebra and the Baxter Hopf algebra studied by S.~Law and N.~Reading~\cite{LawReading} and S.~Giraudo~\cite{Giraudo}. We present again the construction of~$\BaxCamb$ as a subalgebra of~$\FQSym_\pm$ and that of its dual~$\BaxCamb^*$ as a quotient of~$\FQSym_\pm^*$.


\subsection{Subalgebra of $\FQSym_\pm$}

We denote by~$\BaxCamb$ the vector subspace of~$\FQSym_\pm$ generated by the elements
\[
\PBax_{[\tree_\circ, \tree_\bullet]} \eqdef \sum_{\substack{\tau \in \fS_\pm \\ \surjectionPermAssoBax(\tau) = [\tree_\circ, \tree_\bullet]}} \F_\tau = \sum_{\tau \in \linearExtensions(\unionOp{\tree_\circ\;}{\,\tree_\bullet})} \F_\tau,
\]
for all pairs of twin Cambrian trees~$[\tree_\circ, \tree_\bullet]$. For example, for the pair of twin Cambrian trees of \fref{fig:twinCambrianTrees}\,(left), we have
\[
\PBax_{\left[ \raisebox{-.45cm}{\Tex}, \raisebox{-.45cm}{\TexTwin} \right]} = 
\F_{\down{21}\up{7}\down{5}\up{3}\down{4}\up{6}} + \F_{\down{2}\up{7}\down{15}\up{3}\down{4}\up{6}} + \F_{\down{2}\up{7}\down{51}\up{3}\down{4}\up{6}} + \F_{\up{7}\down{215}\up{3}\down{4}\up{6}} + \F_{\up{7}\down{251}\up{3}\down{4}\up{6}} + \F_{\up{7}\down{521}\up{3}\down{4}\up{6}}.
\]

\begin{theorem}
\label{thm:baxSubalgebra}
$\BaxCamb$ is a Hopf subalgebra of~$\FQSym_\pm$.
\end{theorem}

\begin{proof}
The proof of this theorem is left to the reader as it is very similar to that of Theorem~\ref{thm:cambSubalgebra}. Exchanges in a permutation~$\tau$ of the product~$\PBax_{[\tree_\circ, \tree_\bullet]} \product \PBax_{[\tree'_\circ, \tree'_\bullet]}$ are due to exchanges either in the linear extensions of~$\unionOp{\tree_\circ}{\tree_\bullet}$ and $\unionOp{\tree'_\circ}{\tree'_\bullet}$ or in the shuffle product of these linear extensions. The coproduct is treated similarly.
\end{proof}

\enlargethispage{-.7cm}
As for the Cambrian algebra, we can describe combinatorially the product and coproduct of $\PBax$-basis elements of~$\BaxCamb$ in terms of operations on pairs of twin Cambrian trees.

\para{Product}
The product in the Baxter-Cambrian algebra~$\BaxCamb$ can be described in terms of intervals in Baxter-Cambrian lattices.

\begin{proposition}
For any two pairs~$[\tree_\circ, \tree_\bullet]$ and~$[\tree_\circ', \tree_\bullet']$ of twin Cambrian trees, the product~$\PBax_{[\tree_\circ, \tree_\bullet]} \product \PBax_{[\tree_\circ', \tree_\bullet']}$ is given by 
\[
\PBax_{[\tree_\circ, \tree_\bullet]} \product \PBax_{[\tree_\circ', \tree_\bullet']} = \sum_{[\tree[S]_\circ, \tree[S]_\bullet]} \PBax_{[\tree[S]_\circ, \tree[S]_\bullet]},
\]
where~$[\tree[S]_\circ, \tree[S]_\bullet]$ runs over the interval between~$\left[ \raisebox{-6pt}{$\tree_\circ$} \nearrow \raisebox{4pt}{$\bar\tree'_\circ$},\raisebox{4pt}{$\tree_\bullet$} \nwarrow \raisebox{-6pt}{$\bar \tree_\bullet'$} \right]$ and~$\left[ \raisebox{4pt}{$\tree_\circ$} \nwarrow \raisebox{-6pt}{$\bar\tree_\circ'$}, \raisebox{-6pt}{$\tree_\bullet$} \nearrow \raisebox{4pt}{$\bar\tree_\bullet'$} \right]$ in the~$\signature(\tree_\circ)\signature(\tree_\circ')$-Baxter-Cambrian lattice.
\end{proposition}

\begin{proof}
The result relies on the fact that the $\signature$-Baxter-Cambrian classes are intervals of the weak order on~$\fS^\signature$, and that the shuffle product of two intervals of the weak order is again an interval of the weak order. See the similar proof of Proposition~\ref{prop:product}.
\end{proof}

For example, we can compute the product
\begin{gather*}
\hspace*{-1.7cm}
\PBax_{\left[ \raisebox{-.25cm}{\includegraphics{exmTreeYAg}}, \raisebox{-.25cm}{\includegraphics{exmTreeAYd}} \right]} \product \PBax_{\left[ \raisebox{-.3cm}{\includegraphics{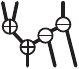}}, \raisebox{-.3cm}{\includegraphics{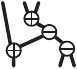}} \right]} = \F_{\up{2}\down{1}} \product \big( \F_{\up{2}\down{34}\up{1}} + \F_{\up{2}\down{3}\up{1}\down{4}} \big) \hspace*{\textwidth}
\\[-.2cm]
\hspace*{-2cm}
\begin{array}{@{$\quad{} = {}$}c@{+}c@{+}c@{+}c@{+}c@{}c@{}c}
\begin{pmatrix} \quad \F_{\up{2}\down{1}\up{4}\down{5}\up{3}\down{6}} + \F_{\up{24}\down{15}\up{3}\down{6}} \\ + \F_{\up{24}\down{51}\up{3}\down{6}} + \F_{\up{2}\down{1}\up{4}\down{56}\up{3}} \\ + \F_{\up{24}\down{156}\up{3}} + \F_{\up{24}\down{516}\up{3}} \\ + \F_{\up{24}\down{561}\up{3}} \end{pmatrix}
& \begin{pmatrix} \F_{\up{24}\down{5}\up{3}\down{16}} + \F_{\up{24}\down{5}\up{3}\down{61}} \\ + \F_{\up{24}\down{56}\up{3}\down{1}} \end{pmatrix}
& \begin{pmatrix} \quad \F_{\up{42}\down{15}\up{3}\down{6}} + \F_{\up{42}\down{51}\up{3}\down{6}} \\ + \F_{\up{42}\down{156}\up{3}} + \F_{\up{42}\down{516}\up{3}} \\ + \F_{\up{42}\down{561}\up{3}} + \F_{\up{4}\down{5}\up{2}\down{1}\up{3}\down{6}} \\ + \F_{\up{4}\down{5}\up{2}\down{16}\up{3}} + \F_{\up{4}\down{5}\up{2}\down{61}\up{3}} \\ + \F_{\up{4}\down{56}\up{2}\down{1}\up{3}} \end{pmatrix}
& \begin{pmatrix} \quad \F_{\up{42}\down{5}\up{3}\down{16}} + \F_{\up{42}\down{5}\up{3}\down{61}} \\ + \F_{\up{42}\down{56}\up{3}\down{1}} + \F_{\up{4}\down{5}\up{23}\down{16}} \\ + \F_{\up{4}\down{5}\up{23}\down{61}} + \F_{\up{4}\down{5}\up{2}\down{6}\up{3}\down{1}} \\ + \F_{\up{4}\down{56}\up{23}\down{1}} \end{pmatrix}
& \begin{pmatrix} \quad \F_{\up{4}\down{5}\up{32}\down{16}} + \F_{\up{4}\down{5}\up{32}\down{61}} \\ + \F_{\up{4}\down{5}\up{3}\down{6}\up{2}\down{1}} + \F_{\up{4}\down{56}\up{32}\down{1}} \end{pmatrix}
\\[.8cm]
\PCamb_{\left[ \raisebox{-.35cm}{\includegraphics{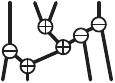}}, \raisebox{-.35cm}{\includegraphics{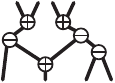}} \right]} & \PCamb_{\left[ \raisebox{-.35cm}{\includegraphics{exmProductTwin1}}, \raisebox{-.35cm}{\includegraphics{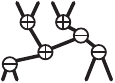}} \right]} & \PCamb_{\left[ \raisebox{-.35cm}{\includegraphics{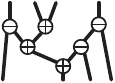}}, \raisebox{-.35cm}{\includegraphics{exmProductTwin2}} \right]} & \PCamb_{\left[ \raisebox{-.35cm}{\includegraphics{exmProductTwin4}}, \raisebox{-.35cm}{\includegraphics{exmProductTwin3}} \right]} & \PCamb_{\left[ \raisebox{-.35cm}{\includegraphics{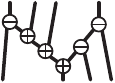}}, \raisebox{-.35cm}{\includegraphics{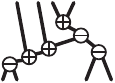}} \right]} & .
\end{array}
\end{gather*}

\begin{remark}[Multiplicative bases]
\label{rem:multiplicativeBasesBaxter}
Similar to the multiplicative bases defined in Section~\ref{sec:multiplicativeBases}, the bases~$\ECamb^{[\tree_\circ, \tree_\bullet]}$ and~$\HCamb^{[\tree_\circ, \tree_\bullet]}$ defined by
\[
\ECamb^{[\tree_\circ, \tree_\bullet]} \eqdef \sum_{[\tree_\circ, \tree_\bullet] \le [\tree'_\circ, \tree'_\bullet]} \PCamb_{[\tree'_\circ, \tree'_\bullet]}
\qquad\text{and}\qquad
\HCamb^{[\tree_\circ, \tree_\bullet]} \eqdef \sum_{[\tree'_\circ, \tree'_\bullet] \le [\tree_\circ, \tree_\bullet]} \PCamb_{[\tree'_\circ, \tree'_\bullet]}
\]
are multiplicative since
\[
\ECamb^{[\tree_\circ, \tree_\bullet]} \product \ECamb^{[\tree'_\circ, \tree'_\bullet]} = \ECamb^{\left[ \raisebox{-5pt}{\scriptsize$\tree_\circ$}\nearrow \raisebox{4pt}{\scriptsize$\bar \tree'_\circ$}, \raisebox{4pt}{\scriptsize$\tree_\bullet$}\nwarrow \raisebox{-5pt}{\scriptsize$\bar \tree'_\bullet$} \right]}
\qquad\text{and}\qquad
\HCamb^{[\tree_\circ, \tree_\bullet]} \product \HCamb^{[\tree'_\circ, \tree'_\bullet]} = \HCamb^{\left[ \raisebox{4pt}{\scriptsize$\tree_\circ$}\nwarrow \raisebox{-5pt}{\scriptsize$\bar \tree'_\circ$}, \raisebox{-5pt}{\scriptsize$\tree_\bullet$}\nearrow \raisebox{4pt}{\scriptsize$\bar \tree'_\bullet$} \right]}.
\]
The $\ECamb$-indecomposable elements are precisely the pairs~$[\tree_\circ, \tree_\bullet]$ for which all linear extensions of~$\graphTwin$ are indecomposable. In particular, $[\tree_\circ, \tree_\bullet]$ is $\ECamb$-indecomposable as soon as~$\tree_\circ$ is \mbox{$\ECamb$-inde}\-composable or~$\tree_\bullet$ is $\HCamb$-indecomposable. This condition is however not necessary. For example $\surjectionPermAssoBax(\down{3142})$ is $\ECamb$-indecomposable while~$\surjectionPermAsso(\down{3142}) = \surjectionPermAsso(\down{1342})$ is $\ECamb$-decomposable and~${\surjectionPermAsso(\down{2413}) = \surjectionPermAsso(\down{4213})}$ is $\HCamb$-decomposable. The enumerative and structural properties studied in Section~\ref{sec:multiplicativeBases} do not hold anymore for the set of $\ECamb$-indecomposable pairs of twin Cambrian trees: they form an ideal of the Baxter-Cambrian lattice, but this ideal is not principal as in Proposition~\ref{prop:upperIdeal}, and they are not counted by simple formulas as in Proposition~\ref{prop:numberIndecomposables}. Let us however mention that 
\begin{itemize}
\item the numbers of $\ECamb$-indecomposable elements with constant signature~$(-)^n$ are given by $1$, $1$, $3$, $11$, $47$, $221$, $\dots$ See~\href{https://oeis.org/A217216}{\cite[A217216]{OEIS}}.
\item the numbers of $\ECamb$-indecomposable elements with constant signature~$({+}{-})^{n/2}$ are given by~$1$, $1$, $3$, $9$, $29$, $97$, $333$, $1165$, $4135$, $\dots$ These numbers are the coefficients of the Taylor series of~$\frac{1}{x + \sqrt{1-4x}}$. See~\href{https://oeis.org/A081696}{\cite[A081696]{OEIS}} for references and details.
\end{itemize}
\end{remark}

\para{Coproduct}
A \defn{cut} of a pair of twin Cambrian trees~$[\tree[S]_\circ, \tree[S]_\bullet]$ is a pair~$\gamma = [\gamma_\circ, \gamma_\bullet]$ where~$\gamma_\circ$ is a cut of~$\tree[S]_\circ$ and~$\gamma_\bullet$ is a cut of~$\tree[S]_\bullet$ such that the labels of~$\tree[S]_\circ$ below~$\gamma_\circ$ coincide with the labels of~$\tree[S]_\bullet$ above~$\gamma_\bullet$. Equivalently, it can be seen as a lower set of~$\graphTwin$. An example is illustrated in \fref{fig:exampleCoproductTwin}.

\begin{figure}[h]
  \centerline{\includegraphics{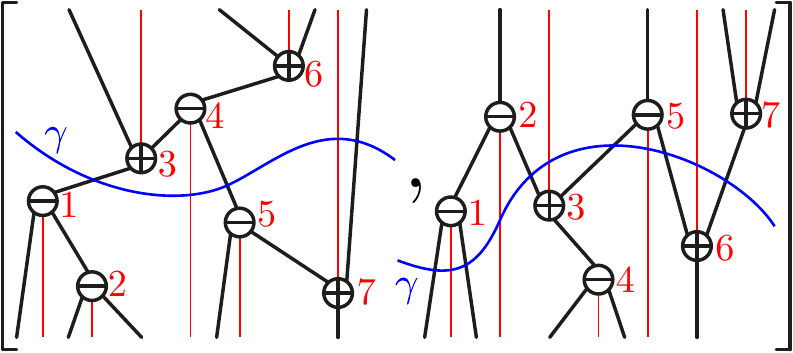}}
  \caption{A cut~$\gamma$ of a pair of twin Cambrian trees.}
  \label{fig:exampleCoproductTwin}
\end{figure}

We denote by~$AB([\tree[S]_\circ, \tree[S]_\bullet], [\gamma_\circ, \gamma_\bullet])$ the set of pairs~$[A_\circ, B_\bullet]$, where~$A_\circ$ appears in the product~$\prod_{\tree \in A(\tree[S]_\circ)} \PCamb_{\tree}$ while~$B_\bullet$ appears in the product~$\prod_{\tree \in B(\tree[S]_\circ)} \PCamb_{\tree}$, and~$A_\circ$ and~$B_\bullet$ are twin Cambrian trees. We define~$BA([\tree[S]_\circ, \tree[S]_\bullet], [\gamma_\circ, \gamma_\bullet])$ similarly exchanging the role of~$A$ and~$B$. We obtain the following description of the coproduct in the Baxter-Cambrian algebra~$\BaxCamb$.

\begin{proposition}
For any pair of twin Cambrian trees~$[\tree[S]_\circ, \tree[S]_\bullet]$, the coproduct~$\coproduct \PBax_{[\tree[S]_\circ, \tree[S]_\bullet]}$ is given~by
\[
\coproduct \PBax_{[\tree[S]_\circ, \tree[S]_\bullet]} = \sum_{\gamma} \bigg( \sum_{[B_\circ, A_\bullet]} \PBax_{[B_\circ, A_\bullet]} \bigg) \otimes \bigg( \sum_{[A_\circ, B_\bullet]} \PBax_{[A_\circ, B_\bullet]} \bigg),
\]
where~$\gamma$ runs over all cuts of~$[\tree[S]_\circ, \tree[S]_\bullet]$, $[B_\circ, A_\bullet]$ runs over~$BA([\tree[S]_\circ, \tree[S]_\bullet], [\gamma_\circ, \gamma_\bullet])$ and $[A_\circ, B_\bullet]$ runs over~$AB([\tree[S]_\circ, \tree[S]_\bullet], [\gamma_\circ, \gamma_\bullet])$.
\end{proposition}

\begin{proof}
The proof is similar to that of Proposition~\ref{prop:coproduct}. The difficulty here is to describe the linear extensions of the union of the forest~$A(\tree[S]_\circ, \gamma_\circ)$ with the opposite of the forest~$B(\tree[S]_\bullet, \gamma_\bullet)$. This difficulty is hidden in the definition of~$AB([\tree[S]_\circ, \tree[S]_\bullet], [\gamma_\circ, \gamma_\bullet])$.
\end{proof}

For example, we can compute the coproduct
\begin{align*}
& \hspace*{-1cm} \coproduct \PBax_{\left[ \raisebox{-.3cm}{\includegraphics{exmProductTwinA}}, \raisebox{-.3cm}{\includegraphics{exmProductTwinB}} \right]} = \coproduct \big( \F_{\up{2}\down{3}\up{1}\down{4}} + \F_{\up{2}\down{34}\up{1}} \big) \hspace*{\textwidth}
\\
& \hspace*{-.5cm} = 1 \otimes \big( \F_{\up{2}\down{3}\up{1}\down{4}} + \F_{\up{2}\down{34}\up{1}} \big)
+ \F_{\up{1}} \otimes \big( \F_{\down{2}\up{1}\down{3}} + \F_{\down{23}\up{1}} \big)
+ \F_{\up{1}\down{2}} \otimes \big( \F_{\up{1}\down{2}} + \F_{\down{2}\up{1}} \big)
+ \F_{\up{2}\down{3}\up{1}} \otimes \F_{\down{1}}
+ \F_{\up{1}\down{2}\down{3}} \otimes \F_{\up{1}}
+ \big( \F_{\up{2}\down{3}\up{1}\down{4}} + \F_{\up{2}\down{34}\up{1}} \big) \otimes 1
\\
& \hspace*{-.5cm} = 1 \otimes \PBax_{\left[ \raisebox{-.3cm}{\includegraphics{exmProductTwinA}}, \raisebox{-.3cm}{\includegraphics{exmProductTwinB}} \right]}
\;\; + \;\; \PBax_{\left[ \raisebox{-.15cm}{\includegraphics{exmTreeY}}, \raisebox{-.15cm}{\includegraphics{exmTreeY}} \right]} \otimes \big( \PBax_{\left[ \raisebox{-.25cm}{\includegraphics{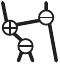}}, \raisebox{-.25cm}{\includegraphics{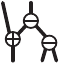}} \right]} + \PBax_{\left[ \raisebox{-.25cm}{\includegraphics{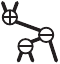}}, \raisebox{-.25cm}{\includegraphics{exmCoproductTwin2}} \right]} \big)
\;\; + \;\; \PBax_{\left[ \raisebox{-.25cm}{\includegraphics{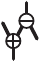}}, \raisebox{-.25cm}{\includegraphics{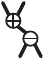}} \right]} \otimes \big( \PBax_{\left[ \raisebox{-.25cm}{\includegraphics{exmTreeYAd}}, \raisebox{-.25cm}{\includegraphics{exmTreeAYg}} \right]} + \PBax_{\left[ \raisebox{-.25cm}{\includegraphics{exmTreeAYg}}, \raisebox{-.25cm}{\includegraphics{exmTreeYAd}} \right]} \big)
\\[.2cm]
& \qquad\qquad\qquad + \PBax_{\left[ \raisebox{-.25cm}{\includegraphics{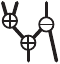}}, \raisebox{-.25cm}{\includegraphics{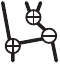}} \right]} \otimes \PBax_{\left[ \raisebox{-.15cm}{\includegraphics{exmTreeA}}, \raisebox{-.15cm}{\includegraphics{exmTreeA}} \right]}
\;\; + \;\; \PBax_{\left[ \raisebox{-.25cm}{\includegraphics{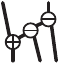}}, \raisebox{-.25cm}{\includegraphics{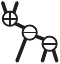}} \right]} \otimes \PBax_{\left[ \raisebox{-.15cm}{\includegraphics{exmTreeY}}, \raisebox{-.15cm}{\includegraphics{exmTreeY}} \right]}
\;\; + \;\; \PBax_{\left[ \raisebox{-.3cm}{\includegraphics{exmProductTwinA}}, \raisebox{-.3cm}{\includegraphics{exmProductTwinB}} \right]} \otimes 1.
\end{align*}

\medskip
In the result line, we have grouped the summands according to the six possible cuts of the pair of twin Cambrian trees~$\left[ \raisebox{-.3cm}{\includegraphics{exmProductTwinA}}, \raisebox{-.3cm}{\includegraphics{exmProductTwinB}} \right]$.

\medskip
\para{Matriochka algebras}
As the Baxter-Cambrian classes refine the Cambrian classes, the Baxter-Cambrian Hopf algebra is sandwiched between the Hopf algebra on signed permutations and the Cambrian Hopf algebra. It completes our sequence of subalgebras:
\[
\Rec \subset \Camb \subset \BaxCamb \subset \FQSym_\pm.
\]


\subsection{Quotient algebra of $\FQSym_\pm^*$}

As for the Cambrian algebra, the following result is automatic from Theorem~\ref{thm:baxSubalgebra}.

\begin{theorem}
The graded dual~$\BaxCamb^*$ of the Baxter-Cambrian algebra is isomorphic to the image of~$\FQSym_\pm^*$ under the canonical projection
\[
\pi : \C\langle A \rangle \longrightarrow \C\langle A \rangle / \equiv\Bax,
\]
where~$\equiv\Bax$ denotes the Baxter-Cambrian congruence. The dual basis~$\QBax_{[\tree_\circ, \tree_\bullet]}$ of~$\PCamb_{[\tree_\circ, \tree_\bullet]}$ is expressed as~$\QBax_{[\tree_\circ, \tree_\bullet]} = \pi(\G_\tau)$, where~$\tau$ is any linear extension of~$\graphTwin$.
\end{theorem}

We now describe the product and coproduct in~$\BaxCamb^*$ by combinatorial operations on pairs of twin Cambrian trees. We use the definitions and notations introduced in Section~\ref{subsec:quotientAlgebra}.

\para{Product}
The product in~$\BaxCamb^*$ can be described using gaps and laminations similarly to Proposition~\ref{prop:productDual}. An example is illustrated on \fref{fig:exampleProductDualTwin}. For two Cambrian trees~$\tree$ and~$\tree'$ and a shuffle~$s$ of the signatures~$\signature(\tree)$ and~$\signature(\tree')$, we still denote by~$\tree \,{}_s\!\backslash \tree'$ the tree described in Section~\ref{subsec:quotientAlgebra}.

\begin{proposition}
For any two pairs of twin Cambrian trees~$[\tree_\circ, \tree_\bullet]$ and~$[\tree'_\circ, \tree'_\bullet]$, the product~$\QBax_{[\tree_\circ, \tree_\bullet]} \product \QBax_{[\tree'_\circ, \tree'_\bullet]}$ is given by
\[
\QBax_{[\tree_\circ, \tree_\bullet]} \product \QBax_{[\tree'_\circ, \tree'_\bullet]} = \sum_s \QBax_{[\tree_\circ \,{}_s\!\backslash \tree'_\circ, \tree'_\bullet \,{}_s\!\backslash \tree_\bullet]},
\]
where~$s$ runs over all shuffles of the signatures~$\signature(\tree_\circ) = \signature(\tree_\bullet)$ and~$\signature(\tree'_\circ) = \signature(\tree'_\bullet)$.
\end{proposition}

\begin{proof}
The proof follows the same lines as that of Proposition~\ref{prop:productDual}. The only difference is that if~$\tau \in \linearExtensions(\unionOp{\tree_\circ}{\tree_\bullet})$, $\tau' \in \linearExtensions(\unionOp{\tree'_\circ}{\tree'_\bullet})$, and~$\sigma \in \tau \convolution \tau'$, then~$\tree_\circ = \surjectionPermAsso(\tau)$ appears below~$\tree'_\circ = \surjectionPermAsso(\tau')$ in~$\surjectionPermAsso(\sigma)$ since~$\sigma$ is inserted from left to right in~$\surjectionPermAsso(\sigma)$, while $\tree_\bullet = \surjectionPermAsso(\mirror{\tau})$ appears above~$\tree'_\bullet = \surjectionPermAsso(\mirror{\tau}\!\!')$ in~$\surjectionPermAsso(\mirror{\sigma})$ since~$\sigma$ is inserted from right to left in~$\surjectionPermAsso(\mirror{\sigma})$.
\end{proof}

\begin{figure}[h]
  \centerline{\includegraphics{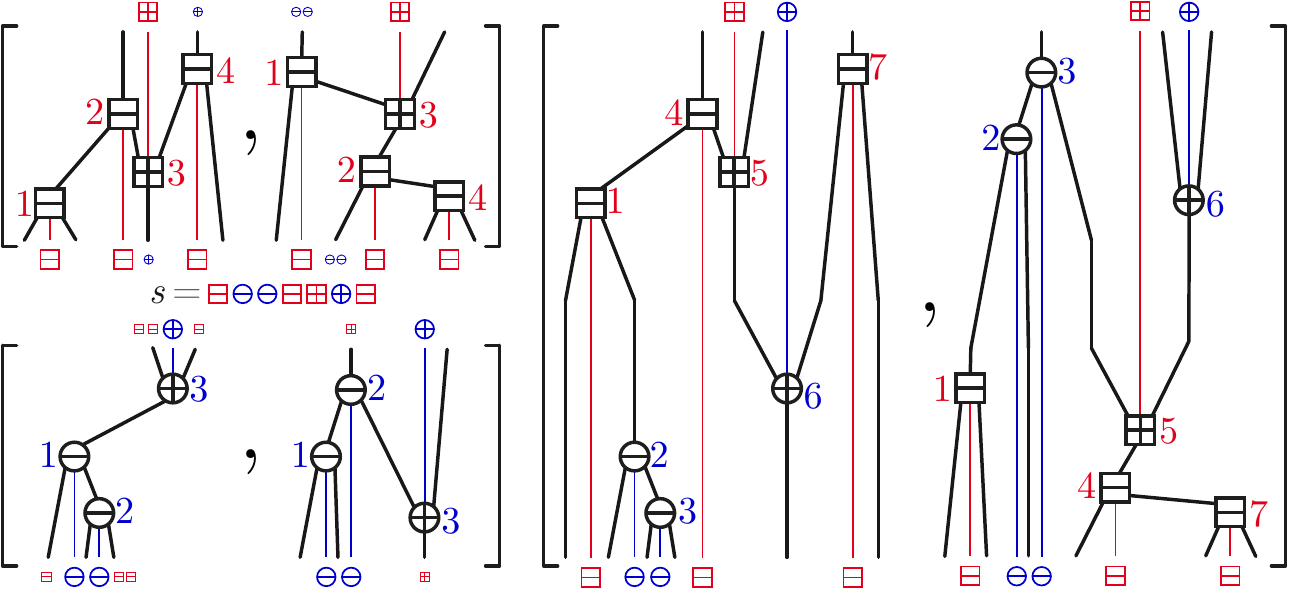}}
  \caption{Two pairs of twin Cambrian trees~$[\tree_\circ, \tree_\bullet]$ and~$[\tree'_\circ, \tree'_\bullet]$ (left), and a pair of twin Cambrian tree which appear in the product~$\QBax_{[\tree_\circ, \tree_\bullet]} \product \QBax_{[\tree'_\circ, \tree'_\bullet]}$ (right).}
  \label{fig:exampleProductDualTwin}
\end{figure}

For example, we can compute the product
\begin{align*}
& \hspace*{-.5cm} \QBax_{\left[ \raisebox{-.25cm}{\includegraphics{exmTreeYAg}}, \raisebox{-.25cm}{\includegraphics{exmTreeAYd}} \right]} \product \QBax_{\left[ \raisebox{-.25cm}{\includegraphics{exmTreeYAd}}, \raisebox{-.25cm}{\includegraphics{exmTreeAYg}} \right]} = \G_{\up{2}\down{1}} \product \G_{\up{1}\down{2}} \hspace*{\textwidth}
\\
& \hspace*{-.5cm} 
\begin{array}{@{$\quad{} = {}$}c@{+\,}c@{+\,}c@{+\,}c@{+\,}c@{+\,}c}
\G_{\up{2}\down{1}\up{3}\down{4}}
& \G_{\up{3}\down{1}\up{2}\down{4}}
& \G_{\up{4}\down{1}\up{2}\down{3}}
& \G_{\up{3}\down{2}\up{1}\down{4}}
& \G_{\up{4}\down{2}\up{1}\down{3}}
& \G_{\up{4}\down{3}\up{1}\down{2}}
\\
\QBax_{\left[ \raisebox{-.28cm}{\includegraphics{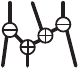}}, \raisebox{-.28cm}{\includegraphics{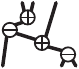}} \right]}
& \QBax_{\left[ \raisebox{-.28cm}{\includegraphics{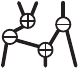}}, \raisebox{-.28cm}{\includegraphics{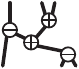}} \right]}
& \QBax_{\left[ \raisebox{-.28cm}{\includegraphics{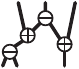}}, \raisebox{-.28cm}{\includegraphics{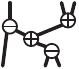}} \right]}
& \QBax_{\left[ \raisebox{-.28cm}{\includegraphics{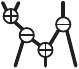}}, \raisebox{-.28cm}{\includegraphics{exmProductTwin10}} \right]}
& \QBax_{\left[ \raisebox{-.28cm}{\includegraphics{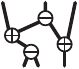}}, \raisebox{-.28cm}{\includegraphics{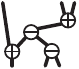}} \right]}
& \QBax_{\left[ \raisebox{-.28cm}{\includegraphics{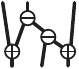}}, \raisebox{-.28cm}{\includegraphics{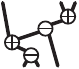}} \right]}.
\end{array}
\end{align*}

\para{Coproduct}
The coproduct in~$\BaxCamb^*$ can be described combinatorially as in Proposition~\ref{prop:coproductDual}. For a Cambrian tree~$\tree[S]$ and a gap~$\gamma$ between two consecutive vertices of~$\tree[S]$, we still denote by~$L(\tree[S], \gamma)$ and~$R(\tree[S], \gamma)$ the left and right Cambrian subtrees of~$\tree[S]$ when split along the path~$\lambda(\tree[S], \gamma)$.

\begin{proposition}
For any pair of twin Cambrian trees~$[\tree[S]_\circ, \tree[S]_\bullet]$, the coproduct~$\coproduct\QCamb_{[\tree[S]_\circ, \tree[S]_\bullet]}$ is given~by
\[
\coproduct\QCamb_{[\tree[S]_\circ, \tree[S]_\bullet]} = \sum_{\gamma} \QCamb_{[L(\tree[S]_\circ,\gamma), L(\tree[S]_\bullet, \gamma)]} \otimes \QCamb_{[R(\tree[S]_\circ,\gamma), R(\tree[S]_\bullet, \gamma)]},
\]
where~$\gamma$ runs over all gaps between consecutive positions in~$[n]$.
\end{proposition}

\begin{proof}
The proof is identical to that of Proposition~\ref{prop:coproductDual}.
\end{proof}

For example, we can compute the coproduct
\begin{align*}
\coproduct \QBax_{\left[ \raisebox{-.3cm}{\includegraphics{exmProductTwinA}}, \raisebox{-.3cm}{\includegraphics{exmProductTwinB}} \right]}
& = \coproduct \G_{\up{2}\down{34}\up{1}}
\\[-.3cm]
& = 1 \otimes \G_{\up{2}\down{34}\up{1}}
+ \G_{\up{1}} \otimes \G_{\up{1}\down{23}}
+ \G_{\up{21}} \otimes \G_{\down{12}}
+ \G_{\up{2}\down{3}\up{1}} \otimes \G_{\down{1}}
+ \G_{\up{2}\down{34}\up{1}} \otimes 1
\\[.2cm]
& = 1 \otimes \QBax_{\left[ \raisebox{-.3cm}{\includegraphics{exmProductTwinA}}, \raisebox{-.3cm}{\includegraphics{exmProductTwinB}} \right]}
\;\; + \;\; \QBax_{\left[ \raisebox{-.15cm}{\includegraphics{exmTreeY}}, \raisebox{-.15cm}{\includegraphics{exmTreeY}} \right]} \otimes \QBax_{\left[ \raisebox{-.25cm}{\includegraphics{exmCoproductTwin6}}, \raisebox{-.25cm}{\includegraphics{exmCoproductTwin7}} \right]}
\;\; + \;\; \QBax_{\left[ \raisebox{-.25cm}{\includegraphics{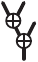}}, \raisebox{-.25cm}{\includegraphics{exmTreeYYd}} \right]} \otimes \QBax_{\left[ \raisebox{-.25cm}{\includegraphics{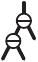}}, \raisebox{-.25cm}{\includegraphics{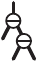}} \right]} \\
& \qquad\qquad\qquad + \;\; \QBax_{\left[ \raisebox{-.3cm}{\includegraphics{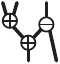}}, \raisebox{-.3cm}{\includegraphics{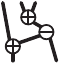}} \right]} \otimes \QBax_{\left[ \raisebox{-.15cm}{\includegraphics{exmTreeA}}, \raisebox{-.15cm}{\includegraphics{exmTreeA}} \right]}
\;\; + \;\; \QBax_{\left[ \raisebox{-.3cm}{\includegraphics{exmProductTwinA}}, \raisebox{-.3cm}{\includegraphics{exmProductTwinB}} \right]} \otimes 1.
\end{align*}

\smallskip


\section{Cambrian tuples}
\label{sec:tuples}

This section is devoted to a natural extension of our results on twin Cambrian trees and the Baxter-Cambrian algebra to arbitrary intersections of Cambrian congruences.
Since the results presented here are straightforward generalizations of that of Sections~\ref{sec:BaxterTrees} and~\ref{sec:BaxterCambrianAlgebra}, all proofs of this section are left to the reader.


\subsection{Combinatorics of Cambrian tuples}
As observed in Remark~\ref{rem:reversing}, pairs of twin Cambrian trees can as well be thought of as pairs of Cambrian trees on opposite signature whose union is acyclic. We extend this idea to arbitrary signatures. For an $\ell$-tuple~$\tuple$ and~$k \in [\ell]$, we denote by~$\tuple_{[k]}$ the $k$th element of~$\tuple$.

\begin{definition}
A \defn{Cambrian $\ell$-tuple} is a $\ell$-tuple~$\tuple$ of Cambrian trees~$\tuple_{[k]}$ on the same vertex set, and whose union forms an acyclic graph. The \defn{signature} of~$\tuple$ is the $\ell$-tuple of signatures $\signatures(\tuple) \eqdef \big[ \signature \big( \tuple_{[1]} \big), \dots, \signature \big( \tuple_{[\ell]} \big) \big]$. Let~$\CambTrees(\signatures)$ denote the set of Cambrian $\ell$-tuples of signature~$\signatures$.
\end{definition}

\begin{definition}
A \defn{leveled Cambrian $\ell$-tuple} is a $\ell$-tuple~$\tuple$ of leveled Cambrian trees~$\tuple_{[k]}$ with labelings~$p_{[k]}, q_{[k]}$, and such that~$q_{[k]} \circ {p_{[k]}}^{-1}$ is independent of~$k$. In other words, it is a Cambrian $\ell$-tuple endowed with a linear extension of the union of its trees.
\end{definition}

For example, pairs of twin (leveled) Cambrian trees are particular (leveled) Cambrian $2$-tuples. A Cambrian $2$-tuple and a leveled Cambrian $2$-tuple with signature~$[{-}{-}{+}{-}{-}{+}{+}, {+}{+}{-}{+}{-}{-}{+}]$ are represented in \fref{fig:CambrianPair}.

\begin{figure}[t]
  \centerline{\includegraphics{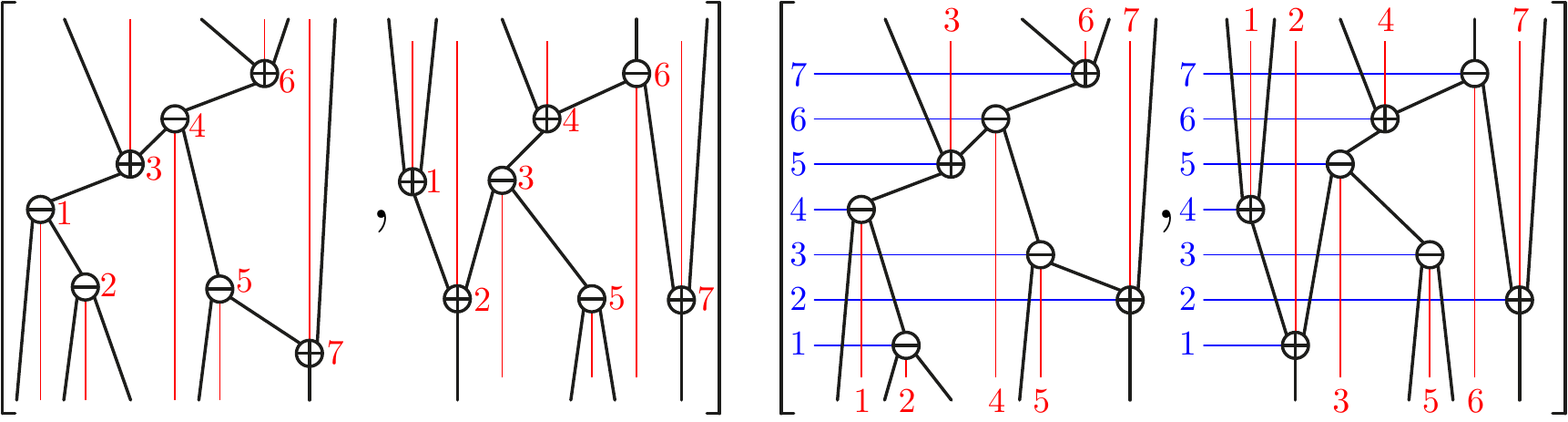}}
  \caption{A Cambrian $2$-tuple (left), and a leveled Cambrian $2$-tuple (right).}
  \label{fig:CambrianPair}
\end{figure}

We now want to define an analogue of the Cambrian correspondence. For this, we need permutations recording~$\ell$ different signatures. Call \defn{$\ell$-signed permutation} a permutation table where each dot receives an~$\ell$-tuple of signs. In other words, it is an element of the wreath product of~$\fS$ by~$(\Z_2)^\ell$. For example,
\[
\uptilde{\downw{2}} \uptilde{\upw{7}} \downtilde{\downw{5}} \uptilde{\downw{1}} \downtilde{\upw{3}} \uptilde{\downw{4}} \downtilde{\upw{6}}
\]
is a $2$-signed permutation whose signatures are marked with~$\up{\phantom{1}}/\down{\phantom{1}}$ and~$\uptilde{\phantom{1}}/\downtilde{\phantom{1}}$ respectively. For a $\ell$-signed permutation~$\tau$ and~$k \in [\ell]$, we denote by~$\tau_{[k]}$ the signed permutation where we only keep the~$k$-th signature. For example
\[
\uptilde{\downw{2}} \uptilde{\upw{7}} \downtilde{\downw{5}} \uptilde{\downw{1}} \downtilde{\upw{3}} \uptilde{\downw{4}} \downtilde{\upw{6}}_{[1]} = \down{2}\up{7}\down{51}\up{3}\down{4}\up{6}
\qquad\text{and}\qquad
\uptilde{\downw{2}} \uptilde{\upw{7}} \downtilde{\downw{5}} \uptilde{\downw{1}} \downtilde{\upw{3}} \uptilde{\downw{4}} \downtilde{\upw{6}}_{[2]} = \simtilde{2} \simtilde{7} \downtilde{5} \simtilde{1} \downtilde{3} \simtilde{4} \downtilde{6}.
\]
We denote by~$\fS_{\pm^\ell}$ the set of all $\ell$-signed permutations and by~$\fS_\signatures$ (resp.~$\fS^\signatures$) the set of $\ell$-signed permutations with p-signatures (resp.~v-signatures)~$\signatures$. Applying the Cambrian correspondences in parallel yields a map form $\ell$-signed permutations to Cambrian $\ell$-tuples.

\begin{proposition}
The map~$\CambCorresp_\ell$ defined by~$\CambCorresp_\ell(\tau) \eqdef \big[\CambCorresp \big( \tau_{[1]} \big), \dots, \CambCorresp \big( \tau_{[\ell]} \big) \big]$ is a bijection from $\ell$-signed permutations to leveled Cambrian $\ell$-tuples.
The map~$\surjectionPermAsso_\ell$ defined by~$\surjectionPermAsso_\ell(\tau) \eqdef \big[ \surjectionPermAsso \big( \tau_{[1]} \big), \dots, \surjectionPermAsso \big( \tau_{[\ell]} \big) \big]$ is a surjection from $\ell$-signed permutations to Cambrian $\ell$-tuples.
\end{proposition}

For example, the Cambrian $2$-tuple and the leveled Cambrian $2$-tuple on \fref{fig:CambrianPair} are
\[
\surjectionPermAsso_2 \left( \uptilde{\downw{2}} \uptilde{\upw{7}} \downtilde{\downw{5}} \uptilde{\downw{1}} \downtilde{\upw{3}} \uptilde{\downw{4}} \downtilde{\upw{6}} \right)
\qquad\text{and}\qquad
\CambCorresp_2 \left( \uptilde{\downw{2}} \uptilde{\upw{7}} \downtilde{\downw{5}} \uptilde{\downw{1}} \downtilde{\upw{3}} \uptilde{\downw{4}} \downtilde{\upw{6}} \right).
\]
As in the Cambrian and Baxter-Cambrian situation, the permutations with the same $\PSymbol\Bax$-symbol define the Cambrian congruence classes on $\ell$-signed permutations.

\begin{definition}
For a signature $\ell$-tuple~$\signatures$, the \defn{$\signatures$-Cambrian congruence} on~$\fS^\signatures$ is the intersection~$\equiv_\signatures \eqdef \bigcap_{k \in [\ell]} \equiv_{\signatures_{[k]}}$ of all $\signatures_{[k]}$-Cambrian congruences. In other words, it is the transitive closure of the rewriting rule $UacV \equiv_\signatures UcaV$ if for all~$k \in [\ell]$, there exists~$a < b_{[k]} < c$ such that~$(\signatures_{[k]})_{b_{[k]}} = +$ and~$b_{[k]}$ appears in~$U$, or~$(\signatures_{[k]})_{b_{[k]}} = -$ and~$b_{[k]}$ appears in~$V$. The \defn{Cambrian congruence} on~$\fS_{\pm^\ell}$ is the equivalence relation~$\equiv_\ell$ on all $\ell$-signed permutations obtained as the union of all $\signatures$-Cambrian congruences.
\end{definition}

\enlargethispage{-.4cm}
For example, the $[{-}{+}{-}{-}, {+}{-}{-}{-}]$-Cambrian classes are represented on \fref{fig:latticesTer}.

\begin{figure}[h]
  \centerline{\includegraphics[width=.48\textwidth]{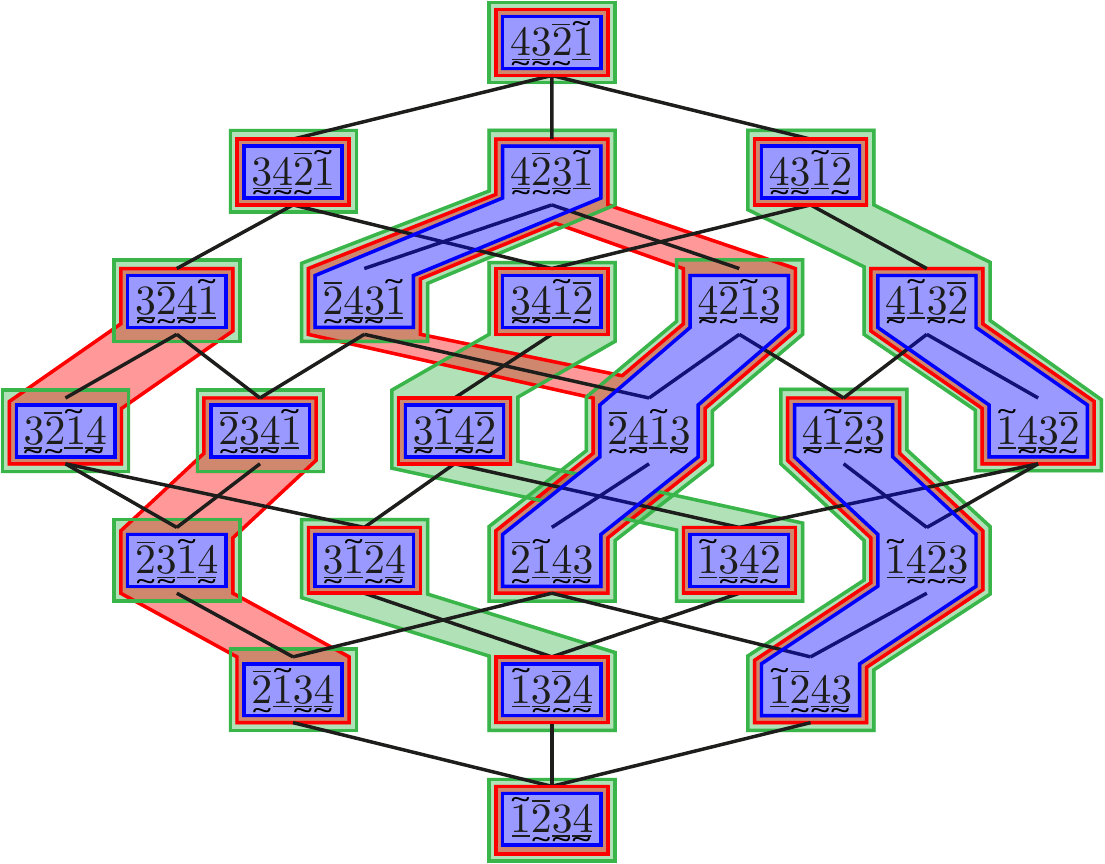}\quad\includegraphics[width=.48\textwidth]{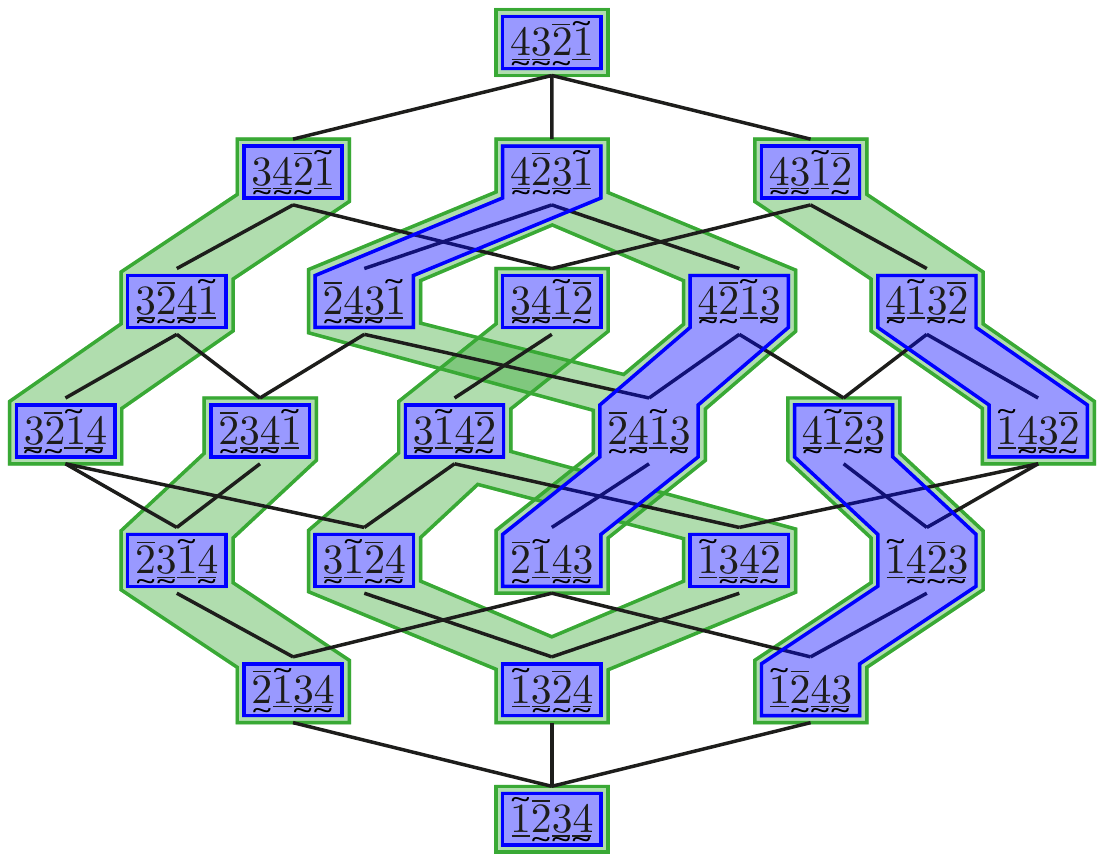}}
  \caption{(Left)~The $[{-}{+}{-}{-}, {+}{-}{-}{-}]$-Cambrian classes (blue) are the intersections of the $({-}{+}{-}{-})$-Cambrian classes (red) and the $({+}{-}{-}{-})$-Cambrian classes (green). (Right) $[{-}{+}{-}{-}, {+}{-}{-}{-}]$-Cambrian (blue) and boolean (green) congruence classes on the weak order.}
  \label{fig:latticesTer}
\end{figure}

\begin{proposition}
Two $\ell$-signed permutations~$\tau, \tau' \in \fS_{\pm^\ell}$ are Cambrian congruent if and only if they have the same~$\PSymbol_\ell$-symbol:
\[
\tau \equiv_\ell \tau' \iff \surjectionPermAsso_\ell(\tau) = \surjectionPermAsso_\ell(\tau').
\]
\end{proposition}

\begin{proposition}
The $\signature$-Cambrian class indexed by the Cambrian $\ell$-tuple~$\tuple$ is the intersection of the $\signatures_{[k]}$-Cambrian classes indexed by~$\tuple_{[k]}$ over~$k \in [\ell]$.
\end{proposition}

We now present the rotation operation on Cambrian $\ell$-tuples.

\begin{definition}
\label{def:rotationTuples}
Let~$\tuple$ be a Cambrian $\ell$-tuple and consider an edge~$i \to j$ of the union~$\bigcup_{k \in [\ell]} \tuple_{[k]}$. We say that the edge~$i \to j$ is \defn{rotatable} if either~$i \to j$ is an edge or $i$ and~$j$ are incomparable in each tree~$\tuple_{[k]}$ (note that $i \to j$ is an edge in at least one of these trees since it belongs to their union). If~$i \to j$ is rotatable in~$\tuple$, its \defn{rotation} transforms~$\tuple$ to the $\ell$-tuple of trees~$\tuple' \eqdef \big[ \tuple'_{[1]}, \dots, \tuple'_{[\ell]} \big]$,~where~$\tuple'_{[k]}$ is obtained by rotation of~$i \to j$ in~$\tuple_{[k]}$ if possible and~$\tuple'_{[k]} = \tuple_{[k]}$ otherwise.
\end{definition}

\begin{proposition}
Rotating a rotatable edge~$i \to j$ in a Cambrian $\ell$-tuple~$\tuple$ yields a Cambrian $\ell$-tuple~$\tuple'$ with the same signature.
\end{proposition}

Consider the \defn{increasing rotation graph} whose vertices are $\signatures$-Cambrian tuples and whose arcs are increasing rotations~$\tuple \to \tuple'$, \ie for which~$i < j$ in Definition~\ref{def:rotationTuples}. This graph is illustrated on \fref{fig:tupleCambrian} for the signature $2$-tuple~$\signatures = [{-}{+}{-}{-}, {+}{-}{-}{-}]$.

\begin{figure}[t]
  \centerline{\includegraphics[width=.92\textwidth]{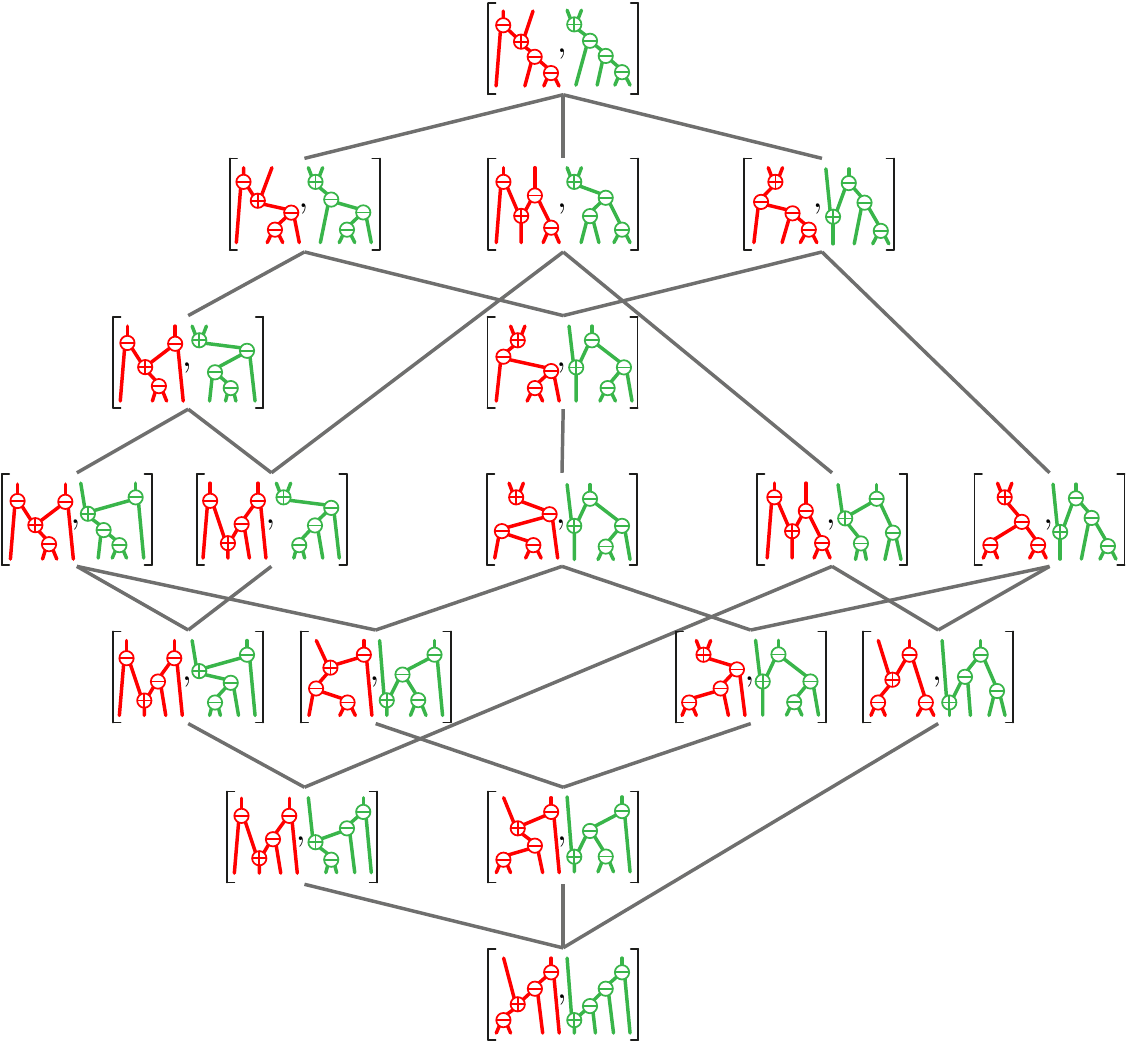}}
  \caption{The $[{-}{+}{-}{-}, {+}{-}{-}{-}]$-Cambrian lattice on Cambrian tuples. See also \fref{fig:latticesTer}.}
  \label{fig:tupleCambrian}
\end{figure}

\begin{proposition}
For any cover relation~$\tau < \tau'$ in the weak order on~$\fS^\signatures$, either~$\surjectionPermAsso_\ell(\tau) = \surjectionPermAsso_\ell(\tau')$ or~$\surjectionPermAsso_\ell(\tau) \to \surjectionPermAsso_\ell(\tau')$ in the increasing rotation graph.
\end{proposition}

It follows that the increasing rotation graph on $\signatures$-Cambrian tuples is acyclic. We call \defn{\mbox{$\signatures$-Cambrian} poset} its transitive closure. In other words, the previous statement says that the map~$\surjectionPermAsso_\ell$ defines a poset homomorphism from the weak order on~$\fS^\signatures$ to the $\signatures$-Cambrian poset. This homomorphism is in fact a lattice homomorphism.

\begin{proposition}
The~$\signatures$-Cambrian poset is a lattice quotient of the weak order on~$\fS^\signatures$.
\end{proposition}

The $\signatures$-Cambrian lattice has natural geometric realizations, similar to the geometric realizations of the Baxter-Cambrian lattice. Namely, for a signature $\ell$-tuple~$\signatures$, the cones~$\Cone(\tuple) \eqdef \Cone(\bigcup_{k \in [\ell]} \tuple_{[k]})$ for all~$\signatures$-Cambrian tuples~$\tuple$ form (together with all their faces) a complete polyhedral fan that we call the $\signatures$-Cambrian fan.  It is the common refinement of the $\signatures_{[k]}$-Cambrian fans for~$k \in [\ell]$. It is therefore the normal fan of the Minkowski sum of the associahedra~$\Asso[\signatures_{[k]}]$ for~$k \in [\ell]$. An example is illustrated on \fref{fig:MinkowskiSumsTuple}. The $1$-skeleton of this polytope, oriented in the direction of~$(n, \dots, 1)-(1, \dots, n) = \sum_{i \in [n]} (n+1-2i) \, e_i$, is the Hasse diagram of the $\signatures$-Cambrian lattice. Finally, the $\signatures$-Cambrian $\PSymbol_\ell$-symbol can be read geometrically~as
\[
\tuple = \surjectionPermAsso_\ell(\tau) \iff \Cone(\tuple) \subseteq \Cone(\tau) \iff \Cone\polar(\tuple) \supseteq \Cone\polar(\tau).
\]

\begin{figure}[h]
  \vspace*{.8cm}
  \begin{overpic}[scale=.35]{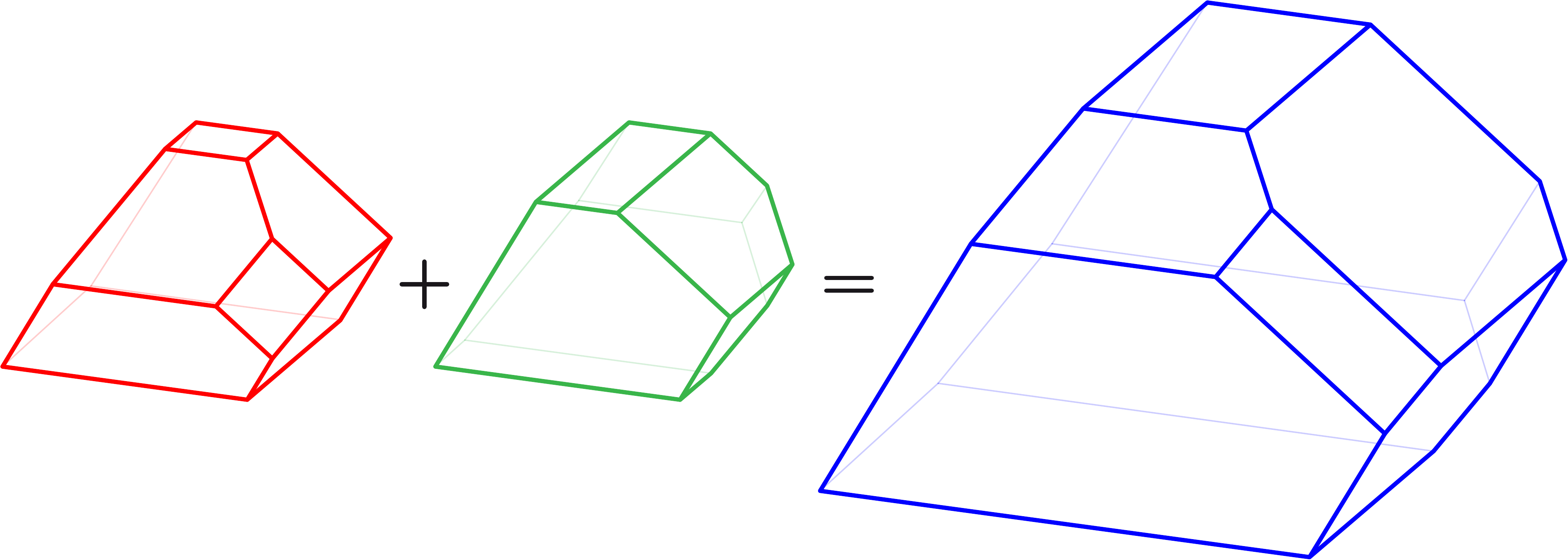}
	\put( 6,29){$\red\Asso[{-}{+}{-}{-}]$}
	\put(32,29){$\green\Asso[{+}{-}{-}{-}]$}
	\put(62,37){$\darkblue\Asso[{-}{+}{-}{-}] + \Asso[{+}{-}{-}{-}]$}
  \end{overpic}
  \vspace{.1cm}
  \caption{The Minkowski sum (blue, right) of the associahedra~$\Asso[{-}{+}{-}{-}]$ (red, left) and~$\Asso[{+}{-}{-}{-}]$ (green, middle) gives a realization of the $[{-}{+}{-}{-}, {+}{-}{-}{-}]$-Cambrian lattice, represented in Figures~\ref{fig:latticesTer} and~\ref{fig:tupleCambrian}.}
  \label{fig:MinkowskiSumsTuple}
\end{figure}


\subsection{Cambrian tuple Hopf Algebra}
\label{subsec:CambrianTupleAlgebra}

In this section, we construct a Hopf algebra indexed by Cambrian $\ell$-tuples, similar to the Baxter-Cambrian algebra. Exactly as we needed to consider the Hopf algebra~$\FQSym_\pm$ on signed permutations when constructing the Cambrian algebra to keep track of the signature, we now need to consider a natural extension of~$\FQSym$ on $\ell$-signed permutation to keep track of the $\ell$ signatures of~$\signatures$.

The \defn{shifted shuffle product}~$\tau \shiftedShuffle \tau'$ (resp.~\defn{convolution product}~$\tau \convolution \tau'$) of two $\ell$-signed permutations~$\tau, \tau'$ is still defined as the shifted product (resp.~convolution product) where signs travel with their values (resp.~stay at their positions). When~$\ell = 2$ and the two signatures are marked with~$\up{\phantom{1}}/\down{\phantom{1}}$ and~$\uptilde{\phantom{1}}/\downtilde{\phantom{1}}$ respectively, we have for example
\begin{align*}
{\red \uptilde{\upw{1}}}{\red \uptilde{\downw{2}}} \shiftedShuffle {\darkblue \downtilde{\downw{2}}}{\darkblue \uptilde{\downw{3}}}{\darkblue \downtilde{\upw{1}}} & =
\{
	{\red \uptilde{\upw{1}}}{\red \uptilde{\downw{2}}}{\darkblue \downtilde{\downw{4}}\uptilde{\downw{5}}}{\darkblue \downtilde{\upw{3}}},
	{\red \uptilde{\upw{1}}}{\darkblue \downtilde{\downw{4}}}{\red \uptilde{\downw{2}}}{\darkblue \uptilde{\downw{5}}}{\darkblue \downtilde{\upw{3}}},
	{\red \uptilde{\upw{1}}}{\darkblue \downtilde{\downw{4}}\uptilde{\downw{5}}}{\red \uptilde{\downw{2}}}{\darkblue \downtilde{\upw{3}}},
	{\red \uptilde{\upw{1}}}{\darkblue \downtilde{\downw{4}}\uptilde{\downw{5}}}{\darkblue \downtilde{\upw{3}}}{\red \uptilde{\downw{2}}},
	{\darkblue \downtilde{\downw{4}}}{\red \uptilde{\upw{1}}}{\red \uptilde{\downw{2}}}{\darkblue \uptilde{\downw{5}}}{\darkblue \downtilde{\upw{3}}},
	{\darkblue \downtilde{\downw{4}}}{\red \uptilde{\upw{1}}}{\darkblue \uptilde{\downw{5}}}{\red \uptilde{\downw{2}}}{\darkblue \downtilde{\upw{3}}},
	{\darkblue \downtilde{\downw{4}}}{\red \uptilde{\upw{1}}}{\darkblue \uptilde{\downw{5}}}{\darkblue \downtilde{\upw{3}}}{\red \uptilde{\downw{2}}},
	{\darkblue \downtilde{\downw{4}}\uptilde{\downw{5}}}{\red \uptilde{\upw{1}}}{\red \uptilde{\downw{2}}}{\darkblue \downtilde{\upw{3}}},
	{\darkblue \downtilde{\downw{4}}\uptilde{\downw{5}}}{\red \uptilde{\upw{1}}}{\darkblue \downtilde{\upw{3}}}{\red \uptilde{\downw{2}}},
	{\darkblue \downtilde{\downw{4}}\uptilde{\downw{5}}}{\darkblue \downtilde{\upw{3}}}{\red \uptilde{\upw{1}}}{\red \uptilde{\downw{2}}}
\}, \\
{\red \uptilde{\upw{1}}}{\red \uptilde{\downw{2}}} \convolution {\darkblue \downtilde{\downw{2}}}{\darkblue \uptilde{\downw{3}}}{\darkblue \downtilde{\upw{1}}} & =
\{
	{\red \uptilde{\upw{1}}}{\red \uptilde{\downw{2}}}{\darkblue \downtilde{\downw{4}}}{\darkblue \uptilde{\downw{5}}}{\darkblue \downtilde{\upw{3}}}, 
	{\red \uptilde{\upw{1}}}{\red \uptilde{\downw{3}}}{\darkblue \downtilde{\downw{4}}}{\darkblue \uptilde{\downw{5}}}{\darkblue \downtilde{\upw{2}}}, 
	{\red \uptilde{\upw{1}}}{\red \uptilde{\downw{4}}}{\darkblue \downtilde{\downw{3}}}{\darkblue \uptilde{\downw{5}}}{\darkblue \downtilde{\upw{2}}}, 
	{\red \uptilde{\upw{1}}}{\red \uptilde{\downw{5}}}{\darkblue \downtilde{\downw{3}}}{\darkblue \uptilde{\downw{4}}}{\darkblue \downtilde{\upw{2}}}, 
	{\red \uptilde{\upw{2}}}{\red \uptilde{\downw{3}}}{\darkblue \downtilde{\downw{4}}}{\darkblue \uptilde{\downw{5}}}{\darkblue \downtilde{\upw{1}}}, 
	{\red \uptilde{\upw{2}}}{\red \uptilde{\downw{4}}}{\darkblue \downtilde{\downw{3}}}{\darkblue \uptilde{\downw{5}}}{\darkblue \downtilde{\upw{1}}}, 
	{\red \uptilde{\upw{2}}}{\red \uptilde{\downw{5}}}{\darkblue \downtilde{\downw{3}}}{\darkblue \uptilde{\downw{4}}}{\darkblue \downtilde{\upw{1}}}, 
	{\red \uptilde{\upw{3}}}{\red \uptilde{\downw{4}}}{\darkblue \downtilde{\downw{2}}}{\darkblue \uptilde{\downw{5}}}{\darkblue \downtilde{\upw{1}}}, 
	{\red \uptilde{\upw{3}}}{\red \uptilde{\downw{5}}}{\darkblue \downtilde{\downw{2}}}{\darkblue \uptilde{\downw{4}}}{\darkblue \downtilde{\upw{1}}}, 
	{\red \uptilde{\upw{4}}}{\red \uptilde{\downw{5}}}{\darkblue \downtilde{\downw{2}}}{\darkblue \uptilde{\downw{3}}}{\darkblue \downtilde{\upw{1}}}
\}.
\end{align*}

We denote by~$\FQSym_{\pm^\ell}$ the Hopf algebra with basis~$(F_\tau)_{\tau \in \fS_{\pm^\ell}}$ indexed by $\ell$-signed permutations and whose product and coproduct are defined by
\[
\F_\tau \product \F_{\tau'} = \sum_{\sigma \in \tau \shiftedShuffle \tau'} \F_\sigma
\qquad\text{and}\qquad
\coproduct \F_\sigma = \sum_{\sigma \in \tau \convolution \tau'} \F_\tau \otimes \F_{\tau'}.
\]

\begin{remark}[Cambrian algebra \versus $G$-colored binary tree algebra]
Checking that these product and coproduct produce a Hopf algebra is standard. It even extends to a Hopf algebra~$\FQSym_G$ on $G$-colored permutations for an arbitrary semigroup~$G$, see \eg \cite{NovelliThibon-coloredHopfAlgebras, BaumannHohlweg, BergeronHohlweg}. In these papers, the authors use this algebra~$\FQSym_G$ to defined $G$-colored subalgebras from congruence relations on permutations, see in particular~\cite{BergeronHohlweg}. Note that our construction of the Cambrian algebra and of the tuple Cambrian algebra really differs from the construction of~\cite{BergeronHohlweg} as our congruence relations depend on the signs, while their congruences do not.
\end{remark}

We denote by~$\Camb_\ell$ the vector subspace of~$\FQSym_{\pm^\ell}$ generated by the elements
\[
\PCamb_{\tuple} \eqdef \sum_{\substack{\tau \in \fS_{\pm^\ell} \\ \surjectionPermAsso_\ell(\tau) = \tuple}} \F_\tau = \sum_{\tau \in \linearExtensions\big(\bigcup\limits_{k \in [\ell]} \tuple_{[k]}\big)} \F_\tau,
\]
for all Cambrian $\ell$-tuples~$\tuple$.
\renewcommand{\uptilde}[1]{\accentset{\vspace{-.05cm}\scalebox{.55}{$\sim$}}{#1}}
\renewcommand{\downtilde}[1]{\underaccent{\,\scalebox{.55}{$\sim$}}{#1}}
For example, for the Cambrian tuple of \fref{fig:CambrianPair}\,(left), we have
\[
\PCamb_{\left[ \raisebox{-.45cm}{\Tex}, \raisebox{-.45cm}{\TexPair} \right]} = \F_{\uptilde{\downw{2}} \uptilde{\downw{1}} \uptilde{\upw{7}} \downtilde{\downw{5}} \downtilde{\upw{3}} \uptilde{\downw{4}} \downtilde{\upw{6}}} + \F_{\uptilde{\downw{2}} \uptilde{\upw{7}} \uptilde{\downw{1}} \downtilde{\downw{5}} \downtilde{\upw{3}} \uptilde{\downw{4}} \downtilde{\upw{6}}} + \F_{\uptilde{\downw{2}} \uptilde{\upw{7}} \downtilde{\downw{5}} \uptilde{\downw{1}} \downtilde{\upw{3}} \uptilde{\downw{4}} \downtilde{\upw{6}}} + 
\F_{\uptilde{\upw{7}} \uptilde{\downw{2}} \uptilde{\downw{1}} \downtilde{\downw{5}} \downtilde{\upw{3}} \uptilde{\downw{4}} \downtilde{\upw{6}}} + 
\F_{\uptilde{\upw{7}} \uptilde{\downw{2}} \downtilde{\downw{5}} \uptilde{\downw{1}} \downtilde{\upw{3}} \uptilde{\downw{4}} \downtilde{\upw{6}}} + 
\F_{\uptilde{\upw{7}} \downtilde{\downw{5}} \uptilde{\downw{2}} \uptilde{\downw{1}} \downtilde{\upw{3}} \uptilde{\downw{4}} \downtilde{\upw{6}}}.
\]

\begin{theorem}
\label{thm:cambTupleSubalgebra}
$\Camb_\ell$ is a Hopf subalgebra of~$\FQSym_{\pm^\ell}$.
\end{theorem}

As for the Cambrian algebra, the product and coproduct of~$\PCamb$-basis elements of the $\ell$-Cambrian algebra~$\Camb_\ell$ can be described directly in terms of combinatorial operations on Cambrian \mbox{$\ell$-tuples}.

\para{Product}
The product in the $\ell$-Cambrian algebra~$\Camb_\ell$ can be described in terms of intervals in $\ell$-Cambrian lattices. We denote by~$\signatures\signatures' \eqdef [\signatures_{[1]}\signatures'_{[1]}, \dots, \signatures_{[\ell]}\signatures'_{[\ell]} ]$ the componentwise concatenation of two signature $\ell$-tuples~$\signatures,\signatures'$. Similarly, for two Cambrian $\ell$-tuples~$\tuple, \tuple'$, we define
\[
\raisebox{-6pt}{$\tuple$} \nearrow \raisebox{4pt}{$\bar\tuple'$} \eqdef \bigg[ \raisebox{-6pt}{$\tuple_{[1]}$} \nearrow \raisebox{4pt}{$\bar\tuple'_{[1]}$}, \dots, \raisebox{-6pt}{$\tuple_{[\ell]}$} \nearrow \raisebox{4pt}{$\bar\tuple'_{[\ell]}$} \bigg]
\qquad\text{and}\qquad
\raisebox{4pt}{$\tuple$} \nwarrow \raisebox{-6pt}{$\bar\tuple'$} \eqdef \bigg[ \raisebox{4pt}{$\tuple_{[1]}$} \nwarrow \raisebox{-6pt}{$\bar\tuple'_{[1]}$}, \dots, \raisebox{4pt}{$\tuple_{[\ell]}$} \nwarrow \raisebox{-6pt}{$\bar\tuple'_{[\ell]}$} \bigg].
\]

\begin{proposition}
For any two Cambrian $\ell$-tuples~$\tuple$ and~$\tuple'$, the product~$\PBax_{\tuple} \product \PBax_{\tuple'}$ is given by 
\[
\PBax_{\tuple} \product \PBax_{\tuple'} = \sum_{\tuple[S]} \PBax_{\tuple[S]},
\]
where~$\tuple[S]$ runs over the interval between~$\raisebox{-6pt}{$\tuple$} \nearrow \raisebox{4pt}{$\bar\tuple'$}$ and~$\raisebox{4pt}{$\tuple$} \nwarrow \raisebox{-6pt}{$\bar\tuple'$}$ in the~$\signatures(\tuple)\signatures(\tuple')$-Cambrian lattice.
\end{proposition}

\begin{remark}[Multiplicative bases]
\enlargethispage{.2cm}
Similar to the multiplicative bases defined in Section~\ref{sec:multiplicativeBases} and Remark~\ref{rem:multiplicativeBasesBaxter}, the bases~$\ECamb^{\tuple}$ and~$\HCamb^{\tuple}$ defined by
\[
\ECamb^{\tuple} \eqdef \sum_{\tuple \le \tuple'} \PCamb_{\tuple'}
\qquad\text{and}\qquad
\HCamb^{\tuple} \eqdef \sum_{\tuple' \le \tuple} \PCamb_{\tuple'}
\]
are multiplicative since
\[
\ECamb^{\tuple} \product \ECamb^{\tuple'} = \ECamb^{\raisebox{-5pt}{\scriptsize$\tuple$}\nearrow \raisebox{4pt}{\scriptsize$\bar \tuple'$}}
\qquad\text{and}\qquad
\HCamb^{\tuple} \product \HCamb^{\tuple'} = \HCamb^{\raisebox{4pt}{\scriptsize$\tuple$}\nwarrow \raisebox{-5pt}{\scriptsize$\bar \tuple'$}}.
\]
The $\ECamb$-indecomposable elements are precisely the Cambrian $\ell$-tuples~$\tuple$ such that all linear extensions of the union~$\bigcup_{k \in [\ell]} \tuple_{[k]}$ are indecomposable. In particular, $\tuple$ is $\ECamb$-indecomposable as soon as one of the~$\tuple_{[k]}$ is $\ECamb$-indecomposable, but this condition is not necessary. The $\ECamb$-indecomposable $\signatures$-Cambrian tuples form an ideal of the $\signatures$-Cambrian lattice, but this ideal is not principal.
\end{remark}

\para{Coproduct}
A \defn{cut}~$\gamma$ of a Cambrian $\ell$-tuple~$\tuple[S]$ is a cut of the union~$\bigcup_{k \in [\ell]} \tuple[S]_{[k]}$. It defines a cut~$\gamma_{[k]}$ on each Cambrian tree~$\tuple[S]_{[k]}$. We denote by
\[
A(\tuple[S], \gamma) \eqdef A(\tuple[S]_{[1]}, \gamma_{[1]}) \times \dots \times A(\tuple[S]_{[\ell]}, \gamma_{[\ell]})
\qquad\text{and}\qquad
B(\tuple[S], \gamma) \eqdef B(\tuple[S]_{[1]}, \gamma_{[1]}) \times \dots \times B(\tuple[S]_{[\ell]}, \gamma_{[\ell]}).
\]

\begin{proposition}
For any Cambrian $\ell$-tuple~$\tuple[S]$, the coproduct~$\coproduct \PBax_{\tuple[S]}$ is given~by
\[
\coproduct \PBax_{\tuple[S]} = \sum_{\gamma} \bigg( \sum_{\tuple[B] \in B(\tuple[S], \gamma)} \PBax_{\tuple[B]} \bigg) \otimes \bigg( \sum_{\tuple[A] \in A(\tuple[S], \gamma)} \PBax_{\tuple[A]} \bigg),
\]
where~$\gamma$ runs over all cuts of~$\tuple[S]$.
\end{proposition}



\subsection{Dual Cambrian tuple Hopf Algebra}
\label{subsec:dualCambrianTupleAlgebra}

We now consider the dual Hopf algebra of~$\Camb_\ell$. Again, the following statement is automatic from Theorem~\ref{thm:cambTupleSubalgebra}.

\begin{theorem}
The graded dual~$\Camb_\ell^*$ of the $\ell$-Cambrian algebra is isomorphic to the image of~$\FQSym_{\pm^\ell}^*$ under the canonical projection
\[
\pi : \C\langle A \rangle \longrightarrow \C\langle A \rangle / \equiv_\ell,
\]
where~$\equiv_\ell$ denotes the $\ell$-Cambrian congruence. The dual basis~$\QCamb_{\tuple}$ of~$\PCamb_{\tuple}$ is expressed as~${\QCamb_{\tuple} = \pi(\G_\tau)}$, where~$\tau$ is any linear extension of~$\bigcup_{k \in [\ell]} \tuple_{[k]}$.
\end{theorem}

We now describe the product and coproduct in~$\Camb_\ell^*$ in terms of combinatorial operations on Cambrian $\ell$-tuples. We use the definitions and notations introduced in Section~\ref{subsec:quotientAlgebra}.

\para{Product}
The product in~$\Camb_\ell^*$ can be described using gaps and laminations similarly to Proposition~\ref{prop:productDual}. For two Cambrian trees~$\tree$ and~$\tree'$ and a shuffle~$s$ of the signatures~$\signature(\tree)$ and~$\signature(\tree')$, we still denote by~$\tree \,{}_s\!\backslash \tree'$ the tree described in Section~\ref{subsec:quotientAlgebra}. For two Cambrian $\ell$-tuples~$\tuple$ and~$\tuple'$, with trees of size~$n$ and~$n'$ respectively, and for a shuffle~$s$ of~$[n]$ and~$[n']$, we write
\[
\tuple \,{}_s\!\backslash \tuple' \eqdef \big[ \tuple_{[1]} \,{}_s\!\backslash \tuple'_{[1]}, \dots, \tuple_{[\ell]} \,{}_s\!\backslash \tuple'_{[\ell]} \big],
\]
where we see~$s$ as a shuffle of the signatures~$\signature(\tuple_{[k]})$ and~$\signature(\tuple'_{[k]})$.

\begin{proposition}
For any two Cambrian $\ell$-tuples~$\tuple, \tuple'$, the product~$\QBax_{\tuple} \product \QBax_{\tuple'}$ is given by
\[
\QBax_{\tuple} \product \QBax_{\tuple'} = \sum_s \QBax_{\tuple \,{}_s\!\backslash \tuple'},
\]
where~$s$ runs over all shuffles of~$[n]$ and~$[n']$ (where~$n$ and~$n'$ denote the respective sizes of the trees of~$\tuple$ and~$\tuple'$).
\end{proposition}

\para{Coproduct}
The coproduct in~$\Camb_\ell^*$ can be described combinatorially as in Proposition~\ref{prop:coproductDual}. For a Cambrian $\ell$-tuple~$\tuple[S]$, with trees of size~$n$, and a gap~$\gamma \in \{0,\dots,n\}$, we define
\[
L(\tuple[S], \gamma) = \big[ L(\tuple_{[1]}, \gamma), \dots, L(\tuple_{[\ell]}, \gamma) \big]
\qquad\text{and}\qquad
R(\tuple[S], \gamma) = \big[ R(\tuple_{[1]}, \gamma), \dots, R(\tuple_{[\ell]}, \gamma) \big].
\]

\begin{proposition}
For any Cambrian $\ell$-tuple~$\tuple[S]$, the coproduct~$\coproduct\QCamb_{\tuple[S]}$ is given~by
\[
\coproduct\QCamb_{\tuple[S]} = \sum_{\gamma} \QCamb_{L(\tuple[S], \gamma)} \otimes \QCamb_{R(\tuple[S], \gamma)},
\]
where~$\gamma$ runs over all gaps between consecutive positions in~$[n]$ (where~$n$ denotes the size of the trees of~$\tuple$).
\end{proposition}


\part{The Schr\"oder-Cambrian Hopf Algebra}
\label{part:SchroderCambrian}


\section{Schr\"oder-Cambrian trees}

We already insisted on the fact that the bases of M.~Malvenuto and C.~Reutenauer's algebra on permutations, of J.-L.~Loday and M.~Ronco's algebra on binary trees, and of L.~Solomon's descent algebra correspond to the vertices of the permutahedra, of the associahedra, and of the cubes respectively. In~\cite{Chapoton}, F.~Chapoton generalized these algebras to three Hopf algebras with bases indexed by all faces of the permutahedra, of the associahedra, and of the cubes. To conclude the paper, we show that F.~Chapoton's construction extends as well to the Cambrian setting. We obtain the Schr\"oder-Cambrian Hopf algebra with basis indexed by all faces of all associahedra of C.~Hohlweg and C.~Lange. It is also a good occasion to observe relevant combinatorial properties of Schr\"oder-Cambrian trees, which correspond to the faces of these associahedra.


\subsection{Schr\"oder-Cambrian trees}

The faces of J.-L.~Loday's $n$-dimensional associahedron correspond to \defn{Schr\"oder trees} with $n+1$ leaves, \ie trees where each internal node has at least $2$ children. We first recall the Cambrian counterpart of these trees, which correspond to all faces of C.~Hohlweg and C.~Lange's associahedra (see Section~\ref{subsec:geometrySchroderTrees}). These trees were defined in~\cite{LangePilaud} as ``spines'' of dissections of polygons, see Remark~\ref{rem:dissections}.

\begin{definition}
Consider a signature~$\signature \in \pm^n$. For~$X \subseteq [n]$, we denote by~$X^+ \eqdef \set{x \in X}{\signature_x = +}$ and~$X^- \eqdef \set{x \in X}{\signature_x = -}$. A \defn{Schr\"oder $\signature$-Cambrian tree} is a directed tree~$\tree$ with vertex set~$\ground$ endowed with a vertex labeling~$p : \ground \to 2^{[n]} \ssm \varnothing$ such that
\begin{enumerate}[(i)]
\item the labels of~$\tree$ partition~$[n]$, \ie $v \ne w \in \ground \implies p(v) \cap p(w) = \varnothing$ and~$\bigcup_{v \in \ground} p(v) = [n]$;
\item each vertex~$v \in \ground$ has one incoming (resp.~outgoing) subtree~$\tree_{v,I}$ for each interval~$I$ of~${[n] \ssm p(v)^-}$ (resp.~of~$[n] \ssm p(v)^+$) and all labels of~$\tree_{v,I}$ are subsets of~$I$.
\end{enumerate}
For~$p \ge 0$ and~$\signature \in \pm^n$, we denote by~$\SchrCambTrees^{\ge p}(\signature)$ the set of Schr\"oder $\signature$-Cambrian trees with at most~$n-p$ internal nodes, and we define~$\SchrCambTrees^{\ge p}(n) \eqdef \bigsqcup_{\signature \in \pm^n} \SchrCambTrees^{\ge p}(\signature)$ and $\SchrCambTrees^{\ge p} \eqdef \bigsqcup_{n \in \N} \SchrCambTrees^{\ge p}(n)$. They correspond to associahedron faces of dimension at least~$p$. Finally, we simply omit the superscript~$^{\ge p}$ in the previous notations to denote Schr\"oder-Cambrian trees with arbitrary many internal nodes. Note that this defines a filtration
\[
\SchrCambTrees = \bigcup_{p \ge 1} \SchrCambTrees^{\ge p}
\qquad\text{with}\qquad
\SchrCambTrees^{\ge 0} \supset \SchrCambTrees^{\ge 1} \supset \dots
\]
\end{definition}

\begin{definition}
A \defn{$k$-leveled Schr\"oder $\signature$-Cambrian tree} is a directed tree with vertex set~$\ground$ endowed with two labelings~$p : \ground \to 2^{[n]} \ssm \varnothing$ and~$q : \ground \to [k]$ which respectively define a Schr\"oder $\signature$-Cambrian tree and an increasing tree (meaning that $q$ is surjective and~$v \to w$ in~$\tree$ implies that~$q(v) < q(w)$).
\end{definition}

A Schr\"oder-Cambrian tree and a $3$-leveled Schr\"oder-Cambrian tree are represented in \fref{fig:leveledSchroderCambrianTree}. Note that each level of a $k$-leveled Schr\"oder $\signature$-Cambrian tree may contain more than one node.

\begin{figure}[h]
  \centerline{\includegraphics{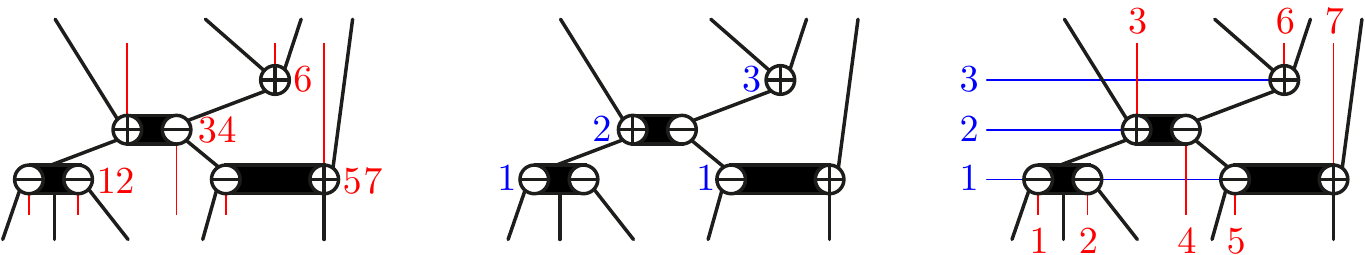}}
  \caption{A Schr\"oder-Cambrian tree (left), an increasing tree (middle), and a $3$-leveled Schr\"oder-Cambrian tree (right).}
  \label{fig:leveledSchroderCambrianTree}
\end{figure}

\begin{remark}[Spines of dissections]
\label{rem:dissections}
\enlargethispage{-.4cm}
Exactly as $\signature$-Cambrian trees correspond to triangulations of the $(n+2)$-gon~$\polygon$ (see Remark~\ref{rem:triangulation}), Schr\"oder $\signature$-Cambrian trees correspond to all dissections of~$\polygon$. See \fref{fig:dissection} and refer to~\cite{LangePilaud} for details.
\begin{figure}[h]
  \centerline{\includegraphics{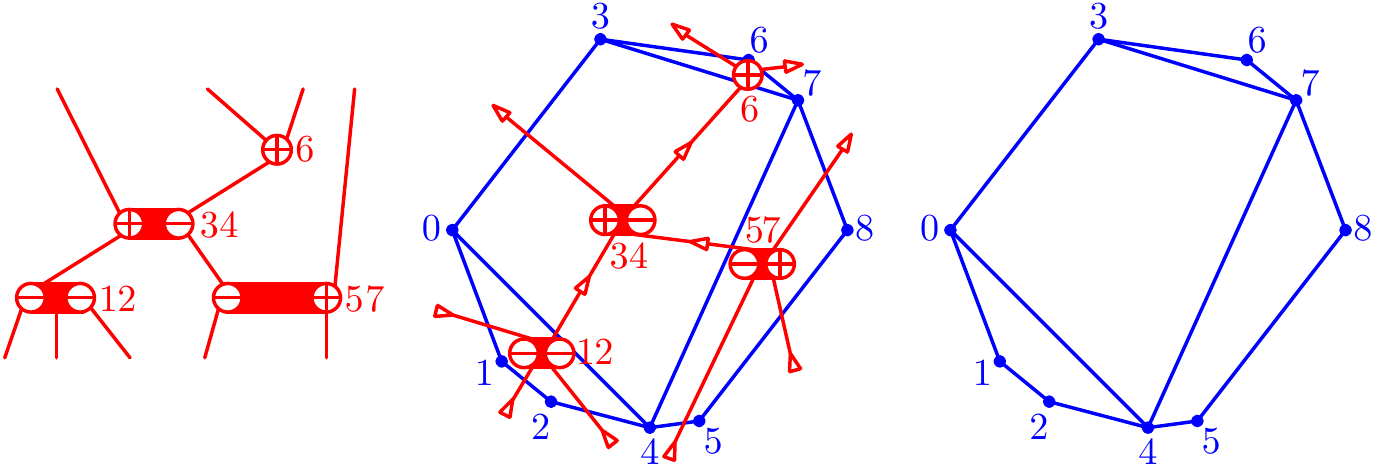}}
  \caption{Schr\"oder-Cambrian trees (left) and dissections (right) are dual to each other (middle).}
  \label{fig:dissection}
\end{figure}
\end{remark}

Remark~\ref{rem:dissections} immediately implies that the number of Schr\"oder $\signature$-Cambrian trees with $k$ nodes is the number of $(n-k)$-dimensional faces of the associahedron, and is therefore independent of the signature~$\signature$. An alternative proof based on generating trees is mentioned in Remark~\ref{rem:patternAvoidanceSchroder}.

\begin{proposition}
\label{prop:SchroderCambrianNumbers}
For any signature~$\signature \in \pm^n$, the number of Schr\"oder $\signature$-Cambrian trees with $k$ internal nodes is
\[
\frac{1}{k+1}\binom{n+2+k}{k+1}\binom{n-1}{k+1},
\]
see~\href{https://oeis.org/A033282}{\cite[A033282]{OEIS}}.
\end{proposition}


\subsection{Schr\"oder-Cambrian correspondence}

We now define an analogue of the Cambrian correspondence and Cambrian $\PSymbol$-symbol, which will map the faces of the permutahedron to the faces of C.~Hohlweg and C.~Lange's associahedra. Recall first that the $(n-k)$-dimensional faces of the $n$-dimensional permutahedron correspond to \defn{surjections} from~$[n]$ to~$[k]$, or equivalently to \defn{ordered partitions} of~$[n]$ into $k$ parts. See \fref{fig:permutahedra}. We denote (abusively) by~$\pi^{-1}$ the ordered partition corresponding to a surjection~$\pi : [n] \to [k]$, \ie given by~$\pi^{-1} \eqdef \pi^{-1}(\{1\}) \bsep \pi^{-1}(\{2\}) \bsep \cdots \bsep \pi^{-1}(\{k\})$. Conversely, we denote (abusively) by~$\lambda^{-1}$ the surjection corresponding to an ordered partition~$\lambda = \lambda_1 \sep \lambda_2 \sep \cdots \sep \lambda_k$, \ie such that each~$i$ belongs to the part~$\lambda_{\lambda^{-1}(i)}$. We represent graphically a surjection~$\pi : [n] \to [k]$ by the~$(k \times n)$-table with a dot in row~$\pi(j)$ in each column~$j$. Therefore, we represent an ordered partition~$\lambda \eqdef \lambda_1 \sep \cdots \sep \lambda_k$ of~$[n]$ by the~$(k \times n)$-table with a dot in row~$i$ and column~$j$ for each~$j \in \lambda_i$. See \fref{fig:insertionAlgorithmSchroder}\,(left). In this paper, we work with ordered partitions rather than surjections to match better the presentation of the previous sections: the permutations of~$[n]$ used in the previous sections have to be thought of as ordered partitions of~$[n]$ into~$n$ parts. We denote by~$\fP_n^{\ge p}$ the set of ordered partitions of~$[n]$ into at most~$n-p$ parts, and we set~$\fP^{\ge p} \eqdef \bigsqcup_{n \in \N} \fP_n^{\ge p}$. It correspond to permutahedron faces of dimension at least~$p$. As previously, we omit the superscript~$^{\ge p}$ in these notations to forget this dimension restriction.

\begin{figure}[h]
  \capstart
  \centerline{\includegraphics[width=\textwidth]{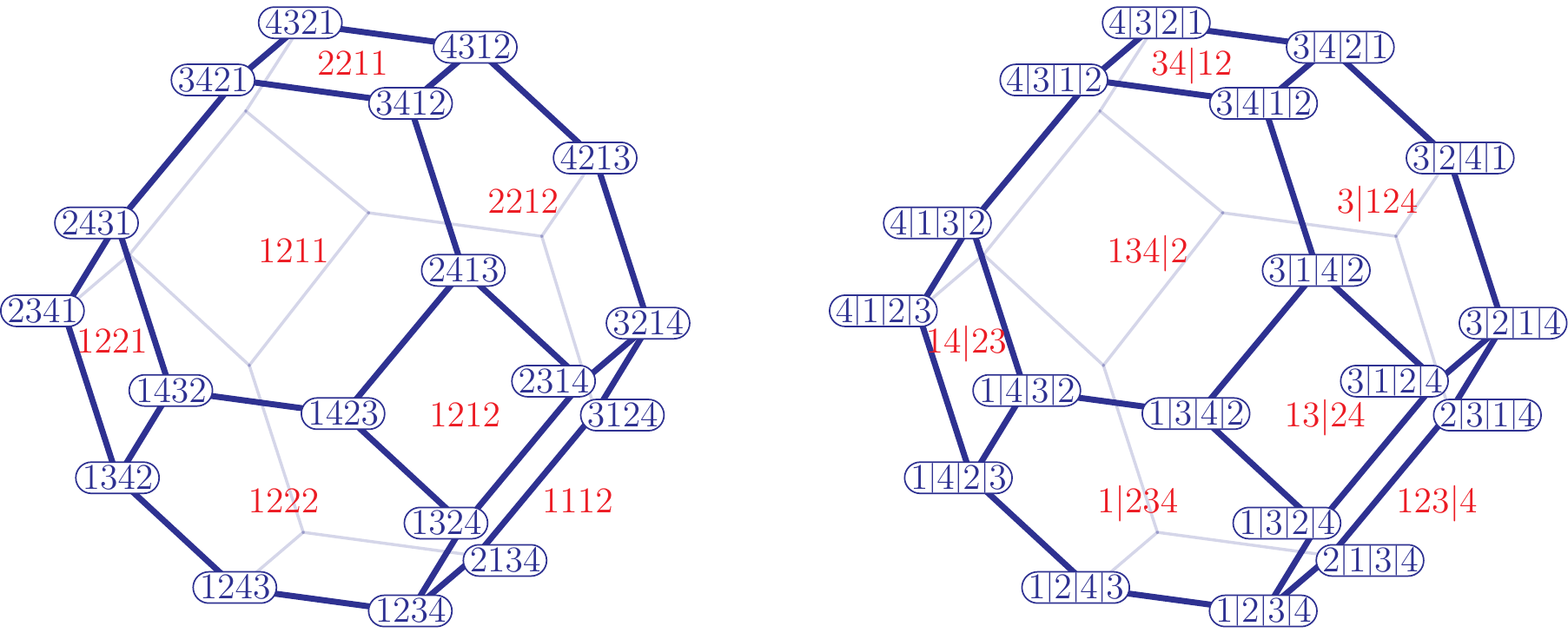}}
  \caption{The $3$-dimensional permutahedron~$\Perm[\protect{[4]}]$. Its $(4-k)$-dimensional faces correspond equivalently to the surjections from~$[4]$ to~$[k]$ (left), or to the ordered partitions of~$[4]$ into~$k$ parts (right). Vertices are in blue and facets in red. The reader is invited to label the edges accordingly.}
  \label{fig:permutahedra}
\end{figure}

A \defn{signed ordered partition} is an ordered partition table where each dot receives a~$+$ or~$-$ sign. For a signature~$\signature \in \pm^n$, we denote by~$\fP_\signature^{\ge p}$ the set of ordered partitions of~$[n]$ into at most~$n-p$ parts signed by~$\signature$, and we set
\[
\fP_\pm^{\ge p} \eqdef \bigsqcup_{n \in \N, \signature \in \pm^n} \fP_\signature^{\ge p}.
\]
We omit again the superscript~$^{\ge p}$ in the previous notations to denote signed ordered partitions with arbitrarily many parts. This gives again a filtration
\[
\fP_\pm = \bigcup_{p \ge 1} \fP_\pm^{\ge p}
\qquad\text{with}\qquad
\fP_\pm^{\ge 0} \supset \fP_\pm^{\ge 1} \supset \dots
\]

Given such a signed ordered partition~$\lambda$, we construct a leveled Schr\"oder-Cambrian tree~$\SchrCambCorresp(\lambda)$ as follows. As a preprocessing, we represent the table of~$\lambda$, we draw a vertical wall below the negative dots and above the positive dots, and we connect into nodes the dots at the same level which are not separated by a wall. Note that we might obtain several nodes per level. We then sweep the table from bottom to top as follows. The procedure starts with an incoming strand in between any two consecutive negative values. At each level, each node~$v$ (connected set of dots) gathers all strands in the region below and visible from~$v$ (\ie not hidden by a vertical wall) and produces one strand in each region above and visible from~$v$. The procedure finished with an outgoing strand in between any two consecutive positive values. See \fref{fig:insertionAlgorithmSchroder}.

\begin{figure}[h]
  \centerline{\includegraphics[width=\textwidth]{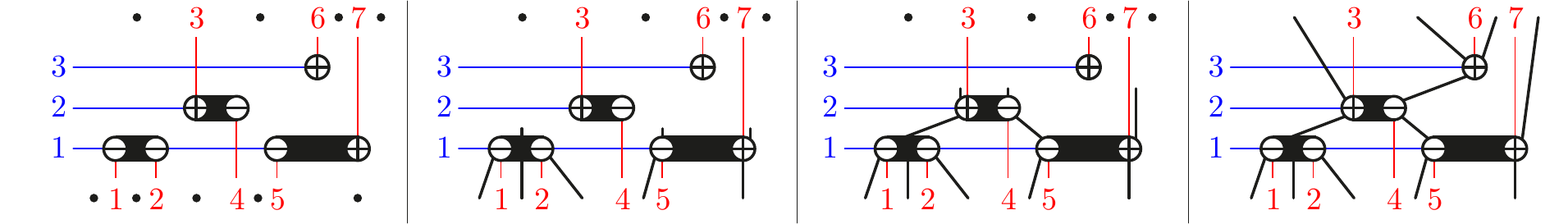}}
  \caption{The insertion algorithm on the signed ordered partition~$\down{125}\up{7} \sep \up{3}\down{4} \sep \up{6}$.}
  \label{fig:insertionAlgorithmSchroder}
\end{figure}

\begin{proposition}[\cite{LangePilaud}]
The map~$\SchrCambCorresp$ is a bijection from signed ordered partitions to leveled Schr\"oder-Cambrian trees.
\end{proposition}

We define the $\PSymbol^\star$-symbol of a signed ordered partition~$\lambda$ as the Schr\"oder-Cambrian tree~$\surjectionSchrPermAsso(\lambda)$ defined by~$\SchrCambCorresp(\lambda)$. Note that an ordered partition of~$[n]$ into $k$ parts is sent to a Schr\"oder-Cambrian tree with at least~$k$ internal nodes, since some levels can be split into several nodes. In other words, the fibers of the Schr\"oder-Cambrian $\PSymbol^\star$-symbol respect the filtrations~$(\fP_\pm^{\ge p})_{p \in \N}$ and~$(\SchrCambTrees^{\ge p})_{p \in \N}$, in the sense that
\[
(\surjectionSchrPermAsso)^{-1} \big( \SchrCambTrees^{\ge p} \big) \subseteq \fP_\pm^{\ge p}.
\]

The following characterization of the fibers of the map~$\surjectionSchrPermAsso$ is immediate from the description of the Schr\"oder-Cambrian correspondence. For a Schr\"oder-Cambrian tree~$\tree$, we write~$i \to j$ in~$\tree$ if the node of~$\tree$ containing~$i$ is below the node of~$\tree$ containing~$j$, and~$i \sim j$ in~$\tree$ if~$i$ and~$j$ belong to the same node of~$\tree$. We say that~$i$ and~$j$ are \defn{incomparable} in~$\tree$ when~$i \not\to j$, $j \not\to i$, and~$i \not\sim j$.

\begin{proposition}
\label{prop:mergeLinearExtensions}
For any Schr\"oder $\signature$-Cambrian tree~$\tree$ and signed ordered partition~$\lambda \in \fP_\signature$, we have~$\surjectionSchrPermAsso(\lambda) = \tree$ if and only if~$i \sim j$ in~$\tree$ implies $\lambda^{-1}(i) = \lambda^{-1}(j)$ and $i \to j$ in~$\tree$ implies $\lambda^{-1}(i) < \lambda^{-1}(j)$.
In other words, $\lambda$ is obtained from a linear extension of~$\tree$ by merging parts which label incomparable vertices of~$\tree$.
\end{proposition}

\begin{example}
When~$\signature = ({+})^n$, the Schr\"oder-Cambrian tree~$\surjectionSchrPermAsso(\lambda)$ is the increasing tree of~$\lambda^{-1}$. Here, the \defn{increasing tree}~$\increasingTree(\pi)$ of a surjection~$\pi = \pi^{(1)} 1 \pi^{(2)} 1 \dots 1 \pi^{(p)}$ is defined inductively by grafting the increasing trees~$\increasingTree(\pi^{(1)}), \dots, \increasingTree(\pi^{(p)})$ from left to right on a bottom root labeled by~$1$. Similarly, when~$\signature = ({-})^n$, the Schr\"oder-Cambrian tree~$\surjectionSchrPermAsso(\lambda)$ is the decreasing tree of~$\lambda^{-1}$. Here, the \defn{decreasing tree}~$\decreasingTree(\pi)$ of a surjection~$\pi = \pi^{(1)} k \pi^{(2)} k \dots k \pi^{(p)}$ from~$[n]$ to~$[k]$ is defined inductively by grafting the decreasing trees~$\decreasingTree(\pi^{(1)}), \dots, \decreasingTree(\pi^{(p)})$ from left to right on a top root labeled by~$k$. See \fref{fig:constantSignsSchroder}.

\begin{figure}[h]
  \centerline{\includegraphics{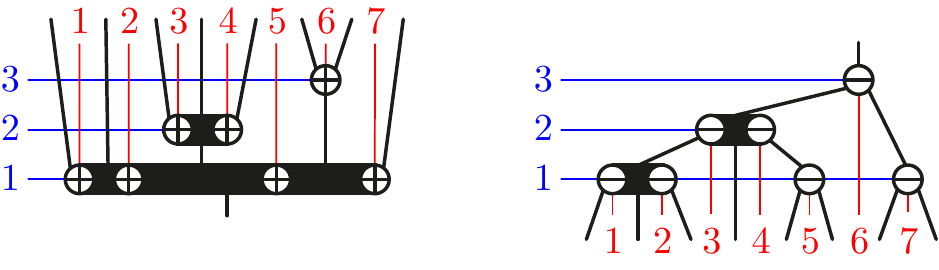}}
  \caption{The insertion procedure produces increasing Schr\"oder trees when the signature is constant positive (left) and decreasing Schr\"oder trees when the signature is constant negative (right).}
  \label{fig:constantSignsSchroder}
\end{figure}
\end{example}

\begin{remark}[Schr\"oder-Cambrian correspondence on dissections]
Similar to Remark~\ref{rem:CambrianCorrespondenceTriangulations}, we can describe the map~$\surjectionSchrPermAsso$ on the dissections of the polygon~$\polygon$. Namely the dissection dual to the Schr\"oder-Cambrian tree~$\surjectionSchrPermAsso(\lambda)$ is the union of the paths~$\pi_0, \dots, \pi_n$ where~$\pi_i$ is the path between vertices~$0$ and~$n+1$ of~$\polygon$ passing through the vertices in the symmetric difference~$\signature^{-1}(-) \symdif \big( \bigcup_{j \in [i]} \lambda_j \big)$.
\end{remark}


\subsection{Schr\"oder-Cambrian congruence}

As for the Cambrian congruence, the fibers of the Schr\"oder-Cambrian $\PSymbol^\star$-symbol define a congruence relation on signed ordered partitions, which can be expressed by rewriting rules.

\begin{definition}
For a signature~$\signature \in \pm^n$, the Schr\"oder $\signature$-Cambrian congruence is the equivalence relation on~$\fP_\signature$ defined as the transitive closure of the rewriting rules
\[
U \sep \b{a} \sep \b{c} \sep V \equiv^\star_\signature U \sep \b{a}\b{c} \sep V \equiv^\star_\signature U \sep \b{c} \sep \b{a} \sep V,
\]
where $\b{a}, \b{c}$ are parts while~$U, V$ are sequences of parts of~$[n]$, and there exists~$\b{a} < b < \b{c}$ such that~$\signature_b = +$ and~$b \in \bigcup U$, or~$\signature_b = -$ and~$b \in \bigcup V$. The Schr\"oder-Cambrian congruence is the equivalence relation~$\equiv^\star$ on~$\fP_{\pm}$ obtained as the union of all Schr\"oder $\signature$-Cambrian congruences.
\end{definition}

For example, $\down{125}\up{7} \sep \up{3}\down{4} \sep \up{6} \equiv^\star \down{12} \sep \down{5}\up{7} \sep \up{3}\down{4} \sep \up{6} \equiv^\star  \down{5}\up{7} \sep \down{12} \sep \up{3}\down{4} \sep \up{6} \not\equiv^\star \down{5}\up{7} \sep \up{3}\down{4} \sep \down{12} \sep \up{6}$.

\begin{proposition}
Two signed ordered partitions~$\lambda, \lambda' \in \fP_{\pm}$ are Schr\"oder-Cambrian congruent if and only if they have the same $\PSymbol^\star$-symbol:
\[
\lambda \equiv^\star \lambda' \iff \surjectionSchrPermAsso(\lambda) = \surjectionSchrPermAsso(\lambda').
\]
\end{proposition}

\begin{proof}
It boils down to observe that two consecutive parts~$\b{a}$ and~$\b{c}$ of an ordered partition~$U \sep \b{a} \sep \b{c} \sep V$ in a fiber~$(\surjectionSchrPermAsso)^{-1}(\tree)$ can be merged to~$U \sep \b{a}\b{c} \sep V$ and even exchanged to~$U \sep \b{c} \sep \b{a} \sep V$ while staying in~$(\surjectionSchrPermAsso)^{-1}(\tree)$ precisely when they belong to distinct subtrees of a node of~$\tree$. They are therefore separated by the vertical wall below (resp.~above) a value~$b$ with~$\b{a} < b < \b{c}$ and such that~$\signature_b = -$ and~$b \in V$ (resp.~$\signature_b = +$ and~$b \in U$).
\end{proof}


\subsection{Weak order on ordered partitions and Schr\"oder-Cambrian lattices}

In order to define the Schr\"oder counterpart of the Cambrian lattice, we first need to extend the weak order on permutations to all ordered partitions. This was done by D.~Krob, M.~Latapy, J.-C.~Novelli, H.~D.~Phan and~S.~Schwer in~\cite{KrobLatapyNovelliPhanSchwer}. See also~\cite{BoulierHivertKrobNovelli} for representation theoretic properties of this order and \cite{PalaciosRonco} for an extension to all finite Coxeter systems.

\begin{definition}
The \defn{coinversion map}~${\coinv(\lambda) : \binom{[n]}{2} \to \{-1,0,1\}}$ of an ordered partition~$\lambda \in \fP_n$ is the map defined for~$i < j$ by
\[
\coinv(\lambda)(i,j) = 
\begin{cases}
-1 & \text{if } \lambda^{-1}(i) < \lambda^{-1}(j), \\
\phantom{-}0 & \text{if } \lambda^{-1}(i) = \lambda^{-1}(j), \\
\phantom{-}1 & \text{if } \lambda^{-1}(i) > \lambda^{-1}(j).
\end{cases}
\]
It is also called the \defn{inversion map} of the surjection~$\lambda^{-1}$.
\end{definition}

\begin{definition}
There are two natural poset structures on~$\fP_n$:
\begin{itemize}
\item The \defn{refinement poset}~$\contractionPoset$ defined by~$\lambda \contractionPoset \lambda'$ if~$|\coinv(\lambda)(i,j)| \ge |\coinv(\lambda')(i,j)|$ for all~${i < j}$. It is isomorphic to the face lattice of the permutahedron~$\Perm$, and respects the filtration~$(\fP_n^{\ge p})_{p \in [n]}$.
\item The \defn{weak order}~$\le$ defined by~$\lambda \le \lambda'$ if~$\coinv(\lambda)(i,j) \le \coinv(\lambda')(i,j)$ for all~$i < j$.
\end{itemize}
\end{definition}

\enlargethispage{.3cm}
These two posets are represented in \fref{fig:latticesOrderedPartitions} for~$n = 3$.

\begin{figure}[h]
  \centerline{\includegraphics[width=.7\textwidth]{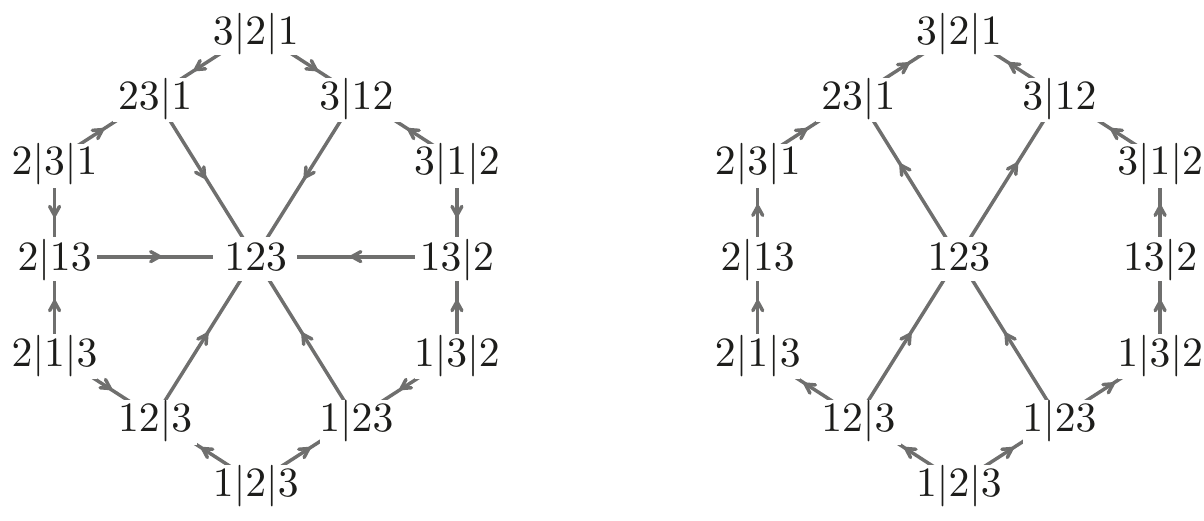}}
  \caption{The refinement poset (left) and the weak order (right) on~$\fP_3$.}
  \label{fig:latticesOrderedPartitions}
\end{figure}

Note that the restriction of the weak order to~$\fS_n$ is the classical weak order on permutations, which is a lattice. This property was extended to the weak order on~$\fP_n$ in~\cite{KrobLatapyNovelliPhanSchwer}.

\begin{proposition}[\cite{KrobLatapyNovelliPhanSchwer}]
The weak order~$<$ on the set of ordered partitions~$\fP_n$  is a lattice.
\end{proposition}

In the following statement and throughout the remaining of the paper, we define for~$X, Y \subset \N$
\[
X \ll Y \; \iff \; \max(X) < \min(Y) \; \iff \; x < y \text{ for all~$x \in X$ and~$y \in Y$.}
\]

\begin{proposition}[\cite{KrobLatapyNovelliPhanSchwer}]
The cover relations of the weak order~$<$ on~$\fP_n$ are given by
\begin{gather*}
\lambda_1 \sep \cdots \sep \lambda_i \sep \lambda_{i+1} \sep \cdots \sep \lambda_k \;\; < \;\; \lambda_1 \sep \cdots \sep \lambda_i\lambda_{i+1} \sep \cdots \sep \lambda_k \qquad\text{if } \lambda_i \ll \lambda_{i+1}, \\
\lambda_1 \sep \cdots \sep \lambda_i\lambda_{i+1} \sep \cdots \sep \lambda_k \;\; < \;\; \lambda_1 \sep \cdots \sep \lambda_i \sep \lambda_{i+1} \sep \cdots \sep \lambda_k \qquad\text{if } \lambda_{i+1} \ll \lambda_i.
\end{gather*}
\end{proposition}

We now extend the Cambrian lattice on Cambrian trees to a lattice on all Schr\"oder-Cambrian trees. For the constant signature~$\signature = (-)^n$, the order considered below was already defined by P.~Palacios and M.~Ronco in~\cite{PalaciosRonco}, although its lattice structure (Proposition~\ref{prop:SchroderCambrianLattice}) was not discussed in there.

\begin{definition}
Consider a Schr\"oder $\signature$-Cambrian tree~$\tree$, and an edge~$e=\{v,w\}$ of~$\tree$. We denote by~$\tree/e$ the tree obtained by contracting~$e$ in~$\tree$. It is again a Schr\"oder $\signature$-Cambrian tree. We say that the contraction is \defn{increasing} if~$p(u) \ll p(v)$ and \defn{decreasing} if~$p(v) \ll p(u)$. Otherwise, we say that the contraction is \defn{non-monotone}.
\end{definition}

\begin{definition}
There are two natural poset structures on~$\SchrCambTrees_\signature$:
\begin{itemize}
\item The \defn{contraction poset}~$\contractionPoset$ defined as the transitive closure of the relation~$\tree \contractionPoset \tree/e$ for any Schr\"oder $\signature$-Cambrian tree~$\tree$ and edge~$e \in \tree$. It is isomorphic to the face lattice of the associahedron~$\Asso$, and respects the filtration~$(\SchrCambTrees_\signature^{\ge p})_{p \in [n]}$.
\item The \defn{Schr\"oder $\signature$-Cambrian poset}~$\SchrCambPoset$ defined as the transitive closure of the relation~$\tree \SchrCambPoset \tree/e$ (resp.~$\tree/e \SchrCambPoset \tree$) for any Schr\"oder $\signature$-Cambrian tree~$\tree$ and any edge~$e \in \tree$ defining an increasing (resp.~decreasing) contraction.
\end{itemize}
\end{definition}

These two posets are represented in \fref{fig:SchroderCambrianLattices} for the signature~${+}{-}{-}$. Observe that there are two non-monotone contractions. Note also that the restriction of the Schr\"oder $\signature$-Cambrian poset~$\SchrCambPoset$ to the $\signature$-Cambrian trees is the $\signature$-Cambrian lattice.

\begin{figure}[h]
  \centerline{\includegraphics[width=\textwidth]{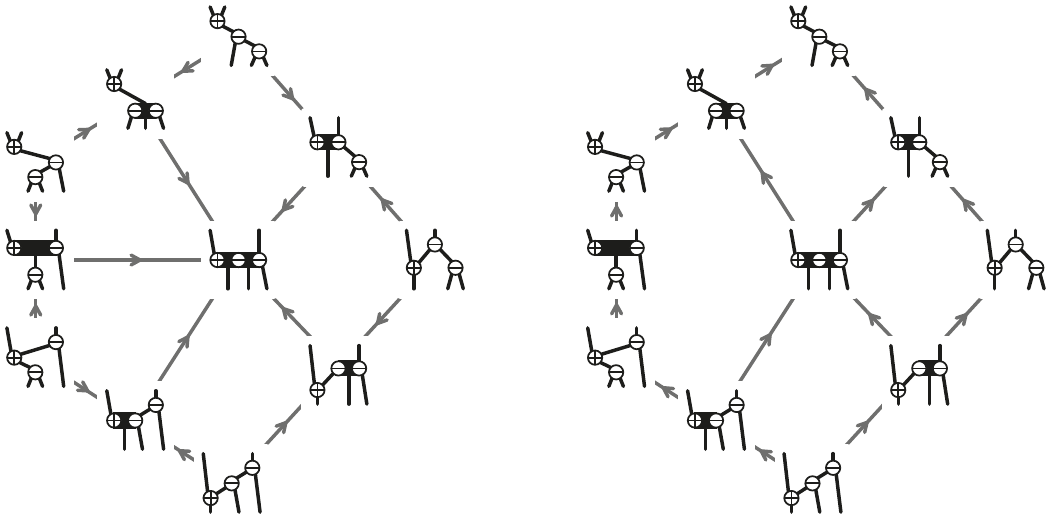}}
  \caption{The contraction poset (left) and the Schr\"oder $({+}{-}{-})$-Cambrian poset (right) on Schr\"oder $({+}{-}{-})$-Cambrian trees.}
  \label{fig:SchroderCambrianLattices}
\end{figure}

\begin{proposition}
The map~$\surjectionSchrPermAsso$ defines a poset homomorphism from the weak order on~$\fP_\signature$ to the Schr\"oder $\signature$-Cambrian poset on~$\SchrCambTrees(\signature)$.
\end{proposition}

\begin{proof}
Let~$\lambda < \lambda'$ be a cover relation in the weak order on~$\fP_\signature$. Assume that~$\lambda'$ is obtained by merging the parts~$\lambda_i \ll \lambda_{i+1}$ of~$\lambda$ (the other case being symmetric). Let~$u$ denote the rightmost node of~$\surjectionSchrPermAsso(\lambda)$ at level~$i$, and~$v$ the leftmost node of~$\surjectionSchrPermAsso(\lambda)$ at level~$i+1$. If~$u$ and~$v$ are not comparable, then~$\surjectionSchrPermAsso(\lambda) = \surjectionSchrPermAsso(\lambda')$. Otherwise, there is an edge~$u \to v$ in~$\surjectionSchrPermAsso(\lambda)$ and~$\surjectionSchrPermAsso(\lambda')$ is obtained by the increasing contraction of~$u \to v$ in~$\surjectionSchrPermAsso(\lambda)$.
\end{proof}

\begin{proposition}
\label{prop:SchroderCambrianLattice}
For any signature~$\signature \in \pm^n$, the Schr\"oder $\signature$-Cambrian poset on Schr\"oder $\signature$-Cambrian trees is a lattice quotient of the weak order on ordered partitions of~$[n]$.
\end{proposition}

This proposition is proved by the following two lemmas, similar to N.~Reading's approach~\cite{Reading-CambrianLattices}.

\begin{lemma}
The Schr\"oder $\signature$-Cambrian classes are intervals of the weak order.
\end{lemma}

\begin{proof}
Let~$\tree$ be a Schr\"oder $\signature$-Cambrian tree, with vertex labeling~$p : \ground \to 2^{[n]} \ssm \varnothing$. Consider a linear extension of~$\tree$, \ie an ordered partition~$\lambda$ whose parts are the labels of~$\tree$ and such that~$p(v)$ appears before~$p(w)$ for~$v \to w$ in~$\tree$. If~$v$ and~$w$ are two incomparable nodes of~$\tree$, then either~$p(v) \ll p(w)$ or~$p(w) \ll p(v)$ since they are separated by a wall. By successive exchanges, there exists a linear extension~$\lambda_{\min}$ (resp.~$\lambda_{\max}$) of~$\tree$ such that~$p(v)$ appears before (resp.~after) $p(w)$ for any two incomparable nodes~$v$ and~$w$ such that~$p(v) \ll p(w)$. By construction, the $(i,j)$-entries of the coinversion tables of~$\lambda_{\min}$ and~$\lambda_{\max}$ are given for~$i < j$ by
\[
\coinv(\lambda_{\min})(i,j) =
\begin{cases}
1 & \text{if $j \to i$ in $\tree$,} \\
0 & \text{if $i \sim j$ in $\tree$,} \\
-1 & \text{otherwise,}
\end{cases}
\qquad\text{and}\qquad
\coinv(\lambda_{\max})(i,j) =
\begin{cases}
-1 & \text{if $i \to j$ in $\tree$,} \\
0 & \text{if $i \sim j$ in $\tree$,} \\
1 & \text{otherwise.}
\end{cases}
\]
It follows that the fiber of~$\tree$ under~$\surjectionSchrPermAsso$ is the weak order interval~$[\lambda_{\min}, \lambda_{\max}]$.
\end{proof}

\begin{lemma}
Let~$\lambda$ and~$\lambda'$ be two signed ordered partitions from distinct Schr\"oder $\signature$-Cambrian classes~$C$ and~$C'$. If~$\lambda < \lambda'$ then~$\min(C) < \min(C')$ and~$\max(C) < \max(C')$ (all in weak order).
\end{lemma}

\begin{proof}
We prove the result for maximums, the proof for the minimums being similar. We first observe that we can assume that $\lambda'$ covers~$\lambda$ in weak order, so that there exists a position~$i$ such that either~$\lambda'_i = \lambda_i \cup \lambda_{i+1}$ and~$\lambda_i \ll \lambda_{i+1}$, or $\lambda_i = \lambda'_i \cup \lambda'_{i+1}$ and~$\lambda'_{i+1} \ll \lambda'_i$. The proof then works by induction on the weak order distance between~$\lambda$ and~$\max(C)$. If~$\lambda = \max(C)$, the result is immediate as~$\max(C) = \lambda < \lambda' \le \max(C')$. Otherwise, consider an ordered partition~$\mu$ in~$C$ covering~$\lambda$ in weak order. There exists a position~$j \ne i$ such that~$\mu_j = \lambda_j \cup \lambda_{j+1}$ and~$\lambda_j \ll \lambda_{j+1}$, or $\lambda_j = \mu_j \cup \mu_{j+1}$ and~$\mu_{j+1} \ll \mu_j$. We now distinguish four cases, according to the relative positions of~$i$ and~$j$:
\begin{enumerate}[(1)]
\item If~$|i - j| > 1$, then the local changes from~$\lambda$ to~$\lambda'$ at position~$i$ and from~$\lambda$ to~$\mu$ at position~$j$ are independent. Define~$\mu'$ to be the ordered partition obtained from~$\lambda$ by performing both local changes at~$i$ and at~$j$. We then check that~$\lambda' \equiv^\star \mu'$ since any witness for the equivalence~$\lambda \equiv^\star \mu$ is also a witness for the equivalence~$\lambda' \equiv^\star \mu'$. Moreover,~$\mu < \mu'$.
\item Otherwise, the local changes at~$i$ and~$j$ are not independent anymore. We therefore need to treat various cases separately, depending on whether the local changes from~$\lambda$ to~$\lambda'$ and from~$\lambda$ to~$\mu$ are merging or splitting, and on the respective positions of these local changes. In all cases below, $\b{a}, \b{b}, \b{c}$ are parts of~$[n]$ such that~$\b{a} \ll \b{b} \ll \b{c}$, while~$U, V$ are sequences of parts of~$[n]$.
	\begin{itemize}
	\item If~$\lambda = U \sep \b{a} \sep \b{b} \sep \b{c} \sep V$, $\lambda' = U \sep \b{a}\b{b} \sep \b{c} \sep V$, and~$\mu = U \sep \b{a} \sep \b{b}\b{c} \sep V$, then define~$\mu' \eqdef U \sep \b{a}\b{b}\b{c} \sep V$. Any witness for the Schr\"oder-Cambrian congruence~$\lambda \equiv^\star \mu$ is also a witness for the Schr\"oder-Cambrian congruence~$\lambda' \equiv^\star \mu'$. Moreover, we have~$\mu < \mu'$ since~$\b{a} \ll \b{bc}$. The same arguments yield the same conclusions in the following cases:
		\item[--] if~$\lambda = U \sep \b{a} \sep \b{b} \sep \b{c} \sep V$, $\lambda' = U \sep \b{a} \sep \b{b}\b{c} \sep V$, and~$\mu = U \sep \b{a}\b{b} \sep \b{c} \sep V$, then define~$\mu' \eqdef U \sep \b{a}\b{b}\b{c} \sep V$.
		\item[--] if~$\lambda = U \sep \b{a}\b{b}\b{c} \sep V$, $\lambda' = U \sep \b{c} \sep \b{a}\b{b} \sep V$, and~$\mu = U \sep \b{b}\b{c} \sep \b{a} \sep V$, then define~$\mu' \eqdef U \sep \b{c} \sep \b{b} \sep \b{a} \sep V$.
		\item[--] if~$\lambda = U \sep \b{a}\b{b}\b{c} \sep V$, $\lambda' = U \sep \b{b}\b{c} \sep \b{a} \sep V$, and~$\mu = U \sep \b{c} \sep \b{a}\b{b} \sep V$, then define~$\mu' \eqdef U \sep \b{c} \sep \b{b} \sep \b{a} \sep V$.
	\item If~$\lambda = U \sep \b{a} \sep \b{b}\b{c} \sep V$, $\lambda' = U \sep \b{a}\b{b}\b{c} \sep V$, and~$\mu = U \sep \b{a} \sep \b{c} \sep \b{b} \sep V$, then define~$\mu' \eqdef U \sep \b{c} \sep \b{a}\b{b} \sep V$. Any witness for the Schr\"oder-Cambrian congruence~$\lambda \equiv^\star \mu$ is also a witness for the Schr\"oder-Cambrian congruence~$\lambda' \equiv^\star \mu'$. Moreover~$\mu < \mu'$ by comparison of the coinversion tables. The same arguments yield the same conclusions in the case:
		\item[--] if~$\lambda = U \sep \b{a}\b{b} \sep \b{c} \sep V$, $\lambda' = U \sep \b{a}\b{b}\b{c} \sep V$, and~$\mu = U \sep \b{b} \sep \b{a} \sep \b{c} \sep V$, then define~$\mu' \eqdef U \sep \b{b}\b{c} \sep \b{a} \sep V$.
	\item If~$\lambda = U \sep \b{a} \sep \b{b}\b{c} \sep V$, $\lambda' = U \sep \b{a} \sep \b{c} \sep \b{b} \sep V$, and~$\mu = U \sep \b{a}\b{b}\b{c} \sep V$, then define~$\mu' \eqdef U \sep \b{c} \sep \b{a}\b{b} \sep V$. Let~$d$ be a witness for the Schr\"oder-Cambrian congruence~$\lambda \equiv^\star \mu$, that is, $\b{a} < d < \b{bc}$ and either~$\signature_d = -$ and~$d \in V$, or~$\signature_d = +$ and~$d \in U$. Then~$d$ is also a witness for the Schr\"oder-Cambrian congruences~$\lambda' = U \sep \b{a} \sep \b{c} \sep \b{b} \sep V \equiv^\star U \sep \b{a}\b{c} \sep \b{b} \sep V \equiv^\star U \sep \b{c} \sep \b{a} \sep \b{b} \sep V \equiv^\star U \sep \b{c} \sep \b{a}\b{b} \sep V = \mu'$. Moreover, we have~$\mu < \mu'$ since~$\b{ab} \ll \b{c}$. The same arguments yield the same conclusions in the case:
		\item[--] if~$\lambda = U \sep \b{a}\b{b} \sep \b{c} \sep V$, $\lambda' = U \sep \b{b} \sep \b{a} \sep \b{c} \sep V$, and~$\mu = U \sep \b{a}\b{b}\b{c} \sep V$, then define~$\mu' \eqdef U \sep \b{b}\b{c} \sep \b{a} \sep V$.
	\end{itemize}
\end{enumerate}
In all cases, we have~$\lambda \equiv^\star \mu < \mu' \equiv^\star \lambda'$. Since~$\mu$ is closer to~$\max(C)$ than~$\lambda$, we obtain that~$\max(C) < \max(C')$ by induction hypothesis. The proof for minimums is identical.
\end{proof}

\begin{remark}[Extremal elements and pattern avoidance]
\label{rem:patternAvoidanceSchroder}
Since the Schr\"oder-Cambrian classes are generated by rewriting rules, their minimal elements are precisely the ordered partitions avoiding the patterns~$\b{c} \sep \b{a} \text{ -- } \down{b}$ and~$\up{b} \text{ -- } \b{c} \sep \b{a}$, while their maximal elements are precisely the ordered partitions avoiding the patterns~$\b{a} \sep \b{c} \text{ -- } \down{b}$ and~$\up{b} \text{ -- } \b{a} \sep \b{c}$. This enables us to construct a generating tree for these permutations. Similar arguments as in Section~\ref{subsec:CambrianClasses} could thus provide an alternative proof of Proposition~\ref{prop:SchroderCambrianNumbers}.
\end{remark}


\subsection{Canopy}

We define the canopy of a Schr\"oder-Cambrian tree using the same observation as for Cambrian trees: in any Sch\"oder-Cambrian tree, the numbers $i$ and~$i+1$ appear either in the same label, or in two comparable labels.

\begin{definition}
The \defn{canopy} of a Schr\"oder-Cambrian tree~$\tree$ is the sequence~$\surjectionSchrAssoPara(\tree) \in \{-,0,+\}$ defined by
\[
\surjectionSchrAssoPara(\tree)_i = 
\begin{cases}
- & \text{if~$i$ appears above~$i+1$ in~$\tree$,} \\
0 & \text{if~$i$ and~$i+1$ appear in the same label in~$\tree$,} \\
+ & \text{if~$i$ is below~$i+1$ in~$\tree$.}
\end{cases}
\]
\end{definition}

For example, the canopy of the Schr\"oder-Cambrian tree of \fref{fig:leveledSchroderCambrianTree}\,(left) is~$0{+}0{-}{+}{-}$. The following statement provides an immediate analogue of Proposition~\ref{prop:commutativeDiagram} in the Schr\"oder-Cambrian setting. We define the \defn{recoil} sequence of an ordered partition~$\lambda \in \fP_n$ as~${\surjectionSchrPermPara(\lambda) \in \{-,0,+\}^{n-1}}$, where~${\surjectionSchrPermPara(\lambda)_i = \coinv(\lambda)(i,i+1)}$.

\begin{proposition}
\label{prop:commutativeDiagramSchroder}
The maps~$\surjectionSchrPermAsso$, $\surjectionSchrAssoPara$, and~$\surjectionSchrPermPara$ define the following commutative diagram of lattice homomorphism
\[
\begin{tikzpicture}
  \matrix (m) [matrix of math nodes,row sep=1.5em,column sep=5em,minimum width=2em]
  {
     \fP_\signature  	&					& \{-,0,+\}^{n-1}	\\
						& \SchrCambTrees(\signature) &			\\
  };
  \path[->>]
    (m-1-1) edge node [above] {$\surjectionSchrPermPara$} (m-1-3)
                 edge node [below] {$\surjectionSchrPermAsso\quad$} (m-2-2.west)
    (m-2-2.east) edge node [below] {$\quad\surjectionSchrAssoPara$} (m-1-3);
\end{tikzpicture}
\]
\end{proposition}

\fref{fig:latticesQua}\,(left) illustrates this proposition for the signature~${+}{-}{-}$.

\begin{figure}[h]
  \centerline{\includegraphics[width=.9\textwidth]{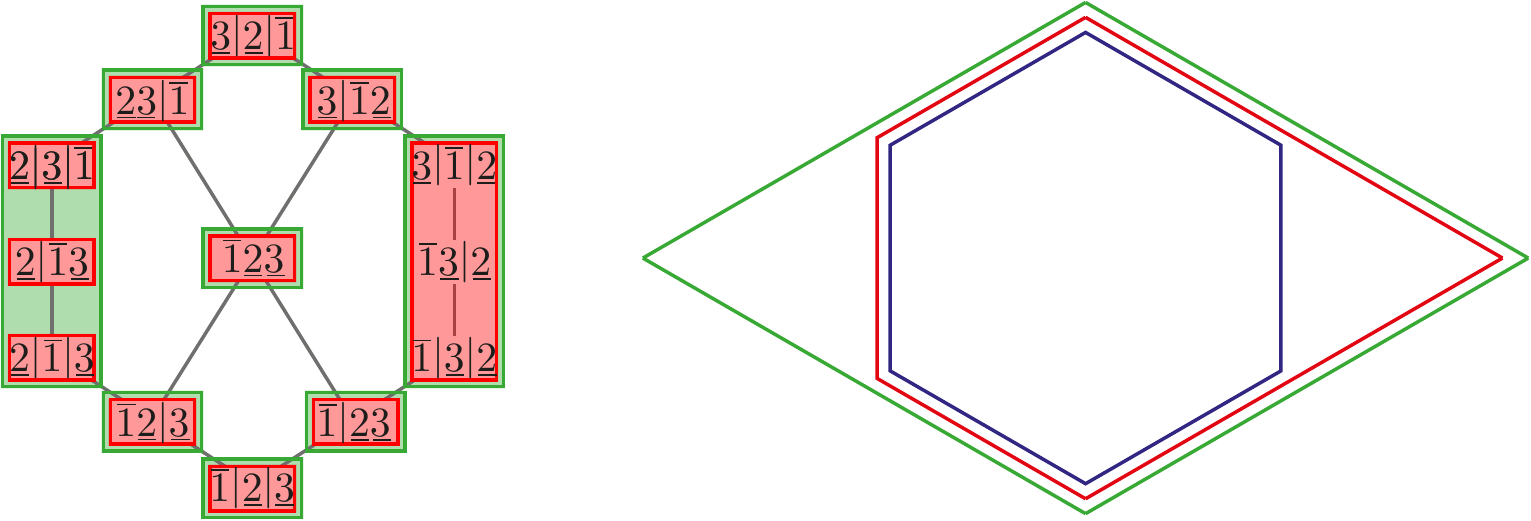}}
  \caption{The fibers of the maps~$\surjectionSchrPermAsso$ (red) and~$\surjectionSchrPermPara$ (green) on the weak orders of~$\fS_{{+}{-}{-}}$ (left), and the geometric realization of these maps (right).}
  \label{fig:latticesQua}
\end{figure}


\subsection{Geometric realizations}
\label{subsec:geometrySchroderTrees}

\enlargethispage{.1cm}
We close this section with the geometric motivation of Schr\"oder-Cambrian trees. More details can be found in~\cite{LangePilaud}. We still denote by~$e_1, \dots, e_n$ the canonical basis of~$\R^n$ and by~$\HH$ the hyperplane of~$\R^n$ orthogonal to~$\sum e_i$. We define the \defn{incidence cone}~$\Cone(\tree)$ and the \defn{braid cone}~$\Cone\polar(\tree)$ of a directed tree~$\tree$ with vertices labeled by subsets of~$[n]$ as
\begin{align*}
\Cone(\tree) & \eqdef \cone\set{e_i-e_j}{\text{for all } i \to j \text{ or } i \sim j \text{ in } \tree} \quad\text{and} \\
\Cone\polar(\tree) & \eqdef \set{\b{x} \in \HH}{x_i \le x_j \text{ for all } i \to j \text{ or } i \sim j \text{ in } \tree}.
\end{align*}
These two cones lie in the space~$\HH$ and are polar to each other. Note that if~$\tree$ has~$k$ nodes, then~$\dim(\Cone\polar(\tree)) = k-1$. For an ordered partition~${\lambda \in \fP_n}$, we denote by~$\Cone(\lambda)$ and~$\Cone\polar(\lambda)$ the incidence and braid cone of the chain~$\lambda_1 \to \dots \to \lambda_n$. Finally, for a vector~$\chi \in \{-, 0, +\}^{n-1}$, we denote by~$\Cone(\tau)$ and~$\Cone\polar(\tau)$ the incidence and braid cone of the oriented path~$1 - \dots - n$, where~$i \to i+1$ if~$\chi_i = +$, $i \leftarrow i+1$ if~$\chi_i = -$, and $i$ and~$i+1$ are merged to the same~node~if~$\chi_i = 0$.

As explained in Section~\ref{subsec:geometricRealizations}, the collections of cones
\[
\set{\Cone\polar(\lambda)}{\lambda \in \fP_n},
\qquad
\set{\Cone\polar(\tree)}{\tree \in \SchrCambTrees_\signature},
\qquad\text{and}\qquad
\set{\Cone\polar(\chi)}{\chi \in \{-,0,+\}^{n-1}}
\]
form complete simplicial fans, which are the normal fans of the classical permutahedron~$\Perm$, of C.~Hohlweg and C.~Lange's associahedron~$\Asso$, and of the parallelepiped~$\Para$ respectively. See Figures~\ref{fig:latticesQua}\,(right) and~\ref{fig:permutahedraAssociahedraCubes} for $2$- and $3$-dimensional examples of these polytopes.

The incidence and braid cones also characterize the maps~$\surjectionSchrPermAsso$, $\surjectionSchrAssoPara$, and~$\surjectionSchrPermPara$ as follows
\begin{gather*}
\tree = \surjectionSchrPermAsso(\lambda) \iff \Cone(\tree) \subseteq \Cone(\lambda) \iff \Cone\polar(\tree) \supseteq \Cone\polar(\lambda), \\
\chi = \surjectionSchrAssoPara(\tree) \iff \Cone(\chi) \subseteq \Cone(\tree) \iff \Cone\polar(\chi) \supseteq \Cone\polar(\tree), \\
\chi = \surjectionSchrPermPara(\lambda) \iff \Cone(\chi) \subseteq \Cone(\lambda) \iff \Cone\polar(\chi) \supseteq \Cone\polar(\lambda).
\end{gather*}

Finally, the weak order, the Schr\"oder-Cambrian lattice and the boolean lattice correspond to the lattice of faces of the permutahedron~$\Perm$, associahedron~$\Asso$ and parallelepiped~$\Para$, oriented in the direction~$(n, \dots, 1) - (1, \dots, n) = \sum_{i \in [n]} (n+1-2i) \, e_i$.


\section{Schr\"oder-Cambrian Hopf Algebra}
\label{sec:SchroderCambrianAlgebra}

In this section, we define the Schr\"oder-Cambrian Hopf algebra~$\SchrCamb$, extending simultaneously the Cambrian Hopf algebra and F.~Chapoton's Hopf algebra on Schr\"oder trees~\cite{Chapoton}. We construct the algebra~$\SchrCamb$ as a subalgebra of a signed version of F.~Chapoton's Hopf algebra on ordered partitions~\cite{Chapoton}. We then also consider the dual algebra~$\SchrCamb^*$ as a quotient of the dual Hopf algebra on signed ordered partitions.


\subsection{Shuffle and convolution products on signed ordered partitions}

We define here a natural analogue of the shifted shuffle and convolution products of Section~\ref{subsec:products} on ordered partitions. Equivalent definitions in the world of surjections can be found in~\cite{Chapoton}. Here, we stick to ordered partitions to match our presentation of the Cambrian algebra in Section~\ref{sec:CambrianAlgebra}.

We first define two restrictions on ordered partitions. Consider an ordered partition~$\mu$ of~$[n]$ into~$k$ parts. As already mentioned earlier, we represent graphically~$\mu$ by the $(k \times n)$-table with a dot in row~$i$ and column~$j$ for each~$j \in \mu_i$. For~${I \subseteq [k]}$, we let~$n_I \eqdef |\set{j \in [n]}{\exists \, i \in I, \; j \in \mu_i}|$ and we denote by~$\mu_{|I}$ the ordered partition of~$[n_I]$ into $|I|$ parts whose table is obtained from the table of~$\mu$ by deleting all rows not in~$I$ and standardizing to get a~$(|I| \times n_I)$-table. Similarly, for~$J \subseteq [n]$, we let~$k_J \eqdef |\set{i \in [k]}{\exists \, j \in J, \; j \in \mu_i}|$ and we denote by~$\mu^{|J}$ the ordered partition of~$[|J|]$ into~$k_J$ parts whose table is obtained from the table of~$\mu$ by deleting all columns not in~$J$ and standardizing to get a~$(k_J \times |J|)$-table. These restrictions are illustrated in \fref{fig:restrictionOrderedPartition}.

\begin{figure}[h]
  \centerline{\includegraphics{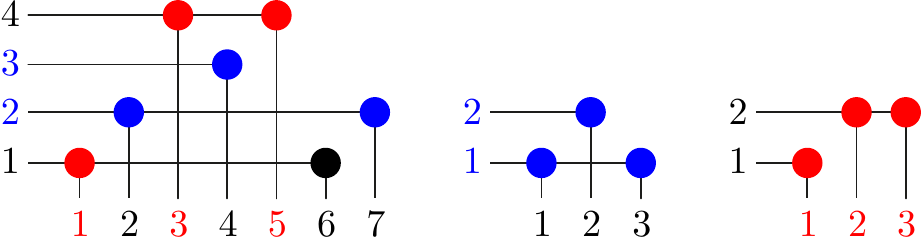}}
  \caption{The tables of the ordered partitions~$\lambda = 16 \sep 27 \sep 4 \sep 35$ (left) and of its restrictions~$\lambda_{|\{2,3\}}$ (middle) and~$\lambda^{|\{1,3,5\}}$ (right).}
  \label{fig:restrictionOrderedPartition}
\end{figure}

We define the \defn{shifted concatenation}~$\lambda \bar\lambda'$, the \defn{shifted shuffle product}~$\lambda \shiftedShuffle \lambda'$, and the \defn{convolution product}~$\lambda \convolution \lambda'$ of two (unsigned) ordered partitions~$\lambda \in \fP_n$ and~$\lambda' \in \fP_{n'}$ as
\begin{align*}
\lambda \bar\lambda' & \eqdef \lambda_1 \bsep \cdots \bsep \lambda_k \bsep n + \lambda'_1 \bsep \cdots \bsep n + \lambda'_{k'}, \qquad \text{where } n + \lambda'_i \eqdef \set{n+j}{j \in \lambda'_i}\\
\lambda \shiftedShuffle \lambda' & \eqdef \set{\mu \in \fP_{n+n'}}{\mu^{|\{1, \dots, n\}} = \lambda \text{ and } \mu^{|\{n+1, \dots, n+n'\}} = \lambda'}, \\
\text{and}\qquad \lambda \convolution \lambda' & \eqdef \set{\mu \in \fP_{n+n'}}{\mu_{|\{1, \dots, n\}} = \lambda \text{ and } \mu_{|\{n+1, \dots, n+n'\}} = \lambda'}
\end{align*}

For example,
\begin{align*}
{\red 1} \sep {\red 2} \shiftedShuffle {\darkblue 2} \sep {\darkblue 31} = \{ &
	{\red 1} \sep {\red 2} \sep {\darkblue 4} \sep {\darkblue 53}, \;
	{\red 1} \sep {\red 2}{\darkblue 4} \sep {\darkblue 53}, \;
	{\red 1} \sep {\darkblue 4} \sep {\red 2} \sep {\darkblue 53}, \;
	{\red 1} \sep {\darkblue 4} \sep {\red 2}{\darkblue 53}, \;
	{\red 1} \sep {\darkblue 4} \sep {\darkblue 53} \sep {\red 2}, \;
	{\red 1}{\darkblue 4} \sep {\red 2} \sep {\darkblue 53}, \;
	{\red 1}{\darkblue 4} \sep {\red 2}{\darkblue 53}, \\[-.1cm]
&	{\red 1}{\darkblue 4} \sep {\darkblue 53} \sep {\red 2}, \;
	{\darkblue 4} \sep {\red 1} \sep {\red 2} \sep {\darkblue 53}, \;
	{\darkblue 4} \sep {\red 1} \sep {\red 2}{\darkblue 53}, \;
	{\darkblue 4} \sep {\red 1} \sep {\darkblue 53} \sep {\red 2}, \;
	{\darkblue 4} \sep {\red 1}{\darkblue 53} \sep {\red 2}, \;
	{\darkblue 4} \sep {\darkblue 53} \sep {\red 1} \sep {\red 2} \}, \\[.2cm]
{\red 1} \sep {\red 2} \convolution {\darkblue 2} \sep {\darkblue 31} = \{ &
	{\red 1} \sep {\red 2} \sep {\darkblue 4} \sep {\darkblue 53}, \;
	{\red 1} \sep {\red 3} \sep {\darkblue 4} \sep {\darkblue 52}, \;
	{\red 1} \sep {\red 4} \sep {\darkblue 3} \sep {\darkblue 52}, \;
	{\red 1} \sep {\red 5} \sep {\darkblue 3} \sep {\darkblue 42}, \;
	{\red 2} \sep {\red 3} \sep {\darkblue 4} \sep {\darkblue 51}, \\[-.1cm]
&	{\red 2} \sep {\red 4} \sep {\darkblue 3} \sep {\darkblue 51}, \;
	{\red 2} \sep {\red 5} \sep {\darkblue 3} \sep {\darkblue 41}, \;
	{\red 3} \sep {\red 4} \sep {\darkblue 2} \sep {\darkblue 51}, \;
	{\red 3} \sep {\red 5} \sep {\darkblue 2} \sep {\darkblue 41}, \;
	{\red 4} \sep {\red 5} \sep {\darkblue 2} \sep {\darkblue 31} \}.
\end{align*}

Graphically, the table of the shifted concatenation~$\lambda \bar\lambda'$ contains the table of~$\lambda$ as the bottom left block and the table of~$\lambda'$ as the top right block. The tables in the shifted shuffle product~$\lambda \shiftedShuffle \lambda'$ (resp.~in the convolution product~$\lambda \convolution \lambda'$) are obtained by shuffling the rows (resp.~columns) of the table of~$\lambda \bar\lambda'$. See \fref{fig:shuffleConvolutionOrderedPartitions}.

\begin{figure}[h]
  \centerline{\includegraphics{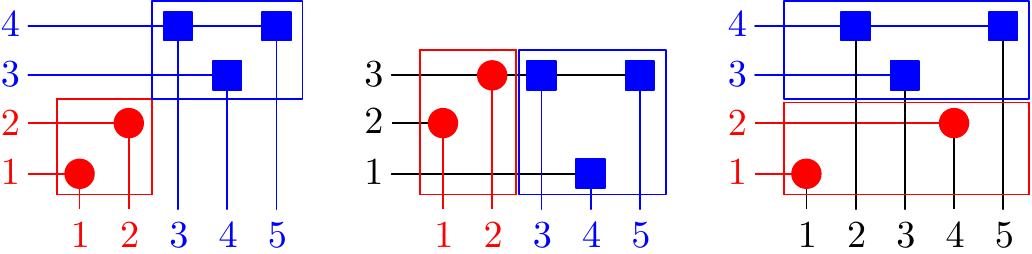}}
  \caption{The table of the shifted concatenation~$\lambda\bar\lambda'$ (left) has two blocks containing the tables of the ordered partitions~$\lambda = 1 \sep 2$ and~$\lambda' = 2 \sep 31$. Elements of the shifted shuffle product~$\lambda \shiftedShuffle \lambda'$ (middle) and of the convolution product~${\lambda \convolution \lambda'}$~(right) are obtained by shuffling respectively the rows and columns of the table of~$\lambda\bar\lambda'$.}
  \label{fig:shuffleConvolutionOrderedPartitions}
\end{figure}

\begin{remark}
\begin{enumerate}[(i)]
\item Note that the shifted shuffle and convolution products are compatible with the filtration~$(\fP_n^{\ge p})_{p \in [n]}$:
\[
\fP_n^{\ge p} \shiftedShuffle \fP_{n'}^{\ge p'} \subseteq \fP_{n+n'}^{\ge p+p'}
\qquad\text{and}\qquad
\fP_n^{\ge p} \convolution \fP_{n'}^{\ge p'} \subseteq \fP_{n+n'}^{\ge p+p'}.
\]
\item By projection on the quotient~$\fP / \fP^{\ge 1} \simeq \fS$, the (signed) shifted shuffle and convolution products coincide with the description of Section~\ref{subsec:products}.
\item The shifted product of ordered partitions preserves intervals of the weak order. Namely,
\[
[\lambda, \mu] \shiftedShuffle [\lambda', \mu'] = [\lambda \bar\lambda', \bar\mu' \mu].
\]
\end{enumerate}
\end{remark}

These definitions extend to signed ordered partitions: signs travel with their values in the signed shifted shuffle product, and stay at their positions in the signed convolution product.



\subsection{Subalgebra of~$\OrdPart_\pm$}
\label{subsec:SchroderCambrianAlgebra}

We denote by~$\OrdPart_\pm$ the Hopf algebra with basis~$(\F_\lambda)_{\lambda \in \fP_{\pm}}$ and whose product and coproduct are defined by
\[
\F_\lambda \product \F_{\lambda'} = \sum_{\mu \in \lambda \shiftedShuffle \lambda'} \F_\mu
\qquad\text{and}\qquad
\coproduct \F_\mu = \sum_{\mu \in \lambda \convolution \lambda'} \F_\lambda \otimes \F_{\lambda'}.
\]
Note that the Hopf algebra~$\FQSym_\pm$ is isomorphic to the quotient~$\OrdPart_\pm / \OrdPart_\pm^{\ge 1}$. Note also that the unsigned version of~$\OrdPart_\pm$ is the dual of the algebra~$\WQSym$ of word quasi-symmetric functions (also denoted~$\NCQSym$ for non-commutative quasi-symmetric functions), see~\cite{BergeronZabrocki, NovelliThibon-trigebres}.

\begin{remark}
The proof that~$\OrdPart_\pm$ is indeed a Hopf algebra is left to the reader: it consists in translating F.~Chapoton's proof~\cite{Chapoton} from surjections to signed ordered partitions. In fact, F.~Chapoton's Hopf algebras on faces of the permutahedra, associahedra, and cubes could be decorated by an arbitrary group, similar to the constuctions in~\cite{NovelliThibon-coloredHopfAlgebras, BaumannHohlweg, BergeronHohlweg}. Once again the main point here is that the Schr\"oder-congruence relations depend on the decoration.
\end{remark}

We denote by~$\SchrCamb$ the vector subspace of~$\OrdPart_\pm$ generated by the elements
\[
\PCamb_{\tree} \eqdef \sum_{\substack{\lambda \in \fP_\pm \\ \surjectionSchrPermAsso(\lambda) = \tree}} \F_\lambda
\]
for all Schr\"oder-Cambrian trees~$\tree$. For example, for the Schr\"oder-Cambrian tree of \fref{fig:leveledSchroderCambrianTree}\,(left), we have
\[
\PCamb_{\!\!\!\!\TexSchroder} = \F_{\down{12}|\down{5}\up{7}|\up{3}\down{4}|\up{6}} + \F_{\down{125}\up{7}|\up{3}\down{4}|\up{6}} + \F_{\down{5}\up{7}|\down{12}|\up{3}\down{4}|\up{6}}.
\]
Note that the Hopf algebra~$\Camb$ is isomorphic to the quotient~$\SchrCamb_\pm / \SchrCamb_\pm^{\ge 1}$.

\begin{theorem}
\label{thm:SchrCambSubalgebra}
$\SchrCamb$ is a Hopf subalgebra of~$\OrdPart_\pm$.
\end{theorem}

\begin{proof}
Similar to the proof of Theorem~\ref{thm:cambSubalgebra}.
\end{proof}

As for the Cambrian algebra, the product and coproduct of~$\PCamb$-basis elements of the Schr\"oder-Cambrian algebra~$\SchrCamb$ can be described directly in terms of combinatorial operations on Schr\"oder-Cambrian trees.

\para{Product}
The product in the Schr\"oder Cambrian algebra~$\SchrCamb$ can again be described in terms of intervals in the Sch\"oder-Cambrian lattice. Given two Schr\"oder-Cambrian trees~$\tree, \tree'$, we denote by~$\raisebox{-6pt}{$\tree$}\nearrow \raisebox{4pt}{$\bar \tree'$}$ the Schr\"oder $\signature(\tree)\signature(\tree')$-Cambrian tree obtained by grafting the rightmost outgoing leaf of~$\tree$ on the leftmost incoming leaf of~$\tree$ and shifting all labels of~$\tree'$. We define similarly~$\raisebox{4pt}{$\tree$} \nwarrow \raisebox{-6pt}{$\bar \tree'$}$.

\begin{proposition}
For any Schr\"oder-Cambrian trees~$\tree, \tree'$, the product~$\PCamb_{\tree} \product \PCamb_{\tree'}$ is given by
\[
\PCamb_{\tree} \product \PCamb_{\tree'}  = \sum_{\tree[S]} \PCamb_{\tree[S]},
\]
where~$\tree[S]$ runs over the interval between~$\raisebox{-6pt}{$\tree$}\nearrow \raisebox{4pt}{$\bar \tree'$}$ and~$\raisebox{4pt}{$\tree$} \nwarrow \raisebox{-6pt}{$\bar \tree'$}$ in the Schr\"oder $\signature(\tree)\signature(\tree')$-Cambrian lattice.
\end{proposition}

\begin{proof}
Similar to that of Proposition~\ref{prop:product}.
\end{proof}

For example, we can compute the product
\[
\begin{array}{@{}c@{${} = {}$}c@{+}c@{+}c@{}@{}c}
\PCamb_{\!\!\includegraphics{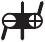}} \product \PCamb_{\includegraphics{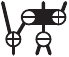}} &
\multicolumn{3}{l}{\F_{\down{1}\up{2}} \product \big( \F_{\up{1} \sep \down{3} \sep \down{2}\up{4}} + \F_{\up{1}\down{3} \sep \down{2}\up{4}} + \F_{\down{3} \sep \up{1} \sep \down{2}\up{4}} \big)}
\\[-.3cm]
& \begin{pmatrix} \quad \F_{\down{1}\up{2} \sep \up{3} \sep \down{5} \sep \down{4}\up{6}} + \F_{\down{1}\up{2} \sep \up{3}\down{5} \sep \down{4}\up{6}} \\
+ \; \F_{\down{1}\up{2} \sep \down{5} \sep \up{3} \sep \down{4}\up{6}} + \F_{\down{1}\up{2}\down{5} \sep \up{3} \sep \down{4}\up{6}} \\
+ \; \F_{\down{5} \sep \down{1}\up{2} \sep \up{3} \sep \down{4}\up{6}} \end{pmatrix}
& \begin{pmatrix} \quad \F_{\down{1}\up{2}\up{3} \sep \down{5} \sep \down{4}\up{6}} \\
+ \; \F_{\down{1}\up{2}\up{3} \down{5} \sep \down{4}\up{6}} \\
+ \; \F_{\down{5} \sep \down{1}\up{2}\up{3} \sep \down{4}\up{6}} \end{pmatrix}
& \begin{pmatrix} \quad \F_{\up{3} \sep \down{1}\up{2} \sep \down{5} \sep \down{4}\up{6}} + \F_{\up{3} \sep \down{1}\up{2}\down{5} \sep \down{4}\up{6}} + \F_{\up{3} \sep \down{5} \sep \down{1}\up{2} \sep \down{4}\up{6}} \\
+ \; \F_{\up{3}\down{5} \sep \down{1}\up{2} \sep \down{4}\up{6}} + \F_{\down{5} \sep \up{3} \sep \down{1}\up{2} \sep \down{4}\up{6}} + \F_{\up{3} \sep \down{5} \sep \down{1}\up{2}\down{4}\up{6}} \\
+ \; \F_{\up{3}\down{5} \sep \down{1}\up{2}\down{4}\up{6}} + \F_{\down{5} \sep \up{3} \sep \down{1}\up{2}\down{4}\up{6}} + \F_{\up{3} \sep \down{5} \sep \down{4}\up{6} \sep \down{1}\up{2}} \\
+ \; \F_{\up{3}\down{5} \sep \down{4}\up{6} \sep \down{1}\up{2}} + \F_{\down{5} \sep \up{3} \sep \down{4}\up{6} \sep \down{1}\up{2}}
\end{pmatrix}
\\[.8cm]
& \PCamb_{\!\!\includegraphics{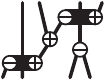}} & \PCamb_{\!\!\includegraphics{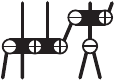}} & \PCamb_{\!\!\includegraphics{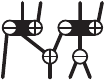}} & .
\end{array}
\]

\para{Coproduct}
The coproduct in the Schr\"oder Cambrian algebra~$\SchrCamb$ can again be described in terms of cuts. A \defn{cut} of a Schr\"oder-Cambrian tree~$\tree[S]$ is a set~$\gamma$ of edges such that any geodesic vertical path in~$\tree[S]$ from a down leaf to an up leaf contains precisely one edge of~$\gamma$. We denote again by~$A(\tree[S],\gamma)$ and~$B(\tree[S],\gamma)$ the two Schr\"oder-Cambrian forests above and below~$\gamma$ in~$\tree[S]$. 

\begin{proposition}
For any Schr\"oder-Cambrian tree~$\tree[S]$, the coproduct~$\coproduct \PCamb_{\tree[S]}$ is given by
\[
\coproduct \PCamb_{\tree[S]} = \sum_{\gamma} \bigg( \prod_{\tree \in B(\tree[S],\gamma)} \PCamb_{\tree} \bigg) \otimes \bigg( \prod_{\tree' \in A(\tree[S], \gamma)} \PCamb_{\tree'} \bigg),
\]
where~$\gamma$ runs over all cuts of~$\tree[S]$.
\end{proposition}

\begin{proof}
Similar to that of Proposition~\ref{prop:coproduct}.
\end{proof}

For example, we can compute the coproduct
\[
\begin{array}{@{}c@{${} = {}$}c@{$\;+\;$}c@{$\;+\;$}c@{$\;+\;$}c@{$\;+\;$}c@{}}
\coproduct \PCamb_{\includegraphics{exmProductSchroderB}} & \multicolumn{5}{l}{\coproduct \big( \F_{\up{1} \sep \down{3} \sep \down{2}\up{4}} + \F_{\up{1}\down{3} \sep \down{2}\up{4}} + \F_{\down{3} \sep \up{1} \sep \down{2}\up{4}} \big)}
\\
& 1 \otimes \begin{pmatrix} \quad \F_{\up{1} \sep \down{3} \sep \down{2}\up{4}} \\ + \; \F_{\up{1}\down{3} \sep \down{2}\up{4}} \\ + \; \F_{\down{3} \sep \up{1} \sep \down{2}\up{4}} \end{pmatrix} 
& \F_{\up{1}} \otimes \F_{\down{2} \sep \down{1}\up{3}}
& \F_{\down{1}} \otimes \F_{\up{1} \sep \down{2}\up{3}}
& \begin{pmatrix} \quad \F_{\up{1} \sep \down{2}} \\ + \; \F_{\up{1}\down{2}} \\ + \; \F_{\down{2} \sep \up{1}} \end{pmatrix} \otimes \F_{\down{1}\up{2}}
& \begin{pmatrix} \quad \F_{\up{1} \sep \down{3} \sep \down{2}\up{4}} \\ + \; \F_{\up{1}\down{3} \sep \down{2}\up{4}} \\ + \; \F_{\down{3} \sep \up{1} \sep \down{2}\up{4}} \end{pmatrix} \otimes 1
\\[.8cm]
& 1 \otimes \PCamb_{\!\includegraphics{exmProductSchroderB}}
& \PCamb_{\includegraphics{exmTreeY}} \otimes \PCamb_{\!\!\includegraphics{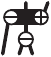}}
& \PCamb_{\includegraphics{exmTreeA}} \otimes \PCamb_{\!\includegraphics{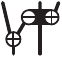}}
& \big( \PCamb_{\!\includegraphics{exmTreeY}} \product \PCamb_{\includegraphics{exmTreeA}} \big)\otimes \PCamb_{\includegraphics{exmProductSchroderA}}
& \PCamb_{\!\includegraphics{exmProductSchroderB}} \otimes 1.
\end{array}
\]

\para{Matriochka algebras}
To conclude, we connect the Schr\"oder-Cambrian algebra to F.~Chapoton's algebra on faces of the cubes defined in~\cite{Chapoton}. We call trilean Hopf algebra the Hopf subalgebra~$\Tril$ of~$\OrdPart_\pm$ generated by the elements
\[
\XRec_\chi \eqdef \sum_{\substack{\lambda \in \fP_\pm \\ \surjectionSchrPermPara(\lambda) = \chi}} \F_\lambda
\]
for all~$\chi \in \{-,0,+\}^{n-1}$. The commutative diagram of Proposition~\ref{prop:commutativeDiagramSchroder} ensures that
\[
\XRec_\chi = \sum_{\substack{\tree \in \SchrCambTrees \\ \surjectionSchrAssoPara(\tree) = \chi}} \PCamb_{\tree},
\]
and thus that~$\Tril$ is a subalgebra of~$\SchrCamb$. In other words, the Schr\"oder-Cambrian algebra is sandwiched between the signed ordered partitions algebra and the trilean algebra~$\Tril \subset \SchrCamb \subset \OrdPart_\pm$.


\subsection{Quotient algebra of~$\OrdPart_\pm^*$}
\label{subsec:dualSchroderCambrianAlgebra}

We switch to the dual Hopf algebra~$\OrdPart_\pm^*$ with basis $(\G_\lambda)_{\lambda \in \fP_\pm}$ and whose product and coproduct are defined by
\[
\G_\lambda \product \G_{\lambda'} = \sum_{\mu \in \lambda \convolution \lambda'} \G_\mu
\qquad\text{and}\qquad
\coproduct \G_\mu = \sum_{\mu \in \lambda \shiftedShuffle \lambda'} \G_\lambda \otimes \G_{\lambda'}.
\]
Note that the unsigned version of~$\OrdPart_\pm^*$ is the algebra~$\WQSym$ of word quasi-symmetric functions (also denoted~$\NCQSym$ for non-commutative quasi-symmetric functions), see~\cite{BergeronZabrocki, NovelliThibon-trigebres}.The following statement is automatic from Theorem~\ref{thm:SchrCambSubalgebra}.

\begin{theorem}
The graded dual~$\SchrCamb^*$ of the Schr\"oder-Cambrian algebra is isomorphic to the image of~$\OrdPart_\pm^*$ under the canonical projection
\[
\pi : \C\langle A \rangle \longrightarrow \C\langle A \rangle / \equiv,
\]
where~$\equiv$ denotes the Schr\"oder-Cambrian congruence. The dual basis~$\QCamb_{\tree}$ of~$\PCamb_{\tree}$ is expressed as~$\QCamb_{\tree} = \pi(\G_\lambda)$, where~$\lambda$ is any ordered partition such that~$\surjectionSchrPermAsso(\lambda) = \tree$.
\end{theorem}

Similarly as in the previous section, we can describe combinatorially the product and coproduct of $\QCamb$-basis elements of~$\SchrCamb^*$ in terms of operations on Schr\"oder-Cambrian trees.

\para{Product}
We define \defn{gaps} and \defn{laminations} of Schr\"oder-Cambrian trees exactly as we did for Cambrian trees in Section~\ref{subsec:quotientAlgebra}. Note that laminations may or may not split the nodes of a Schr\"oder-Cambrian tree, see \fref{fig:exampleProductDualSchroder}\,(c) for examples. Given two Schr\"oder-Cambrian trees~$\tree$ and~$\tree'$ on~$[n]$ and~$[n']$ respectively, and a shuffle~$s$ of their signature defining multisets $\Gamma$ of gaps of~$[n]$ and $\Gamma'$ of gaps of~$[n']$, we still denote by~$\tree \,{}_s\!\backslash \tree'$ the Schr\"oder-Cambrian tree obtained by connecting the up leaves of~$\lambda(\tree,\Gamma)$ to the down leaves of the forest defined by the lamination~$\lambda(\tree',\Gamma')$. See \fref{fig:exampleProductDualSchroder}.

\begin{figure}[h]
  \centerline{\includegraphics{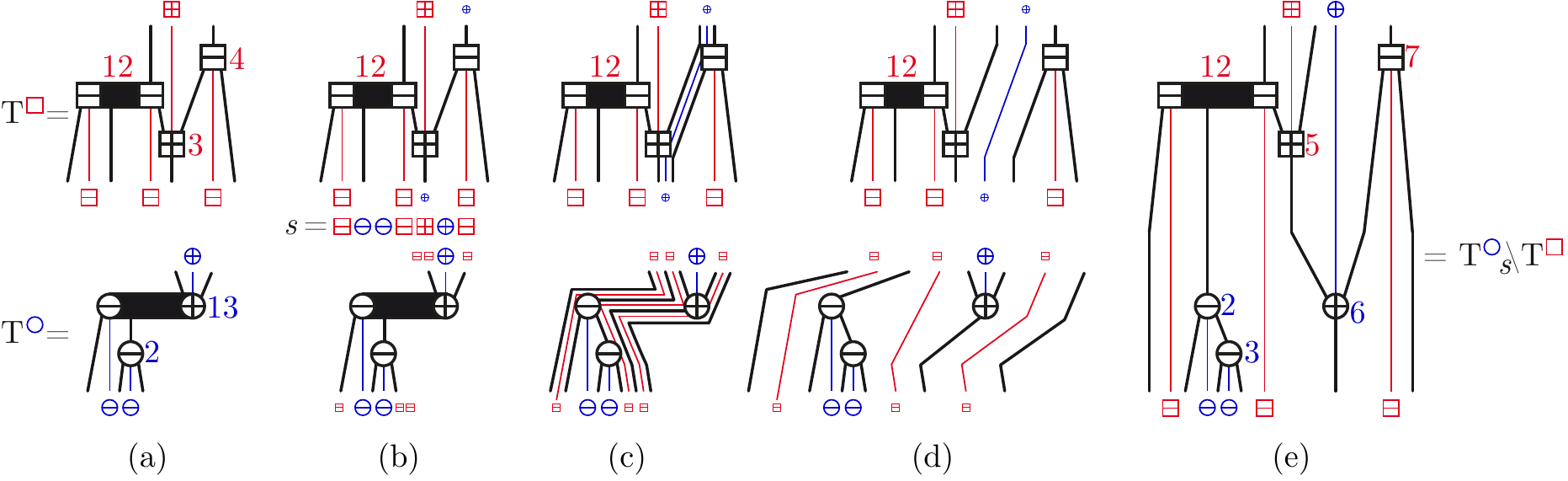}}
  \caption[Combinatorial interpretation of the product in the~$\QCamb$-basis of~$\SchrCamb^*$.]{(a) Two Schr\"oder-Cambrian trees~$\tree^\blueCirc$ and~$\tree^\redSquare$. (b) Given the shuffle $s = \redMinus\blueMinus\blueMinus\redMinus\redPlus\bluePlus\redMinus$, the positions of the~$\redMinus$ are reported in~$\tree^\blueCirc$ and the positions of the~$\bluePlus$ are reported in~$\tree^\redSquare$. (c) The corresponding laminations. (d) The trees are split according to the laminations. (e) The resulting Schr\"oder-Cambrian tree~$\tree^\blueCirc {}_s\!\backslash \tree^\redSquare$.}
  \label{fig:exampleProductDualSchroder}
\end{figure}

\begin{proposition}
For any Schr\"oder-Cambrian trees~$\tree, \tree'$, the product~$\QCamb_{\tree} \product \QCamb_{\tree'}$ is given by
\[
\QCamb_{\tree} \product \QCamb_{\tree'} = \sum_s \QCamb_{\tree \,{}_s\!\backslash \tree'},
\]
where~$s$ runs over all shuffles of the signatures of~$\tree$ and~$\tree'$.
\end{proposition}

\begin{proof}
Similar to that of Proposition~\ref{prop:productDual}.
\end{proof}

For example, we can compute the product
\[
\hspace*{-1cm}\begin{array}{@{}ccc@{$\;+\;$}c@{$\;+\;$}c@{$\;+\;$}c@{$\;+\;$}c@{$\;+\;$}c@{$\;+\;$}c@{$\;+\;$}c@{}}
\QCamb_{\includegraphics{exmProductSchroderA}} \product \QCamb_{\includegraphics{exmProductSchroderB}}
& = & \multicolumn{8}{l}{\G_{\down{1}\up{2}} \product \G_{\up{1}\down{3} \sep \down{2}\up{4}}}
\\[-.2cm]
& = & \quad \G_{\down{1}\up{2} \sep \up{3}\down{5} \sep \down{4}\up{6}} & \G_{\down{1}\up{3} \sep \up{2}\down{5} \sep \down{4}\up{6}} & \G_{\down{1}\up{4} \sep \up{2}\down{5} \sep \down{3}\up{6}} & \G_{\down{1}\up{5} \sep \up{2}\down{4} \sep \down{3}\up{6}} & \G_{\down{1}\up{6} \sep \up{2}\down{4} \sep \down{3}\up{5}} & \G_{\down{2}\up{3} \sep \up{1}\down{5} \sep \down{4}\up{6}} & \G_{\down{2}\up{4} \sep \up{1}\down{5} \sep \down{3}\up{6}} & \G_{\down{2}\up{5} \sep \up{1}\down{4} \sep \down{3}\up{6}}
\\
& & + \; \G_{\down{2}\up{6} \sep \up{1}\down{4} \sep \down{3}\up{5}} & \G_{\down{3}\up{4} \sep \up{1}\down{5} \sep \down{2}\up{6}} & \G_{\down{3}\up{5} \sep \up{1}\down{4} \sep \down{2}\up{6}} & \G_{\down{3}\up{6} \sep \up{1}\down{4} \sep \down{2}\up{5}} & \G_{\down{4}\up{5} \sep \up{1}\down{3} \sep \down{2}\up{6}} & \G_{\down{4}\up{6} \sep \up{1}\down{3} \sep \down{2}\up{5}} & \multicolumn{2}{@{}l}{\G_{\down{5}\up{6} \sep \up{1}\down{3} \sep \down{2}\up{4}}}
\\[.2cm]
& = & \quad \QCamb_{\!\includegraphics{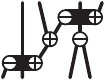}} & \QCamb_{\!\includegraphics{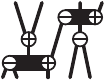}} & \QCamb_{\!\includegraphics{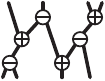}} & \QCamb_{\!\includegraphics{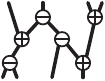}} & \QCamb_{\!\includegraphics{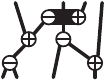}} & \QCamb_{\!\includegraphics{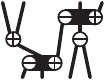}} & \QCamb_{\!\includegraphics{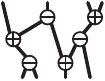}} & \QCamb_{\!\includegraphics{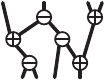}} \\
& & + \; \QCamb_{\!\includegraphics{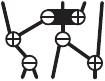}} & \QCamb_{\!\includegraphics{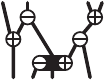}} & \QCamb_{\!\includegraphics{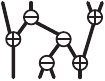}} & \QCamb_{\!\includegraphics{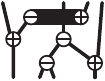}} & \QCamb_{\!\includegraphics{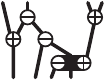}} & \QCamb_{\!\includegraphics{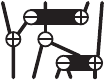}} & \multicolumn{2}{@{}l}{\QCamb_{\!\includegraphics{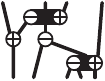}} \; .} 
\end{array}
\]

\para{Coproduct}
For a gap~$\gamma$ of a Schr\"oder-Cambrian tree~$\tree[S]$, we still denote by~$L(\tree[S],\gamma)$ and~$R(\tree[S],\gamma)$ the left and right Schr\"oder-Cambrian subtrees of~$\tree[S]$ when split along the path~$\lambda(\tree[S],\gamma)$.

\begin{proposition}
For any Cambrian tree~$\tree[S]$, the coproduct~$\coproduct\QCamb_{\tree[S]}$ is given by
\[
\coproduct\QCamb_{\tree[S]} = \sum_{\gamma} \QCamb_{L(\tree[S],\gamma)} \otimes \QCamb_{R(\tree[S],\gamma)},
\]
where~$\gamma$ runs over all gaps between vertices of~$\tree[S]$.
\end{proposition}

\begin{proof}
Similar to that of Proposition~\ref{prop:coproductDual}.
\end{proof}

For example, we can compute the coproduct
\[
\hspace*{-1cm}\begin{array}{@{}c@{${} = {}$}c@{${} + {}$}c@{${} + {}$}c@{${} + {}$}c@{${} + {}$}c@{}}
\coproduct \QCamb_{\includegraphics{exmProductSchroderB}} &
 \multicolumn{5}{l}{\coproduct \G_{\up{1}\down{3} \sep \down{2}\up{4}}}
 \\[-.2cm]
& 1 \otimes \G_{\up{1}\down{3} \sep \down{2}\up{4}}
& \G_{\up{1}} \otimes \G_{\down{2} \sep \down{1}\up{3}}
& \G_{\up{1} \sep \down{2}} \otimes \G_{\down{1} \sep \up{2}}
& \G_{\up{1}\down{3} \sep \down{2}} \otimes \G_{\up{1}}
& \G_{\up{1}\down{3} \sep \down{2}\up{4}} \otimes 1
\\[.2cm]
& 1 \otimes \QCamb_{\includegraphics{exmProductSchroderB}}
& \QCamb_{\includegraphics{exmTreeY}} \otimes \QCamb_{\includegraphics{exmTreeT}}
& \QCamb_{\includegraphics{exmTreeYAd}} \otimes \QCamb_{\includegraphics{exmTreeAYd}}
& \QCamb_{\includegraphics{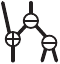}} \otimes \QCamb_{\includegraphics{exmTreeY}}
& \QCamb_{\includegraphics{exmProductSchroderB}} \otimes 1.
\end{array}
\]


\section{Schr\"oder-Cambrian tuples}

\enlargethispage{-.8cm}
As a conclusion, we just want to mention that it would also be possible to extend simultaneously the Cambrian tuple algebra and the Schr\"oder-Cambrian algebra. The objects are \defn{Schr\"oder-Cambrian $\ell$-tuples}, \ie $\ell$-tuples of Schr\"oder-Cambrian trees whose union is acyclic. 

The first step is then to describe the combinatorics of these tuples:
\begin{itemize}
\item applying Schr\"oder-Cambrian correspondences in parallel yields a correspondence~$\SchrCambCorresp_\ell$ between $\ell$-signed ordered partitions and leveled Schr\"oder-Cambrian $\ell$-tuples, and thus defines a surjection~$\surjectionSchrPermAsso_\ell$ from $\ell$-signed ordered partitions to Schr\"oder-Cambrian $\ell$-tuples;
\item the fibers of~$\surjectionSchrPermAsso_\ell$ are intersections of Schr\"oder-Cambrian congruences, and thus define a lattice congruence of the weak order on ordered partitions;
\item the $\ell$-tuples of Schr\"oder-Cambrian trees correspond to all faces of a Minkowski sum of $\ell$ associahedra of~\cite{HohlwegLange}.
\end{itemize}
An interesting combinatorial problem is to count the number of $\ell$-tuples of Schr\"oder-Cambrian trees, in particular the number of Baxter-Schr\"oder-Cambrian trees.

The second step is to define as usual the Schr\"oder-Cambrian $\ell$-tuple Hopf algebra~$\SchrCamb_\ell$ as a subalgebra of the Hopf algebra~$\OrdPart_{\pm^\ell}$ of $\ell$-signed ordered partitions, and its dual~$\SchrCamb_\ell^*$ as a quotient of~$\OrdPart_{\pm^\ell}^*$. The product and coproduct both~$\SchrCamb_\ell$ and~$\SchrCamb_\ell^*$ can then directly be described by the combinatorial operations on Schr\"oder-Cambrian $\ell$-tuples, similar to the operations described in Sections~\ref{subsec:CambrianTupleAlgebra}, \ref{subsec:dualCambrianTupleAlgebra}, \ref{subsec:SchroderCambrianAlgebra} and~\ref{subsec:dualSchroderCambrianAlgebra}.


\section*{Acknowledgements}

We are grateful to the participants of the \emph{Groupe de travail de Combinatoire Alg\'ebrique de l'Universit\'e de Marne-la-Vall\'ee} for helpful discussions and comments on preliminary stages of this work. In particular, we are indebted to J.-Y.~Thibon for showing us relevant research directions and encouraging our ideas, and to J.-C.~Novelli for uncountable technical suggestions and comments (leading in particular to the results presented in Proposition~\ref{prop:GeneratingTree}, in Section~\ref{subsec:BaxterCambrianNumbers}, and in Section~\ref{sec:SchroderCambrianAlgebra}). We also thank V.~Pons, whose suggestions lead to the current presentation of the Cambrian correspondence~$\CambCorresp$ in Section~\ref{subsec:CambrianCorrespondence}, adapted from~\cite{LangePilaud}. The second author thanks C.~Hohlweg for introducing him to this algebraic combinatorics crowd. Finally, we are grateful to M.~Bousquet-Melou, C.~Hohlweg and N.~Reading for pointing out to us various relevant references.


\bibliographystyle{alpha}
\bibliography{CambrianAlgebra}
\label{sec:biblio}

\end{document}